\documentclass[a4paper,8pt]{article}
\usepackage[francais]{babel}
\usepackage{geometry}
\usepackage[T1]{fontenc}
\usepackage[latin1]{inputenc}
\usepackage{amsmath,amssymb,mathrsfs,amsthm}
\usepackage{soul}
\usepackage{epsfig}
\usepackage{hyperref}
\usepackage{graphicx}
\usepackage{xypic}
\usepackage{datetime}
\usepackage{fancyhdr}
\fancypagestyle{plain}{
\lhead{}
\rhead{}
}
\newtheorem{theorem}{Th\'eor\`eme}[section]
\newtheorem{proposition}[theorem]{Proposition}
\newtheorem{lemma}[theorem]{Lemme}

\newtheorem{Corollaire}[theorem]{Corollaire}
\newtheorem{definition}[theorem]{D\'efinition}%\newenvironment{proof}[1][D\'emonstration]{\begin{trivlist}
\newtheorem{example}[theorem]{Exemple}
\newtheorem{remarque}[theorem]{Remarque}

\newcommand{\vc}{\|\cdot\|}
\newcommand{\CC}{\mathbb{C}}
\newcommand{\mc}{\mathcal{O}(1)}

\newcommand{\s}{\mathbb{S}^1}
\newcommand{\lra}{\longrightarrow}
\newcommand{\al}{\alpha}
\newcommand{\la}{\lambda}

\newcommand{\R}{\mathbb{R}}
\newcommand{\cl}{\mathcal{C}^\infty}
\newcommand{\p}{\mathbb{P}}
\newcommand{\eps}{\varepsilon}

\newcommand{\vf}{\varphi}
\newcommand{\si}{\sigma}
\newcommand{\h}{\mathcal{H}}

\newcommand{\N}{\mathbb{N}}
\newcommand{\z}{\overline{z}}
\newcommand{\pt}{\partial}
\newcommand{\dif}{\frac{\pt}{\pt z}}
\title{ Sur la th\'eorie spectrale des  m\'etriques int\'egrables sur une surface  de Riemann compacte}
\date{}
\author{Mounir Hajli}
\begin{document}

\maketitle

\begin{abstract}
Dans ce texte, on continue notre \'etude  de la th\'eorie spectrale associ\'ee aux m\'etriques int\'egrables 
sur une surface de Riemann compacte, qu'on a commenc\'e  dans \cite{Mounir8}.  Nous introduisons une classe 
de m\'etriques continues  sur les fibr\'es en droites
sur une surface de Riemann compacte qu'on appellera  m\'etriques $1$-int\'egrables. Nous g\'en\'eralisons   la
th\'eorie spectrale des op\'erateurs Laplaciens associ\'es aux m\'etriques de classe $\cl$ \`a  l'ensemble  
des fibr\'es en droites munis de ces m\'etriques.\\

Comme application, on retrouve l'\'egalit\'e suivante:
\[
\zeta'_{\Delta_{\overline{\mathcal{O}(m)}_\infty}}(0)=T_g\bigl((\p^1,\omega_\infty); \overline{\mathcal{O}(m)}_\infty \bigr),
\]
obtenue   par des calculs directs dans \cite{Mounir2}.
 \end{abstract}

\section{Introduction}
\subsection{Principaux r\'esultats}
En G\'eom\'etrie d'Arakelov, developp\'ee par Gillet-Soule et Bismut, les m\'etriques hermitiennes sont suppos\'ees de classe $\cl$. Cela
est n\'ecessaire, par exemple, afin de d\'efinir les parties de nature spectrale entrant dans la formulation du th\'eor\`eme de Riemann
Roch Arithm\'etique, voir \cite{ARR}. Or, il s'av\`ere que ces suppositions sont tr\`es restrictives. Un exemple de bonnes m\'etriques, sont
les m\'etriques canoniques sur les  vari\'et\'es toriques qui sont continues mais pas n\'ecessairement $\cl$, voir \cite{Zhang},
\cite{Maillot} pour la construction de ce genre de m\'etriques.

Apr\`es une premi\`ere approche due \`a  Zhang, voir \cite{Zhang}. Maillot introduit  une G\'eom\'etrie d'Arakelov qui tient en compte des
m\'etriques admissibles et int\'egrables, voir \cite{Maillot}. Dans ce texte on compl\`ete cette th\'eorie, au moins pour les surfaces de
Riemann compactes, en construisant une th\'eorie spectrale associ\'ee \`a  ces m\'etriques int\'egrables qui g\'en\'eralise la th\'eorie classique. Ce qui nous permettra de d\'efinir une th\'eorie de torsion analytique holomorphe dans ce cas.\\

Rappelons qu'un  fibr\'e en droites hermitien $(E,h_E)$ sur une vari\'et\'e projective complexe non-singuli\`ere, est
dit
admissible si
$h_E$ est limite uniforme d'une suite $(h_n)_{n\in \N}$  de m\'etriques $\cl$ et positives. On dit qu'il est
int\'egrable si $h_E$ est
quotient de deux m\'etriques admissibles, voir \eqref{rappelmetint}.\\

Soit $(X,h_X)$ une surface de Riemann compacte o\`u $h_X$ est une m\'etrique hermitienne continue. Au paragraphe
\eqref{rappelLAPCLA},
on rappelle la construction du Laplacien dans le cas classique agissant sur $A^{(0,0)}(X,E)$.  Soit
$\overline{E}_\infty=(E,h_{E,\infty})$ un fibr\'e en droites int\'egrable  sur $X$,  nous montrons qu'on peut
associer \`a  $h_{E,\infty}$ un op\'erateur lin\'eaire qui agit
qui agit sur $A^{(0,0)}(X,E)$ et \'etend la notion de Laplacien pour les m\'etriques $\cl$. On l'appellera
l'op\'erateur Laplacien associ\'e \`a  $h_{E,\infty}$ et on le note par $\Delta_{\overline{E}_\infty}$. Plus
pr\'ecis\'ement, on \'etablit le r\'esultat suivant, (voir  th\'eor\`eme \eqref{lapintconv}):

\begin{theorem} Soit $(X,h_X)$ une surface de Riemann compacte avec $h_X$ une m\'etrique hermitienne continue.
Soit $\overline{E}_\infty=(E,h_{E,\infty})$ un fibr\'e en droites  int\'egrable.  On note par
$\overline{A^{0,0}(X,E)}_\infty$ le compl\'et\'e de $A^{0,0}(X,E)$ pour  la m\'etrique $L^2_\infty$ induite par $h_X$
et $h_{E,\infty}$.

Pour toute d\'ecomposition de $\overline{E}_\infty=\overline{E}_{1,\infty}\otimes \overline{E}_{2,\infty}^{-1}$ en fibr\'es admissibles
$\overline{E}_{1,\infty}$ et $\overline{E}_{2,\infty}$, et pour tout choix de suites $\bigl(h_{1,n}\bigr)_{n\in \N}$  (resp.
$\bigl(h_{2,n}\bigr)_{n\in \N}$) de m\'etriques positives $\cl$ sur $\overline{E}_{1,\infty}$ (resp. sur $\overline{E}_{2,\infty}$)
qui converge uniform\'ement vers $h_{1,\infty}$ (resp. $h_{2,\infty}$) et si l'on pose
$\overline{E}_n:=\overline{E}_{1,n}\otimes \overline{E}_{2,n}^{-1}$. Alors
 \begin{enumerate}

\item Pour tout $\xi \in A^{(0,0)}(X,E)$, la suite
$\bigl(\Delta_{\overline{E}_n}\xi\bigr)_{n\in \N}$ converge, pour la norme $L^2_\infty$, lorsque $n\mapsto 
\infty$ vers une limite qu'on note par $\Delta_{\overline{E}_\infty}\xi$ dans
$\overline{A^{0,0}(X,E)}_\infty$. Cette limite ne d\'epend par du choix de la d\'ecomposition ni de celui 
de la suite.
 \item $\Delta_{\overline{E}_\infty}$ est un op\'erateur lin\'eaire de $A^{0,0}(X,E)$ vers $\overline{A^{0,0}(X,E)}_\infty$.

 \item
 \[
\begin{split}
\bigl(\Delta_{\overline{E}_\infty}(f\otimes \si), g\otimes \tau\bigr)_{L^2,\infty}&=\bigl(f\otimes \si, \Delta_{\overline{E}_\infty}(g\otimes \tau)\bigr)_{L^2,\infty}\\
&= \frac{i}{2\pi }\int_X h_{E,\infty}(\si,\tau)\frac{\pt f}{\pt \z}\,\frac{\pt \overline{g}}{\pt z}\,  dz\wedge d\z.
\end{split}
\]
 $\forall f,g\in A^{0,0}(X)$ et $\si,\tau$ deux sections locales holomorphes de $E$ tels que $f\otimes \si$ et $g\otimes \tau$ soient dans $A^{(0,0)}\bigl(X,E\bigr)$. En particulier,
 \[
 \bigl(\Delta_{\overline{E}_\infty}\xi, \xi'\bigr)_{L^2,\infty}=\bigl(\xi, \Delta_{\overline{E}_\infty}\xi'\bigr)_{L^2,\infty},
 \]
  $\forall \xi,\xi'\in A^{(0,0)}(X,E)$.
 \item
 \[\bigl(\Delta_{\overline{E}_\infty}\xi,\xi\bigr)_{L^2, \infty}\geq  0 \quad \forall\, \xi\in A^{0,0}(X,E).\\
 \]
 \end{enumerate}

\end{theorem}

Si l'on choisit convenablement $\bigl(h_n \bigr)_{n}$ une suite  de m\'etriques de classe $\cl$ qui converge uniform\'ement vers
$h_\infty$, alors nous sommes capables de relier  $e^{-t\Delta_{\overline{E}_\infty}}$ \`a
$\bigl(e^{-t\Delta_{\overline{E}_n}}\bigr)_{n
\in \N}$; la suite de noyaux de Chaleur associ\'es \`a  $\bigl(h_n \bigr)_{n}$. En fait, si $(h_n)_{n\in \N}$ v\'erifie les deux
conditions suivantes:
\begin{enumerate}
\item \[
\sup_{n\in \N}\biggl\| h_X\Bigl(\dif,\dif\Bigr)^{-\frac{1}{2}}\dif \log \frac{h_{n+1}}{h_n}\biggr\|_{\sup}<\infty
\]
o\`u $\bigl\{\dif \bigr\}$ est une base locale de $TX$. Rappelons que $\biggl|
h_X\Bigl(\dif,\dif\Bigr)^{-\frac{1}{2}}\dif \log \frac{h_{n+1}}{h_n}\biggr|$ ne d\'epend pas du choix de la base locale.
\item
\[
 \sum_{n=1}^\infty \biggl\| \frac{h_n}{h_{n-1}}-1\biggr\|_{\sup}^{\frac{1}{2}}<\infty.\\
\]

\end{enumerate}
Alors, nous montrons le th\'eor\`eme suivant, (voir  th\'eor\`eme  \eqref{ggg222}):

\begin{theorem}
On a,
\[
\bigl(e^{-t\Delta_{\overline{E}_n}}\bigr)_{n\in \N} \xrightarrow[n\mapsto \infty]{} e^{-t\Delta_{\overline{E}_\infty}}.
\]
dans l'espace des op\'erateurs born\'es sur le compl\'et\'e de $A^{(0,0)}(X,E)$ pour la m\'etrique $L^2_\infty$.
\end{theorem}

Apr\`es nous montrons que $\Delta_{\overline{E}_\infty}$ poss\`ede un spectre discret, infini et positif. 
L'\'etape suivante consiste \`a \'etudier  l'op\'erateur $ P^\infty e^{-t\Delta_{L,\infty}}$, o\`u $P^\infty$ 
est la projection orthogonale \`a  $\ker \Delta_{E,\infty}$ pour $L^2_\infty$.  Nous montrons qu'il est 
nucl\'eaire,
on peut alors introduire   la fonction Th\^eta associ\'ee en posant:
\[
 \theta_\infty(t):=\mathrm{Tr}\bigl(P^\infty e^{-t\Delta_{E,\infty}}\bigr)\quad \forall \, t>0,
\]
Nous montrons qu'on a convergence simple de:
\[
 \bigl(\theta_n(t)\bigr)_{n\in \N}\xrightarrow[n\mapsto\infty]{} \theta_\infty(t)\quad \forall\, t>0,
\]
o\`u $\theta_n$ est la fonction Th\^eta associ\'ee \`a  $\Delta_{E,h_n}$.

Ces r\'esultats sont regroup\'es dans le th\'eor\`eme suivant, (voir  th\'eor\`eme  \eqref{EEnuclear}):
\begin{theorem}
Pour tout $t>0$,  $e^{-t\Delta_{E,\infty}}$ est un op\'erateur nucl\'eaire. On a
\[
 \underset{u\mapsto \infty}{\lim}\theta_u(t)=\theta_\infty(t).
\]
 La fonction Z\^eta $\zeta_\infty$ d\'efinit par:
\[
 \zeta_\infty(s)=\frac{1}{\Gamma(s)}\int_0^\infty \theta_{\infty}(t) t^{s-1}dt,
\]
est holomorphe sur $\mathrm{Re}(s)>1$ et
\[
\zeta_\infty(s)=\sum_{k=1}^\infty \frac{1}{\la_{\infty,k}^s}\quad \forall\; \mathrm{Re}(s)>1.
\]
 admet un prolongement analytique au voisinage de 0 et on a
\[
 \zeta'_{\infty}(0)=\lim_{u\mapsto \infty}\zeta'_u(0).
\]
avec $\zeta_u$ est la fonction Z\^eta  associ\'ee \`a  $\Delta_{E,u}$.
\end{theorem}

On notera  qu'afin de  d\'emontrer   ce th\'eor\`eme on aura besoin d'une borne inf\'erieure positive non nulle pour la suite
$(\la_{1,n})_{n\in \N}$, o\`u $\la_{1,n}$ est la premi\`ere valeur propre non nulle de $\Delta_{\overline{E}_n}$ pour tout $n\in \N$.  L'existence de cette borne sera d\'eduit de la proposition suivante,  qu'on \'etablit dans \eqref{uniformelambda}:
\begin{proposition}
Soit $(X,h_X)$ une surface de Riemann compacte et $E$ un fibr\'e en droites holomorphe. Soit $h_{E,\infty}$ une m\'etrique hermitienne continue sur $L$. Soit $(h_n)_{n}$ une suite de m\'etriques hermitiennes $\cl$ sur $L$ qui converge uniform\'ement vers $h_\infty$. Alors il existe une constante $\al\neq 0$ telle que
\[
\al \leq \frac{\la_{1,q}}{\la_{1,p}}\leq \frac{1}{\al},\quad \forall \,1\ll p\leq q.
\]

\end{proposition}

On termine par montrer que la suite des m\'etriques de Quillen $(\vc_{Q,h_n})_{n\in \N}$ converge vers une limite et que cette derni\`ere est repr\'esent\'ee par la m\'etrique $\vc_{Q,h_\infty}=\vc_{L^2,\infty}\exp({\zeta'_\infty(0)})$.\\

\subsection{Organisation de l'article}

%Soit $X$ une surface de Riemann compacte et $E$ un fibr\'e en droites holomorphe sur $X$. On consacre deux paragraphes
%ind\'ependants \`a  l'\'etude de la variation de la m\'etrique sur $E$ et sur $TX$, ainsi qu'aux objets spectraux associ\'es.
%Le cas de $E$ s'av\`ere \^etre plus d\'elicat, plus concr\`etement, afin d'\'etendre la th\'eorie du noyau de chaleur aux
%Mm\'etriques int\'egrables sur $E$, on est confront\'e \`a  deux  probl\`emes. Le premier  est crucial pour assurer la
%convergence de la
%suite form\'ee par les op\'erateurs de chaleur, voir th\'eor\`eme \eqref{ggg222}. Le second est un outil cl\'e pour prouver en
%particulier que $e^{-t\Delta_\infty}$ est un op\'erateur nucl\'eaire.\\

%Le premier probl\`eme s'\'enonce comme suit: On fixe une m\'etrique $h_X$ sur $X$ et on se donne  une m\'etrique admissible (ou plus g\'en\'eralement int\'egrable) $h_\infty$ sur $E$ alors existe-il $\bigl(h_n\bigr)_{n\in \N}$ une suite de m\'etriques hermitiennes de classe $\cl$ positives qui converge uniform\'ement vers $h_\infty$ sur $X$ entier telle que

% et le deuxi\`eme est de nature spectrale; Si l'on se donne  $(h_n)_n$ une suite de m\'etriques hermitiennes $\cl$ sur $E$ ou $TX$ qui converge uniform\'ement vers $h_\infty$ et si l'on note par $\la_{1,n}$ la premi\`ere valeur propre non nulle de l'op\'erateur $\Delta_{E, h_n}$, alors est ce qu'on peut trouver une constante $\al>0$ telle que
%\begin{equation}\label{isoperproblem111}
 %\al\leq \la_{1,n}\quad \forall\, n\in \N\quad ?\\
%\end{equation}

Dans le chapitre \eqref{paragrapheintegrable}, nous introduisons la notion de m\'etrique $1$-int\'egrable. On consid\`ere alors la d\'efinition suivante:

\begin{definition}
Soit $\bigl(X,\omega\bigr)$ une surface de Riemann compacte et $\omega_X$ une forme de K\"ahler. Soit $\overline{E}=\bigl(E,h_\infty\bigr)$ un fibr\'e en droites int\'egrable sur $X$.

On dit que $\bigl(E,h_\infty \bigr)$ est $1$-int\'egrable s'il existe une suite $\bigl( h_n\bigr)_{n\in \N}$ de m\'etriques de classe $\cl$ sur $E$ qui converge uniform\'ement vers $h_\infty$ v\'erifiant que:
\begin{enumerate}
\item Il existe $\overline{E}_1$ et $\overline{E}_2$ deux fibr\'es en droites admissibles tels que $\overline{E}=\overline{E}_1\otimes \overline{E}_2^{-1}$ et deux suites $\bigl(h_{1,n} \bigr)_{n\in\N}$ et $\bigl(h_{2,n} \bigr)_{n\in \N}$ de m\'etriques positives de classe $\cl$ sur $E_1$ respectivement sur $E_2$ telles que $\bigl( h_{1,n}\bigr)_{n\in \N}$ respectivement $ \bigl(h_{2,n}\bigr)_{n\in \N}$ converge uniform\'ement vers $h_{1,\infty}$ respectivement vers  $h_{2,\infty}$ et que
\[
 h_n=h_{1,n}\otimes h_{2,n}^{-1}\quad \forall n\in \N.
\]
ou il existe  $h'$, une m\'etrique  positive de classe $\cl$ telle que $h_n\otimes h'$ est positive, $\forall n\in \N$.

\item
\[
 \sum_{n=1}^\infty \Bigl\| \frac{h_n}{h_{n-1}}-1\Bigl\|_{\sup}^{\frac{1}{2}}<\infty,
\]
\item
\begin{equation}\label{supsupmetrique}
 \sup_{n\in \N_{\geq 1}}\Bigl\| h_X\Bigl( \dif,\dif\Bigr)^{-\frac{1}{2}}\dif\log \frac{h_n}{h_{n-1}}\Bigr\|_{\sup}<\infty.
\end{equation}
\end{enumerate}

\end{definition}

Les principaux r\'esultats du chapitre \eqref{paragrapheintegrable} concernant cette nouvelle notion de m\'etriques  sont les th\'eor\`emes
\eqref{integrableregular} et \eqref{bellemajorationfaible}. Nous commencons tout d'abord par remarquer que
\[\biggr(
\Bigl|h_X\Bigl(\dif,\dif\Bigr)^{-\frac{1}{2}}\dif \log \frac{h_n}{h_{n-1}}\Bigr|\biggr)_{n\in \N} \]
 est une suite de fonctions bien d\'efinies sur $X$ et qu'elle est $L^2$-born\'ee. On notera \`a  l'aide d'exemples que la convergence uniforme n'est pas suffisante pour garantir \eqref{supsupmetrique}.  Nous  \'etablissons   que toute m\'etrique continue  avec des conditions de r\'egularit\'es faibles sur un fibr\'e en droites sur une surface de Riemann compacte $X$  est int\'egrable, c'est l'objet du th\'eor\`eme \eqref{integrableregular}; L'id\'ee cl\'e est le fait que toute m\'etrique positive est admissible. \\

Dans la proposition \eqref{supsupexemple}, on y trouve une large classe d'exemples de suites de m\'etriques sur $\mathcal{O}(m)$ sur $\p^1$ v\'erifiant \eqref{supsupmetrique}  et on notera qu'on peut avoir aussi $\lim_{n\in \N}\Bigl\| h_X\Bigl(\dif,\dif\Bigr)^{-\frac{1}{2}}\dif \log \frac{h_n}{h_{n-1}}  \Bigr\|_{\sup}=0$.\\

Afin de r\'epondre, partiellement, \`a  la question \eqref{supsupmetrique}, on se restreint \`a  la classe des m\'etriques int\'egrables invariantes par l'action de $\s$, par exemple les m\'etriques canoniques sur $\p^1$. Notons que les r\'esultats \'enonc\'es ici, sont tous v\'erifi\'es par les m\'etriques canoniques sur $\p^1$.
Dans la proposition \eqref{dynlisse1}, on s'int\'eresse \`a  des m\'etriques d'origine dynamique. L'observation des diff\'erents exemples nous am\`ene \`a  introduire la classe des m\'etriques int\'egrables invariantes par $\s$ qui sont $\cl$ sur $\p^1\setminus S$ o\`u $S$ est un compact de $\p^1\setminus \{ 0,\infty\}$, par exemple la m\'etrique $\frac{|\cdot |}{\max(|x_0|,|x_1|) }$ est int\'egrable, invariante par $\s$ et $\cl$ sur $\p^1\setminus \{|z|=1 \}$.  Le principal r\'esultat de ce paragraphe est le th\'eor\`eme \eqref{bellemajorationfaible}, o\`u on montre que \eqref{supsupmetrique} est v\'erifi\'ee si $h_\infty$ est int\'egrable, invariante par $\s$ et de classe $\cl$ sur $\p^1\setminus S$ avec $S$ un compact de $\p^1\setminus\{0,\infty \}$.\\

D'apr\`es le th\'eor\`eme \eqref{bellemajorationfaible}, il est  naturel d'introduire la d\'efinition suivante: Soit $X$ une surface de Riemann compacte et $h_{E,\infty}$ une m\'etrique admissible sur $E$, un fibr\'e en droites sur $X$. On dit que $h_{E,\infty}$ est $1$-admissible s'il existe une suite de m\'etriques positives de classe $\cl$ sur $L$ qui converge uniform\'ement vers $h_{E,\infty}$ telle que
\[
\sup_{n\in \N^\ast}\biggl\|h\Bigl(\dif,\dif\bigr)^{-\frac{1}{2}}\dif \log \frac{h_n}{h_{n-1}} \biggr\|_{\sup}<\infty,
\]
(o\`u $\{\dif\}$ est une base  holomorphe locale de $TX$)
et
\[
 \sum_{n=1}^\infty \biggl\| \frac{h_n}{h_{n-1}}-1\biggr\|^{\frac{1}{2}}<\infty.
\]

et plus g\'en\'eralement, $h_{E,\infty}$ sera dite ``$1$-int\'egrable'' s'il existe $h_{1,\infty}$ et $h_{2,\infty}$ deux m\'etriques ``$1$-admissibles'' telles que
\[
 h_{E,\infty}=h_{1,\infty}\otimes h_{2,\infty}^{-1}.
\]
Avec les notations de la d\'efinition \eqref{1-integrable}, on montre qu'il existe des m\'etriques $1$-int\'egrables
continues mais non $\cl$, c'est l'objet du th\'eor\`eme \eqref{existence1integrable}. Notons que nos r\'esultats s'\'etendent
\`a  toute m\'etrique ayant cette propri\'et\'e.\\

Remarquons que \eqref{integrableregular}, nous fournit une large classe d'exemples d'application pour la th\'eorie
d\'evelopp\'ee dans le texte.\\

Les paragraphes \eqref{LGAMI1} et \eqref{LGAMI2} sont  consacr\'es \`a  l'extension de la notion du Laplacien g\'en\'eralis\'e  aux cas des m\'etriques int\'egrables sur les fibr\'es en droites holomorphes sur une  surface de Riemann compacte. On propose deux approches diff\'erentes, dans la premi\`ere nous montrons que pour tout choix $(h_n)_n$ de suites de m\'etriques positives $\cl$ qui converge uniform\'ement vers une m\'etrique admissible $h_\infty$, la suite d'op\'erateurs $\Delta_{E,h_n}$ converge fortement pour la m\'etrique $L^2$ vers un op\'erateur qu'on notera par $\Delta_{\overline{E}_\infty}$, voir le th\'eor\`eme \eqref{lapintconv}. Lorsque $X=\p^1$, le th\'eor\`eme \eqref{lapintconv22} d\'ecrit la deuxi\`eme approche qui consiste \`a  d\'efinir directement l'op\'erateur $\Delta_{E,h_\infty}$ par une formule qui \'etend l'expression locale du Laplacien dans le cas $\cl$ en montrant que le coefficients de ce dernier ont encore un sens dans le cas de $h_{E,\infty}$.

On termine, par prouver que les deux approches d\'efinissent le m\^eme op\'erateur, qu'on appellera le Laplacien associ\'e \`a   $h_{E,\infty}$.\\

Dans le chapitre \eqref{paragrapheLapX}, on \'etudie la variation de la m\'etrique sur $TX$. On consid\`ere une m\'etrique int\'egrable sur $TX$, qu'on lui associe un op\'erateur Laplacien not\'e $\Delta_{X,\infty}$. on montre qu'il admet un noyau de chaleur $e^{-t\Delta_{X,\infty}}$ qui peut \^etre approch\'e par une suite de noyaus de chaleur $\bigl(e^{-t\Delta_{X,n}}\bigr)_n$ de classe $\cl$.\\

Le \eqref{paragrapheOpcompacts} constitue une annexe regroupant quelques r\'esultats techniques utiles, on y trouve  un rappel sur les op\'erateurs born\'es, compacts et nucl\'eaires  sur un espace de Hilbert. En combinant plusieurs r\'esultats on \'etablit dans \eqref{paragrapheVarMetE}, que
\[
 e^{-t\Delta_{E,h_\infty}},\quad t>0
\]
  est un op\'erateur nucl\'eaire, en d'autres termes nous \'etablissons que
\[
 \theta_\infty(t):=\sum_{k\geq 1}e^{-t\la_{\infty,k}}<\infty
\]
o\`u $\{\la_{\infty,k}\}_{k\in \N}$ sont les valeurs propres de $\Delta_{E,h_\infty}$ compt\'ees avec multiplicit\'es. Si l'on note par
\[
 \zeta_\infty(s)=\frac{1}{\Gamma(s)}\int_0^\infty t^{s-1}\theta_\infty(t)dt,\quad s\in \CC,
\]
alors nous montrons notre principal r\'esultat de ce paragraphe:
\[
 \lim_{n\in\N}\zeta_n'(0)=\zeta_\infty'(0)
\]
est la limite est finie.\\

Cela explique nos pr\'ec\'edents calculs et le fait que
\[
 T_g\bigl((\p^1,\omega_\infty);\overline{\mathcal{O}(m)}_\infty \bigr)=\zeta_{\Delta_{\overline{\mathcal{O}(m)}_\infty}}'(0)
\]
o\`u le terme \`a  gauche est obtenu comme limite de $\zeta_{n}'(0)$.\\

L'appendice \eqref{Quelqueslemmes} regroupe quelques lemmes et r\'esultats qui seront utilis\'es tout au long de cet
article, par exemple on \'etablit un lemme technique qui nous permet de construire \`a  partir d'une  famille discr\`ete ou
une suite $(h_n)_{n\in \N}$ de m\'etriques hermitiennes, une famille \`a  param\`etre continu $u\in [1,\infty[$ qui
interpole $(h_n)_{n\in \N}$ aux points entiers et qui varie de facon $\cl$ en fonction $u$. Ce lemme sera utilis\'e
pour \'etudier en particulier la variation des $e^{-t\Delta_{L,h_n}}$ en fonction de $n$ en utilisant la formule de
Duhamel pour $\frac{\pt }{\pt u}e^{-t\Delta_u}$.\\

 \tableofcontents

\section{Sur les  m\'etriques int\'egrables sur $\p^1$}\label{paragrapheintegrable}

Dans ce paragraphe, on s'int\'eresse au probl\`eme suivant: Etant donn\'ee une m\'etrique $h_\infty$ admissible, plus g\'en\'eralement int\'egrable  (voir  \eqref{rappelmetint} pour la d\'efinition), existe-il  une suite $(h_n)_{n\in \N}$ de m\'etriques hermitiennes positives de classe $\cl$ qui converge uniform\'ement vers $h_\infty$ telle que:
\begin{equation}\label{borneadmissible}
\sup_{n\in \N} \Bigl\| h_X\bigl(\dif,\dif \bigr)^{\frac{1}{2}}\dif \log \frac{h_n}{h_{n-1}}\Bigr\|_{\sup}<\infty
\end{equation}

Si l'on fixe une m\'etrique $L^2$ sur $A^{0,0}(\p^1)$, alors  la suite   $\Bigl(\Bigl|h_X\bigl(\dif,\dif \bigr)^{-\frac{1}{2}}\dif \log \frac{h_n}{h_{n-1}}\Bigr|\Bigr)_{n\in \N}$ est $L^2$-born\'ee, en effet on a
{\allowdisplaybreaks
\begin{align*}
\Bigl\|h_X\bigl(\dif,\dif \bigr)^{\frac{1}{2}}\dif \log \frac{h_n}{h_{n-1}}\Bigr\|_{L^2}^2&=\int_{X}\Bigl|h_X\bigl(\dif,\dif \bigr)^{\frac{1}{2}}\dif\bigl( \log \frac{h_n}{h_{n-1}}\bigr)\Bigr|^2\omega_X\\
&=\int_{X}\Bigl|\dif\bigl( \log \frac{h_n}{h_{n-1}}\bigr)\Bigr|^2\omega_X\\
&=\int_X\log \frac{h_n}{h_{n-1}}\bigl(c_1(L,h_n)-c_1(L,h_{n-1}) \bigr)\\
&\leq \int_X\log \frac{h_n}{h_{n-1}}c_1(L,h_n)+\int_X\log \frac{h_n}{h_{n-1}}c_1(L,h_{n-1}) \\
&\leq 2\deg(L)\Bigl\| \log \frac{h_n}{h_{n-1}}\Bigr\|_{\sup}.
\end{align*}
}
(on a utilis\'e la positivit\'e de $c_1(L,h_n)$).\\

On remarque qu'en g\'en\'eral, le fait que $(h_n)_{n\in \N}$ converge uniform\'ement vers $h_\infty$ n'implique pas n\'ecessairement la propri\'et\'e \eqref{borneadmissible}, en effet, si l'on consid\`ere par exemple le polyn\^ome $P(z)=z^2-2$, et on pose $h_n$ la m\'etrique hermitienne sur $\mathcal{O}(1)$, d\'efinie comme suit:
\[
\|\cdot\|_n:=\frac{|\cdot|}{\bigl(1+|P^{(n)}(z)|^2 \bigr)^\frac{1}{2^{n+1}} }\footnote{o\`u $P^{(n+1)}(z)=P(P^{(n)}(z))$ pour tout $n\in
\N$.}
\]
alors on montrera que $(h_n)_{n\in \N}$ est une suite de m\'etriques hermitiennes positives de classe $\cl$ qui converge uniform\'ement vers une m\'etrique continue qu'on note par $h_\infty$ cf. \cite{Zhang}. Mais, comme
\[
 \dif \log h_n(z)=\frac{1}{2^n}\frac{\bigl(\dif P^{(n)}(z)\bigr)\overline{P^{(n)}(z)} }{1+|P^{(n)}(z)|^2}=\frac{\bigl( \prod_{k=0}^{n-1}P^{(k)}(z)\bigr)\overline{P^{(n)}(z)} }{1+|P^{(n)}(z)|^2}
\]

et puisque $P(2)=2$, alors
{{}
\[
 \dif \log h_n(2)=\frac{2^{n+1} }{5}.
\]
}
Par suite,
\[
\Bigl|h_\infty(\dif,\dif)^{-\frac{1}{2}}\dif \log \frac{h_n}{h_{n-1}}\Bigr|(2)=\frac{4}{5}2^{n}\quad \forall n\in \N.
\]
Il est naturel de s'int\'eresser aux m\'etriques int\'egrables singuli\`eres, puisqu'on les m\'etriques qui sont $\cl$ v\'erifient automatiquement la question ci-dessus. On commence  par \'etudier une classe de m\'etriques associ\'ees \`a  des syst\`emes dynamique sur $\p^1$, on montrera que se sont des m\'etriques int\'egrables singuli\`eres plus pr\'ecis\'ement on va montrer qu'elles sont toujours de classe $\cl$ sur $\p^1\setminus S$, o\`u $S$ est un compact de $\CC$. On en d\'eduira des exemples de m\'etriques v\'erifiant la condition ci-dessus.\\

Plus  g\'en\'eralement, on d\'emontrera que toute m\'etrique int\'egrable invariante par l'action de $\s$ et qui est de classe $\cl$ sur $\p^1\setminus S$ o\`u $S$ est un compact de $\CC^\ast$, par exemple les m\'etriques canoniques sur les fibr\'es en droites sur $\p^1$, v\'erifient la propri\'et\'e \eqref{borneadmissible}. \\

%On suppose qu'il existe $S$ un ferm\'e de $\p^1\setminus{\{0, \infty\}}$ tel que $h_\infty$ soit diff\'erentiable sur $\p^1\setminus S$. {{} Par exemple, si $h_\infty$ est la m\'etrique canonique de $\mathcal{O}(m)$, alors on peut prendre $S=\s$.}\\

On note par $\mathcal{S}$ l'ensemble des compacts de $\p^{1}\setminus\{0,\infty\}$, notons que cet ensemble est stable par union finie. Si $h_1$ (resp. $ h_2$) une m\'etriques continue de classe $\cl$ en dehors $ S_1$ (resp. $S_2$) avec $S_1, S_2\in \mathcal{S}$ alors $h_1\otimes h_2$ est $\cl$  sur $\p^1\setminus \bigl(S_1\cup S_2)$.

\begin{definition}
On note par $\widehat{\mathrm{Pic}}_{int,\mathcal{S}}(\p^1)$ le sous-groupe de $\widehat{\mathrm{Pic}}_{int}(\p^1)$ form\'e des classes d'isomorphie isom\'etrique des fibr\'es hermitiens int\'egrables $\overline{L}=(L,h)$ sur $\p^1$ tels que $h$ soit de classe $\cl$ sur $\p^1\setminus S$ o\`u $S\in \mathcal{S}$.

$\widehat{\mathrm{Pic}}_{int,\mathcal{S},0}(\p^1)$ le sous-groupe de $\widehat{\mathrm{Pic}}_{int}(\p^1)$ engendr\'e par les m\'etriques de classes $\cl$ et les m\'etriques radiales int\'egrables qui sont $\cl$ sur le compl\'ementaire d'un \'element de  $\mathcal{S}$.\\

On note que $\widehat{\mathrm{Pic}}(\p^1)$ est un sous-groupe de $\widehat{\mathrm{Pic}}_{int,\mathcal{S},0}(\p^1)$\\

Aussi, on d\'efinit  $\widehat{CH}^2_{int,\mathcal{S}}(\p^1)$ comme \'etant le sous-groupe de $\widehat{CH}^2_{int}(\p^1) $  engendr\'e par les \'el\'ements de la forme $\al \cdot \widehat{c}_1(\overline{L})^q$, o\`u $\al\in \widehat{CH}^{2-q}(\p^1)$.
\end{definition}

On rappelle le lemme  suivant qui sera utile dans la suite:
\begin{lemma}\label{convexefonction1}
Soit $[a,b]$ un intervalle compact de $\R$ et $\bigl(f_n\bigr)_{n\in \N}$ une suite de fonctions r\'eelles convexes diff\'erentiables sur $[a,b]$. On suppose que $\bigl(f_n\bigr)_{n\in \N}$ converge uniform\'ement vers une fonction $f$ sur $[a,b]$, alors pour tout $[\al,\beta]\subset ]a,b[$ il existe une constante $c$ telle que
{{}
\[
 \bigl|f_n'(x)\bigr|\leq c, \quad \forall x \in[\al,\beta] \quad\forall\, n\in \N.
\]
}
\end{lemma}
\begin{proof}
 C'est un r\'esultat classique.
\end{proof}

\begin{proposition}\label{positifadmissible}
 Soit $X$ une vari\'et\'e complexe projective et $L$ un fibr\'e en droites ample sur $X$. Toute m\'etrique positive sur $L$ est admissible.
\end{proposition}
\begin{proof}
Voir \cite[Th\'eor\`eme 4.6.1]{Maillot}.
\end{proof}

\begin{proposition}\label{Green11} Soit $X$ une vari\'et\'e complexe de dimension $d$ et $A\subset X$ un ouvert relativement compact tel que $\overline{A}$ soit une sous-vari\'et\'e r\'eelle \`a  coins de $X$. Soient $f$ et $g$ deux formes diff\'erentielles de classes $\mathcal{C}^2$ au voisinage de $\overline{A} $ de bidegr\'es $(p,p)$ et $(q,q)$ telles que $p+q=d-1$. On a:

\[
 \int_{A}(fdd^c g-g dd^c f)=\int_{\partial A} (fd^cg-g d^cf)
\]
\end{proposition}
\begin{proof}
 Voir par exemple \cite{DemaillyLivre}.
\end{proof}
\begin{theorem}\label{integrableregular}
Soit $X$ une surface de Riemann compacte. Soient $A_1 \subsetneq A_2\subsetneq \ldots \subsetneq A_m \subset X$ $m$ ouverts non vide relativement compacts tels que $\overline{A_1},\ldots, \overline{A_m}$ soient des sous-vari\'et\'es r\'eelles \`a  coins de $X$.\\
%Soient $\gamma_1,\ldots,\gamma_m$ $m$ sous-vari\'et\'es  r\'eelles compactes \`a  coins de $\CC$  disjointes deux \`a  deux telles que pour tout $i\geq 2$, la composante born\'ee de bord $\gamma_i$ contient strictement $ \gamma_1,\ldots,\gamma_{i-1} $. Pour tout $1\leq i\geq \leq r$, on note par $D_{i}$  le ferm\'e qui a pour bord $\gamma_i\cup\gamma_{i-1}$, par convention on prend $\gamma_0=\emptyset$ et $D_{r+1}$ la composante non born\'ee de bord $\gamma_r$.

Soit $g$ une fonction r\'eelle continue sur $X$. On suppose que pour tout $i\geq 1$, il existe une fonction $g_i$ de classe $\mathcal{C}^2$ sur un voisinage ouvert de $\overline{A_i}\setminus A_{i-1}$, on prend $A_0=\emptyset$,  telles que

\[
g_{|_{\overline{A_i}\setminus A_{i-1}}}={g_i}_{|_{\overline{A_i}\setminus A_{i-1}}}\quad \forall i\geq 1,
\]
et
\[
d^cg_i=d^cg_{i-1} \quad \text{sur}\,\,\, \pt A_i,\quad\forall i\geq 2.
\]

Soit $L$ un fibr\'e en droites sur $X$ et $h'$ une m\'etrique hermitienne de classe $\cl$ sur $L$ alors la m\'etrique suivante:
\[
h_g=h'\exp(g)
\]
 est int\'egrable.

% Soit $g$ une fonction continue sur $\p^1$ invariante par l'action de $\s$
% Si il existe $a_0=0<a_1<a_2<\ldots$ une suite croissante, tels que pour tout $i$, il existe $g_i$ une fonction  deux fois diff\'erentiable au voisinage de $[a_i,a_{i+1}]$ telles que pour tout $i$:
%\begin{enumerate}
%\item $g_i{|_{[a_i,a_{i+1}]}}=g$,
%\item $\frac{\pt}{\pt r}\bigl(r\frac{\pt g_i}{\pt r})(a_i)=\frac{\pt}{\pt r}\bigl(r\frac{\pt g_{i-1}}{\pt r})(a_i)$
%\item $\frac{\pt^2 g_i}{\pt z\pt\z}$ est born\'ee sur $[a_i,a_{i+1}]$
%\end{enumerate}

%ALors la m\'etrique d\'efinie comme suit:
%\[
% h_g:=h_{FS} e^g
%\]
%sur $\mathcal{O}(1)$ est int\'egrable.
\end{theorem}
\begin{proof}
Soit $A^{0,0}(X)_+$ le sous-ensemble de $A^{0,0}(X)$ des fonctions positives. On choisit un plongement de $X$ dans un espace projectif et note par $(\mathcal{O}(1)_{|_{X}},h_{FS})$ la restriction \`a  $X$ de $\overline{\mathcal{O}(1)_{FS}}$ par ce plongement, c'est un fibr\'e en droites  strictement positif  sur $X$. Pour montrer la proposition il suffit de montrer qu'il existe $N\in \N$ tel que le courant

\[
 dd^c\bigl(-\log (h_{FS}^N\otimes h_g)\bigr)\geq 0.
\]

Comme $h'$ est $\cl$
 donc, elle est int\'egrable, alors il suffit de montrer que la fonctionnelle
\[
 \vf\in A^{0,0}(X)_+\setminus\{0\}\lra \frac{\bigl|\int_{X}g dd^c\vf\bigr|}{\int_{X}\vf\omega_X},
\]
est born\'ee, o\`u $\omega_X$ est une forme volume de classe $\cl$.\\
On a
{\allowdisplaybreaks
\begin{align*}
 \int_{X}gdd^c\vf&=\sum_{i=1}^m \int_{\overline{A_i}\setminus A_{i-1}}g dd^c\vf\\
&=\sum_{i=1}^m \int_{\overline{A_i}\setminus A_{i-1}}g_i dd^c\vf\\
&=\sum_{i=1}^m \int_{\overline{A_i}\setminus A_{i-1}}dd^c g_i \vf+\int_{\pt\overline{A_i} }\bigl(g_i d^c\vf-\vf d^c g_i\bigr)-\int_{\pt\overline{A_{i-1}} }\bigl(g_i d^c\vf-\vf d^c g_i\bigr)\quad\eqref{Green11}\\
&=\sum_{i=1}^m \int_{\overline{A_i}\setminus A_{i-1}}dd^c g_i \vf+\sum_{i=1}^m\int_{\pt A_i}(g_i-g_{i+1})d^c\vf+ \int_{\pt A_i}\vf (d^cg_i-d^cg_{i+1})\\
&=\sum_{i=1}^m \int_{\overline{A_i}\setminus A_{i-1}}dd^c g_i \vf
\end{align*}}

%\[
% \begin{split}
%\int_{\p^1} g(r)\frac{\pt}{\pt r}\bigl(r\frac{\pt \vf}{\pt r} \bigr)dr&=\sum_{i}\bigl[ rg_i(r)\frac{\pt \vf}{\pt r} \bigr]_{a_i}^{a_{i+1}}-\sum_i\int_{a_i}^{a_{i+1}}r\frac{\pt g_i}{\pt r}\frac{\pt \vf}{\pt r}dr\\
%&=-\sum_i\bigl[ \frac{\pt}{\pt r}\bigl(r\frac{\pt g_i}{\pt r})\vf \bigr) \bigr]_{a_i}^{a_{i+1}}+\sum_i\int_{a_i}^{a_{i+1}}\frac{\pt }{\pt r}\bigl(r\frac{\pt g_i}{\pt r} \bigr)\vf dr\\
%&=-\sum_i\Bigl(\frac{\pt}{\pt r}\bigl(r\frac{\pt g_i}{\pt r})(a_i)-\frac{\pt}{\pt r}\bigl(r\frac{\pt g_{i-1}}{\pt r})(a_i) \Bigr)\vf(a_i)+\sum_i\int_{a_i}^{a_{i+1}}\frac{\pt }{\pt r}\bigl(r\frac{\pt g_i}{\pt r} \bigr)\vf dr\\
%&=\sum_i\int_{a_i}^{a_{i+1}}\frac{\pt }{\pt r}\bigl(r\frac{\pt g_i}{\pt r} \bigr)\vf dr.
 %\end{split}
%\]
Soit $z$ une coordonn\'ee holomorphe locale et soit $\theta$ un diff\'eomorphisme holomorphe au voisinage de $z$. Si l'on pose $y:=\theta(z)$, alors:
\[
 h_X\Bigl(\dif,\dif\Bigr)^{-1}\frac{\pt^2 g}{\pt z\pt \z}=h_X\Bigl(\frac{\pt}{\pt y},\frac{\pt}{\pt
y}\Bigr)^{-1}\frac{\pt^2 \widetilde{g}}{\pt y\pt \overline{y}},
\]
o\`u $\{\dif\}$ est une base locale de $TX$, $\frac{\pt}{\pt y}=\theta_\ast \dif$  et  $\widetilde{g}(y)=g(\theta^{-1}y)$.\\

Par hypoth\`ese sur $g$ et d'apr\`es la r\`egle de changement de variables pr\'ec\'edente, on a
\[
 h_X\Bigl(\dif,\dif\Bigr)^{-1}\frac{\pt^2 g}{\pt z\pt\z},
\]
est born\'ee, sur chaque carte de $X$. Comme $X$ est compacte, alors on conclut qu'il existe deux constantes $c,c'$ telles que
\[
 c\int_{X}\vf\, \omega_X \leq
  \int_{X}dd^c g\,\vf \leq c' \int_{X} \vf\,\omega_X\quad \forall\, \vf\in A^{0,0}(X)_+.
\]
Il existe $N\in \N$ tel que
\[
 Nc_1(\overline{\mathcal{O}(1)_{FS}}_{|_X})+c\,\omega_X\geq 0.
\]
Par cons\'equent
\[
\int_{X}N\vf c_1({\mathcal{O}(1)_{FS}}_{|_X})+\log h_g(s,s) dd^c\vf\geq \int_{X}\vf \bigl(N c_1({\mathcal{O}(1)_{FS}}_{|_X})+c\, \omega_X\bigr)\geq 0.
\]
On conclut que
\[
h_{FS}^{\otimes N}\otimes h_g,
\]
est positive, donc admissible par \eqref{positifadmissible}.
\end{proof}

\begin{remarque}\rm{
Si l'on pose
\[
g=\log \frac{h_\infty}{h_{FS}},
\]
o\`u $h_\infty$ est la m\'etrique canonique de $\mc$, alors $g$ est un exemple de fonctions v\'erifiant les hypoth\`eses de la proposition.}
\end{remarque}
\begin{remarque}
 \eqref{integrableregular}\rm{ peut \^etre \'etendu en dimension sup\'erieure}.
\end{remarque}

\begin{proposition}\label{supsupexemple} Soit $m$ un entier positif non nul. Sur $\p^1$, on munit le fibr\'e en droites $\mathcal{O}(m)$  de la m\'etrique suivante:
\[
h_{\chi,p}(\cdot,\cdot)([x_0:x_1]):=\frac{|\cdot|^2}{\bigl(|x_0|^{\chi(p)}+|x_1|^{\chi(p)} \bigr)^{\frac{m}{\chi(p)}}}\quad \forall\, [x_0:x_1]\in \p^1,
\]
o\`u $\chi$ est une fonction r\'eelle croissante  d\'efinie sur $\R^+$  telle que $\chi(p)\in \N^\ast$, si $ p\in \N^\ast$. Alors, il existe une constante $c_0$ telle que
{{}
\begin{equation}\label{exemplechi1}
 \Bigl\|\frac{h_{\chi,p}}{h_{\chi,p-1}}-1\Bigr\|_{\sup}\leq c_0\Bigl(\frac{1}{\chi(p-1)}-\frac{1}{\chi(p)} \Bigr) \quad \forall\, p\in \N_{\geq 1}.
\end{equation}
}
En particulier, $(h_{\chi,p})_p$ converge uniform\'ement vers $h_\infty$; la m\'etrique canonique de $\mathcal{O}(m)$, et il existe $c_1$ et $c_2$ deux constantes non nulles telles que
{{}
\[
c_1\biggl|1-\frac{{\chi(p-1)}}{\chi(p)}\biggr|\leq \biggl\|h(\dif,\dif)^{-\frac{1}{2}} \dif \log \frac{h_{\chi,p}}{h_{\chi,p-1}}\biggr\|_{\sup}\leq c_2\log \biggl(\frac{{\chi(p)}}{{\chi(p-1)}}\biggr)\quad \forall\, p\in \N_{\geq 1}.
\]
}
 \end{proposition}

 \begin{proof}
Soit $x>0$, on a
{{}\[
\frac{d}{dt}(1+x^{\chi(t)})^{-\frac{1}{\chi(t)}} =\frac{\chi'(t)}{\chi(t)^2}\Bigl(\frac{\log(1+x^{\chi(t)})}{(1+x^{\chi(t)})^\frac{1}{{\chi(t)}}}-\frac{x^{\chi(t)}\log x^{\chi(t)}}{(1+x^{\chi(t)})^{1+\frac{1}{{\chi(t)}}}}\Bigr).
\]}
Lorsque $0<x<1$, on montre que
{{}\[
\frac{d}{dt}(1+x^{\chi(t)})^{-\frac{1}{\chi(t)}} \leq (2\log 2+e^{-1})\frac{\chi'(t)}{\chi(t)^2}.
\]}
donc, $\forall p\in\N_{\geq 1}$
\[
 \begin{split}
\biggl|\bigl(1+x^{\chi(p)}\bigr)^{-\frac{1}{\chi(p)}}-\bigl(1+x^{\chi(p-1)}\bigr)^{-\frac{1}{\chi(p-1)}}\biggr|&=\biggl|\int_{p-1}^p \frac{\chi'(t)}{\chi(t)^2}\Bigl(\frac{\log(1+x^{\chi(t)})}{(1+x^{\chi(t)})^\frac{1}{{\chi(t)}}}-\frac{x^{\chi(t)}\log x^{\chi(t)}}{(1+x^{\chi(t)})^{1+\frac{1}{{\chi(t)}}}}\Bigr)dt\biggr|\\
%&=\frac{1}{u^2} \Bigl|\int_{p-1}^p \frac{(1+x^u)\log(1+x^u)-x^u\log x^u}{(1+x^u)^{1+\frac{1}{u}}}du\Bigr|\\
&\leq \bigl(2\log 2+e^{-1}\bigr)\int_{p-1}^p \frac{\chi'(t)}{\chi(t)^2}dt\\
&=(2\log 2+e^{-1})\biggl(\frac{1}{\chi(p-1)}-\frac{1}{\chi(p)} \biggr).
 \end{split}
\]
On obtient alors, pour tout  $|z|\leq 1$:
\[
\begin{split} \biggl|\frac{h_{\chi,p}}{h_{\chi,p-1}}(z)-1\biggr|&=\frac{(1+|z|^{\chi(p)})^{\frac{m-1}{\chi(p)}}}{(1+|z|^{\chi(p-1)})^{\frac{m-1}{\chi(p-1)}}}\frac{(1+|z|^{\chi(p)})^{\frac{1}{\chi(p)}}}{(1+|z|^{\chi(p-1)})^{\frac{1}{\chi(p-1)}}} \\
 &\leq 2^{m}\frac{(1+|z|^{\chi(p)})^{\frac{1}{\chi(p)}}}{(1+|z|^{\chi(p-1)})^{\frac{1}{\chi(p-1)}}} \\
 & \leq 2^{m}{(2\log 2+e^{-1})}\Bigl(\frac{1}{\chi(p-1)}-\frac{1}{\chi(p)} \Bigr) (1+|z|^{\chi(p-1)})^{\frac{1}{\chi(p-1)}}\\
 &\leq 2^{m}2^{\frac{1}{\chi(p-1)}}{(2\log 2+e^{-1})}\Bigl(\frac{1}{\chi(p-1)}-\frac{1}{\chi(p)} \Bigr).
\end{split}
\]

Si $|z|>1$, il suffit de remarquer que $\frac{h_{\chi,p}}{h_{\chi,p-1}}(z)=\frac{h_{\chi,p}}{h_{\chi,p-1}}(\frac{1}{z})$ et de se ramener au cas pr\'ec\'edent. On d\'eduit que $(h_{\chi,p})_p$ converge uniform\'ement vers une limite qui n'est autre que $h_\infty$; la m\'etrique canonique de $\mathcal{O}(m)$.\\

On a pour tout $z\in \CC$,

\[
 \max\bigl(1,|z|\bigr)^2\dif \log \frac{h_{\chi,p}}{h_{\chi,p-1}}(z)=m\max\bigl(1,|z|^2\bigr)\frac{1}{z}\Bigl(\frac{|z|^{\chi(p)}}{1+|z|^{\chi(p)}}-\frac{|z|^{\chi(p-1)}}{1+|z|^{\chi(p-1)}}  \Bigr).
\]
 Notons que cette quantit\'e est bien d\'efinie en $z=0$, puisque $\chi(k)\geq 1, \forall k\in \N$.\\

On voit que
\[
\biggl| \max(1,|z|)^2\dif \log \frac{h_{\chi,p}}{h_{\chi,p-1}}(z)\biggr|=m\max(1,|z|^2)\frac{1}{z}\Bigl|\frac{|z|^{\chi(p)}}{1+|z|^{\chi(p)}}-\frac{|z|^{\chi(p-1)}}{1+|z|^{\chi(p-1)}}  \Bigr|\leq 2m \quad \forall z\in \CC.
\]

Montrons qu'il existe deux constantes $c_1$ et $c_2$ telles que
{{}
\begin{equation}\label{a1a}
c_1 \biggl|1-\frac{{\chi(p-1)}}{\chi(p)}\biggr|\leq \biggl\|h(\dif,\dif)^{-\frac{1}{2}} \dif \log \frac{h_{\chi,p}}{h_{\chi,p-1}}\biggr\|_{\sup}\leq c\log \biggl(\frac{{\chi(p)}}{{\chi(p-1)}}\biggr)\quad \forall p\in \N_{\geq 1}.
\end{equation}}
Pour montrer l'in\'egalit\'e \`a  droite, cela  il suffit de montrer que pour tout $x\in ]0,1[$, on a
\[
\biggl|\frac{1}{1+x^{\chi(p)}}-\frac{1}{1+x^{\chi(p-1)}}\biggr|\leq c_2x\log \biggl(\frac{{\chi(p)}}{{\chi(p-1)}}\biggr). \]
Fixons $x$ dans $]0,1[$. On a, $\forall p\in\N_{\geq 1}$
\[
\begin{split}
 \biggl|\frac{1}{1+x^{{\chi(p)}}}-\frac{1}{1+x^{\chi(p-1)}}\biggr|&=\biggl|\int_{p-1}^p \frac{\chi'(t)}{\chi(t)}\frac{x^{\chi(t)}\log x^{\chi(t)}}{(1+x^{\chi(t)})^2}dt\biggr|\\
 &\leq c_2x\int_{p-1}^p\frac{\chi'(t)}{\chi(t)}dt,\quad c_2:=\sup_{y\in [0,1], \, t>0}\frac{
y^{\chi(t)-1}|\log y^{\chi(t)}|}{(1+y^{\chi(t)})^2}\\
&=c_2x\log \frac{\chi(p)}{\chi(p-1)}.
 \end{split}
 \]

Montrons maintenant l'in\'egalit\'e \`a  gauche. On a, pour tout $x\in ]0,1[$:
 \begin{align*}
 \biggl|\frac{1}{1+x^{\frac{{\chi(p-1)}}{\chi(p)}}}-\frac{1}{1+x}\biggr|& =\frac{\bigl|x-x^{\frac{{\chi(p-1)}}{\chi(p)}} \bigr|}{\bigl(1+x\bigr)\bigl(1+x^{\frac{{\chi(p-1)}}{\chi(p)}}\bigr)}\\
&\geq \frac{1}{4}\Bigl|x-x^{\frac{{\chi(p-1)}}{\chi(p)}} \Bigr|\\
&=\Bigl|1-\frac{{\chi(p-1)}}{\chi(p)}\Bigr|\bigl|\log x\bigr|\,x^{c_p} \quad\text{o\`u   }\;c_p\in
\Bigl]\frac{{\chi(p-1)}}{\chi(p)},1\Bigr[  \\
&\geq \Bigl|1-\frac{{\chi(p-1)}}{\chi(p)}\Bigr|\bigl|\log x\bigr|\,x,
\end{align*}
(le $c_p$ r\'esulte de l'utilisation du th\'eor\`eme des accroissements finis) par suite
{\allowdisplaybreaks
\begin{align*}
\biggl\|h(\dif,\dif)^{-\frac{1}{2}} \dif \log \frac{h_{\chi,p}}{h_{\chi,p-1}}\biggr\|_{\sup}&\geq
\sup_{|z|\leq 1}\biggl|h(\dif,\dif)^{-\frac{1}{2}} \dif \log \frac{h_{\chi,p}}{h_{\chi,p-1}}(z)\biggr| \\
&=\sup_{0<x\leq 1}\frac{m}{x} \biggl|\frac{1}{1+x^{{\chi(p)}}}-\frac{1}{1+x^{\chi(p-1)}}\biggr|\\
&\geq m\sup_{0<x\leq 1}\biggl|\frac{1}{1+x^{\frac{{\chi(p-1)}}{\chi(p)}}}-\frac{1}{1+x}\biggr| \\
& \geq m\sup_{0<x\leq 1}\biggl|1-\frac{{\chi(p-1)}}{\chi(p)}\biggr|\bigl|\log x\bigr|x\\
&=e^{-1}m\biggl|1-\frac{{\chi(p-1)}}{\chi(p)}\biggr|.
\end{align*}}

Si l'on pose
{{}
\[
 l:=\underset{p\mapsto \infty }{\limsup} \frac{\chi(p-1)}{\chi(p)}.
\]
}
alors
{{}
\[
 \underset{p\mapsto \infty}{\limsup}\biggl\|h(\dif,\dif)^\frac{1}{2}\log \frac{h_{\chi,p}}{h_{\chi,p-1}}\biggr\|_{\sup}\neq 0.
\]}
si $l<1$, {{} (par exemple, $\chi(p)=2^p$)}. En effet, on a pour tout $z$ fix\'e:
\[
 \underset{p\mapsto \infty }{\limsup}\frac{1}{|z|^{\frac{1}{\chi(p)}}} \Bigl(\frac{1}{1+(|z|^{\frac{1}{\chi(p)}})^{\chi(p)}} -\frac{1}{1+(|z|^{\frac{1}{\chi(p)}})^{\chi(p-1)}} \Bigr)=\frac{1}{1+|z|}-\frac{1}{1+|z|^l
 }.
\]
Si $l=1$, alors de \eqref{a1a}, on d\'eduit que
{{}
\[
 \underset{p\mapsto \infty}{\limsup}\Bigl\|h(\dif,\dif)^\frac{1}{2}\log \frac{h_{\chi,p}}{h_{\chi,p-1}}\Bigr\|_{\sup}= 0.
\]}

\end{proof}
\begin{remarque}
 \rm{
{{}$
\limsup_{p\mapsto \infty }\Bigl\|h(\dif,\dif)^\frac{1}{2}\log \frac{h_{\chi,p}}{h_{\chi,p-1}}\Bigr\|_{\sup}\neq 0,
$} si {{} $\limsup_{p\mapsto \infty}\frac{\chi(p-1)}{\chi(p)}\neq 1$} par exemple pour $\chi(p)=2^p$, et
{{} $
\limsup_{p\mapsto \infty }\Bigl\|h(\dif,\dif)^\frac{1}{2}\log \frac{h_{\chi,p}}{h_{\chi,p-1}}\Bigr\|_{\sup}= 0,
$} si  {{} $\limsup_{p\mapsto \infty}\frac{\chi(p-1)}{\chi(p)}=1$} par exemple si $\chi(p)=p$.
}
 \end{remarque}

\begin{remarque}
\rm{Lorsque $\chi$ est un polyn\^ome, alors
il existe une constante $c\neq 0$ telle que

{{}
\begin{equation}\label{exemple123}
\Bigl\|h(\dif,\dif)^{-\frac{1}{2}}\dif \log \frac{h_{\chi,p}}{h_{\chi,p-1}}\Bigr\|_{\sup}\underset{p\mapsto \infty}{\sim} \frac{c}{p}.
\end{equation}}
}

\end{remarque}

\begin{proposition}\label{dynlisse1}
Soit $L$ un  fibr\'e en droites sur $\p^1$, engendr\'e par ses sections globales,  et  $f:\p^1\rightarrow \p^1 $ un morphisme de degr\'e $d$, d\'efini par un polyn\^ome $P(z)=z^d+a_1z^{d-1}+\cdots+a_d$. Le morphisme $f$ induit un isomorphisme de fibr\'es en droites:  {{} \[\phi:L^d\rightarrow f^\ast L.\]} On consid\`ere $\textbf{h}_1$ une m\'etrique hermitienne positive de classe $\cl$  sur $L$, on construit par r\'ecurrence  une suite $(\textbf{h}_p)_{p\in \N }$ sur $L$ comme suit:

\[
\textbf{h}_p:=\bigl(\phi^\ast f^\ast \textbf{h}_{p-1}\bigr)^{\frac{1}{d}} \quad \forall\, p\in \N_{\geq 1}.
\]
Cette suite v\'erifie les propri\'et\'es suivantes:
\begin{enumerate}
 \item $(\textbf{h}_p)_{p\in \N 0}$   converge uniform\'ement vers une m\'etrique $\textbf{h}_\infty$ continue.
\item $\textbf{h}_\infty$ est admissible.
\item Il existe  $J$,  un sous-ensemble  compact de $\CC$ d'int\'erieur vide tel que $\textbf{h}_\infty$ soit de classe $\cl$ sur $\p^1\setminus J$.
\item Si $a_d=a_{d-1}=0$ alors $0\notin J$. Dans ce cas, on pose
\[
h_n(\cdot,\cdot)=|\cdot|^2e^{\psi_n} \; \text{et}\quad h_\infty(\cdot,\cdot)=|\cdot|^2e^{\psi_\infty},
\]

o\`u
$\psi_n=\int_{\s}\log \textbf{h}_n,\quad \text{et}\quad \psi_\infty=\int_{\s}\log \textbf{h}_\infty$. Alors, on a
\begin{enumerate}
 \item   $h_n$ est invariante par $\s$, positive et de classe $\cl$, $\forall n\in \N$.
\item $h_\infty$ est invariante par $\s$, positive , de classe $\cl$ sur $\p^1\setminus J$ et non de classe $\cl$ au voisinage de $J$.
\item $\bigl(h_n\bigr)_{n\in \N}$ converge uniform\'ement vers $h_\infty$.
\end{enumerate}
En d'autre termes, $\bigl(L,h_\infty\bigr)\in \widehat{Pic}_{int,\mathcal{S},0}(\p^1)$.
\end{enumerate}

\end{proposition}
 \begin{proof}

On a $1)$ et $2)$ sont due \`a  \cite{Zhang}.\\

Soit $c_1(L,\textbf{h}_\infty)$ le courant de Chern associ\'e \`a  $\bigl(L,\textbf{h}_\infty\bigr)$, alors on montre dans \cite{dynamicsibony}, que le support de la puissance maximale de $c_1(L,\textbf{h}_\infty)$, donc $c_1\bigl(L,\textbf{h}_\infty \bigr)$ dans notre cas, co\"incide avec l'ensemble de {}{Julia}  qu'on note $J$, c'est un sous-ensemble compact de $\CC$ d'int\'erieur vide, voir \cite[Corollaire 4.11]{Milnor}.\\

Montrons que $\textbf{h}_\infty$ est de classe $\cl$ sur $\p^1\setminus J$. On pose $u=-\log \textbf{h}_\infty$, $U:=\p^1\setminus J$ et on note par $v$ le courant $dd^c u$.   On a  $v_{|_U}=c_1(L,\textbf{h}_\infty)_{|_U}=0$, en particulier $v$ est de classe $\cl$. D'apr\`es \cite[1.2.1 th\'eor\`eme]{AIT}, il existe $u'$ un courant de degr\'e $(0,0)$ tel que
{{}
\[
 v=dd^c u',\quad \text{et}\quad u'_{|_U}\, \text{est }\, \cl.
\]
}
Comme $dd^c\bigl(u-u' \bigr)=0$ et puisque $\p^1$ est compact K\"ahler alors d'apr\`es \cite[1.2.2 th\'eor\`eme]{AIT}, il existe $\omega$ une forme harmonique, telle que $u-u'=\omega$. Mais $\p^1$ est projectif donc $\omega$ est une fonction constante, par suite, $u_{|_U}$ est $\cl$  et donc $\textbf{h}_\infty$ l'est aussi. Cel\`a  termine la preuve du $3)$.\\

Avant de d\'emontrer le $4)$, on note  qu'il possible d'avoir $ 0\in \mathrm{Supp}\bigl(c_1(L,\textbf{h}_\infty)\bigr)
$, par exemple, si $P(z)=z^2-2$ alors on montre que $J=[-2,2]$, voir \cite[lemme 7.1]{Milnor}.  Supposons que $P$ est de la forme suivante:
{{}
\[
 P(z)=z^d+a_1z^{d-1}+\cdots+a_{d-2}z^2,
\]
}
et montrons que dans ce cas, que $\dif \log  \textbf{h}_\infty$ est nulle au voisinage de z\'ero, ce qui est
suffisant pour montrer $4)$. Par continuit\'e, il existe $\eta>0$ tel que $|P(z)|\leq |z|$ pour tout
$|z|<\eta$. Par suite

\begin{equation}\label{pneta}
\bigl|P^{(n)}(z)\bigr|\leq \frac{\eta}{2} \quad \forall\, |z|<\frac{\eta}{2} \quad\forall n\in \N.
\end{equation}
Soit $|z|<\frac{\eta}{2}$, on a\begin{align*}
 \biggl|\dif \log  \textbf{h}_n(z)\biggr|&=\biggl|\frac{1}{d^n}\frac{\dif \bigl(P^{(n)}(z) \bigr)\overline{P^{(n)}(z)}}{1+|P^{(n)}(z)|^2}\biggr|\\
&\leq \frac{\eta}{2d^n}\Bigl|\dif \bigl(P^{(n)}(z)\bigr) \Bigr|\quad \text{par}\; \eqref{pneta}\\
&=\frac{\eta}{2d^n}\Bigl| \int_{|\xi|=\eta}\frac{P^{(n)}(\xi)}{\bigl(\xi-z\bigr)^2}d\xi \Bigr|\\
&\leq \frac{\eta^2}{2d^n}\Bigl| \int_{|\xi|=\eta}\frac{1}{\bigl||\xi|-|z|\bigr|^2}|d\xi| \Bigr|\quad \text{par}\; \eqref{pneta}\\
&\leq \frac{\eta}{d^n}.
\end{align*}
on a donc, montr\'e que pour tout $|z|<\frac{\eta}{2}$,
\[
 \Bigl|\dif \log  \textbf{h}_n(z)\Bigr|\leq \frac{\eta}{d^n},\quad \forall\, n\in \N.
\]

Par un th\'eor\`eme classique sur la convergence uniforme de suites de d\'eriv\'ees de fonctions diff\'erentiables, on d\'eduit que $ \log \textbf{h}_\infty$ est diff\'erentiable sur $\bigl\{|z|<\frac{\eta}{2}\bigr\}$ et  $\dif \log \textbf{h}_\infty=0$
sur $\bigl\{|z|<\frac{\eta}{2}\bigr\}$. On pose
 \[
 \psi_n:=\int_{\s} \log \textbf{h}_n
\quad \forall\, n\in \N\;\quad\text{et}\quad \psi:=\int_{\s} \log \textbf{h}_\infty.
\]
Donc, \[h_n(\cdot,\cdot):=|\cdot|^2e^{\psi_n},\]
 est de classe $\cl$ et elle est radiale par construction.  On a $dd^c\log h_n=\int_{\s}dd^c \log
\textbf{h}_n $, donc $h_n$ est positive. On v\'erifie que  $\bigl(h_n\bigr)_{n\in \N}$ converge
uniform\'ement vers  $h_\infty(\cdot,\cdot):=|\cdot|^2e^{\psi_\infty}$ qui est $\cl$ sur $\p^1\setminus J$.
On conclut que $h_\infty\in \mathcal{S}$. \\

Notons que $h_\infty$ est n\'ecessairement non $\cl$ au voisinage de $J$; En effet, supposons que $h_\infty$
est $\cl$
sur $\p^1$.  Puisque  $J$ est d'int\'erieur vide, alors l'adh\'erence de $\p^1\setminus J$ est $\p^1$ et par la continuit\'e des d\'eriv\'ees secondes on d\'eduit que $c_1(L,h_\infty)=0$ au voisinage de $J$  ce qui est impossible.

Notons que ce fait, n'est pas trivial puisqu'on   peut trouver une m\'etrique $\cl$ telle que le support
de sa premi\`ere forme de Chern soit un compact de $\p^1\setminus\{0,\infty\}$, en consid\'erant par exemple
une fonction $\psi$ sur $\CC$ telle $\psi=0$  au voisinage de $0$, de
classe $\cl$ en un voisinage ouvert  de $\s$ et $\psi(z)=\log |z|$ si $|z|\gg1$, alors $\psi$ d\'efinie une m\'etrique sur $\mathcal{O}(1)$ (puisque
$\psi-\psi_\infty$ est une fonction born\'ee sur $\p^1$) et montre que le support de sa premi\`ere forme de
Chern est un compact de $\p^1\setminus\{0,\infty\}$.
\end{proof}

\begin{theorem}\label{bellemajorationfaible}
Soit $h$ une m\'etrique int\'egrable sur un fibr\'e en droites $L$ sur $\p^1$. On suppose que $h_\infty$ est invariante par l'action de $\s$ et qu'il existe $S\in \mathcal{S}$ tel que $h$ soit de classe $\cl$ sur $\p^1\setminus S$. Si $h_\infty=h_{1,\infty}\otimes h_{2,\infty}^{-1} $ avec $h_{1,\infty}$ et $h_{2,\infty}$ sont deux m\'etriques admissibles, alors pour tout choix de suite $\bigl( h_{1,n}\bigr)_{n\in \N}$ (resp. $\bigl(h_{2,n} \bigr)_{n\in \N}$) de m\'etriques positives de classe $\cl$ convergeant uniform\'ement vers $h_{1,\infty}$ (resp. $h_{2,\infty}$),  il existe $\bigl(h_{n,\rho}\bigr)_{n\in \N}$ une suite de m\'etriques de classe $\cl$  qui converge uniform\'ement vers $h_\infty $ telle que
\begin{enumerate}
 \item
\begin{equation}\label{bellemajorationfaible1}
 \biggl\| \log \frac{h_{n,\rho}}{h_{n-1,\rho}}\biggr\|_{\sup}\leq \biggl\| \log
\frac{h_n}{h_{n-1}}\biggr\|_{\sup}
       \end{equation}
\item
Il existe $M>0$ tel que
\begin{equation}\label{bellemajorationfaible2}
\biggl\| h_{\p^1}\Bigl(\frac{\pt}{\pt z},\frac{\pt}{\pt z}\Bigr)^{-\frac{1}{2}}\frac{\pt}{\pt z}\Bigl(\log \frac{h_{\rho,n}}{h_{\rho,n-1}}\Bigr) \biggr\|_{\sup}\leq M \quad \forall n\in \N_{\geq 1}.
\end{equation}

\item
Il existe $h'$, une m\'etrique positive de classe $\cl$ telle que $h_{\rho,n}\otimes h' $ est positive pour tout $n\in \N$.
\end{enumerate}

\end{theorem}

\begin{proof}
Soit $h$ une m\'etrique int\'egrable invariante par $\s$ et \`a  singularit\'es contenues dans un ensemble  $S$ appartenant \`a  $\mathcal{S}$.\\

Par d\'efinition de $\mathcal{S}$, il existe deux r\'eels $0<r<R$ tels que
\[
 S\subset \bigl\{ r<|z|<R\bigr\}.
\]

Il existe $h_{\infty,1}$ et $h_{\infty,2}$ deux m\'etriques positives telles que $h_\infty=h_{\infty,1}\otimes h_{2,\infty}^{-1}$ et deux suites   $(h_{1,n})_{n\in \N }$ et $(h_{2,n})_{n\in \N }$ de m\'etriques positives de classe $\cl$ qui converge uniform\'ement vers $h_{1,\infty}$ respectivement vers $h_{2,\infty}$  sur $\p^1$. {{} Biens\^ur on s'assure que toutes ces m\'etriques sont invariantes par $\s$. }\\

On pose $h_n=h_{1,n}\otimes h_{2,n}^{-1}$ $\forall n$,  $\psi=\log h_\infty$, $\psi_n=\log h_n$ pour tout $n\in \N$.  \\

Soit $\rho$ une fonction r\'eelle de classe $\cl$ sur $\p^1$ radiale telle que $0\leq \rho\leq 1$ et v\'erifie:
\[
\rho(z) = \left\{
    \begin{array}{ll}
        1 & \mbox{si } |z|\leq \frac{r}{2}  \\
      0   & r\leq |z|\leq R\\
1&  |z|\leq 2R.
    \end{array}
\right.
\]

On pose
\[
 \psi_{n,\rho}(z):=\rho(z)\psi+\bigl(1-\rho(z)\bigr)\psi_n(z)
\]
et
\[
 h_{n,\rho}(s,s)=|s|^2e^{\psi_{n,\rho}},
\]
pour toute section locale $s$ de $L$. \\

On a, pour tout $n\in \N_{\geq 1}$:
\[
\log \biggl(\frac{h_{n,\rho}}{h_{n-1,\rho}}\biggr)=\psi_{n, \rho}-\psi_{n-1,\rho}=\bigl(1-\rho(z) \bigr)\bigl(\psi_{n-1}-\psi_n \bigr)=\bigl(1-\rho(z) \bigr)\log \biggl(\frac{h_n}{h_\infty}\biggr),
\]
donc,
\[
 \biggl\| \log \frac{h_{n,\rho}}{h_{n-1,\rho}}\biggr\|_{\sup}\leq \biggl\| \log
\frac{h_n}{h_{n-1}}\biggr\|_{\sup}\quad \forall \, n\in \N_{\geq 1}.
       \]

On a aussi,
\[
\log \biggl(\frac{h_{n,\rho}}{h_\infty}\biggr)=\psi_{n, \rho}-\psi_\infty=\bigl(1-\rho(z) \bigr)\bigl(\psi_\infty-\psi_n \bigr)=\bigl(1-\rho(z) \bigr)\log \Bigl(\frac{h_n}{h_\infty}\Bigr) \quad \forall n\in \N,
\]
alors, pour tout $n\in \N$, $h_{n,\rho}$ d\'efinit une m\'etrique hermitienne continue sur $L$, en plus on note que la suite $\bigl(h_{n,\rho}\bigr)_n$ converge uniform\'ement vers $h_\infty$. On a $\psi_{n,\rho}=\psi_n$, sur $\bigl\{r< |z|<R \bigr\}$ et comme $\psi $ est de classe $\cl$ au voisinage de $\bigl\{|z|\leq r\bigr\}\cup \bigl\{|z|\geq R\bigr\}$ alors la m\'etrique $h_{n,\rho}$ est de classe $\cl$.\\

Commencons par remarquer que
\[
 h_X\Bigl(\dif,\dif\Bigr)^{-\frac{1}{2}}\dif \bigl(\log \frac{h_{n,\rho}}{h_{n-1,\rho}} \bigr)=0,
\]
sur $\p^1\setminus \Bigl(\bigl\{|z|\leq \frac{r}{2} \bigr\}\cup \bigl\{|z|\geq 2R\bigr\}\Bigr)$.\\

On pose pour tout $u\in \R$:
\begin{align*}
\mathcal{C}_{n,i}(u)&:=\log h_{n,i}\bigl(\exp(-u)\bigr)\, \quad\text{pour}\; i=1,2,\\
 \mathcal{C}_{n}(u)&:=\mathcal{C}_{n,1}(u)-\mathcal{C}_{n,2}(u)=\bigl(\log h_n\bigl(\exp(-u)\bigr) \bigr),\\
  \mathcal{C}_{n,\rho}(u)&:=\log h_{n,\rho}\bigl(\exp(-u)\bigr),\\
  \mathcal{C}_{\infty}(u)&:=\log h_{\infty}\bigl(\exp(-u)\bigr).
\end{align*}

On a
\[
\begin{split}
 \mathcal{C}_{n,\rho}(u)&:=\rho(\exp(-u))\mathcal{C}_\infty(u)+\bigl(1-\rho(\exp(-u)) \bigr)\mathcal{C}_n(u)\\
&=\widetilde{\rho}(u)\mathcal{C}_\infty(u)+\bigl(1-\widetilde{\rho}(u) \bigr)\mathcal{C}_n(u).\\
\end{split}
\]
Donc,
\[
\mathcal{C}_{n,\rho}'(u)=\widetilde{\rho}'(u)\bigl(\mathcal{C}_\infty(u)-\mathcal{C}_n(u) \bigr)+\widetilde{\rho}(u)\mathcal{C}'_\infty(u)+\bigl(1-\widetilde{\rho}(u)\bigr)\mathcal{C}'_n(u)\quad \forall \,u\in \R,
\]
 Notons que cette quantit\'e est bien d\'efinie.\\

Puisque $\mathcal{C}_{n,i}$ est concave et $\cl$ pour $i=1,2$ et $n\in \N$, alors d'apr\`es
\eqref{convexefonction1}, $\mathcal{C}'_n$ est uniform\'ement born\'ee sur compact de $\R$. En notant que
$\bigl|\frac{\pt \mathcal{C}_\ast}{\pt u}\bigr|=|z|\bigl|\frac{\pt }{\pt z}\log h_\ast \bigr|$ cela permet
d'affirmer  l'assertion  \eqref{bellemajorationfaible2}.\\

Montrons qu'il existe une m\'etrique positive $h'$ telle que $h_{n,\rho}\otimes h'$ soit positive pour
$n>>1$. Donc, en language de la th\'eorie des fonctions concaves, il suffit de trouver une fonction concave
$\mathcal{C}'$ telle que
\[
 \mathcal{C}''_{n,\rho}(u)+\mathcal{C}''(u)\leq 0, \quad \forall n\in \N, \,\forall u\in \R.
\]

Rappelons que
\[
\frac{\pt^2}{\pt z\pt\z}\psi= -\frac{1}{4}\exp(2u)\mathcal{C}''(u),
\]

o\`u $\psi$ est $\cl$ qui v\'erifie $\psi(z)=\psi(|z|)$ et $\mathcal{C}(u):=\psi(\exp(-u)), \forall u\in \R$.\\

 On a
\[
\mathcal{C}_{n,\rho}''(u)=\widetilde{\rho}''(u)\bigl(\mathcal{C}(u)-\mathcal{C}_n(u)\bigr)+2\widetilde{\rho}'(u)\bigl(\mathcal{C}'(u)-\mathcal{C}'_n(u)\bigr)+\widetilde{\rho}(u)\mathcal{C}''(u)+\bigl(1-\widetilde{\rho}(u)\bigr)\mathcal{C}_n''(u)\quad \forall u\in \R.
\]

 Notons encore que cette quantit\'e est bien d\'efinie.\\

Sur $\bigl\{u\leq -\log(2R)  \bigr\}\cup\bigl\{u\geq -\log \frac{r}{2} \bigr\}=\bigl\{|z|\geq 2R\bigr\}\cup \bigl\{|z|\leq \frac{r}{2} \bigr\}$, on a
\[
\mathcal{C}_{n,\rho}''(u)=\mathcal{C}''_\infty(u),
\]
Par suite, $\mathcal{C}_{n,\rho}''$ est n\'egative sur cet ensemble.\\

Sur $\bigl\{\frac{r}{2}\leq |z|\leq 2R \bigr\}=\bigl\{ -\log 2R \leq u\leq -\log \frac{r}{2}   \bigr\}=:K$, $K$ est un compact de $\R$. donc $\mathcal{C}'_n$ est uniform\'ement born\'e sur $K$. On voit donc que $\mathcal{C}_{n,\rho}''(u)$ est une somme de deux termes, un terme qui est n\'egatif sur $\R$, et le deuxi\`eme est born\'e uniform\'ement en $n$ sur $K$.\\

Si l'on consid\`ere par exemple la m\'etrique de Fubini-Study $h_{FS}$ et on pose $\mathcal{C}_{FS}(u)=-\log \bigl(1+e^{-2u} \bigr)$, alors, par un  calcul direct, on montre que
\[
\mathcal{C}_{FS}''(u)=-\frac{4e^{-2u}}{\bigl(1+ e^{-2u} \bigr)^2}\leq \max\bigl(\mathcal{C}_{FS}''(a), \mathcal{C}_{FS}''(1),\mathcal{C}_{FS}''(b) \bigr) \quad \forall a<u<b.
\]

Donc sur $K$, on aura
\[
\mathcal{C}_{FS}''(u)\leq \max\biggl(-\frac{16R^2}{\bigl(1+4R^2 \bigr)^2},-\frac{4e^{-2}}{\bigl(1+e^{-2} \bigr)^2},-\frac{r^2}{\bigl(1+\frac{1}{4}r^2 \bigr)^2} \biggr).
\]

Donc, on peut trouver $N\in \N$ tel que
\[
 \mathcal{C}_{n,\rho}''(u)+N\mathcal{C}_{FS}''(u)\leq 0 \quad \forall u\in \R,
\]

ce qui veut dire que
\[
 h_{n,\rho}\otimes h_{FS}^{\otimes N},
\]

est positive, donc admissible par \eqref{positifadmissible}.
\end{proof}

\begin{definition}\label{1-integrable}
Soit $\bigl(X,\omega\bigr)$ une surface de Riemann compacte et $\omega_X$ une forme de K\"ahler. Soit $\overline{E}=\bigl(E,h_\infty\bigr)$ un fibr\'e en droites int\'egrable sur $X$.

On dit que $\bigl(E,h_\infty \bigr)$ est $1$-int\'egrable s'il existe une suite $\bigl( h_n\bigr)_{n\in \N}$ de m\'etriques de classe $\cl$ sur $E$ qui converge uniform\'ement vers $h_\infty$ v\'erifiant que:
\begin{enumerate}
\item Il existe $\overline{E}_1$ et $\overline{E}_2$ deux fibr\'es en droites admissibles tels que $\overline{E}=\overline{E}_1\otimes \overline{E}_2^{-1}$ et deux suites $\bigl(h_{1,n} \bigr)_{n\in\N}$ et $\bigl(h_{2,n} \bigr)_{n\in \N}$ de m\'etriques positives de classe $\cl$ sur $E_1$ respectivement sur $E_2$ telles que $\bigl( h_{1,n}\bigr)_{n\in \N}$ respectivement $ \bigl(h_{2,n}\bigr)_{n\in \N}$ converge uniform\'ement vers $h_{1,\infty}$ respectivement vers  $h_{2,\infty}$ et que
\[
 h_n=h_{1,n}\otimes h_{2,n}^{-1}\quad \forall n\in \N.
\]
ou il existe  $h'$, une m\'etrique  positive de classe $\cl$ telle que $h_n\otimes h'$ est positive, $\forall n\in \N$.

\item
\[
 \sum_{n=1}^\infty \Bigl\| \frac{h_n}{h_{n-1}}-1\Bigl\|_{\sup}^{\frac{1}{2}}<\infty,
\]
\item
\[
 \sup_{n\in \N_{\geq 1}}\Bigl\| h_X\Bigl( \dif,\dif\Bigr)^{-\frac{1}{2}}\dif\log \frac{h_n}{h_{n-1}}\Bigr\|_{\sup}<\infty.
\]
\end{enumerate}

\end{definition}

On note par $ \widehat{\mathrm{Pic}}_{int,1}\bigl(X\bigr)$ le sous-ensemble de $\widehat{\mathrm{Pic}}_{int}\bigl(X\bigr)$ form\'e par les classes d'isomorphie isom\'etriques des fibr\'es en droites $1$-int\'egrables.
\begin{proposition}
Soit $X$ une surface de Riemann compacte. On a
$\widehat{\mathrm{Pic}}_{int,1}\bigl( X\bigr)$ est un sous-groupe de $\widehat{\mathrm{Pic}}_{int}\bigl(X\bigr)$ qui contient $\widehat{\mathrm{Pic}}\bigl(X \bigr)$. \\

Si $X=\p^1$, alors
\[
\widehat{\mathrm{Pic}}_{int,\mathcal{S},0}\bigl(\p^1\bigr)\subset \widehat{\mathrm{Pic}}_{int,1}\bigl( \p^1\bigr).
\]
\end{proposition}
\begin{proof}
Soit $\bigl(L_1,h_\infty^{(1)}\bigr)$ et $\bigl(L_2,h_\infty^{(2)}\bigr)$ deux fibr\'es en droites munis de
m\'etriques $1$-int\'egrables. On consid\`ere $\bigl( h_{n,1}^{(1)}\bigr)_{n\in \N}$ et $\bigl(h_{n,2}^{(2)}
\bigr)_{n\in \N}$ deux suites comme dans la d\'efinition \eqref{1-integrable} associ\'ees \`a  $h_\infty^{(1)}$
respectivement \`a  $h_\infty^{(2)}$. On a
\begin{align*}
\sup_{n\in \N_{\geq 1}}\biggl\| h_X\Bigl( \dif,\dif\Bigr)^{-\frac{1}{2}}\dif\log \biggl(\frac{h_n^{(1)}\otimes
h_n^{(2)}}{h_{n-1}^{(1)}\otimes h_{n-1}^{(2)}}\biggr)\biggr\|_{\sup}&\leq
 \sup_{n\in \N_{\geq 1}}\biggl\| h_X\Bigl( \dif,\dif\Bigr)^{-\frac{1}{2}}\dif\log
 \frac{h_n^{(1)}}{h_{n-1}^{(1)}}\biggr\|_{\sup}\\
& +\sup_{n\in \N_{\geq 1}}\biggl\| h_X\Bigl( \dif,\dif\Bigr)^{-\frac{1}{2}}\dif\log
\frac{h_n^{(2)}}{h_{n-1}^{(2)}}\biggr\|_{\sup}.
\end{align*}
Du lemme \eqref{lemme}, il existe $C$ et $C'$ deux constantes positives telles que   pour $n\gg 1$
{\allowdisplaybreaks
\begin{align*}
 \biggl\| \biggl(\frac{h_n^{(1)}\otimes h_n^{(2)}}{h_{n-1}^{(1)}\otimes
 h_{n-1}^{(2)}}\biggr)-1\biggr\|_{\sup}^{\frac{1}{2}}&\leq C\biggl\|\log \biggl(\frac{h_n^{(1)}\otimes
 h_n^{(2)}}{h_{n-1}^{(1)}\otimes h_{n-1}^{(2)}}\biggr) \biggr\|_{\sup}^{\frac{1}{2}}\\
&\leq C \biggl(\biggl\| \log\frac{h_n^{(1)}}{h_{n-1}^{(1)}}\biggr\|_{\sup} +\biggl\|
\log\frac{h_n^{(2)}}{h_{n-1}^{(2)}}\biggr\|_{\sup}\biggr)^{\frac{1}{2}}\\
&\leq C\biggl(\biggl\| \log\frac{h_n^{(1)}}{h_{n-1}^{(1)}}\biggr\|_{\sup}^{\frac{1}{2}}+\biggl\|
\log\frac{h_n^{(2)}}{h_{n-1}^{(2)}}\biggr\|_{\sup}^{\frac{1}{2}} \biggr)\\
&\leq C'\biggl(\biggl\| \frac{h_n^{(1)}}{h_{n-1}^{(1)}}-1\biggr\|_{\sup}^{\frac{1}{2}}+\biggl\|
\frac{h_n^{(2)}}{h_{n-1}^{(2)}}-1\biggr\|_{\sup}^{\frac{1}{2}} \biggr),
\end{align*}}
ce qui donne,
\[
\sum_{n\in \N} \biggl\| \frac{h_n^{(1)}\otimes h_n^{(2)}}{h_{n-1}^{(1)}\otimes h_{n-1}^{(2)}}-1\biggr\|_{\sup}^{\frac{1}{2}}<\infty.
\]
Donc,
\[
\bigl(L_1,h^{(1)}\bigr)\otimes \bigl(L_2,h^{(2)}\bigr)\in \widehat{\mathrm{Pic}}_{int,1}\bigl(X\bigr).
\]

Soit $\bigl( L_1^{\ast},h^{(1)\ast}\bigr)$ le dual de $\bigl(L_1,h^{(1)}\bigr)$, on a
\begin{align*}
 \sum_{n=1}^\infty \Bigl\| \frac{h_n^{(1)\ast}}{h_{n-1}^{(1)\ast}}-1\Bigl\|_{\sup}^{\frac{1}{2}}&\leq  \sum_{n=1}^\infty \Bigl\| \frac{h_{n-1}^{(1)}}{h_n^{(1)}}\Bigr\|_{sup}^{\frac{1}{2}}\Bigl\| \frac{h_n^{(1)}}{h_{n-1}^{(1)}}-1\Bigl\|_{\sup}^{\frac{1}{2}}\\
 &\leq \sup_{k\in \N_{\geq 1}}\Bigl\| \frac{h_{k-1}}{h_k}\Bigr\|_{sup}^{\frac{1}{2}}\sum_{n=1}^\infty \Bigl\| \frac{h_n^{(1)}}{h_{n-1}^{(1)}}-1\Bigl\|_{\sup}^{\frac{1}{2}},
\end{align*}
comme $\bigl(h_n^{(1)}\bigr)_{n\in \N}$ converge uniform\'ement vers $h_\infty^{(1)}$ alors $\sup_{k\in \N_{\geq 1}}\Bigl\| \frac{h_{k-1}}{h_k}\Bigr\|_{sup}$ est fini. On conclut que
\[
\bigl( L_1^{\ast},h^{(1)\ast}\bigr)\in \widehat{\mathrm{Pic}}_{int,1}\bigl(X\bigr).
\]
On a montr\'e donc que $\widehat{\mathrm{Pic}}_{int,1}\bigl(X\bigr)$ est un sous-groupe de $\widehat{\mathrm{Pic}}_{int}\bigl(X\bigr)$
 D'apr\`es le th\'eor\`eme \eqref{bellemajorationfaible},
\[
\widehat{\mathrm{Pic}}_{int,\mathcal{S},0}\bigl(\p^1\bigr)\subset \widehat{\mathrm{Pic}}_{int,1}\bigl( \p^1\bigr).
\]

\end{proof}

\begin{theorem}\label{existence1integrable}
Soit $X$ une surface de Riemann compacte. On a $\widehat{\mathrm{Pic}}\big(X\bigr)$ est un sous-groupe propre de $\widehat{\mathrm{Pic}}_{int,1}\big(X\bigr)$.
\end{theorem}
\begin{proof}
Soit $U$ une carte locale de $X$. On peut prendre dans \eqref{integrableregular}, $X=U$. Par suite, on peut trouver
une  fonction $g$ continue non $\cl$ sur $U$ v\'erifiant les hypoth\`eses de ce th\'eor\`eme. Comme $U$ est isomorphe \`a  un ouvert de
$\CC$, on peut supposer que $g$ est invariant par $\s$ par cet isomorphisme sur un ouvert assez petit de $U$. En
recollant convenablement cette fonction, on obtient une fonction continue sur $X$ qui est $\cl$ sur $X\setminus
\overline{U}$ et dont sa restriction sur un ouvert de $X$ correspond \`a  $g$. On consid\`ere $L$ un fibr\'e en droites holomorphe sur $X$ et on le munit de le munit de la m\'etrique suivante
\[
 h_g=h\exp(g),
\]
o\`u $h$ est une m\'etrique hermitienne de classe $\cl$.  Apr\`es on adapte la preuve du  th\'eor\`eme \eqref{bellemajorationfaible}, pour montrer que cette m\'etrique est $1$-int\'egrable.\\
\end{proof}

\begin{lemma}\label{lemme}
Soit $X$ un espace topologique compact. On d\'esigne par $\mathcal{C}^0(X,\R)$ l'espace des fonctions continues sur $X$ \`a  valeurs r\'eelles muni de la norme sup. Soit $\phi\in \mathcal{C}^0(X,\R)$ v\'erifiant $\bigl\|\phi-1\bigr\|_{\sup}<\eps<\frac{1}{2}$, alors   on

\[
\frac{1}{1+2\eps}\bigl| \log \phi(x)\bigr|\leq \log \bigl|\phi(x)-1\bigr|\leq \frac{1}{1-2\eps} \bigl|\log \phi(x)\bigr| \quad \forall\, x\in X.
\]
\end{lemma}

\begin{proof}
Puisque $\bigl\|\phi-1\bigr\|_{\sup}<\eps<\frac{1}{2}$ alors $0<1-\eps<\phi(x)<1+\eps$, $\forall \, x\in X$.

Donc, pour tout $x\in X$
{\allowdisplaybreaks
\begin{align*}
\log\phi(x)&=\log\Bigl( \phi(x)-1+1\Bigr)\\
&=\sum_{l\geq 1}\frac{(-1)^{l+1}}{l}\Bigl(\phi(x)-1\Bigr)^{l}\\
&=\bigl(\phi(x)-1\bigr)\Bigl(1+\sum_{l\geq 2}\frac{(-1)^{l+1}}{l}(\phi(x)-1)^{l-1}\Bigr),
\end{align*}
}
donc
\begin{align*}
\frac{|\log \phi(x)|}{\Bigl|1+\bigl|\sum_{l\geq 2}\frac{(-1)^{l+1}}{l}(\phi(x)-1)^{l-1}\bigr|\Bigr|} \leq \bigl|\phi(x)-1\bigr|\leq \frac{\bigl|\log \phi(x)\bigr|}{\Bigl|1-\bigl|\sum_{l\geq 2}\frac{(-1)^{l+1}}{l}(\phi(x)-1)^{l-1}\bigr|\Bigr|}\quad \forall \, x\in X.
\end{align*}

Comme \[\Bigl|\sum_{l\geq 2}\frac{(-1)^{l+1}}{l}(\phi(x)-1)^{l-1}\Bigr|\leq \frac{1}{\eps}\sum_{l\geq 1}\frac{\eps^{l+1}}{l+1}=\frac{1}{\eps}\Bigl(-\log(1-\eps)-\eps\Bigr)\leq 2\eps\quad\forall\, x\in X, \]
alors
\begin{align*}
\frac{|\log \phi(x)|}{1+2\eps} \leq \bigl|\phi(x)-1\bigr|\leq \frac{\bigl|\log \phi(x)\bigr|}{1-2\eps}\quad \forall \, x\in X.
\end{align*}

\end{proof}

\section{Le Laplacien g\'en\'eralis\'e sur une surface de Riemann compacte}
Soit $(X,h_X)$ une surface de Riemann compacte munie d'une m\'etrique continue $h_X$ et $\overline{E}=(E,h_E)$ un fibr\'e hermitien holomorphe.
Le but de ce paragraphe est l'extension de la notion du Laplacien g\'en\'eralis\'e $\Delta_{\overline{E}}$ associ\'e aux m\'etriques $\cl$ aux m\'etriques int\'egrables
$h_E$ sur $E$ et agissant sur $A^{(0,0)}(X,E)$, en fixant $h_X$ qu'on suppose continue.\\

On commence par rappeler  la construction de  l'op\'erateur  Laplacien $\Delta_{\overline{E}}$ agissant sur $A^{(0,0)}(X,E)$,
 o\`u $h_E$ sera $\cl$ et notera que
cette construction est valable si $h_X$ est uniquement continue. On en donnera une expression locale. Par approximation et en
utilisant la positivit\'e, on montre qu'on peut \'etendre cette notion aux m\'etriques int\'egrables $h_E$ sur $E$, c'est le th\'eor\`eme
\eqref{lapintconv}. Lorsqu'on  consid\`ere une m\'etrique int\'egrable invariante par $\s$, alors on montre dans \eqref{lapintconv22}
qu'on peut d\'efinir directement un op\'erateur   qui \'etend celui dans le cas $\cl$. On notera \`a  l'aide de l'exemple
\eqref{contreexemplelap11}, que l'int\'egrabilit\'e est \'el\'ement important dans cette th\'eorie.\\

\subsection{Le Laplacien g\'en\'eralis\'e associ\'e aux m\'etriques $\cl$ sur $E$, rappel}\label{rappelLAPCLA}
Soit $h_X$ une m\'etrique hermitienne continue sur $TX$, et $h_E$ une m\'etrique hermitienne $\cl$ sur $E$.
Soit $\omega_0$ la forme de K\"ahler normalis\'ee associ\'ee \`a  $h_X$, donn\'ee dans chaque carte locale sur $X$ par
\[
 \omega_0=\frac{i}{2\pi}h_X\Bigl(\frac{\partial}{\partial z_\al},\frac{\partial}{\partial z_\al}\Bigr)dz_\al\wedge d\z_{\al}.
\]

Cette m\'etrique induit une m\'etrique sur les formes diff\'erentielles de type $(0,1)$. En tensorisant par la m\'etrique hermitienne de $E$, on obtient un produit scalaire ponctuel en $x\in X$: $(s(x),t(x))$ pour deux sections de $A^{0,q}(X,E)=A^{0,q}(X)\otimes _{\cl(X)}A^0(X,E)$, pour $q=0 $ ou $1$.\\

Le produit scalaire $L^2$ de deux sections $s,t\in A^{0,q}(X,E)$ est d\'efini par la formule
{{}
\[
 (s,t)_{L^2}=\int_X \bigl(s(x),t(x)\bigr)\omega_0.
\]}

L'op\'erateur de Cauchy-Riemann $\overline{\partial}_E$ agit sur les formes de types $(0,q)$ \`a  valeurs dans $E$. On obtient le complexe de Dolbeault
{{}
\[
0\lra A^{0,0}(X,E)\overset{\overline{\partial}_E}{\longrightarrow}A^{0,1}(X,E)\lra 0 %\overset{\overline{\partial}_E}{\longrightarrow}\cdots  \overset{\overline{\partial}_E}{\longrightarrow}A^{0,q}(X,E)\overset{\overline{\partial}_E}{\longrightarrow}\cdots
\]
}
Sa cohomologie calcule la cohomologie du faisceau de $X$ \`a  coefficients dans $E$, cf. par exemple \cite{GH}.\\

 On va rappeler la construction de $\overline{\pt}_{E}^\ast$: l'adjoint de $\overline{\pt}_{E}$
. Comme, par exemple, dans \cite[Chapitre. 5]{Voisin}, on consid\`ere les deux applications suivantes:
\[
 \ast_{0,E}:A^{0,0}(X,E)\lra A^{1,1}(X,E^\ast),\]
et
\[
\ast_{1,E}: A^{0,1}(X,E)\lra A^{1,0}(X,E^\ast).\\
\]

Ce sont les uniques applications qui v\'erifient:
\[
 \bigl(f(x)\sigma_x\bigr)\wedge \ast_{0,E} \bigl(g(x)\tau_x\bigr)=\bigl(\sigma(x),\tau (x)\bigr)_x\omega_x,
\]
et
\[
\bigl(f d\overline{z}\otimes \si\bigr) \wedge\bigl(\ast_{1,E}(gd\overline{z}\otimes \tau\bigr) )=\bigl(f\,d\overline{z}(x),g\,d\overline{z}(x)\bigr)_xh_E(\si,\tau)\omega_0(x),
\]
pour tout $x\in X$ et pour tous $f,g\in A^{0,0}(X)$ et $\si,\tau$ deux sections locales de $E$ tels que $f\otimes \si$ et $g\otimes
\tau$ soient des \'el\'ements de $A^{(0,0)}\bigl(X,E\bigr)$. Notons que pour d\'efinir ces deux applications, on n'a pas besoin que
$h_X$ soit $\cl$.\\

On montre que ces morphisme s'\'ecrivent  respectivement sur une carte locale, comme suit:

\[
 \ast_{0,E} (g\otimes \tau)=\overline{g}\omega_0\otimes h_E(\cdot,\tau).
\]
et
\begin{equation}\label{starhodge}
\ast_{1,E}(gd\overline{z}\otimes \tau)=-\overline{g} dz\otimes h_E(\cdot, \tau).\\
\end{equation}

L'op\'erateur $\overline{\pt}_E$ poss\`ede un adjoint pour le produit scalaire $L^2$; c'est \`a  dire une application
{{}
\[
 \overline{\partial}^\ast_E:A^{0,1}(X,E)\lra A^{0,0}(X,E)
\]}
qui v\'erifie
\[
 \bigl(s,\overline{\partial}^\ast_E t\bigr)_{L^2}=\bigl(\overline{\partial}_Es,t\bigr)_{L^2}.
\]

pour tout $s\in A^{0,q}(X,E)$ et $t\in A^{0,q+1}(X,E)$.\\

L'op\'erateur $\overline{\partial}^\ast_E:A^{0,1}(X,E)\lra A^{0,0}(X,E) $ est, par d\'efinition, donn\'e par la formule
\[
\overline{\partial}^\ast_E=-\ast_{0,E}^{-1}\overline{\partial}_{K_X\otimes E^\ast}\ast_{1,E}.\\\]

On note par $\Delta_{\overline{E}}^0$, ou plus simplement $\Delta_{\overline{E}}$,  l'op\'erateur $\overline{\partial}^\ast_E
\overline{\partial}_E$ sur $A^{0,0}(X,E)$, et l' appelle l'op\'erateur  Laplacien g\'en\'eralis\'e, cf. par exemple \cite[Chapitre 5]{Voisin}.\\

\begin{remarque}
\rm{ Remarquons que m\^eme si $h_X$ n'est pas $\cl$, alors $\Delta_{\overline{E}}=\overline{\partial}^\ast_E
\overline{\partial}_E$ est bien d\'efini.}
\end{remarque}

%Pour tout entier $q\geq 0$, l'op\'erateur $\Delta^q_E=\overline{\partial}_E \overline{\partial}^\ast_E+\overline{\partial}^\ast_E \overline{\partial}_E$ sur $A^{0,q}(X,E)$ sera appel\'e l'op\'erateur  Laplacien g\'en\'eralis\'e, cf. par exemple \cite[§. 5]{Voisin}.\\

%On suppose que $E$ est  engendr\'e par ses sections globales, et on montre le lemme ci-dessus qui sera utilis\'e pour expliciter l'expression du Laplacien g\'en\'eralis\'e $\Delta_E$.

\begin{lemma}\label{decompositionsection}
Soit $E$ une fibr\'e en droites engendr\'e par ses sections globales et $\{e_1,\ldots,e_r\}$ une base de $H^0(X,E)$ sur $\CC$, alors
\[
A^{0,0}(X,E)=\sum_{i=1}^rA^{0,0}(X)\otimes e_i.
\]
\end{lemma}
\begin{proof} Soit $\{e_1,\ldots,e_r\}$ une base de $H^0(X,E)$.

On pose $U_i:=X\setminus \mathrm{div}(e_i)$, pour $i=1,\ldots,r$. Comme $E$ est engendr\'e par ses sections globales alors
{{}
\[
X=\bigcup_{i=1}^r U_i.
\]}
Puisque $X$ est compacte, alors  il existe une partition de l'unit\'e subordonn\'ee au recouvrement $(U_i)_{1\leq i\leq r}$ c'est \`a  dire il existe $\rho_1, \rho_2,\ldots,\rho_r$ des fonctions $\cl$ sur $X$ telles que:

\begin{enumerate}
\item  Le support de $\rho_i$ est inclus dans $U_i$ et $0\leq \rho_i\leq 1$, pour $i=1,\ldots,r$.
\item $\sum_{i=1}^r \rho_i(x)=1$, $\forall x\in X$.\\
\end{enumerate}

 Soit $\xi\in A^{0,0}(X,E)$. Sur $U_i$ pour $i=1,\ldots,r$, la section $\xi$ corresponds, apr\`es trivialisation de $E$ par $e_i$, \`a  une fonction de classe $\cl$. En d'autres termes, il existe $f_i$, une fonction de classe  $\cl$ d\'efinie sur $U_i$ telle que
{{}
\begin{equation}\label{exploc}
\xi_i=f_i\otimes e_i.
\end{equation}}
 On a $\rho_i f_i$ est une fonction $\cl$ d\'efinie sur $X$ entier.\\

  V\'erifions que
{{}
\[
 \sum_{i=1}^r(\rho_i f_i) \otimes e_i=\xi.
 \]}
 Soit $x\in X$, on pose $I_x:=\{i\in \{1,\ldots,r\}|\, x\in U_i\}$. On a
{{}
\[
 \begin{split}
 \sum_{i=1}^r (\rho_i f_i\otimes e_i)(x)&= \sum_{i\in I_x} (\rho_i f_i\otimes e_i)(x)+\sum_{i\in \{1,\ldots,r\}\setminus I_x}(\rho_i f_i \otimes e_i)(x)\\
 &=\sum_{i\in I_x}\rho_i(x) \xi(x)+0, \quad \text{par}\quad \eqref{exploc}\\
&=\bigl(\sum_{i\in I_x}\rho_i(x)\bigr)\xi(x)\\
 &=\bigl(\sum_{i\in I_x}\rho_i(x)+\sum_{i\notin I_x}\rho_i(x)\bigr)\xi(x)\\
 &=\xi(x).
 \end{split}
 \]}
{{}  Pour la derni\`ere \'egalit\'e, r\'esulte du fait  {{} $1=\sum_{i=1}^r\rho_i(x)=\sum_{i\in I_x}\rho_i(x)+\sum_{i\in \{1,\ldots,r\}\setminus I_x}\rho_i(x)=\sum_{i\in I_x}\rho_i(x)$.}}\\

\end{proof}

%\subsection{Expression locale du Laplacien}\label{paragraphelapE1}

\begin{lemma}
Soit $(X,h_X)$ une surface de Riemann compacte avec $h_X$ est continue, et $(E,h_E)$ un fibr\'e hermitien avec $h_E$ de classe $\cl$.
L'op\'erateur $\Delta_{\overline{E}}$ est d\'etermin\'e localement par:
\begin{equation}\label{explap11}
\begin{split}
 \Delta_{\overline{E}}(f\otimes \si)=
&=-h_X\Bigl(\frac{\partial}{\partial z},\frac{\partial}{\partial z}\Bigr)^{-1}h_E\bigl(\sigma,\sigma\bigr)^{-1}\frac{\partial}{\partial z}\Bigl(h_E(\si,\si)\frac{\partial f}{\partial \overline{z}}  \Bigr)\otimes \si\\
&= -h_X\Bigl(\frac{\partial}{\partial z},\frac{\partial}{\partial z}\Bigr)^{-1}\frac{\pt^2 f}{\pt z \pt \z}\otimes \si-h_X\Bigl(\frac{\partial}{\partial z},\frac{\partial}{\partial z}\Bigr)^{-1}\dif \bigl(\log h_E(\si,\si)\bigr)\,\frac{\partial f}{\partial \overline{z}}  \otimes \si.
\end{split}
\end{equation} $\forall f\in A^{0,0}(X)$ et $\si$ une section locale holomorphe de $E$ tels que $f\otimes \si\in A^{(0,0)}\bigl(X,E\bigr)$,
o\`u $\{\dif\}$ est une base locale de $TX$.

%Si $X=\p^1$ ou $TX$ et $E$ sont engendr\'es  par leurs sections globales, alors l'expression pr\'ec\'edente est valable sur $X$ entier.

\end{lemma}

\begin{proof}
%Soit $\xi\in A^{0,0}(X,E)$, %par le lemme pr\'ec\'edant, on peut supposer que  $f\otimes e$ o\`u $f\in \cl(X)$ et $e$ une section globale de $E$, donc il suffit de d\'eterminer l'expression de $\Delta_E(f\otimes \si)$.\\
Soient $f\in A^{0,0}(X)$ et $\si$ une section locale holomorphe tels que $f\otimes \si\in A^{(0,0)}\bigl(X,E\bigr)$.

On a localement
{\allowdisplaybreaks
\begin{align*}
\Delta_{\overline{E}}(f\otimes\sigma)&= \overline{\pt}_E^\ast \overline{\pt}_E\bigl(f\otimes \si \bigr) \\
&=-\ast_{0,E}^{-1}\overline{\partial}_{K_X\otimes E^\ast}\ast_{1,E}\bigl( \overline{\partial}_E(f\otimes \sigma)\bigr)\\
&=-\ast^{-1}_{0,E}\overline{\partial}_{K_X\otimes E^\ast}\ast_{1,E}\Bigl(\frac{\partial f}{\partial\overline{z}}d\overline{z}\otimes \sigma \Bigr)\\
&=-\ast^{-1}_{0,E}\overline{\partial}_{K_X\otimes E^\ast}\Bigl(  -\frac{\partial \overline{f}}{\partial z} dz\otimes h_E(\cdot, \sigma)    \Bigr)\\
&=-\ast^{-1}_{0,E}\overline{\partial}_{K_X\otimes E^\ast}\Bigl(  -\frac{\partial \overline{f}}{\partial z}  dz\otimes\frac{h_E(\sigma,\sigma)}{\sigma^\ast(\sigma)}\sigma^\ast    \Bigr)\\
&=-\ast^{-1}_{0,E}\overline{\partial}_{K_X\otimes E^\ast}\Bigl(  -\frac{h_E(\sigma,\sigma)}{\sigma^\ast(\sigma)}\frac{\partial \overline{f}}{\partial z}  dz\otimes\sigma^\ast    \Bigr)\\
&=-\ast^{-1}_{0,E}\biggl( \frac{\partial}{\partial \overline{z}}\Bigl(  -\frac{h_E(\sigma,\sigma)}{\sigma^\ast(\sigma)}\frac{\partial \overline{f}}{\partial z}  dz \Bigr)    \otimes \sigma^\ast\biggr)\\
&=-\ast^{-1}_{0,E}\biggl( \frac{1}{\sigma^\ast(\sigma)}\frac{\partial}{\partial \overline{z}}\Bigl(  -h_E(\sigma,\sigma)\frac{\partial \overline{f}}{\partial z}  dz \Bigr)    \otimes \sigma^\ast\biggr)\quad \quad \text{puisque}\, \si^\ast(\si)\, \text{est holomorphe},\\
&=-\ast^{-1}_{0,E}\biggl( \frac{\partial}{\partial \overline{z}}\Bigl(  -h_E(\sigma,\sigma)\frac{\partial \overline{f}}{\partial z}  dz \Bigr)    \otimes \frac{\sigma^\ast}{\sigma^\ast(\sigma)}\biggr)\\
&=-\ast^{-1}_{0,E}\biggl( \frac{\partial}{\partial \overline{z}}\Bigl(  -h_E(\sigma,\sigma)\frac{\partial \overline{f}}{\partial\overline{z}}\Bigr)d\overline{z}\wedge  dz     \otimes \frac{1}{h_E(\si,\si)}h_E(\cdot,\sigma)\biggr)\\
&=-h_X\Bigl(\dif,\dif\Bigr)^{-1}h_E(\sigma,\sigma)^{-1}\frac{\partial}{\partial z}\Bigl(h_E(\si,\si)\frac{\partial f}{\partial \overline{z}}  \Bigr)\otimes \si.
\end{align*}}
On a donc,
\[
\begin{split}
 \Delta_{\overline{E}}(f\otimes \si)&=-h_X\Bigl(\dif,\dif\Bigr)
^{-1}h_E(\sigma,\sigma)^{-1}\frac{\partial}{\partial z}\bigl(h_E(\si,\si)\frac{\partial f}{\partial
\overline{z}}  \bigr)\otimes \si\\
&= -h_X\Bigl(\dif,\dif\Bigr)^{-1}\frac{\pt^2 f}{\pt z \pt \z}\otimes \si-h_X\Bigl(\dif,\dif\Bigr)^{-1}\dif \log h_E(\si,\si)\frac{\partial f}{\partial \overline{z}}  \otimes \si.
\end{split}
\]

%Dans la suite  on notera par $\Delta_h$ le Laplacien associ\'e \`a  $(E,h)$.\\

%\begin{equation}\label{laplacien}
 %\Delta(f\otimes \si)=-h(\frac{\partial}{\partial z},\frac{\partial}{\partial z})^{-1}\frac{\partial \log h(\si,\si)}{\partial z}\frac{\partial f}{\partial \overline{z}}  \otimes \si -h(\frac{\partial}{\partial z},\frac{\partial}{\partial z})^{-1}\frac{\pt^2 f}{\pt z \pt \z}\otimes \si
%\end{equation}
\end{proof}

\begin{lemma}\label{formesimple}
Soient $\si$ et $\tau $ sont  deux sections holomorphes locales de $E$, $f$ et $g$ deux fonctions de classe $\cl$ sur $X$ tels que $f\otimes \si, g\otimes \tau\in A^{(0,0)}\bigl(X,E\bigr)$  alors
{{}
\[
\bigl(\Delta_{\overline{E}}(f\otimes \si),g\otimes \tau \bigr)_{L^2}=\frac{i}{2\pi }\int_X h_E(\si,\tau)\frac{\pt f}{\pt \z}\frac{\pt \overline{g}}{\pt z}dz\wedge d\z.
\]}
\end{lemma}

\begin{proof}
Supposons que $\tau$ est transverse \`a  $\si$ (c'est \`a  dire $div(\si)\cap div(\tau)$ est finie). Il existe localement une fonction holomorphe $\phi$ telle que $\tau=\phi \si$. On suppose que $\si=z$,  une coordonn\'ee locale. On pose $U_\eps:=\{|z|>\eps\}$, avec $0<\eps\ll 1$
{\allowdisplaybreaks
\begin{align*}
\bigl(\Delta_{\overline{E}}(f\otimes \si),g\otimes \tau\bigr)_{L^2}&=-\frac{i}{2\pi}\int_U\frac{\pt}{\pt
z}\Bigl(h(\si,\si\bigr)\frac{\pt f}{\pt \z}\Bigr)\frac{h(\si,\tau)}{h(\si,\si)}\,\overline{g}\,dz\wedge d\z\\
&=-\frac{i}{2\pi}\int_U \frac{\pt}{\pt z}\Bigl(h(\si,\si)\frac{\pt f}{\pt
\z}\Bigr)\,\overline{\phi}\,\overline{g}\, dz\wedge d\z\\
&=-\frac{i}{2\pi}\lim_{\eps\mapsto 0}\int_{U_\eps}\frac{\pt}{\pt z}\Bigl(h(\si,\si)\frac{\pt f}{\pt
\z}\Bigr)\overline{\phi}\,\overline{g}\,dz\wedge d\z\\
&=-\frac{i}{2\pi}\lim_{\eps\mapsto 0}\int_{|z|=\eps}h(\si,\si)\frac{\pt f}{\pt
\z}\,\overline{\phi}\,\overline{g}\,d\z+\frac{i}{2\pi}\lim_{\eps\mapsto 0}\int_{U_\eps}h(\si,\si)\frac{\pt f}{\pt \z}
\,\overline{\phi}\,\frac{\pt \overline{g}}{\pt z}dz\wedge d\z\\
&=-\frac{i}{2\pi}\lim_{\eps\mapsto 0}\int_{|z|=\eps}h(\si,\tau)\frac{\pt f}{\pt \z}\,
\overline{g}\, d\z+\frac{i}{2\pi }\lim_{\eps\mapsto 0}\int_{U_\eps}h(\si,\tau)\frac{\pt f}{\pt \z}\, \frac{\pt \overline{g}}{\pt z}\,dz\wedge
d\z\\
&=\frac{i}{2\pi}\int_X h\bigl(\si,\tau\bigr)\frac{\pt f}{\pt \z}\,\frac{\pt \overline{g}}{\pt z}\,dz\wedge d\z.
\end{align*}
}
\end{proof}

%$=-\frac{i}{2\pi}\int_U \frac{\pt}{\pt z}\Bigl(h(\si,\si)\frac{\pt }{\pt\z}\Bigr)\,\overline{\phi}\,\overline{g}\, dz\wedge  d\z
%$

%Si $h_0$ est une autre m\'etrique hermitienne de classe $\cl$ sur $E$ et Si l'on note par $\vf$ la fonction de classe $\cl$ qui v\'erifie
%\[
 %h_0=e^{\vf}h
%\]

%alors un calcul direct donne que
%{{}\[
 %\Delta_h(f\otimes \si)=\Delta_{h_0}(f\otimes \si)+h(\frac{\partial}{\partial z},\frac{\partial}{\partial z})^{-1}\frac{\partial \vf}{\partial z}\frac{\partial f}{\partial \overline{z}}  \otimes \si
%\]}

%On a
%{{}
%\[
 %\Bigl\| \frac{\partial \vf}{\partial z}\frac{\partial f}{\partial \overline{z}}  \otimes \si \Bigr\|^2_{L^2, h_0}
%\]}

%{{}
%\[
 %\Bigl\| h(\frac{\pt}{\pt z},\frac{\pt}{\pt z})^{-1}\frac{\partial \vf}{\partial z}\frac{\partial f}{\partial \overline{z}}  \otimes \si \Bigr\|^2_{L^2, h_0}=\int_{X} \bigl| \frac{\pt \vf}{\pt z}\bigr|^2 \bigl| \frac{\pt f}{\pt \z}\bigr|^2 \frac{h_0(\si,\si)}{h(\frac{\pt}{\pt z},\frac{\pt}{\pt z})}dz\wedge d\z
%\]}

\begin{lemma}\label{simple}
 Soient $\vf$ et $\psi$ deux fonctions r\'eelles de classe $\cl$ sur $X$. On a

\[
\frac{i}{2\pi}\int_X\Bigl|\frac{\partial \vf}{\partial \overline{z}}\Bigr|^2\psi{} dz\wedge d\overline{z}=-\frac{i}{2\pi}\int_X \vf\frac{\partial ^2 \vf}{\partial z\partial \overline{z}}\,\psi\, dz\wedge d\overline{z}+\frac{i}{4\pi}\int_X \vf^2 \frac{\partial ^2 \psi{}}{\partial z\partial \overline{z}} dz\wedge d\overline{z}.
\]
\end{lemma}

\begin{proof}
Soient $\psi$ et $\varphi$ deux fonctions r\'eelles de classe $\cl$ sur $X$. On a,
{{}\[
 \begin{split}
\int_X\Bigl|\frac{\partial \vf}{\partial \overline{z}}\Bigr|^2\psi\, dz\wedge d\overline{z}&=
\int_X\Bigl(\frac{\partial \vf}{\partial \overline{z}}\Bigr)\overline{\Bigl(\frac{\partial \vf}{\partial \overline{z}}\Bigr)} \psi{} dz\wedge d\overline{z}\\
&= \int_X\frac{\partial \vf}{\partial \overline{z}}\frac{\partial \vf}{\partial z} \psi{} dz\wedge d\overline{z} \\
&=\int_{X} \frac{\partial}{\partial z}\Bigl( \vf \frac{\partial \vf}{\partial \overline{z}} \psi\Bigr)dz\wedge d\overline{z}-\int_X \vf \frac{\partial ^2 \vf}{\partial z\partial \overline{z}} \,\psi\,dz\wedge d\z\\
&-\int_X \vf\frac{\pt \vf}{\pt \z} \frac{\pt \psi{}}{\pt z}dz\wedge d\z\\
&=-\int_X \vf \frac{\partial ^2 \vf}{\partial z\partial \overline{z}} \,\psi\,dz\wedge d\z
-\int_X \vf\frac{\pt \vf}{\pt \z} \frac{\pt \psi{}}{\pt z}dz\wedge d\z\qquad \text{par le th\'eor\`eme de {}{ Stokes}}\\
&=-\int_X \vf \frac{\partial ^2 \vf}{\partial z\partial \overline{z}} \,\psi\,dz\wedge d\z-\frac{1}{2}\int_X\frac{\pt \bigl(\vf^2\bigr)}{\pt \z}\,\frac{\pt \psi{}}{\pt z}\,dz \wedge d\z\\
&=-\int_X \vf \frac{\partial ^2 \vf}{\partial z\partial \overline{z}} \,\psi\,dz\wedge d\z+\frac{1}{2}\int_X \vf^2 \frac{\pt^2 \psi{} }{\pt z \pt\z}\,dz\wedge d\z \qquad \text{par le th\'eor\`eme de  {}{Stokes}}
 \end{split}
\]}

\end{proof}

%\subsection{Op\'erateurs \`a  coefficients discontinus}

%\textbf{Prendre l expression du laplacien comme d\'efinition, v\'erifie autodjointivit\'e...}

%\textbf{L'ensemble de discontinuit\'e est  n\'egligeable}.\\
\subsection{Le Laplacien g\'en\'eralis\'ee associ\'e aux m\'etriques int\'egrables sur $E$. (I)}\label{LGAMI1}

Soit $(X,h_X)$ une surface de Riemann compacte et $h_X$ une m\'etrique hermitienne continue non necessairement
$\cl$ sur $X$, (fix\'ee
dans ce paragrahe). \\

Soit maintenant $\overline{E}_\infty=(E,h_{E,\infty})$  un fibr\'e en droites int\'egrable. Par d\'efinition, il
existe une d\'ecomposition
$\overline{E}_\infty=\overline{E}_{1,\infty}\otimes \overline{E}_{2,\infty}^{-1}$ o\`u $
\overline{E}_{1,\infty}=(E_1,h_{1,\infty})$
 et $ \overline{E}_{2,\infty}=(E_2,h_{2,\infty})$ sont deux fibr\'es en droites admissibles, c'est \`a  dire
 qu'il existe
 $(h_{1,n})_{n\in \N}$ (resp. $(h_{2,n})_{n\in \N}$) une suite de m\'etriques $\cl$ et positives qui converge
 uniform\'ement vers $h_{1,\infty}$ (resp. $h_{2,\infty}$).

Si l'on pose $\overline{E}_n:=(E,h_n:=h_{1,n}\otimes h_{2,n}^{-1})$, pour tout $n\in \N$. Il est clair que $h_n$
est $\cl$, alors on
peut consid\'erer l'op\'erateur Laplacien $\Delta_{\overline{E}_n}$ construit dans le paragraphe pr\'ec\'edent,
 agissant sur $A^{(0,0)}(X,E)$ et
associ\'e \`a  la donn\'ee $h_X$ et $h_n$.\\

Le th\'eor\`eme suivant permet d'\'etudier la suite $(\Delta_{\overline{E}_n})_{n\in \N}$. Cela nous permettra d'associer \`a
$\overline{E}_\infty$, par un proc\'ed\'e
d'approximation, un op\'erateur lin\'eaire qui sera not\'e $\Delta_{\overline{E}_\infty}$. Cette construction \'etend la notion
d'op\'erateur Laplacien aux fibr\'es en droites int\'egrables.\\

\begin{theorem}\label{lapintconv} Soit $(X,h_X)$ une surface de Riemann compacte avec $h_X$ une m\'etrique hermitienne continue.
Soit $\overline{E}_\infty=(E,h_{E,\infty})$ un fibr\'e en droites  int\'egrable.  On note par $\overline{A^{0,0}(X,E)}_\infty$ le compl\'et\'e de $A^{0,0}(X,E)$ pour  la m\'etrique $L^2_\infty$ induite par $h_X$ et $h_{E,\infty}$.

Pour toute d\'ecomposition de $\overline{E}_\infty=\overline{E}_{1,\infty}\otimes \overline{E}_{2,\infty}^{-1}$ en fibr\'es admissibles
$\overline{E}_{1,\infty}$ et $\overline{E}_{2,\infty}$, et pour tout choix de suites $\bigl(h_{1,n}\bigr)_{n\in \N}$  (resp.
$\bigl(h_{2,n}\bigr)_{n\in \N}$) de m\'etriques positives $\cl$ sur $\overline{E}_{1,\infty}$ (resp. sur $\overline{E}_{2,\infty}$)
qui converge uniform\'ement vers $h_{1,\infty}$ (resp. $h_{2,\infty}$), pour tout $\xi \in A^{0,0}(X,E)$, et si l'on pose
$\overline{E}_n:=\overline{E}_{1,n}\otimes \overline{E}_{2,n}^{-1}$. Alors,
\begin{enumerate}
\item la suite
$\bigl(\Delta_{\overline{E}_n}\xi\bigr)_{n\in \N}$ converge, pour la norme $L^2_\infty$, lorsque $n\mapsto \infty$ vers une limite
$\Delta_{\overline{E}_\infty}\xi$ dans
$\overline{A^{0,0}(X,E)}_\infty$. Cette limite ne d\'epend par du choix de la d\'ecomposition ni de celui de la suite.

\item $\Delta_{\overline{E}_\infty}$ est un op\'erateur lin\'eaire de $A^{0,0}(X,E)$ vers $\overline{A^{0,0}(X,E)}_\infty$.

 \item
 \begin{equation}\label{ff11}
\begin{split}
\bigl(\Delta_{\overline{E}_\infty}(f\otimes \si), g\otimes \tau\bigr)_{L^2,\infty}&=\bigl(f\otimes \si, \Delta_{\overline{E}_\infty}(g\otimes \tau)\bigr)_{L^2,\infty}\\
&= \frac{i}{2\pi }\int_X h_{E,\infty}(\si,\tau)\frac{\pt f}{\pt \z}\,\frac{\pt \overline{g}}{\pt z}\,  dz\wedge d\z.
\end{split}
\end{equation}
 $\forall f,g\in A^{0,0}(X)$ et $\si,\tau$ deux sections locales holomorphes de $E$ tels que $f\otimes \si$ et $g\otimes \tau$ soient dans $A^{(0,0)}\bigl(X,E\bigr)$. En particulier,
 \[
 \bigl(\Delta_{\overline{E}_\infty}\xi, \xi'\bigr)_{L^2,\infty}=\bigl(\xi, \Delta_{\overline{E}_\infty}\xi'\bigr)_{L^2,\infty},
 \]
  $\forall \xi,\xi'\in A^{(0,0)}(X,E)$.
 \item
 \[\bigl(\Delta_{\overline{E}_\infty}\xi,\xi\bigr)_{L^2, \infty}\geq  0 \quad \forall\, \xi\in A^{0,0}(X,E).\]
 \end{enumerate}

\end{theorem}

\begin{proof}

On consid\`ere une d\'ecomposition de $\overline{E}_\infty=\overline{E}_{1,\infty}\otimes \overline{E}_{2,\infty}^{-1}$ en fibr\'es admissibles
$\overline{E}_{1,\infty}$ et $\overline{E}_{2,\infty}$. Soit  $\bigl(h_{1,n}\bigr)_{n\in \N}$  (resp.
$\bigl(h_{2,n}\bigr)_{n\in \N}$) de m\'etriques positives $\cl$ sur $\overline{E}_{1,\infty}$ (resp. sur $\overline{E}_{2,\infty}$)
qui converge uniform\'ement vers $h_{1,\infty}$ (resp. $h_{2,\infty}$).

On pose $\overline{E}_n=(E,h_n:=h_{1,n}\otimes h_{2,n}^{-1})$
pour tout $n\in \N$, et on consid\`ere $\Delta_{\overline{E}_n}$, l'op\'erateur Laplacien associ\'e \`a  $h_X$ et $h_{\overline{E}_n}$.
D'apr\`es \eqref{explap11}, on a pour tout $p,q\in \N$:
\[
  \Delta_{\overline{E}_p}\bigl(f\otimes \si\bigr)-\Delta_{\overline{E}_q}\bigl(f\otimes \si \bigr)=h_X\Bigl(\frac{\partial}{\partial z},\frac{\partial}{\partial z}\Bigr)^{-1}\frac{\partial \vf_{p,q}}{\partial z}\frac{\partial f}{\partial \overline{z}}  \otimes \si,
\]
o\`u $f\in A^{(0,0)}(X)$ et $\si$ est une section locale holomorphe et  $\vf_{p,q}:=\log \frac{h_{\overline{E}_q}}{h_{\overline{E}_p}}$, (notons que $\vf_{p,q}$ ne d\'epend pas de $\si$, ni de $f$).\\

Fixons une m\'etrique hermitienne $h_E$ de classe $\cl$ sur $E$ et notons par $\bigl(\cdot,\cdot\bigr)_{L^2,h_E,h_X}$ le produit hermitien $L^2$ induit par $h_X$ et $h_E$ sur $A^{0,0}(X,E)$. On a

\begin{equation}\label{DpDq}
 \Bigl\| \Delta_{\overline{E}_p}(f\otimes \si)-\Delta_{\overline{E}_q}(f\otimes \si )\Bigr\|^2_{L^2,h_E,h_X}=\frac{i}{2\pi}\int_{X} \Bigl| \frac{\pt \vf_{p,q}}{\pt z}\Bigr|^2 \Bigl| \frac{\pt f}{\pt \z}\Bigr|^2 \frac{h_E(\si,\si)}{h_X\bigl(\frac{\pt}{\pt z},\frac{\pt}{\pt z}\bigr)}dz\wedge d\z.
 \end{equation}

Comme $X$ est compact,  on peut supposer que   $h_X$ est de classe $\cl$ (En effet, si $h'_X$ est une m\'etrique $\cl$ sur $TX$. On a
$\frac{h_X}{h'_X}$ est une fonction continue qui ne s'annule pas sur $X$, donc les normes $\vc_{L_2,h_E,h_X}$ et
$\vc_{L^2,h_E,h'_X}$ sont \'equivalentes). \\

On pose $\psi:=  \bigl| \frac{\pt f}{\pt \z}\bigr|^2 h_X(\frac{\pt}{\pt z},\frac{\pt}{\pt z})^{-1}h_E(\si,\si)$. Montrons que $\psi\in A^{0,0}(X,\R)$,  il suffit de montrer que $ \bigl| \frac{\pt f}{\pt \z}\bigr|^2 {h_X(\frac{\pt}{\pt z},\frac{\pt}{\pt z})^{-1}}$ est  une fonction d\'efinie globalement. En effet, soit $\phi$ un changement de coordonn\'ees locales, on pose $z=\phi(y)$,  on a
\begin{align*}
 \Bigl|\frac{\pt f}{\pt \z}\Bigr|^2 h_X\Bigl( \dif,\dif\Bigr)^{-1}= \Bigl|\frac{\pt (f\circ \phi )}{\pt \overline{y}}\Bigr|^2 h_X\Bigl( \frac{\pt}{\pt y},\frac{\pt}{\pt y}\Bigr)^{-1}.
\end{align*}
On  applique maintenant le lemme \eqref{simple} et on obtient:
 \[
 \begin{split}
\Bigl\| &\Delta_{\overline{E}_p}(f\otimes \si)-\Delta_{\overline{E}_q}(f\otimes \si )\Bigr\|^2_{L^2,h_E}=\frac{i}{2\pi}\int_{X} \biggl| \frac{\pt \vf_{p,q}}{\pt z}\biggr|^2 \Bigl| \frac{\pt f}{\pt \z}\Bigr|^2 \frac{h_E(\si,\si)}{h_X(\frac{\pt}{\pt z},\frac{\pt}{\pt z})}dz\wedge d\z \\
&=\frac{i}{2\pi}\int_X \Bigl| \frac{\pt \vf_{p,q}}{\pt z}\Bigr|^2 \,\psi\, dz\wedge d\z\\
&= -\frac{i}{2\pi}\int_X \vf_{p,q}\,\frac{\partial ^2 \vf_{p,q}}{\partial z\partial \overline{z}}\,\psi\, dz\wedge d\overline{z}+\frac{i}{4\pi}\int_X \vf_{p,q}^2\, \frac{\partial ^2 \psi{}}{\partial z\partial \overline{z}}\, dz\wedge d\overline{z}\quad \text{par} \, \eqref{simple}\\
 &= -\frac{i}{2\pi}\int_X \log \Bigl(\frac{h_p}{h_{q}}\Bigr)\,\psi\,\bigl(c_1(E,h_p)-c_1(E,h_{q}) \bigr)+\frac{i}{4\pi}\int_X \biggl|\log \frac{h_p}{h_q} \biggr|^2 \frac{\partial ^2 \psi{}}{\partial z\partial \overline{z}}\,dz\wedge d\z.
 \end{split}
\]
On  conclut, par exemple par la th\'eorie de Bedford-Taylor, que
\[\Bigl\| \Delta_{\overline{E}_p}\bigl(f\otimes \si\bigr)-\Delta_{\overline{E}_q}\bigl(f\otimes \si \bigr)\Bigr\|^2_{L^2,h_E,h_X}\xrightarrow[p,q\mapsto \infty]{} 0.\]
Par suite,
\[\Bigl\| \Delta_{\overline{E}_p}\bigl(f\otimes \si\bigr)-\Delta_{\overline{E}_q}\bigl(f\otimes \si \bigr)\Bigr\|^2_{L^2,h_{E,\infty},h_X}\xrightarrow[p,q\mapsto \infty]{} 0.\]
Donc, on a montr\'e que pour tout $\xi \in A^{(0,0)}(X,E)$, la suite $\bigl(\Delta_{\overline{E}_p}\xi\bigr)_{p\in \N}$ converge pour la norme $L^2$ vers une limite dans $\overline{A^{(0,0)}(X,E)}$, on notera cette limite par $\Delta_{\overline{E}_\infty}\xi$. On v\'erifie que $\Delta_{\overline{E}_\infty}\xi$ est un op\'erateur lin\'eaire.\\

Soient $f\otimes \si$ et $g\otimes \tau $ deux \'el\'ements de  $A^{0,0}(X,E)$. On a
 {{}\[
 \begin{split}
\bigl(\Delta_{\overline{E}_n}\bigl(f\otimes \si\bigr), g\otimes \tau\bigr)_{L^2,n}&=\frac{i}{2\pi}\int_X h_{\overline{E}_n}\bigl(\si,\tau\bigr)\,\frac{\pt f}{\pt \z}\,\frac{\pt \overline{g}}{\pt z}\,  dz\wedge d\z \quad \text{par}\,\, \eqref{formesimple},\\
 \end{split}
\]}
donc par passage \`a  la limite on d\'eduit \eqref{ff11}.\\

En particulier
\[
0\leq \bigl(\Delta_{\overline{E}_n}\bigl(f\otimes \si\bigr), f\otimes \si\bigr)_{L^2,n}=\frac{i}{2\pi}\int_X h_{\overline{E}_n}(\si,\si)\Bigl|\frac{\pt f}{\pt \z}\Bigr|^2  dz\wedge d\z,
\]
 donc,
\[
 \begin{split}
  \bigl(\Delta_{\overline{E}_\infty}\xi,\xi\bigr)_{L^2, \infty}&=\frac{i}{2\pi}\int_X \sum_{k,j} h_{\overline{E}_{\infty}}(e_k,e_j)\frac{\pt f_k}{\pt \z}\frac{\pt \overline{f_j}}{\pt z} dz\wedge d\z\\
&=\frac{i}{2\pi}\int_{X} h_{\overline{E}_\infty}\Bigl(\sum_{k=1}^r\frac{\pt f_k}{\pt \z}\otimes e_k, \sum_{k=1}^r\frac{\pt f_k}{\pt \z}\otimes e_k\Bigr)\,dz\wedge d\z\\
&\geq 0,
 \end{split}
\]
pour tout $\xi\in A^{(0,0)}(X,E)$.\\

\end{proof}

%\begin{proposition}\label{croissance}
%Soit $\overline{E}=(E,\vc)$ un fibr\'e en droites admissible sur $X$ compacte. Il existe  $(\vc_n)_{n\in \N} $une suite croissante de m\'etriques $\cl$ positives convergeant uniform\'ement vers $\vc$ et v\'erifiant
%\end{proposition}
%\begin{proof}
%Voir \cite[Proposition 4.5.7]{Maillot}.
%\end{proof}

%On doit montrer que $\Delta_{\overline{E}_\infty}$ est essentiellement auto-adjoint

\begin{lemma}\label{ordre}
Soit $\overline{E}=\bigl(E,h_{E,\infty}\bigr)$ un fibr\'e en droites int\'egrable sur $X$, et $\bigl(h_n\bigr)_{n\in \N}$ une suite croissante qui converge uniform\'ement vers $h_{E,\infty}$. On a
{{}
\[
\bigl(\Delta_{\overline{E}_n}\xi,\xi\bigr)_{L^2,n}\leq
\bigl(\Delta_{\overline{E}_{n+1}}\xi,\xi\bigr)_{L^2,n+1}\leq \bigl(\Delta_{\overline{E}_{\infty}}\xi,\xi\bigr)_{L^2,\infty},
\]}
et
{{}
\[
\ker(\Delta_{\overline{E}_\infty})\simeq H^0(X,E).
\]}
\end{lemma}

\begin{proof}
L'existence d'une suite croissante qui converge vers $h_\infty$ est assur\'ee par \cite[Proposition 4.5.7]{Maillot}. Par \eqref{formesimple}, on a pour tout $n\in \N$
{{}\[
\begin{split}
\bigl(\Delta_{\overline{E}_n}\xi,\xi\bigr)_{L^2,n}&=\frac{i}{2\pi}\int_{X}\sum_{k,j}\frac{\pt f_k}{\pt \z}\frac{\pt \overline{f_j}}{\pt z}h_n(e_k,e_j)dz\wedge d\z\\
&=\frac{i}{2\pi }\int_X h_n\Bigl(\sum_{k=1}^r \frac{\pt f_k}{\pt \z}(x)e_k(x),\sum_{k=1}^r\frac{\pt f_k}{\pt \z}(x)e_k(x)\Bigr)\,dz\wedge d\z\\
&\leq \frac{i}{2\pi }\int_X h_{n+1}\Bigl(\sum_{k=1}^r \frac{\pt f_k}{\pt \z}(x)e_k(x),\sum_{k=1}^r\frac{\pt f_k}{\pt \z}(x)e_k(x)\Bigr)\,dz\wedge d\z\\
&=\frac{i}{2\pi }\int_{X}\sum_{k,j}\frac{\pt f_k}{\pt \z}\frac{\pt \overline{f_j}}{\pt z}h_{n+1}(e_k,e_j)dz\wedge d\z\\
&=\bigl(\Delta_{\overline{E}_{n+1}}\xi,\xi\bigr)_{L^2,n+1}.
\end{split}
\]}
Par passage \`a  la limite, on obtient: $\bigl(\Delta_{\overline{E}_n}\xi,\xi\bigr)_{L^2,n}\leq
\bigl(\Delta_{\overline{E}_{n+1}}\xi,\xi\bigr)_{L^2,n+1}\leq \bigl(\Delta_{\overline{E}_{\infty}}\xi,\xi\bigr)_{L^2,\infty},
$. Si $\xi\in \ker(\Delta_{\overline{E}_{\infty}})$, alors l'in\'egalit\'e pr\'ec\'edente  implique que $(\Delta_{\overline{E}_n}\xi,\xi)=0$, par suite $\xi\in H^0(X,E)$.
\end{proof}

\begin{proposition}\label{isosection}
Soit $\overline{L}$ un fibr\'e en droites hermitien de classe $\cl$ engendr\'e par ses sections globales sur une surface de Riemann compacte, alors
{{}
\[
\ker(\Delta_{\overline{L}}) = 1\otimes H^{0}(X,L).
\]}
(o\`u $1\otimes H^0(X,L)$ est par d\'efinition le sous $\CC$-espace vectoriel de $A^{(0,0)}\bigl(X,E\bigr)$ engendr\'e par les \'el\'ements de la forme $1\otimes e$ avec $e\in H^0(X,L)$).
\end{proposition}

\begin{proof}
D'apr\`es \cite[th\'eor\`eme 5.25]{Voisin}, on sait qu'on a isomorphisme entre les deux $\CC-$espaces ci-dessus. On note par $\{e_1,\ldots,e_r\}$ une base sur $\CC$ de $H^{0}(X,L)$. Il suffit de montrer que $\Delta_{\overline{L}}(1\otimes e_i)=0$ o\`u $1$ est la fonction constante sur $X$ \'egale \`a  $1$. Par \eqref{explap11}, on voit facilement que

\[\Delta_{\overline{L}}\bigl(1\otimes e_i\bigr)=0.\]
\end{proof}

\subsection{Le Laplacien g\'en\'eralis\'e associ\'e \`a  une m\'etrique int\'egrable sur un fibr\'e en droites sur $\p^1$. (II)}\label{LGAMI2}

Dans ce paragraphe, on suppose que $X=\p^1$, la droite projective complexe.  L'action du tore compact $\s$,  sur $\p^1$ permet de
consid\'erer une sous classe de m\'etriques int\'egrables, \`a  savoir les m\'etriques invariantes par $\s$. On expliquera comment on peut
construire directement, dans ce
cas,  un op\'erateur Laplacien. Plus pr\'ecis\'ement, si $h_{\overline{E}_\infty}$ est une m\'etrique int\'egrable invariante par $\s$ sur
 $E$.   on donnera un sens \`a  l'expression suivante:
\[
\Delta'_{\overline{E}_\infty}\bigl(f\otimes \si\bigr)=-h_{\p^1}\Bigl(\dif,\dif\Bigr)^{-1}h_{\overline{E}_\infty}(\si,\si)^{-1}\frac{\pt}{\pt z}\Bigl(h_{\overline{E}_\infty}(\si,\si)\frac{\pt f}{\pt \z} \Bigr)\otimes \si.
\]
avec $f\otimes \si\in A^{(0,0)}(\p^1,E)$.\\

On terminera ce paragraphe en comparant cette approche avec celle introduite dans le paragraphe pr\'ec\'edent.\\

L'exemple suivant montre l'importance de la notion d'int\'egrabilit\'e de la m\'etrique dans l'extension de la notion du Laplacien \`a  cette classe de m\'etriques.

\begin{example}\label{contreexemplelap11}
Soit $\bigl(\p^1,\omega_{FS}\bigr)$ la droite projective complexe munie de la m\'etrique de Fubini-Study.  Soit $ \varrho$ une
fonction r\'eelle de classe $\cl$ sur $\p^1$ \'egale \`a  $1$ au voisinage de $\s$ nulle aux voisinages de $0$ et $\infty$. On consid\`ere
la fonction suivante $
\exp\Bigl(\varrho(z)\sqrt{\bigl|1-|z|\bigr|}\Bigr)$.  Elle d\'efinie une m\'etrique hermitienne continue et de classe $\cl$ sur $\p^1\setminus \s$ sur le fibr\'e en droites trivial sur $\p^1$ en posant:
\[
h_{\varrho}\bigl(1,1\bigr)(z)=\exp\Bigl(\varrho(z)\sqrt{\bigl|1-|z|\bigr|}\Bigr)\quad \forall \,z\in \p^1.
\]
 On consid\`ere l'op\'erateur suivant d\'efini sur $\p^1\setminus \s$ par:
\[
\Delta_\varrho \bigl(f\otimes 1\bigr):=-h_{\p^1}\Bigl(\dif,\dif\Bigr)^{-1}h_{\varrho}(1,1)^{-1}\frac{\pt}{\pt z}\Bigl(h_{\varrho}(1,1)\frac{\pt f}{\pt \z} \Bigr)\otimes 1\quad \forall\, f\in A^{(0,0)}\bigl( \p^1\bigr).
\]
Alors $\Delta_{\varrho}$ n'est pas un op\'erateur \`a  valeurs dans $\overline{A^{(0,0)}\bigl(\p^1 \bigr)}_{\varrho}$,
o\`u on a not\'e par $\overline{A^{(0,0)}\bigl(\p^1 \bigr)}_{\varrho}$ le compl\'et\'e de $A^{(0,0)}\bigl( \p^1\bigr)$ pour la m\'etrique
$L^2$ induite par $h_{\varrho}$ et $\omega_{FS}$. On va montrer qu'il existe $f\in A^{(0,0)}\bigl(\p^1 \bigr)$ tel que
\[
\Delta_\varrho \bigl(f\otimes 1\bigr)\notin \overline{A^{(0,0)}\bigl(\p^1 \bigr)}_{\varrho}.
\]

Soit $f$ une fonction de classe $\cl$ sur $\p^1$ telle que $f(z)=\z$ sur un voisinage ouvert de $\s$. On a
{\allowdisplaybreaks
\begin{align*}
 \bigl\|
\Delta_{\varrho}f\bigr\|_{L^2}^2&=\frac{i}{2\pi}\int_{\p^1}h_X\Bigl(\dif,\dif\Bigr)^{-1}h_{\varrho}\bigl(1,1\bigr)^{ -1}\Bigl|\dif\bigl(h_{\varrho}(1,1)\frac{\pt f}{\pt \z} \bigr) \Bigr|^2 dz\wedge d\z\\
&\geq \frac{i}{2\pi}\int_{1+\eps>|z|>1+\eps^2}(1+|z|^2)^2h_{\varrho}\bigl(1,1\bigr)^{-1}\Bigl|\dif\bigl(h_{\varrho}(1,1)\frac{\pt f}{\pt \z} \bigr) \Bigr|^2 dz\wedge d\z\quad \text{pour}\; 0<\eps\ll 1\\
&\geq \frac{i}{2\pi}\int_{1+\eps>|z|>1+\eps^2}h_{\varrho}\bigl(1,1\bigr)^{-1}\Bigl|\dif\bigl(h_{\varrho}(1,1) \bigr) \Bigr|^2 dz\wedge d\z\\
&\geq \frac{i}{2\pi}\int_{1+\eps>|z|>1+\eps^2}h_{\varrho}\bigl(1,1\bigr)\Bigl|\dif\bigl(\log h_{\varrho}(1,1) \bigr) \Bigr|^2 dz\wedge d\z\\
&\geq \frac{i}{2\pi}\int_{1+\eps>|z|>1+\eps^2}\exp\Bigl(\sqrt{\bigl|1-|z| \bigr|} \Bigr)\Bigl|\dif \sqrt{\bigl|1-|z| \bigr|} \Bigr|^2 dz\wedge d\z\\
&\geq \int_{1+\eps>r>1+\eps^2}\exp\Bigl(\sqrt{\bigl|1-|r| \bigr|} \Bigr)\frac{rdr}{4\bigl(r-1\bigr)}\\
& \geq \exp\bigl(\eps^2\bigr)\bigl(\eps-\eps^2-\log \eps\bigr).
\end{align*}}
Donc,
\[
 \bigl\|
\Delta_{\varrho}f\bigr\|_{L^2}=\infty.
\]

\end{example}

Le r\'esultat suivant sera utilis\'e dans la suite:
\begin{theorem}\label{rockafellar}
Soit $U$ un ouvert convexe de $\R^d$, et $f$ une fonction convexe et diff\'erentiable sur $U$. Soient $(f_{n})_{n\in \N_{\geq 1}}$
une suite de fonctions convexes et diff\'erentiables sur $U$ telle que $\underset{n\mapsto \infty}{\lim} f_n(x)=f(x)$ pour tout $x\in U$. Alors
{{}
\[
\underset{n\mapsto \infty}{\lim}  \nabla f_n(x)=\nabla f(x),\quad \forall x\in U.
\]}
On montre que la suite $(\nabla f_n)_{n\in \N_{\geq 1}}$ converge uniform\'ement vers $\nabla f$ sur tout sous ensemble ferm\'e born\'e et convexe de $U$.

\end{theorem}

\begin{proof}
Voir \cite[Th\'eor\`eme 25.7]{convex}.
\end{proof}

Soit $\p^1$ la droite projective complexe et on consid\`ere l'application suivante:
\begin{align*}
\kappa:\R &\longrightarrow \p^1\\
u& \mapsto \exp(-u).
\end{align*}
Dans la proposition suivante on d\'ecrit le lien entre les m\'etriques int\'egrables et invariantes par $\s$ sur les fibr\'es en droites sur $\p^1$ et les fonctions concaves sur $\R$.
\begin{proposition}\label{Cuuuu}
Soit $h_\infty$ une m\'etrique int\'egrable sur un fibr\'e en droites sur $\p^1$. On suppose que $h$ est invariante par $\s$ et on pose
\[
\mathcal{C}_\infty(u):=\kappa^\ast\Bigl(\log h\bigl(1,1\bigr)\Bigr)(u)=\log h\bigl(1,1\bigr)(\exp(-u))\quad \forall\, u\in \R.
\]
Alors,
il existe deux fonctions concaves sur $\R$, $\mathcal{C}_1$ et $\mathcal{C}_2$ telles que
\[
\mathcal{C}(u)=\mathcal{C}_1(u)-\mathcal{C}_2(u)\quad \forall \, u\in \R.
\]
La fonction  $ \log h_\infty(1,1)$ est diff\'erentiable presque partout sur $\CC$ et si l'on note par $\frac{\pt}{\pt z}\log
h_\infty(1,1)$ sa d\'eriv\'ee  d\'efinie presque partout, alors elle est born\'ee sur tout compact de $\p^1\setminus\{0,\infty \}$. En
plus, si $\bigl(h_{n,1} \bigr)_{n\in \N}$ et $\bigl(h_{2,n} \bigr)_{n\in \N}$ deux suites de m\'etriques positives $\cl$ qui
convergent vers $h_{\infty,1}$ respectivement vers $h_{\infty,2}$ telles que $h_\infty=h_{\infty,1}\otimes h_{\infty,2}^{-1}$ alors
la suite
\begin{equation}\label{suitecompactborne}
 \Bigl(\frac{\pt}{\pt z}\log h_{\infty,1}-\frac{\pt}{\pt z}\log h_{\infty,2} \Bigr)_{n\in \N}
\end{equation}

est uniform\'ement born\'ee sur tout compact de $\p^1\setminus\bigl\{ 0,\infty\bigr\}$.
\end{proposition}

\begin{proof}
Il suffit de montrer la proposition pour $h$ une m\'etrique admissible. Par d\'efinition, il existe $\bigl(
h_n\bigr)_{n\in \N}$ une suite de m\'etriques positives de classe $\cl$ qui converge uniform\'ement vers $h_\infty$. On pose
\[
 \mathcal{C}_n(u):=\log h_n\bigl(1,1\bigr)(\exp(-u))\quad \forall\,n\in \N\cup\bigl\{\infty \bigr\}.
\]
Alors $\bigl( \mathcal{C}_n\bigr)_{n\in \N}$ converge simplement vers $\mathcal{C}_\infty$, donc pour montrer que $\mathcal{C}_\infty$ est concave, il suffit de montrer que $\mathcal{C}_n$ est concave pour tout $n\in \N$, ce qui est une cons\'equence du lemme \eqref{lemmeconvexe}.\\

On rappelle quelques r\'esultats sur les fonctions convexes sur $\R^d$, avec $d\geq 1$. Soit $f$ une fonction convexe sur $\R^d$. On appelle sous-diff\'erentielle de $f$ en $x$ l'ensemble (\'eventuellement vide) suivant

\[
 \bigl(\pt f\bigr)(x):=\bigl\{v\in \R^d\, |\, f(y)\geq f(x)+(v,y-x),\, \forall y\in \R^d \bigr\}
\]
o\`u $(\,,\,)$ est le produit scalaire standard sur $\R^d$. Notons que lorsque $f$ est diff\'erentiable alors $(\pt f)(x)=\{(\nabla f)(x)\}$, o\`u $(\nabla f)(x)$ est  le gradient de $f$ en $x$ pour la m\'etrique standard de $\R^d$.\\

Soit $(f_n)_{n\in \N}$ une suite de fonctions r\'eelles convexes (non n\'ecessairement diff\'erentiables) sur un intervalle $I$ de $\R$, si cette suite converge simplement vers une fonction finie $f$ sur $I$, alors on montre que $f$ est une fonction convexe sur $I$ et l'ensemble $\{ f, f_n\, ,n\geq 0 \}$ est localement \'equilipschitzien sur $I$. Dans notre situation, $f_n$ seront des fonctions de classe $\cl$, par ce qui pr\'ec\`ede l'ensemble
\[\bigl\{(\pt f)(x),\, (\nabla f_n)(x)\;\forall\, n\geq 0  \bigr\},\]
 est born\'e au voisinage de $x$.\\

Un th\'eor\`eme bien connu d\^u   Rademacher affirme que toute fonction r\'eelle lipschtizienne sur un ouvert $U$ non vide de $\R^d$ est diff\'erentiable presque partout sur $U$. Ce r\'esultat est valable si l'on consid\`ere les fonctions convexes puisqu'elles sont localement lipschitziennes. \\

On termine la d\'emonstration de la proposition en notant que
\[
\Bigl|\frac{\pt \mathcal{C}_n}{\pt u} \Bigr|= |z|\Bigl|\frac{\pt}{\pt z}\log h_n\bigl(1,1\bigr)\Bigr|\quad \forall\, z\in \CC,\;\forall \, n\in \N.
\]

\end{proof}
\begin{remarque}
\rm{Rappelons que la correspondance entre les m\'etriques positives invariantes par $\s$ ( plus g\'en\'eralement par un tore compact) et les fonctions concaves r\'eelles    a \'et\'e utilis\'ee  dans  \cite{Burgos2} dans le but d'\'etudier l'arithm\'etique des vari\'et\'es toriques projectives.}
\end{remarque}

\begin{lemma}\label{lemmeconvexe}
Soit $\psi$ une fonction de classe $\cl$ sur $\CC$ v\'erifiant sur $\CC$:

\begin{enumerate}
\item $\psi(z)=\psi(|z|)$,

\item $\frac{\pt^2\psi}{\pt z\pt \z }\geq 0$.
\end{enumerate}

Alors, la fonction $\mathcal{C}$ d\'efinie par $\mathcal{C}(u)=\psi(\exp(-u))$, $\forall u\in \R$ est concave sur $\R$ et on a \begin{equation}\label{derivenul}
\frac{\pt\psi}{\pt z}(0)=0.
\end{equation}

\end{lemma}

\begin{proof}
On a la fonction $z\mapsto \frac{\pt^2 \psi}{\pt z\pt\z}$ est continue sur un voisinage de $0$ dans $\CC$, donc si l'on fixe $ 0<\eta\ll 1$, alors il existe une constante $c$ telle que

\[
 \Bigl| \frac{\pt^2 \psi}{\pt z\pt\z}(z)\Bigr|\leq c \quad \forall\, |z|\leq \eta,
\]
mais $ \frac{\pt^2 \psi}{\pt z\pt\z}(z)= \frac{1}{4r}\frac{\pt}{\pt r}(r\frac{\pt \psi}{\pt r})$, par cons\'equent
\[
 \Bigl| \frac{\pt}{\pt r}\Bigl(r\frac{\pt \psi}{\pt r}\Bigr)\Bigr|\leq 4c r\quad \forall\, 0<r\leq \eta.
\]
Soit $0<\eps<\eta$, par int\'egration entre $r$ et $\eps$, on obtient
\[
\Bigl|r\frac{\pt \psi}{\pt r}(r)-\eps\frac{\pt \psi}{\pt r}(\eps)\Bigr|\leq 2c|r^2-\eps^2|\quad \forall\, 0<r\leq \eta
\]
Comme $\psi$ est $\cl$, donc on peut faire tendre $\eps$ vers $0$ et on trouve que
\[
 \Bigl|\frac{\pt \psi}{\pt r}(r)\Bigr|\leq 2cr\quad \forall\, 0<r\leq \eta,
\]
On conclut que
\[
 \frac{\pt \psi}{\pt z }(0)=0.
\]

On note par $\mathcal{C}$ la fonction d\'efinie sur $\R$ par $\mathcal{C}(u)=\psi(\exp(-u))$ $\forall u\in \R$. On va montrer que cette fonction est concave. En effet, on a

\[
\begin{split}
\frac{\pt^2\psi}{\pt z\pt \z }&=-\frac{1}{4r}\frac{\pt}{\pt r}(r\frac{\pt \psi}{\pt r}), \quad r:=|z|\\
&=-\frac{1}{4r^2}r\frac{\pt}{\pt r}\bigl(r\frac{\pt \psi}{\pt r} \bigr)\\
&=-\frac{1}{4}\exp(2u)\frac{\pt^2 \mathcal{C}}{\pt u^2}, \quad \text{puisque} \quad \frac{\pt \mathcal{C}}{\pt u}=-r\frac{\pt \psi}{\pt r}=-\exp(-u)\frac{\pt \psi}{\pt r},\\
\end{split}
\]
donc $\mathcal{C}$ est concave.\\
\end{proof}

\begin{example}\label{contreexempletrivial} l'exemple suivant peut \^etre vu comme un contre exemple  \`a  l'assertion \eqref{rockafellar}.\\

Soit $\rho$ une fonction r\'eelle de classe $\cl$  sur $\R^+$ v\'erifiant:
\begin{enumerate}
\item $\rho(r)=0$ pour $r>>1$.
\item $\sup_{r\in \R^+}|\rho(r)|\leq 1$.
\item $\rho'(1)\neq 0$.
\end{enumerate}

Soit $a>0$. Alors la suite de m\'etriques sur $\mathcal{O}$

{{}\[
h_p(1,1)(z)=\exp(-\frac{1}{p^a}\rho(|z|^p), \forall z\in \CC.
\]}
converge uniform\'ement vers la m\'etrique triviale de $\mathcal{O}$, c'est \`a  dire la m\'etrique $h_\infty(1,1):=1$.\\

Remarquons que
{{}
\[
\bigl|\dif \log h_p(1,1)(z)\bigr|=\bigl|p^{1-a}\frac{\pt \rho}{\pt r}(|z|^p)|z|^{p-1}\frac{\z}{2|z|} \bigr|=\frac{p^{1-a}}{2}\bigl| \frac{\pt \rho}{\pt r}(|z|^p)\bigr||z|^{p-1}
\]}
donc si $|z|=1$, alors
{{}
\[
\bigl|\dif \log h_p(1,1)(z)\bigr|=\frac{p^{1-a}}{2}\bigl| \frac{\pt \rho}{\pt r}(1)\bigr|
\]}
On a
{\allowdisplaybreaks
\begin{align*}
\int_X\widetilde{ch}(\mathcal{O},h_p,h_\infty)&=\frac{1}{2}\int_X\log h_p(1,1) c_1(\mathcal{O},h_p)\\
&=\frac{1}{2}\int_X\bigl| \dif \log h_p(1,1)\bigr|^2dz\wedge d\z\\
&=\frac{1}{2}\int_{\CC}\frac{1}{4}p^{2-2a}\bigl| \frac{\pt \rho}{\pt r}(|z|^p)\bigr|^2|z|^{2p-2}dz\wedge d\z\\
&=\frac{1}{2}\int_0^\infty \frac{1}{4}p^{2-2a}\bigl| \frac{\pt \rho}{\pt r}(r^p)\bigr|^2r^{2p-2}rdr\\
&=\frac{1}{2}\int_0^\infty\frac{1}{4}p^{2-2a}\bigl| \theta(r^p)\bigr|^2r^{2p-1}dr, \quad \text{o\`u}\;\theta(r):=\frac{\pt \rho}{\pt r}(r)\\
&=\frac{1}{2}\int_0^\infty\frac{1}{4}p^{2-2a}\bigl| \theta(r)\bigr|^2\frac{1}{p}rdr\\
&=\frac{1}{8}p^{1-2a}\int_0^\infty\bigl| \theta(r)\bigr|^2rdr.\\
\end{align*}
}
Par suite,
{{}\[
\begin{split}
\int_X\widetilde{ch}(\mathcal{O},h_p,h_q)&=\int_X\widetilde{ch}(\mathcal{O},h_p,h_\infty)-\widetilde{ch}(\mathcal{O},h_q,h_\infty), \quad \cite[(1.3.4.2)]{Character} \\
&=\frac{1}{8}(q^{1-2a}-p^{1-2a})\int_0^\infty\bigl| \theta(r)\bigr|^2rdr
\end{split}
\]}
On a donc
{{}\[
\begin{split}
\Bigl[\int_X\widetilde{ch}(\mathcal{O},h_p,h_q)Td(\overline{TX})\Bigr]^{(0)}&=\frac{1}{8}(q^{1-2a}-p^{1-2a})\int_0^\infty\bigl| \theta(r)\bigr|^2rdr
+\frac{1}{12}\int_X\log \frac{h_p}{h_q}c_1(\overline{TX})\\
&=\frac{1}{8}(q^{1-2a}-p^{1-2a})\int_0^\infty\bigl| \theta(r)\bigr|^2rdr
+\frac{1}{12}\int_X (\frac{\rho(|z|^q)}{q^{a}}-\frac{\rho(|z|^p)}{p^{a}})c_1(\overline{TX})\\
\end{split}
\]}
Si l'on choisit $0<a<\frac{1}{2}$, alors cette suite n'est pas de Cauchy. %\textbf{En conclusion, la consid\'eration de m\'etriques g\'en\'erales est sans int\'er\^et!}.
\end{example}

\begin{remarque}
\rm{Dans la proposition \eqref{Cuuuu}, on a montr\'e que \eqref{suitecompactborne} est born\'ee sur tout compact de $\p^1\setminus\bigl\{0,\infty \bigr\}$. Le r\'esultat suivant a pour but d'\'etendre ce r\'esultat aux  compacts de $\p^1\setminus\bigl\{\infty\bigr\}$:

 Soit $\bigl(\psi_n\bigr)_{n\in\N}$ une suite  de fonctions de classe $\cl$ sur $\CC$ v\'erifiant
\begin{enumerate}
 \item $\psi_n\leq \psi_{n+1}$, $\forall\, n\in \N$.
\item $\psi_n$ converge simplement vers une fonction $\psi_\infty$ telle que $z\mapsto \Bigl|\frac{\psi_\infty(z)-\psi_\infty(z)}{z}\Bigr|$ est born\'ee au voisinage de $z=0$.
\item $\psi_n(0)=0$, $\forall n\in \N$.
\item $\psi_n(z)=\psi_n(|z|)$, $\forall\, z\in \CC$ $\forall \,n\in \N$.
\item $\frac{\pt^2\psi_n}{\pt z\pt \z}(z)\geq 0$, $\forall\, z\in \CC$.

\end{enumerate}
Alors, pour tout $p_0\in \N$ et tout compact $K$ de $\CC$ il existe $k$ une constante r\'eelle telle que

\[
 \biggl|\frac{\pt \psi_p}{\pt z}(z)\biggr|\leq k \quad \forall p\geq p_0\;\forall\, z\in K.
\]

En effet, soit $\eps$ un r\'eel fix\'e. On pose
\[
 \mathcal{C}_p(u)=\psi_p\bigl( \exp(-u)\bigr)\quad \forall \, u\in \R.
\]
 Par concavit\'e de $\mathcal{C}_p$, on a
\[
 \mathcal{C}_p(u+\eps)-\mathcal{C}_p(u)\leq \eps \frac{\pt \mathcal{C}_p}{\pt u}(u) \quad \forall\, u\in \R\;\forall\ p\in \N,
\]
d'o\`u
\[
\frac{1}{r} \Bigl(\psi_p(r\exp(-\eps))-\psi_p(r)\Bigr)\leq -\eps \frac{\pt \psi_p }{\pt r}(r)\quad\text{avec}\; \bigl(r=\exp(-u)\bigr).
\]
Donc,
\[
\begin{split}
 \eps \psi'_p(r)&\leq \frac{\psi_p(r)}{r}-\frac{\psi_p(r\exp(-\eps))}{r}\\
&\leq \frac{\psi_\infty(r)}{r}-\frac{\psi_{p_0}(r\exp(-\eps))}{r} \quad \text{pour}\;\forall \,p\geq p_0\\
&=\frac{\psi_\infty(r)-\psi_\infty(0)}{r}-\frac{\psi_{p_0}(r\exp(-\eps))-\psi_{p_0}(0)}{r} \quad \forall p\geq p_0\\
&\leq \sup_{x\in K}\biggl| \frac{\psi_\infty(x)-\psi_\infty(0)}{x}\biggr|+\exp(-\eps)\sup_{x\in K}\biggl| \frac{\psi_{p_0}(x)-\psi_{p_0}(0)}{x}\biggr| \quad \text{si}\, r\in K,\; \forall p\geq  p_0.
\end{split}
\]
avec $K$ un compact non vide de $\CC$. Comme $\eps $ est arbitraire, on d\'eduit que

\[
\bigl|\psi'_p(z) \bigr|\leq \sup_{x\in K}\biggl| \frac{\psi_\infty(x)-\psi_\infty(0)}{x}\biggr|+e\sup_{x\in K}\biggl| \frac{\psi_{p_0}(x)-\psi_{p_0}(0)}{x}\biggr| \quad \forall z\in K\; \forall p\geq  p_0.
\]
Un exemple de ces suites est donn\'e par:
\[
 \psi_n(z):=\log \Bigl(1+|z|^n \Bigr)^{\frac{1}{n}}\quad \forall\,z\in\CC\;\forall \, n\in \N_{\geq 1}.\\
\]}

\end{remarque}

On suppose maintenant que $X=\p^1$ qu'on munit d'une m\'etrique k\"ahlerienne et $E$ un fibr\'e en droites sur $\p^1$. Soit $h_{E,\infty}$ une m\'etrique int\'egrable sur $E$ invariante par l'action de $\s$. On va montrer qu'il existe un moyen naturel d'associer \`a  cette classe m\'etriques un op\'erateur Laplacien \'etendant la d\'efinition classique. \\

On va montrer, d'abord, que si $h_E$ est une m\'etrique positive invariante par l'action de $\s$ de classe $\mathcal{C}^\infty$ sur $E$, alors
\[
\Bigl( \frac{\pt}{\pt z}\log h_E(1,1)\Bigr)(0)=0.
\]
En effet, si l'on pose $\psi=\log h_E(1,1)$ alors $\psi$ v\'erifie les hypoth\`eses du \eqref{lemmeconvexe}. Si l'on note par $\Delta_{E,h}$ le Laplacien g\'en\'eralis\'e associ\'e alors on a, pour tout $z\in \p^1$, $f\in A^{(0,0)}\bigl(\p^1\bigr)$ et $\si \in H^0\bigl(\p^1,E \bigr)$:

\begin{align*}
 \Delta_{E,h}\bigl(f\otimes \si\bigr)(z)&=-h_X\Bigl(\dif,\dif \Bigr)^{-1}h_E\bigl(\si,\si \bigr)^{-1}\dif\biggl(h_E(\si,\si)\frac{\pt f}{\pt \z} \biggr)\otimes \si \\
&-h_X\Bigl(\dif,\dif\Bigr)^{-1}\frac{\pt^2 f}{\pt z\pt \z}\otimes \si-h_X\Bigl(\dif,\dif\Bigr)^{-1} \Bigl(h(\si,\si)^{-1}\dif h(\si,\si)\Bigr)\otimes \si
%&-h_X\Bigl(\dif,\dif\Bigr)^{-1}\frac{\pt^2 f}{\pt z\pt \z}\otimes \si-h_X\Bigl(\dif,\dif\Bigr)^{-1} \biggl(\frac{\dif \si }{\si}+ \frac{\z}{2r^2}\frac{\pt \mathcal{C}_{d,\infty}}{\pt u}(u) \biggr)\otimes \si\quad \forall z\neq 0,\, u=-\log|z|,\\
\end{align*}
Or, on a \[
 \begin{split}
\Bigl(h(\si,\si)^{-1}\dif h(\si,\si)\Bigr)\otimes \si &= \frac{\dif (|\si|^2e^{-\psi })}{|\si|^2e^{-\psi}}\\
&=\frac{\frac{\pt \si }{\pt z}\overline{\si}e^{-\psi}+|\si|^2\dif e^{-\psi}}{|\si|^2 e^{-\psi}}\otimes \si\\
&=\biggl(\frac{\dif \si }{\si}+\frac{\dif e^{-\psi}}{e^{-\psi}}  \biggr)\otimes \si\\
&=\biggl(\frac{\dif \si }{\si}+\frac{ \frac{\z}{2r}\frac{\pt }{\pt r}e^{-\psi(r)}}{e^{-\psi}}  \biggr)\otimes \si\\
&=\biggl(\frac{\dif \si }{\si}- \frac{\z}{2r}\frac{\pt \psi}{\pt r}(r) \biggr)\otimes \si\\
&=\biggl(\frac{\dif \si }{\si}+ \frac{\z}{2r^2}\frac{\pt \mathcal{C}}{\pt u}(u) \biggr)\otimes \si,
 \end{split}
\]
on a pos\'e $\mathcal{C}(u)=\log h\bigl(1,1\bigr)$.
Donc,
\begin{align*}
 \Delta_{E,h}\bigl(f\otimes \si\bigr)(z)&=-h_X\Bigl(\dif,\dif\Bigr)^{-1}\frac{\pt^2 f}{\pt z\pt \z}\otimes \si-h_X\Bigl(\dif,\dif\Bigr)^{-1} \biggl(\frac{\dif \si }{\si}+ \frac{\z}{2r^2}\frac{\pt \mathcal{C}}{\pt u}(u) \biggr)\otimes \si\quad \forall z\neq 0,\, u=-\log|z|,
\end{align*}

%Soit $\tau$ une fonction holomorphe qui ne s'annule pas en $x$ \textbf{C'est possible $E$ est engendr\'e par ses sections globales..} Suppose que $\tau=1$, $x=0$, alors la fonction $\psi=-\log h(1,1)$ v\'erifie les conditions du lemme \eqref{lemmeconvexe}.\\

Commencons par supposer que $h_{E,\infty}$ est admissible. On pose
\[
\mathcal{C}_\infty(u):=\log h_{\infty}\bigl(1,1\bigr)(\exp(-u)) \quad \forall\, u\in \R.
\]
On rappelle que $\mathcal{C}_\infty$ est une fonction concave sur $\R$ et  on note par
$\frac{\pt \mathcal{C}_{d,\infty}}{\pt u}$ sa d\'eriv\'ee \`a  droite, qui est finie. On a alors le r\'esultat suivant:
\begin{theorem}\label{lapintconv22}
On pose pour tout $f\in A^{0,0}\bigl(\p^1\bigr)$ et pour tout $\si\in H^0(\p^1,E)$
\[
\Delta'_{\overline{E}_\infty}\bigl(f\otimes \si\bigr)(z):=-h_X\Bigl(\dif,\dif\Bigr)^{-1}\frac{\pt^2 f}{\pt z\pt \z}\otimes \si-h_X\Bigl(\dif,\dif\Bigr)^{-1} \biggl(\frac{\dif \si }{\si}+ \frac{\z}{2r^2}\frac{\pt \mathcal{C}_{d,\infty}}{\pt u}(u) \biggr)\otimes \si\quad \forall z\neq 0,\, u=-\log|z|,
\]
et
\[
\Delta'_{\overline{E}_\infty}\bigl(f\otimes \si\bigr)(0):=-h_X\Bigl(\dif,\dif\Bigr)^{-1}\frac{\pt^2 f}{\pt z\pt
\z}\otimes \si -h_X\Bigl(\dif,\dif\Bigr)^{-1} \biggl(\frac{\dif \si }{\si}\biggr)\frac{\pt f}{\pt
\z}\otimes \si.\footnote {Notons que la derni\`ere expression est tout a fait raisonnable, puisqu'on a
montr\'e, voir lemme \eqref{lemmeconvexe}, que $\frac{\pt \psi}{\pt z}(0)=0$ pour une m\'etrique de classe
$\cl$ o\`u $\psi$ une fonction qui d\'efinie la m\'etrique au voisinage de z\'ero.}.
\]

Alors, on a
\begin{enumerate}
\item  $\Delta'_{\overline{E}_\infty}$ est un op\'erateur lin\'eaire de $A^{(0,0)}(\p^1,E)$
vers $\overline{A^{(0,0)}(\p^1,E)}_{\infty}$, en d'autres termes
\[\Bigl\|\Delta'_{\overline{E}_\infty}\bigl(f\otimes \si\bigr)\Bigr\|^2_{L^2,\infty}<\infty.\]
\item
\[
\Delta'_{\overline{E}_\infty}= \Delta_{\overline{E}_\infty},
\]
dans $\overline{A^{0,0}\bigl(\p^1,E\bigr)}_\infty$.
\end{enumerate}

\end{theorem}
\begin{proof}
Soient $f\in A^{0,0}(\p^1)$ et $\si \in H^0(\p^1,E)$. \\

  Puisque $\mathcal{C}_\infty$ est diff\'erentiable presque partout  sur $\R$, alors  on a :
\[
\begin{split}
\Delta'_{\overline{E}_\infty}(f\otimes \si)&=-h_X\Bigl(\dif,\dif\Bigr)^{-1}\frac{\pt^2 f}{\pt z\pt \z}\otimes \si -h_X\Bigl(\dif,\dif\Bigr)^{-1} \log h_\infty(\si,\si)\frac{\pt f}{\pt \z}\otimes \si, \quad \text{(par d\'efinition)}\\
&=-h_X\Bigl(\dif,\dif\Bigr)^{-1}h_\infty(\si,\si)^{-1}\frac{\pt}{\pt z}\bigl(h_\infty(\si,\si)\frac{\pt f}{\pt \z} \bigr)\otimes \si,
\end{split}
\]
presque partout sur $\CC$. Par suite
\[
\Bigl\|\Delta'_{\overline{E}_\infty}(f\otimes \si)\Bigr\|^2_{L^2,\infty}=\frac{i}{2\pi}\int_X
h_X\Bigl(\dif,\dif\Bigr)^{-1}h_\infty(\si,\si)^{-1}\Bigl|\frac{\pt}{\pt z}\Bigl(h_\infty(\si,\si)\frac{\pt f}{\pt \z}
\Bigr)\Bigr|^2dz\wedge d\z
 \]
Or  c'est la limite de $\Bigl(\bigl\|\Delta_{\overline{E}_p}(f\otimes \si)\bigr\|^2_{L^2,p}\Bigr)_{p\in \N}$ quand $p$ tend vers $\infty$, donc
 \[
\bigl\|\Delta'_{\overline{E}_\infty}(f\otimes \si)\bigr\|^2_{L^2,\infty}<\infty.
\]

  C'est donc un op\'erateur lin\'eaire de $A^{0,0}(X,E)$ vers $\overline{A^{0,0}(X,E)}$.\\

 On a
\[
\begin{split}
\bigl(\Delta'_{\overline{E}_\infty}(f\otimes \si), g\otimes \tau \bigr)_{L^2,\infty}&=-\int_X h_\infty\bigl(\si,\si\bigr)^{-1}\frac{\pt}{\pt z}\bigl(h_\infty(\si,\si)\frac{\pt f}{\pt \z} \bigr) \,\overline{g}\,h_\infty(\si,\tau)dz\wedge d\z\\
&=-\frac{i}{2\pi}\lim_{p\mapsto \infty}\int_X h_p(\si,\si)^{-1}\frac{\pt}{\pt z}\bigl(h_p(\si,\si)\,\frac{\pt f}{\pt \z} \bigr)\, \overline{g}\,h_p(\si,\tau)\,dz\wedge d\z\\
&=\frac{i}{2\pi}\lim_{p\mapsto \infty}\int_X h_p(\si,\tau)\,\frac{\pt f}{\pt\z}\,\frac{\pt \overline{g}}{\pt z}\,dz\wedge d\z\\
&=\frac{i}{2\pi}\int_X h_{E,\infty}(\si,\tau)\,\frac{\pt f}{\pt\z}\,\frac{\pt \overline{g}}{\pt z}\,dz\wedge d\z
\end{split}
\]

  On en d\'eduit que \[
\bigl(\Delta'_{\overline{E}_\infty}\xi,\eta)_{L^2,\infty}=
\bigl(\xi,\Delta'_{\overline{E}_\infty}\eta)_{L^2,\infty}
\]
et
\[
 \bigl(\Delta'_{\overline{E}_\infty}\xi,\xi)_{L^2,\infty}\geq 0,
\]
pour tout $\xi$ et $\eta$ dans $A^{(0,0)}(X,E)$.\\

De \eqref{ff11}, on conclut que
\[
 \Delta_{\overline{E}_\infty}=\Delta'_{\overline{E}_\infty},
\]
dans $\overline{A^{0,0}\bigl(\p^1,E\bigr)}$.
\end{proof}

\begin{definition}
On appelle $\Delta_{\overline{E}_\infty}$ le Laplacien g\'en\'eralis\'e associ\'e \`a  $h_{E,\infty}$. C'est un
op\'erateur lin\'eaire et positif de $A^{0,0}(\p^1,E) $ \`a  valeurs dans
$\overline{A^{0,0}\bigl(\p^1,E\bigr)}_\infty$.
\end{definition}

\section{Variation de la m\'etrique sur $E$}\label{paragrapheVarMetE}

Soit $\overline{E}_\infty:=(E,h_{E,\infty})$ un fibr\'e en droites $1$-int\'egrable sur $X$, une surface de Riemann compacte. Par d\'efinition (voir \eqref{1-integrable}) il existe  $\overline{E}_{1,\infty}:=(E_1,h_{1,\infty}) $ et $\overline{E}_{2,\infty}:=(E_2,h_{2,\infty})$ deux fibr\'es en droites admissibles sur $X$ tels que $\overline{E}_\infty=\overline{E}_{1,\infty}\otimes \overline{E}_{2,\infty}^{-1}$. Soient $(h_{n,1})_{n\in \N}$ (resp. $(h_{n,2})_{n\in \N}$) une suite de m\'etrique hermitiennes positives de classe $\cl$ sur ${E}_{1}$ (resp. sur ${E}_{2}$) qui converge uniform\'ement vers $h_{1,\infty}$ (resp. $h_{2,\infty}$) sur $X$ entier. On note par $(h_n)_{n\in \N}$ la suite de m\'etrique sur $E$ d\'efinie par $h_n=h_{n,1}\otimes h_{n,2}^{-1}$, $\forall n\in \N$ v\'erifiant:

\begin{enumerate}
\item \[
\sup_{n\in \N}\biggl\| h_X\Bigl(\dif,\dif\Bigr)^{-\frac{1}{2}}\dif \log \frac{h_{n+1}}{h_n}\biggr\|_{\sup}<\infty
\]
o\`u $\bigl\{\dif \bigr\}$ est une base locale de $TX$. Rappelons que $\biggl| h_X\Bigl(\dif,\dif\Bigr)^{-\frac{1}{2}}\dif \log \frac{h_{n+1}}{h_n}\biggr|$ ne d\'epend pas du choix de la base locale.
\item
\[
 \sum_{n=1}^\infty \biggl\| \frac{h_n}{h_{n-1}}-1\biggr\|_{\sup}^{\frac{1}{2}}<\infty.\\
\]

\end{enumerate}

On consid\`ere l'op\'erateur Laplacien $\Delta_{\overline{E}_\infty}$ associ\'e \`a  $\overline{E}_\infty$, voir  \eqref{lapintconv}, on a
montr\'e qu'il est positif. Afin d'\'etudier ses propri\'et\'es spectrales, on a besoin d'\'etendre $\Delta_{\overline{E}_\infty}$ en un
op\'erateur autoadjoint.  On prend alors, $
\Delta_{\overline{E}_\infty,F}$, l'extension de Friedrichs de $\Delta_{\overline{E}_\infty}$. C'est
un op\'erateur autoadjoint qui \'etend $\Delta_{\overline{E}_\infty}$, (voir \cite[Appendice C.]{Ma} pour la
construction). Comme $  \Delta_{\overline{E}_\infty}$ est positif, alors $\Delta_{\overline{E}_\infty,F}$
l'est aussi. D'apr\`es \eqref{semi}, on peut consid\'erer on peut consid\'erer le semi-groupe associ\'e qu'on note par $e^{-t\Delta_{\overline{E}_\infty}}$.\\

Dans ce paragraphe, nous \'etablissons  la connexion entre $\bigl(e^{-t\Delta_{E,h_n}}\bigr)_n$ et $e^{-t\Delta_{E,h_\infty}}$, on va montrer que
\[
 \Bigl(e^{-t\Delta_{\overline{E}_n}}\Bigr)_{n\in \N}\xrightarrow[n\mapsto\infty]{} e^{-t\Delta_{\overline{E}_\infty}}.
\]
pour la norme $L^2_\infty$ (induite par $h_{E,\infty}$ et $h_X$).\\

Dans la suite, on note par $\bigl(h_u\bigr)_{u\geq 1}$ la famille  associ\'ee \`a  la suite $(h_{n})_{n\in \N}$
construite dans  \eqref{suitefamille}. Pour simplifier, on notera par $\Delta_{E,n}$ (resp. $\Delta_{E,u}$) au lieu de $\Delta_{\overline{E}_n}$
(resp. $\Delta_{\overline{E}_u}$).\\

Soit $x\in X$ et $\si$ une section locale holomorphe de $E$ non nulle en $x$. On pose
\[
k_\si(u)(x')=\biggl|\frac{\pt}{\pt u}\frac{\pt}{\pt z}\log h_u(\si,\si)(x')\biggr|,
\]
pour tout $x'$ dans un voisinage ouvert assez petit de $x$. Au voisinage de $x$, on a
\[
\begin{split}
\frac{\pt}{\pt u} \dif \log h_u(\si,\si)&=\frac{\pt}{\pt u} \dif \log \Bigl(\rho_{p-1}(u)h_{p-1}(\si,\si)+(1-\rho_{p-1}(u))h_{p}(\si,\si)  \Bigr) \quad\forall u\in[p-1,p]\\
&=\frac{\pt}{\pt u} \dif \log \Bigl( \rho_{p-1}(u)|\si|^2e^{-\psi_{p-1}(z)}+(1-\rho_{p-1}(u))|\si|^2e^{-\psi_p(z)} \Bigr)\\
&=\frac{\pt}{\pt u}\dif\log |\si|^2+  \frac{\pt }{\pt u}\dif\log \Bigl( \rho_{p-1}(u)e^{-\psi_{p-1}(z)}+(1-\rho_{p-1}(u))e^{-\psi_p(z)} \Bigr)\\
&=\frac{\pt}{\pt u}   \dif\log \Bigl( \rho_{p-1}(u)e^{-\psi_{p-1}(z)}+(1-\rho_{p-1}(u))e^{-\psi_p(z)} \Bigr).\\
\end{split}
\]
Par cons\'equent, on peut d\'efinir une fonction r\'eelle sur $X$ entier en posant
{{}
 \[k(u)(x)=k_\si(u)(x)\quad \forall x\in X \,\forall\, \si\in H^0(X,E)\;\text{telle que}\; \si(x)\neq 0.\]}

et on pose
\[
\pi_{E }(u):=\sup_{x\in X} h_X\Bigl(\dif,\dif\Bigr)^{-\frac{1}{2}}\bigl|k(u)(x)\bigr|, \quad \forall u\in [1,\infty[.
\]

\begin{lemma}\label{bornepi}
Pour tout $u\geq 1$, on a $h(\dif,\dif)^{-\frac{1}{2}}|k(u)|$ est une fonction continue sur $X$.
 Il existe une constante r\'eelle $M_3$ telle que
\[
 \pi_E(u)\leq M_3,\quad \forall u>1.
\]
\end{lemma}

\begin{proof}
Soit $x\in X$ et $\bigl\{ \dif\bigr\}$ une base locale de $TX$ au voisinage de $x$. Par d\'efinition, il existe $\si$    une section holomorphe locale de $E$  telle que $\si(x)\neq 0$ et que dans un voisinage de $x$:
{{}
 \[k(u)=k_\si(u)=\Bigl|\frac{\pt}{\pt u}\frac{\pt}{\pt z}\log h_u(\si,\si)\Bigr|,\]
}
donc la continuit\'e de $k$ et par suite celle $h(\dif,\dif)^{-\frac{1}{2}}k(u)$ au voisinage de $x$ r\'esulte du fait que $h_u$ est $\cl$.

 On a dans un voisinage de $x$:
{\allowdisplaybreaks
\begin{align*}
&k(u)=\dif\frac{\pt}{\pt u}(\log h_u(\si,\si))\\
&=\frac{\pt \rho_{p-1}}{\pt u}\biggl( \frac{\dif h_p-\dif h_{p-1}}{h_u}-\frac{h_p-h_{p-1}}{h_u}\dif \log h_u \biggr)\\
&=\frac{\pt \rho_{p-1}}{\pt u}\frac{h_p}{h_u}\biggl(\dif \log \frac{h_p}{h_{p-1}}-\Bigl(\frac{h_{p-1}}{h_p}-1\Bigr)\dif \log h_{p-1} +\Bigl(\frac{h_{p-1}}{h_p}-1\Bigr)\dif \log h_u \biggr)\quad \text{o\`u}\;\bigl(p-1=[u]\bigr)\\
&=\frac{\pt \rho_{p-1}}{\pt u}\frac{h_p}{h_u}\biggl(\dif \log \frac{h_p}{h_{p-1}}-\Bigl(\frac{h_{p-1}}{h_p}-1\Bigr)\dif \log h_{p-1}+\Bigl(\frac{h_{p-1}}{h_p}-1\Bigr)\frac{(1-\rho_{p-1})\dif h_{p-1}+\rho_{p-1}\dif h_p }{h_u} \Bigr)\\
&=\frac{\pt \rho_{p-1}}{\pt u}\frac{h_p}{h_u}\biggl(\dif \log \frac{h_p}{h_{p-1}}-\Bigl(\frac{h_{p-1}}{h_p}-1\Bigr)\dif \log h_{p-1}\\
&+\Bigl(\frac{h_{p-1}}{h_p}-1\Bigr)\frac{h_p}{h_u}\bigl((1-\rho_{p-1})\frac{h_{p-1}}{h_p}\dif \log h_{p-1}+\rho_{p-1}\dif \log h_p\bigr) \biggr)\\
&=\frac{\pt \rho_{p-1}}{\pt u}\frac{h_p}{h_u}\biggl(\dif \log \frac{h_p}{h_{p-1}}-\Bigl(\frac{h_{p-1}}{h_p}-1\Bigr)\dif \log h_{p-1}
+(1-\rho_{p-1})\Bigl(\frac{h_{p-1}}{h_p}-1\Bigr)\frac{h_{p-1}}{h_u}\dif \log h_{p-1}\\ &+\rho_{p-1}\Bigl(\frac{h_{p-1}}{h_p}-1\Bigr)\frac{h_{p}}{h_u}\dif \log h_p\bigr) \biggr)\\
&=\frac{\pt \rho_{p-1}}{\pt u}\frac{h_p^2}{h_u^2}\biggl(\dif \log \frac{h_p}{h_{p-1}}-\frac{h_u}{h_p}\Bigl(\frac{h_{p-1}}{h_p}-1\Bigr)\dif \log h_{p-1}
+(1-\rho_{p-1})\bigl(\frac{h_{p-1}}{h_p}-1 \bigr)\frac{h_{p-1}}{h_p}\dif \log h_{p-1}\\ &+\rho_{p-1}\Bigl(\frac{h_{p-1}}{h_p}-1\Bigr)\dif \log h_p\bigr) \biggr).\\
\end{align*}
}
Comme les fonctions $\rho_{p}$ sont uniform\'ement born\'ees et que la suite $\bigl(h_u\bigr)_{u\in [1,\infty[}$ converge uniform\'ement vers $h_\infty$, on peut donc trouver $M'_1,M'_2$ et $M'_3$ trois constantes r\'eelles telles que
\[
\begin{split}
 h_X\Bigl(\dif,\dif\Bigr)^{-\frac{1}{2}}\bigl|k(u)\bigr|&\leq M'_1
 h_X\Bigl(\dif,\dif\Bigr)^{-\frac{1}{2}}\Bigl|\dif \log \frac{h_p}{h_{p-1}} \Bigr|+M'_2
 h_X\Bigl(\dif,\dif\Bigr)^{-\frac{1}{2}}\Bigl| \Bigl(\frac{h_{p-1}}{h_p}-1\Bigr)\dif \log h_{p-1}\Bigr|\\
&+M'_3
 h_X\Bigl(\dif,\dif\Bigr)^{-\frac{1}{2}}\Bigl| \Bigl(\frac{h_{p-1}}{h_p}-1\Bigr)\dif \log h_{p}\Bigr|.
\end{split}
\]
On a
\begin{align*}
\biggl| h_X\Bigl(\dif,\dif\Bigr)^{-\frac{1}{2}}&\Bigl(\frac{h_{p-1}}{h_p}-1\Bigr)\dif \log h_p\biggr|=\biggl| h_X\Bigl(\dif,\dif\Bigr)^{-\frac{1}{2}}\Bigl(\frac{h_{p-1}}{h_p}-1\Bigr)\Bigl(\sum_{k=2}^p\dif \log \frac{h_k}{h_{k-1}}+\dif \log h_1\Bigl)\biggr|\\
&\leq \Bigl|\frac{h_{p-1}}{h_p}-1\Bigr| \sum_{k=2}^p h_X\Bigl(\dif,\dif\Bigr)^{-\frac{1}{2}}\Bigl|\dif \log \frac{h_k}{h_{k-1}}\Bigr|+h_X\Bigl(\dif,\dif\Bigr)^{-\frac{1}{2}}\Bigl|\dif \log h_1\Bigl|\\
&\leq (p-1)\sup_{n\in \N^\ast} \Bigl\|h_X\Bigl(\dif,\dif\Bigr)^{-\frac{1}{2}}\dif \log \frac{h_n}{h_{n-1}}\Bigr\|_{\sup}\Bigl|\frac{h_{p-1}}{h_p}-1\Bigr|\\
&+h_X\Bigl(\dif,\dif\Bigr)^{-\frac{1}{2}}\Bigl|\dif \log h_1\Bigl|\Bigl|\frac{h_{p-1}}{h_p}-1\Bigr|,\\
\end{align*}
de la m\^eme facon, on obtient:
\begin{align*}
\biggl| h_X\Bigl(\dif,\dif\Bigr)^{-\frac{1}{2}}\Bigl(\frac{h_{p-1}}{h_p}-1\Bigr)\dif \log h_{p-1}\biggr|&\leq (p-2)\sup_{n\in \N^\ast} \Bigl\|h_X\Bigl(\dif,\dif\Bigr)^{-\frac{1}{2}}\dif \log \frac{h_n}{h_{n-1}}\Bigr\|_{\sup}\Bigl|\frac{h_{p-1}}{h_p}-1\Bigr|\\
&+h_X\Bigl(\dif,\dif\Bigr)^{-\frac{1}{2}}\Bigl|\dif \log h_1\Bigl|\Bigl|\frac{h_{p-1}}{h_p}-1\Bigr|.
\end{align*}
% Par hypoth\`ese, on peut trouver $0<K<K'$ des constantes r\'eelles telles que, par \eqref{bellemajorationfaible}:

%Commencons par noter que

Donc, on a
\begin{align*}
 h_X\Bigl(\dif,\dif\Bigr)^{-\frac{1}{2}}\bigl|k(u)\bigr|&\leq M'_1
 \sup_{n\in \N^\ast}\Bigl\|h_X\Bigl(\dif,\dif\Bigr)^{-\frac{1}{2}}\Bigl|\dif \log \frac{h_n}{h_{n-1}} \Bigr\|_{\sup}\\
&+M''_2(2p-3)\Bigl|\frac{h_{p-1}}{h_p}-1\Bigr|\sup_{n\in \N^\ast} \Bigl\|h_X\Bigl(\dif,\dif\Bigr)^{-\frac{1}{2}}\dif \log \frac{h_n}{h_{n-1}}\Bigr\|_{\sup}
 \\
&+M''_3 h_X\Bigl(\dif,\dif\Bigr)^{-\frac{1}{2}}\Bigl|\dif \log h_1\Bigl|\Bigl|\frac{h_{p-1}}{h_p}-1\Bigr|
\end{align*}
Puisque   $k(u)$ ne d\'epend pas du choix de la section locale $\si$ et que $X$ est compacte, on peut trouver $\bigl( U_j\bigr)_{j\in J}$ un recouvrement ouvert fini de $X$ tel que  d\'eduire que pour tout $j\in J$, $k(u)=k_{\si_j}(u)$ sur $U_j$ et $\si_j$ une section locale non nulle et que
\[
 \Bigl|h_X\Bigl(\dif,\dif\Bigr)^{-\frac{1}{2}}\dif \log h_1(s_j,s_j)\Bigl|,
\]
soit born\'ee sur $U_j$. Par suite, il existe $M''_4$, une constante r\'eelle telle que
\begin{align*}
 h_X\Bigl(\dif,\dif\Bigr)^{-\frac{1}{2}}\bigl|k(u)\bigr|&\leq M'_1
 \sup_{n\in \N^\ast}\Bigl\|h_X\Bigl(\dif,\dif\Bigr)^{-\frac{1}{2}}\Bigl|\dif \log \frac{h_n}{h_{n-1}} \Bigr\|_{\sup}\\
&+M''_2(2p-3)\Bigl|\frac{h_{p-1}}{h_p}-1\Bigr|\sup_{n\in \N^\ast} \Bigl\|h_X\Bigl(\dif,\dif\Bigr)^{-\frac{1}{2}}\dif \log \frac{h_n}{h_{n-1}}\Bigr\|_{\sup}
 \\
&+M''_4 \Bigl|\frac{h_{p-1}}{h_p}-1\Bigr|.
\end{align*}
Montrons maintenant le reste du lemme, on a:
\begin{align*}
\pi_E(u)&=\sup_{x\in X} h_X\Bigl(\dif,\dif\Bigr)^{-\frac{1}{2}}\bigl|k(u)\bigr|\\
&\leq M'_1
 \sup_{n\in \N^\ast}\Bigl\|h_X\Bigl(\dif,\dif\Bigr)^{-\frac{1}{2}}\dif \log \frac{h_n}{h_{n-1}} \Bigr\|_{\sup}\\
&+M''_2(2p-3)\Bigl\|\frac{h_{p-1}}{h_p}(x)-1\Bigr\|_{\sup}\sup_{n\in \N^\ast} \Bigl\|h_X\Bigl(\dif,\dif\Bigr)^{-\frac{1}{2}}\dif \log \frac{h_n}{h_{n-1}}\Bigr\|_{\sup}
 \\
&+M''_4 \Bigl\|\frac{h_{p-1}}{h_p}(x)-1\Bigr\|_{\sup}.
\end{align*}

Comme $(h_p)_{p\in \N}$ converge uniform\'ement vers $h_\infty$, alors on peut supposer que
\[
\sup_{p\in \N^\ast} \Bigl(p\Bigl\|\frac{h_{p-1}}{h_p}-1\Bigr\|_{\sup}\Bigr)<\infty.
\]

On conclut qu'il existe $M_3$ une constante r\'eelle telle que
\[
 \pi_E(u)\leq M_3, \quad \forall u>1.
\]
\end{proof}

\begin{remarque}
\rm{ Par \eqref{exemple123}, on a, il existe $c_4\in \R$ telle que
\[
 \bigl\|\pi_E(u)\bigr\|_{\sup}\underset{u\mapsto \infty}{\sim} \frac{c_4}{u}.
\]
En effet, on a
\[
 \begin{split}
\bigl|h(\dif,\dif)^{-\frac{1}{2}} k(u)\bigr|+\bigl| O\bigl(
 \frac{h_{p-1}}{h_p}-1 \bigr)\bigr|&\geq  \bigl|h(\dif,\dif)^{-\frac{1}{2}} k(u)- O\bigl(
 \frac{h_{p-1}}{h_p}-1 \bigr)\bigr| \\
&= \bigl|\frac{\pt \rho_{p-1}}{\pt u}
h(\dif,\dif)^{-\frac{1}{2}}\dif \log \frac{h_p}{h_{p-1}}\bigr|\\
&\geq \frac{\pt \rho_{p-1}}{\pt u}\bigl|
h(\dif,\dif)^{-\frac{1}{2}}\dif \log \frac{h_p}{h_{p-1}}\bigr|,
 \end{split}
\]
donc
\[
\begin{split}
\bigl\|h(\dif,\dif)^{-\frac{1}{2}} k(u)\bigr\|_{\sup}+\bigl\| O\bigl(
 \frac{h_{p-1}}{h_p}-1 \bigr)\bigr\|_{\sup}\geq \frac{\pt \rho_{p-1}}{\pt u}\bigl\|
h(\dif,\dif)^{-\frac{1}{2}}\dif \log \frac{h_p}{h_{p-1}}\bigr\|_{\sup}\underset{p\mapsto \infty}{\sim} \frac{c_4}{p},
\end{split}
\]
mais comme $O\bigl(
 \frac{h_{p-1}}{h_p}-1 \bigr)=O(\frac{1}{p^2})$, alors

\[
\underset{ u\mapsto \infty}{\limsup} \Bigl(p\bigl\|h(\dif,\dif)^{-\frac{1}{2}} k(u)\bigr\|_{\sup}\Bigr)\geq c_4.
\]}
 \end{remarque}

\begin{theorem}
 On a, pour tout $\xi \in A^{0,0}(X,E)$:
\begin{equation}\label{formule99}
\begin{split}
\biggl(\frac{\pt \Delta_{E,u}}{\pt u} \xi,\frac{\pt \Delta_{E,u}}{\pt u} \xi\biggr)_{L^2,u}\leq\pi_E^2(u)\bigl(\Delta_{E,u} \xi,\xi\bigr)_{L^2,u}\quad \forall u>1.
\end{split}
\end{equation}
\end{theorem}

\begin{proof}
On a, localement
\[
\Delta_{E,u}(f\otimes \si)=-h_X\Bigl(\dif,\dif\Bigr)^{-1}\frac{\pt^2f}{\pt z\pt\z}\otimes \si-h_X\Bigl(\dif,\dif\Bigr)^{-1}\dif\bigl(\log h_u(\si,\si)\bigr)\frac{\pt f}{\pt \z}\otimes \si.
\]
o\`u $f\in A^{(0,0)}(X)$ et $\si$ est section locale holomorphe non nulle.
donc
{{}
\[
\frac{\pt \Delta_{E,u}}{\pt u}(f\otimes \si)=-h_X\Bigl(\dif,\dif\Bigr)^{-1}\frac{\pt }{\pt u}\Bigl(\dif \log h_u(\si,\si)\Bigr)\frac{\pt f}{\pt \z}\otimes \si.
\]
}

Soit $\xi \in A^{0,0}(X,E)$. Localement, il existe $f_1,\ldots,f_r  \in A^{0,0}(X)$ et $e_1,e_2,\ldots,e_r$ des sections locales holomorphes non nulles  de $E$ telles que $\xi=\sum_{i=1}^r f_i\otimes e_i$. On a
\[
\begin{split}
-\frac{\pt \Delta_{E,u}}{\pt u} \xi &=-\sum_{i=1}^r \frac{\pt \Delta_{E,u} }{\pt u}(f_i\otimes e_i)\\
&=-h(\frac{\pt}{\pt z},\frac{\pt}{\pt z})^{-1}\sum_{i=1}^r \frac{\pt}{\pt u}\bigl(\frac{\pt}{\pt z}\log h_u(e_i,e_i) \bigr)\frac{\pt f_i}{\pt \z}\otimes e_i\\
&=-h_X\Bigl(\dif,\dif\Bigr)^{-1}k(u)\sum_{i=1}^r \frac{\pt f}{\pt \z}\otimes e_i.
\end{split}
\]
Par suite,
\[
\begin{split}
\biggl(\frac{\pt \Delta_{E,u}}{\pt u} \xi,\frac{\pt \Delta_{E,u}}{\pt u} \xi\biggr)_{L^2,u}&=\frac{i}{2\pi} \int_{X}h_X\Bigl(\frac{\pt }{\pt z},\frac{\pt}{\pt z}\Bigr)^{-1} |k(u)|^2\sum_{kj} \frac{\pt f_k}{\pt \z}\frac{\pt \overline{f_j}}{\pt z}h_u(e_k,e_j)dz\wedge d\z.\\
\end{split}
\]

Mais puisque $\sum_{kj}  \frac{\pt f_k}{\pt \z}\frac{\pt \overline{f_j}}{\pt z}h_u(e_k,e_j)\geq 0$, alors
\[
\begin{split}
\biggl(\frac{\pt \Delta_{E,u}}{\pt u} \xi,\frac{\pt \Delta_{E,u}}{\pt u} \xi\biggr)_{L^2,u}&= \frac{i}{2\pi }\int_{X}h_X\Bigl(\dif,\dif\Bigr)^{-1}k(u)^2\sum_{kj}  \frac{\pt f_k}{\pt \z}\frac{\pt \overline{f_j}}{\pt z}h_u(e_i,e_j)dz\wedge d\z\\
&\leq \pi_E^2(u)\int_X \sum_{kj}  \frac{\pt f_k}{\pt \z}\frac{\pt \overline{f_j}}{\pt z}h_u(e_i,e_j)dz\wedge d\z\\
&=\pi_E^2(u)\bigl(\Delta_{E,u} \xi,\xi\bigr)_{L^2,u}.
\end{split}
\]

\end{proof}

\begin{Corollaire}\label{encoreestimation11}
On a
{{}
\begin{equation}
\biggl\|\frac{\pt \Delta_{E,u}}{\pt u} e^{-t\Delta_{E,u}} \biggr\|_{L^2,u}\leq \frac{e^{-\frac{1}{2}}}{\sqrt{t}}\pi_E(u),\quad \forall u>1,
\end{equation}}
pour tout $t>0$ fix\'e.
\end{Corollaire}
\begin{proof}
 Soit $t>0$, et $\eta\in A^{0,0}(X,E)$ et  on pose $\xi:=e^{-t\Delta_{E,u}} \eta$.\\

 De \eqref{formule99}, on a
\begin{equation}\label{inequalitY}
\biggl(\frac{\pt \Delta_{E,u}}{\pt u} e^{-t\Delta_{E,u}} \eta,  \frac{\pt \Delta_{E,u}}{\pt u} e^{-t\Delta_{E,u}} \eta \biggr)_{L^2,u}\leq \pi_E^2(u) \Bigl(\Delta_{E,u}  e^{-t\Delta_{E,u}} \eta,  e^{-t\Delta_{E,u}} \eta\Bigr)_{L^2,u}.
\end{equation}
Si l'on d\'ecompose $\eta$ suivant $(v_{u,k})_{k\in \N}$, une  base orthonormale pour la norme $L^2_u$, form\'ee par les vecteurs propres de $\Delta_{E,u}$: Il existe $a_{u,0},a_{u,1},\ldots$ des r\'eels tels que
\[
\eta=\sum_{k\in \N }a_{u,k} v_{u,k}
\]
alors, dans le compl\'et\'e de $A^{(0,0)}(X,E)$ pour $L^2_u$, on a
\[
\Delta_{E,u} e^{-t\Delta_{E,u} } \eta=\sum_{k\geq 1} -\la_{u,k} e^{-\la_{u,k} t} a_{u,k}v_{u,k}=-\frac{1}{t}\sum_{k\geq 1} -\la_{u,k} t e^{-\la_{u,k} t} a_{u,k}v_{u,k}.
\]
De l\`a , on d\'eduit que
\[
\Bigl\|\Delta_{E,u} e^{-t\Delta_{E,u} } \eta \Bigr  \|^2_{L^2,u}\leq \frac{e^{-1}}{t^2}\sum_{k\geq 1} |a_{u,k}|^2 \leq \frac{c}{t^2}\bigl\| \eta \bigr\|^2_{L^2,u},
\]
donc \eqref{inequalitY} devient
\[
\begin{split}
\Bigl\| \frac{\pt \Delta_{E,u}}{\pt u} e^{-t\Delta_{E,u}} \eta\Bigr\|_{L^2,u}^2&=
\Bigl(\frac{\pt \Delta_{E,u}}{\pt u} e^{-t\Delta_{E,u}} \eta,  \frac{\pt \Delta_{E,u}}{\pt u} e^{-t\Delta_{E,u}} \eta \Bigr)_{L^2,u}\\
&\leq \pi_E^2(u)\bigl(\Delta_{E,u}  e^{-t\Delta_{E,u}} \eta,  e^{-t\Delta_{E,u}} \eta\bigr)_{L^2,u}\\
&\leq \pi_E^2(u) \bigl\| \Delta_{E,u}  e^{-t\Delta_{E,u}} \eta \bigr\|_{L^2,u} \bigl\|  e^{-t\Delta_{E,u}} \eta \bigr\|_{L^2,u},\quad \text{par l'in\'egalit\'e de Cauchy-Schwartz} \\
&\leq \frac{e^{-1}}{t}\pi_E^2(u)\bigl\| \eta \bigr\|_{L^2,u}^2 , \quad \text{car}\;  \bigl\|  e^{-t\Delta_{E,u}} \|_{L^2,u}\leq 1.
\end{split}
\]
On conclut que

\[
\biggl\| \frac{\pt \Delta_{E,u}}{\pt u} e^{-t\Delta_{E,u}} \biggr\|_{L^2,u}
\leq \frac{e^{-\frac{1}{2}}}{\sqrt{t}}\pi_E(u).
\]

\end{proof}

\begin{proposition}
Si $X=\p^1$ et $h_\infty$ est une m\'etrique de classe $\cl$ sur $\mathcal{O}(m)$ invariante par $\s$, alors on peut choisir une sous-suite de $(h_n)_{n \geq 1}$ telle que
\[
 \pi_E(u)=O(\frac{1}{2^u}), \quad \forall u>>1,
\]
par cons\'equent,
\[
 e^{-t\Delta_{E,u}}\xrightarrow[u\mapsto \infty]{} e^{-t\Delta_{E,\infty}}
\]

\end{proposition}

\begin{proof}
Comme $h_\infty$ est $\cl$ et invariante par $\s$ alors la fonction $f_\infty(u):=\log h_\infty(1,1)(\exp(-u))$ est concave et de classe $\cl$. Par application du \eqref{rockafellar}, on d\'eduit que $h(\dif,\dif)^{-\frac{1}{2}}\dif \log h_n(1,1)$  converge uniform\'ement sur tout compact de $\CC^\ast$  vers  $h(\dif,\dif)^{-\frac{1}{2}}\dif \log h_\infty(1,1)$. Par un choix de sous-suite convenable on peut supposer que

\[
 h(\dif,\dif)^{-\frac{1}{2}}\biggl|\dif \log \frac{h_n}{h_{n-1}}\biggr|=O(\frac{1}{2^n}), \quad \forall n>>1.
\]
donc,
\[
\pi_E(u)= h(\dif,\dif)^{-\frac{1}{2}}k(u)=O(\frac{1}{2^u})\quad \forall u>>1.
\]
On conclut de \eqref{encoreestimation11} que
{{}
\[
\Bigl \|\frac{\pt e^{-t\Delta_{E,u}}}{\pt u} \Bigr\|_u\leq \frac{\sqrt{t}}{2^u}O(1),\quad \forall u>>1.
\]}
\end{proof}

On introduit la fonction suivante d\'efinie sur $X$ donn\'ee au voisinage d'un point $x$ par l'expression:
\[
 \frac{\pt }{\pt u}\log h_u(\si,\si)(x),
\]
o\`u $\si$ est une section locale holomorphe de $E$ non nulle  en $x$, et on pose
\[
 \delta_{E}(u)=\sup_{x\in X}\biggl|\frac{\pt }{\pt u}\log h_u(\si,\si)(x) \biggr|.
\]
\begin{proposition}
 $\forall \xi, \eta\in A^{0,0}(X,E)$,
\[
 \Bigl|\Bigl(\frac{\pt \Delta_{E,u}}{\pt u}\xi,\xi\Bigr)_u  \Bigr|\leq 2\delta_{E}(u)\|\Delta_{E,u}\xi\|_u \|\xi\|_u,
\]
et
\begin{equation}\label{fff111}
\begin{split}
\Bigl|\Bigl(\frac{\pt \Delta_{E,u}}{\pt u}\xi,\eta\Bigr)_u  \Bigr|\leq \delta_{E}(u)\Bigl(\|\Delta_{E,u} \xi\|_u^{\frac{1}{2}} \|\Delta_{E,u} \eta\|_u^{\frac{1}{2}} \|\xi\|_u^\frac{1}{2} \|\eta \|_u^\frac{1}{2}+\|\Delta_{E,u} \xi\|_u \|\eta \|_u\Bigr).\\
\end{split}
\end{equation}
Pour tous $t_1,t_2$ et $t_3>0$,
\begin{equation}\label{t1t2t3}
\begin{split}
\Bigl|\Bigl(\frac{\pt \Delta_{E,u}}{\pt u}e^{-t_1\Delta_{E,u}}\xi, &e^{-t_2\Delta_{E,u}}\frac{\pt \Delta_{E,u}}{\pt u}e^{-t_3\Delta_{E,u}}\gamma\Bigr)_u\Bigr|\\
&\leq  \Bigl(\frac{C^2}{\sqrt{t_3t_2t_1}}+\frac{C}{t_1\sqrt{t_3}}\Bigr)\pi_E(u)\delta_{E}(u)\|\xi\|_u\|\gamma\|_u.
\end{split}
\end{equation}

o\`u $C$ est une constante r\'eelle positive.
\end{proposition}

\begin{proof}
 Soit $\xi\in A^{0,0}(X,E)$. On a %il existe $f_1,\ldots,f_r\in A^{0,0}(X)$ telles que $\xi=\sum_{i=1}^r f_i\otimes e_i$. On a
\[
\begin{split}
 \frac{\pt }{\pt u}\Bigl(\bigl(\Delta_{E,u} \xi,\xi\bigr)_u \Bigr)&=\frac{\pt }{\pt u}\biggl(\int_Xh_u(\Delta_{E,u}\xi,\xi)\omega_X \biggr) \\
&=\int_X h_u\Bigl(\frac{\pt \Delta_{E,u}}{\pt u}\xi,\xi\Bigr)\omega_X+\int_X\Bigl(\frac{\pt }{\pt u}\log h_u\Bigr) h_u(\Delta_{E,u}\xi,\xi)\omega_X\\
&=\biggl(\frac{\pt \Delta_{E,u}}{\pt u}\xi,\xi\biggr)_u+\biggl(\Delta_{E,u} \xi,\Bigl(\frac{\pt}{\pt u}\log h_u\Bigr) \xi \biggr)_u.
\end{split}
\]
 notons qu'on a utilis\'e le fait que $\frac{\pt}{\pt u}\log h_u$ est r\'eelle.\\
Or
\[
 \bigl(\Delta_{E,u} \xi,\xi\bigr)_u=\frac{i}{2\pi }\int_X\sum_{kj}\frac{\pt f_k}{\pt \z}\frac{\pt \overline{f_j}}{\pt z}h_u(e_k,e_j)\frac{i}{2\pi}dz\wedge d\z,
\]

donc
\[
\frac{\pt }{\pt u}\Bigl(\bigl(\Delta_{E,u} \xi,\xi\bigr)_u \Bigr)=\int_X\Bigl(\frac{\pt }{\pt u}\log h_u\Bigr)\sum_{kj}\frac{\pt f_k}{\pt \z}\frac{\pt \overline{f_j}}{\pt z}h_u(e_k,e_j)\frac{i}{2\pi}dz\wedge d\z.
\]
En regroupant cela, on trouve que
\[
\begin{split}
 \Bigl|\Bigl(\frac{\pt \Delta_{E,u}}{\pt u}\xi,\xi\Bigr)_u\Bigr|&=\Bigl|\int_X\Bigl(\frac{\pt }{\pt u}\log h_u\Bigr)\sum_{kj}\frac{\pt f_k}{\pt \z}\frac{\pt \overline{f_j}}{\pt z}h_u(e_k,e_j)\frac{i}{2\pi}dz\wedge d\z-\Bigl(\Delta_{E,u} \xi,\bigl(\frac{\pt}{\pt u}\log h_u\bigr) \xi \Bigr)_u\Bigr|\\
&\leq \Bigl|\int_X\bigl(\frac{\pt }{\pt u}\log h_u\bigr)\sum_{kj}\frac{\pt f_k}{\pt \z}\frac{\pt \overline{f_j}}{\pt z}h_u(e_k,e_j)\frac{i}{2\pi}dz\wedge d\z\Bigl|+\Bigl|\Bigl(\Delta_{E,u} \xi,(\frac{\pt}{\pt u}\log h_u) \xi \Bigr)_u\Bigr|\\
&\leq \delta_{E}(u)\int_X\Bigl|\sum_{kj}\frac{\pt f_k}{\pt \z}\frac{\pt \overline{f_j}}{\pt z}h_u(e_k,e_j)\Bigr|\frac{1}{2\pi}|dz\wedge d\z|+\delta_{E}(u) \|\Delta_{E,u}\xi\|_{L^2,u}\|\xi\|_{L^2,u}
\end{split}
\]

mais $\sum_{kj}\frac{\pt f_k}{\pt \z}\frac{\pt \overline{f_j}}{\pt z}h_u(e_k,e_j)\geq 0$, on peut \'ecrire la derni\`ere in\'egalit\'e comme suit
\[
\begin{split}
\Bigl|\Bigl(\frac{\pt \Delta_{E,u}}{\pt u}\xi,\xi\Bigr)_u\Bigr|&\leq\delta_{E}(u)\int_X\sum_{kj}\frac{\pt f_k}{\pt \z}\frac{\pt \overline{f_j}}{\pt z}h_u(e_k,e_j)\frac{1}{2\pi}|dz\wedge d\z|+\delta_{E}(u) \|\Delta_{E,u}\xi\|_{L^2,u}\|\xi\|_{L^2,u}\\
&=\delta_{E}(u)\bigl(\Delta_{E,u}\xi,\xi\bigr)_u+\delta_{E}(u) \|\Delta_{E,u}\xi\|_{L^2,u}\|\xi\|_{L^2,u}\\
&\leq 2\delta_{E}(u) \|\Delta_{E,u}\xi\|_{L^2,u}\|\xi\|_{L^2,u}.
\end{split}
\]

De la m\^eme facon, on \'etablit que
{\allowdisplaybreaks
\begin{align*}
 \Bigl|\Bigl(\frac{\pt \Delta_{E,u}}{\pt u}\xi,\eta\Bigr)_u\Bigr|&=\Bigl|\int_X\bigl(\frac{\pt }{\pt u}\log h_u\bigr)\sum_{kj}\frac{\pt f_k}{\pt \z}\frac{\pt \overline{g_j}}{\pt z}h_u(e_k,e_j)\frac{i}{2\pi}dz\wedge d\z-\Bigl(\Delta_{E,u} \xi,(\frac{\pt}{\pt u}\log h_u) \eta \Bigr)_u\Bigr|\\
&\leq \int_X\Bigl|\frac{\pt }{\pt u}\log h_u\Bigr|\Bigl|\sum_{kj}\frac{\pt f_k}{\pt \z}\frac{\pt \overline{g_j}}{\pt z}h_u(e_k,e_j)\Bigr|\frac{i}{2\pi}dz\wedge d\z+\Bigl|\Bigl(\Delta_{E,u} \xi,(\frac{\pt}{\pt u}\log h_u) \eta \Bigr)_u\Bigr|\\
&=\int_X\bigl|\frac{\pt }{\pt u}\log h_u\bigr|\Bigl|h_u\Bigl(\sum_{k=1}^r\frac{\pt f_k}{\pt \z},\sum_{k=1}^r\frac{\pt g_k}{\pt \z}\Bigr)\Bigr|\frac{i}{2\pi}dz\wedge d\z+\Bigl|\Bigl(\Delta_{E,u} \xi,(\frac{\pt}{\pt u}\log h_u) \eta \Bigr)_u\Bigr|\\
&\leq \delta_{E}(u)\int_X h_u\Bigl(\sum_{k=1}^r \frac{\pt f_k}{\pt \z},\sum_{k=1}^r \frac{\pt f_k}{\pt \z}\Bigr)^\frac{1}{2}h_u\Bigl(\sum_{k=1}^r \frac{\pt g_k}{\pt \z},\sum_{k=1}^r \frac{\pt g_k}{\pt \z}\Bigr)^\frac{1}{2}\frac{i}{2\pi} dz\wedge d\z\\
&+\delta_{E}(u)\|\Delta_{E,u} \xi\|_{L^2,u} \|\eta \|_{L^2,u}\quad \text{On a utilis\'e } |h_u(v,v')|\leq h_u(v,v)^{\frac{1}{2}}h_u(v',v')^{\frac{1}{2}}\\
&\leq \delta_{E}(u)\biggl(\int_X h_u\Bigl(\sum_{k=1}^r \frac{\pt f_k}{\pt \z},\sum_{k=1}^r \frac{\pt f_k}{\pt \z}\Bigr)dz\wedge d\z\biggr)^\frac{1}{2} \biggl(\int_X h_u\Bigl(\sum_{k=1}^r \frac{\pt g_k}{\pt \z},\sum_{k=1}^r \frac{\pt g_k}{\pt \z}\bigr)dz\wedge d\z\biggr)^\frac{1}{2}\\
&+\delta_{E}(u)\|\Delta_{E,u} \xi\|_{L^2,u} \|\eta \|_{L^2,u}\quad \text{Par Cauchy-Schwartz}\\
&=\delta_{E}(u)(\Delta_{E,u}\xi,\xi)^{\frac{1}{2}}(\Delta_{E,u}\eta,\eta)^{\frac{1}{2}}+\delta_{E}(u)\|\Delta_{E,u} \xi\|_{L^2,u} \|\eta \|_{L^2,u}\\
&\leq \delta_{E}(u)\|\Delta_{E,u} \xi\|_{L^2,u}^{\frac{1}{2}} \|\Delta_{E,u} \eta\|_{L^2,u}^{\frac{1}{2}} \|\xi\|_{L^2,u}^\frac{1}{2} \|\eta \|_{L^2,u}^\frac{1}{2}+\delta_{E}(u)\|\Delta_{E,u} \xi\|_{L^2,u} \|\eta \|_{L^2,u}.\\
\end{align*}
}
Si l'on pose $\eta=e^{-t\Delta_{E,u}}\frac{\pt \Delta_{E,u}}{\pt u}\gamma$, alors
\[
\begin{split}
\bigl\|\Delta_{E,u}\eta\bigr\|_{L^2,u}&=\Bigl\|\Delta_{E,u} e^{-t\Delta_{E,u}}\frac{\pt \Delta_{E,u}}{\pt u}\gamma\Bigr\|_{L^2,u}\\
&=\Bigl\|\pt_t e^{-t\Delta_{E,u}} \frac{\pt \Delta_{E,u}}{\pt u}\gamma\Bigr\|_{L^2,u}\\
&\leq \frac{C}{t}\Bigl\|\frac{\pt \Delta_{E,u}}{\pt u}\gamma\Bigr\|_{L^2,u}, \quad (\text{puisque}\, \|\pt_t e^{-t\Delta_{E,u}}v\|\leq \frac{C}{t}\|v\|)
\end{split}
\]
donc,
\[
\Bigl\|\Delta_{E,u} e^{-t\Delta_{E,u}}\frac{\pt \Delta_{E,u}}{\pt u}\gamma\Bigr\|_{L^2,u} \leq \frac{C}{t}\|\frac{\pt \Delta_{E,u}}{\pt u}\gamma\|_{L^2,u}.
\]
On a,
\[
\begin{split}
 \Bigl|\bigl(\frac{\pt \Delta_{E,u}}{\pt u}\xi, e^{-t\Delta_{E,u}}\frac{\pt \Delta_{E,u}}{\pt u}\gamma\bigr)_u\Bigr|&\leq \delta_{E}(u)\|\Delta_{E,u} \xi\|_{L^2,u}^{\frac{1}{2}} \|\Delta_{E,u} e^{-t\Delta_{E,u}}\frac{\pt \Delta_{E,u}}{\pt u}\gamma\|_{L^2,u}^{\frac{1}{2}} \|\xi\|_{L^2,u}^\frac{1}{2} \|e^{-t\Delta_{E,u}}\frac{\pt \Delta_{E,u}}{\pt u}\gamma \|_{L^2,u}^\frac{1}{2}\\
 &+\delta_{E}(u)\|\Delta_{E,u} \xi\|_{L^2,u} \|e^{-t\Delta_{E,u}}\frac{\pt \Delta_{E,u}}{\pt u}\gamma \|_{L^2,u}\\
 &\leq \frac{C}{\sqrt{t}}\delta_{E}(u)\|\Delta_{E,u} \xi\|_{L^2,u}^{\frac{1}{2}} \|\frac{\pt \Delta_{E,u}}{\pt u}\gamma\|^{\frac{1}{2}}_u  \|\xi\|_{L^2,u}^\frac{1}{2} \|e^{-t\Delta_{E,u}}\frac{\pt \Delta_{E,u}}{\pt u}\gamma \|_{L^2,u}^\frac{1}{2}\\
& +\delta_{E}(u)\|\Delta_{E,u} \xi\|_{L^2,u} \|e^{-t\Delta_{E,u}}\frac{\pt \Delta_{E,u}}{\pt u}\gamma \|_{L^2,u}\\
&\leq \frac{C}{\sqrt{t}}\delta_{E}(u) \|\Delta_{E,u} \xi\|_{L^2,u}^{\frac{1}{2}} \|\xi\|_{L^2,u}^\frac{1}{2} \|\frac{\pt \Delta_{E,u}}{\pt u}\gamma \|_{L^2,u} +\delta_{E}(u)\|\Delta_{E,u} \xi\|_{L^2,u} \|\frac{\pt \Delta_{E,u}}{\pt u}\gamma \|_{L^2,u}.\\
\end{split}
\]

On 	a montr\'e que, voir \eqref{encoreestimation11}

\[
\Bigl\|\frac{\pt \Delta_{E,u}}{\pt u} e^{-t\Delta_{E,u}}\Bigr\|_{L^2,u}\leq \frac{\pi_E(u)}{\sqrt{t}}.
\]
Donc
{\small
\[
\begin{split}
\Bigl|\Bigl(\frac{\pt \Delta_{E,u}}{\pt u}e^{-t_1\Delta_{E,u}}\xi, e^{-t_2\Delta_{E,u}}\frac{\pt \Delta_{E,u}}{\pt u}e^{-t_3\Delta_{E,u}}\gamma\Bigr)_u\Bigr|
&\leq \frac{C}{\sqrt{t_2}}\delta_{E}(u) \|\Delta_{E,u} e^{-t_1\Delta_{E,u}}\xi\|_{L^2,u}^{\frac{1}{2}} \|e^{-t_1\Delta_{E,u}}\xi\|_{L^2,u}^\frac{1}{2} \|\frac{\pt \Delta_{E,u}}{\pt u}e^{-t_3\Delta_{E,u}}\gamma \|_{L^2,u}\\
&+\delta_{E}(u)\|\Delta_{E,u} e^{-t_1\Delta_{E,u}}\xi\|_{L^2,u} \|\frac{\pt \Delta_{E,u}}{\pt u}e^{-t_3\Delta_{E,u}}\gamma\|_{L^2,u}\\
&\leq  \frac{C^2\pi_E(u)}{\sqrt{t_3t_2t_1}}\delta_{E}(u)\|\xi\|_{L^2,u}\|\gamma\|_{L^2,u}+\frac{C\pi_E(u)}{t_1\sqrt{t_3}}\delta_{E}(u)\|\xi\|_{L^2,u}\|\gamma\|_{L^2,u}.
\end{split}
\]
}

\end{proof}

\begin{theorem}\label{derivenoyau}
Il existe une constante $c$ telle que
\[
\biggl\| \frac{\pt e^{-t\Delta_{E,u}}}{\pt u}\biggr\|_{L^2,u}\leq c\sqrt{\delta_E(u)}\sqrt{\pi_E(u)}t^{\frac{1}{4}}, \quad \forall u>1.
\]

\end{theorem}
\begin{proof}
On a
\[
\frac{\pt e^{-t\Delta_{E,u}}}{\pt u}=-\int_0^t e^{-(t-s)\Delta_{E,u}}\frac{\pt \Delta_{E,u}}{\pt u}e^{-s\Delta_{E,u}}ds.
\]

Si l'on pose $t_1=s$, $t_2=2(t-s)$ et $t_3=s$ dans \eqref{t1t2t3}, alors
\[
\begin{split}
\Bigl\|e^{-(t-s)\Delta_{E,u}}\frac{\pt \Delta_{E,u}}{\pt u}e^{-s\Delta_{E,u}}\xi\Bigr\|^2_u&=
\biggl(\frac{\pt \Delta_{E,u}}{\pt u}e^{-s\Delta_{E,u}}\xi, e^{-2(t-s)\Delta_{E,u}}\frac{\pt \Delta_{E,u}}{\pt u}e^{-s\Delta_{E,u}}\xi\biggr)_u\\
&\leq \frac{C^2\pi_E(u)}{s\sqrt{2(t-s)}}\delta_E(u)\|\xi\|^2_u+\frac{C\pi_E(u)}{s^{\frac{3}{2}}}\delta_E(u)\|\xi\|^2_u.
\end{split}
\]

Par suite

\begin{align*}
\Bigl\|e^{-(t-s)\Delta_{E,u}}\frac{\pt \Delta_{E,u}}{\pt u}e^{-s\Delta_{E,u}}\Bigr\|^2_{L^2,u}
&\leq  \delta_E(u)\pi_E(u)\bigl(\frac{C^2}{s\sqrt{2(t-s)}}+\frac{C}{s^{\frac{3}{2}}}\Bigr).
\end{align*}

et donc

\[\Bigl\| \frac{\pt e^{-t\Delta_{E,u}}}{\pt u}\Bigr\|_{L^2,u}\leq \int_0^t\Bigl\|e^{-(t-s)\Delta_{E,u}}\frac{\pt \Delta_{E,u}}{\pt u}e^{-s\Delta_{E,u}}\Bigr\|_{L^2,u} ds\leq   \sqrt{\delta_E(u)}\sqrt{\pi_E(u)}\int_0^t \Bigl(\frac{C^2}{s\sqrt{2(t-s)}}+\frac{C}{s^{\frac{3}{2}}}\Bigr)^{\frac{1}{2}}ds
\]
On conclut en remarquant que pour tout $a,b>0$

\[
\begin{split}
\int_{0}^t\biggl(\frac{a}{s\sqrt{2(t-s)}}+\frac{b}{s^{\frac{3}{2}}}\biggr)^\frac{1}{2}ds&
\leq \int_{0}^t \frac{\sqrt{a}}{s^{\frac{1}{2}}{2^\frac{1}{4}(t-s)}^{\frac{1}{4}}}ds+\frac{\sqrt{b}}{s^{\frac{3}{4}}}ds, \quad \text{puisque} \, |x+y|^{\frac{1}{2}}\leq |x|^{\frac{1}{2}}+|y|^{\frac{1}{2}},\\
&=2\sqrt{a}\int_0^{\sqrt{t}}\frac{1}{(t-x^2)^{\frac{1}{4}}}dx+\sqrt{b}t^{\frac{1}{4}} \\
&=2\sqrt{a}t^{\frac{1}{4}}\int_0^{1}\frac{1}{(1-x^2)^{\frac{1}{4}}}dx +\sqrt{b}t^{\frac{1}{4}}  \\
&=2\sqrt{a}t^{\frac{1}{4}}\int_0^{\frac{\pi}{2}}\frac{\cos(\theta)}{\sqrt{1-\sin(\theta)^2}}d\theta +\sqrt{b}t^{\frac{1}{4}}\\
&=2\sqrt{a}t^{\frac{1}{4}}\int_0^{\frac{\pi}{2}}\cos(\theta)^{\frac{1}{2}}d\theta+\sqrt{b}t^{\frac{1}{4}}\\
%&\leq \Bigl( \int_{\frac{t}{2}}^t \frac{{a}}{s^{\frac{3}{2}}}\Bigr)^\frac{1}{2}\Bigl( \int_{\frac{t}{2}}^t \frac{1}{{2^\frac{1}{2}(t-s)}^{\frac{1}{2}}}\Bigr)^\frac{1}{2}+\sqrt{b}\log 2\\
%&=\sqrt{a}(\sqrt{2}-1)\frac{1}{t^{\frac{1}{4}}}t^{\frac{1}{4}}+\sqrt{b}\log 2
&=c t^{\frac{1}{4}}.
\end{split}
\]
\end{proof}

\begin{theorem}\label{ggg222}
On a,
\[
\bigl(e^{-t\Delta_{E,u}}\bigr)_{u} \xrightarrow[u\mapsto \infty]{} e^{-t\Delta_{E,\infty}}.
\]
dans l'espace des op\'erateurs born\'es sur le compl\'et\'e de $A^{(0,0)}(X,E)$ pour la m\'etrique $L^2_\infty$.
\end{theorem}

\begin{proof}

On a montr\'e dans  \eqref{derivenoyau} que
{{}\begin{equation}\label{ggg111}
\biggl\| \frac{\pt e^{-t\Delta_{E,u}}}{\pt u}\biggr\|_{L^2_u}\leq c\sqrt{\delta_E(u)}\sqrt{\pi_E(u)}t^{\frac{1}{4}} \quad \forall\, u>1.
\end{equation}
}
On en d\'eduit qu'il existe une constante $c_5$ telle que
\[
\biggl\| \frac{\pt e^{-t\Delta_{E,u}}}{\pt u}\biggr\|_{L^2_\infty}\leq c_5\sqrt{\delta_E(u)}\sqrt{\pi_E(u)}t^{\frac{1}{4}} \quad \forall\, u>1.
\]
Par \eqref{bornepi},  la fonction $\pi_E$ est born\'ee. Et on a,
\begin{align*}
\delta_E(v)&=\sup_{x\in X}\Bigl|\frac{\pt}{\pt v}\log h_v(\si,\si) \Bigr|\\
&=\sup_{x\in X}\Bigl|\frac{\pt \rho_{[v]} }{\pt v}
\frac{h_{[v]}-h_{[v]-1}}{h_v}
\Bigr|\\
&=O\biggl(\biggl\|\frac{h_{[v]+1} }{h_{[v]}}-1\biggr\|_{\sup} \biggr)\quad \forall \,v\geq 1.
\end{align*}
o\`u $[u]$ est la partie enti\`ere de $u$.
 Donc il existe une constante $M'$ telle que
\begin{align*}
\int_u^{u'}\sqrt{\delta_E(v)}dv&\leq \int_{[u]}^{[u']+1}\sqrt{\delta_E(v)}dv \\
&\leq \sum_{p=[u]}^{[u']}\int_p^{p+1}\sqrt{\delta_E(v)}dv\\
&\leq M'\sum_{p=[u]}^{[u']+1} \Bigl\|\frac{h_p}{h_{p-1}}-1\Bigr\|_{\sup}^{\frac{1}{2}}.
\end{align*}
 Comme on a suppos\'e que (voir d\'ebut de \eqref{paragrapheVarMetE}):
\[
 \sum_{p=0}^{\infty} \Bigl\|\frac{h_p}{h_{p-1}}-1\Bigr\|_{\sup}^{\frac{1}{2}}<\infty,
\]
alors  $\bigl(e^{-t\Delta_{E,u}}\bigr)_{u>1}$ converge vers un op\'erateur qu'on note par $L_t$ pour tout $t$,  $L_t$ est
compact pour tout $t$ fix\'e  par \eqref{operateurcompact} et v\'erifie
{{}
\[
\bigl(\pt_t+\Delta_{E,\infty})L_t=0,
\]}
donc $t\mapsto L_t$ est une solution \`a  l'equation de Chaleur associ\'ee au Laplacien $\Delta_{E,\infty}$, par unicit\'e, on conclut que
{{}
\[
L_t=e^{-t\Delta_{E,\infty}}.
\]}
\end{proof}
\begin{remarque}
\rm{ Voir \eqref{exemplechi1} pour des exemples de suites de m\'etriques v\'erifiant \[\sum_p
\Bigl\|\frac{h_p}{h_{p-1}}-1\Bigr\|_{\sup}^{\frac{1}{2}}<\infty,\]
on peut prendre par exemple avec les notations de \eqref{exemplechi1}, $\chi(p)=2^p$.}
\end{remarque}

\begin{theorem}\label{cvinvlap}
La suite $\bigl((I+\Delta_{E,u})^{-1} \bigr)_{u\geq 1}$ converge vers une limite qu'on note par $(I+\Delta_{E,\infty})^{-1}$. C'est
un op\'erateur compact et autoadjoint.
\end{theorem}

\begin{proof} On va montrer que
\[
 \begin{split}
 \Bigl\|\frac{\pt }{\pt u}(I+\Delta_{E,u})^{-1}\Bigr\|_{L^2,u}\leq 8\delta_E(u), \quad \forall u>1.
\end{split}
\]
 Soit $\gamma\in A^{0,0}(X,E)$.
\[
 \frac{\pt }{\pt u}\Bigl(I+\Delta_{E,u}\Bigr)^{-1}=-\Bigl(I+\Delta_{E,u}\Bigr)^{-1}\frac{\pt \Delta_{E,u}}{\pt u} \Bigl(I+\Delta_{E,u}\Bigr)^{-1}.
\]
On pose {{} $\xi=(I+\Delta_{E,u})^{-1}\gamma$} et {{} $\eta=(I+\Delta_{E,u})^{-2}\frac{\pt \Delta_{E,u} }{\pt u}(I+\Delta_{E,u})^{-1}\gamma$} dans l'in\'egalit\'e \eqref{fff111}, alors
{\footnotesize
\[
 \begin{split}
  \Bigl|&\Bigl(\frac{\pt \Delta_{E,u}}{\pt u}(I+\Delta_{E,u})^{-1}\gamma,(I+\Delta_{E,u})^{-2}\frac{\pt\Delta_{E,u}}{\pt u}(I+\Delta_{E,u})^{-1}\gamma \Bigr)_u\Bigr|=\Bigl|\Bigl(\frac{\pt \Delta_{E,u}}{\pt u}\xi,\eta \Bigr)\Bigr|\\
&\leq \delta_E(u)\|\Delta_{E,u} (I+\Delta_{E,u})^{-1}\gamma\|_{u}^{\frac{1}{2}} \|\Delta_{E,u} (I+\Delta_{E,u})^{-2}\frac{\pt \Delta_{E,u} }{\pt u}(I+\Delta_{E,u})^{-1}\gamma\|_{u}^{\frac{1}{2}} \|(I+\Delta_{E,u})^{-1}\gamma\|_{L^2,u}^\frac{1}{2} \|(I+\Delta_{E,u})^{-2}\frac{\pt \Delta_{E,u} }{\pt u}(I+\Delta_{E,u})^{-1}\gamma \|_{u}^\frac{1}{2}\\
&+\delta_E(u)\|\Delta_{E,u} (I+\Delta_{E,u})^{-1}\gamma\|_{u} \|(I+\Delta_{E,u})^{-2}\frac{\pt \Delta_{E,u} }{\pt u}(I+\Delta_{E,u})^{-1}\gamma \|_{L^2,u}\quad \text{par}\,\eqref{fff111}\\
 &\leq 4\delta_E(u) \|\gamma\|_{L^2,u}^{\frac{1}{2}} \|(I+\Delta_{E,u})^{-1}\frac{\pt \Delta_{E,u} }{\pt u}(I+\Delta_{E,u})^{-1}\gamma\|_{L^2,u}^{\frac{1}{2}} \|\gamma\|_{L^2,u}^\frac{1}{2} \|(I+\Delta_{E,u})^{-1}\frac{\pt \Delta_{E,u} }{\pt u}(I+\Delta_{E,u})^{-1}\gamma \|_{L^2,u}^\frac{1}{2}\\
&+4\delta_E(u)\|\gamma\|_{L^2,u} \|(I+\Delta_{E,u})^{-1}\frac{\pt \Delta_{E,u} }{\pt u}(I+\Delta_{E,u})^{-1}\gamma \|_{L^2,u}.
 \end{split}
\]
}

Or $(I+\Delta_{E,u})^{-1}$ est autoadjoint pour la m\'etrique $L^2_u$, donc
\[
 \begin{split}
\Bigl\|(I+\Delta_{E,u})^{-1}&\frac{\pt\Delta_{E,u}}{\pt u}(I+\Delta_{E,u})^{-1}\gamma \Bigr\|^2_u
=\Bigl|\bigl(\frac{\pt \Delta_{E,u}}{\pt u}(I+\Delta_{E,u})^{-1}\gamma,(I+\Delta_{E,u})^{-2}\frac{\pt\Delta_{E,u}}{\pt u}(I+\Delta_{E,u})^{-1}\gamma \bigr)_u\Bigr|\\
 &\leq 8\delta_E(u) \|\gamma\|_{L^2,u} \biggl\|(I+\Delta_{E,u})^{-1}\frac{\pt \Delta_{E,u} }{\pt u}(I+\Delta_{E,u})^{-1}\gamma\biggr\|_{L^2,u}\, \text{par l'in\'egalit\'e ci-dessus}.
\end{split}
\]
Ce qui donne que
\[
 \begin{split}
 \biggl\|\frac{\pt }{\pt u}(I+\Delta_{E,u})^{-1}\biggr\|_{L^2,u}=\biggl\|(I+\Delta_{E,u})^{-1}\frac{\pt\Delta_{E,u}}{\pt u}(I+\Delta_{E,u})^{-1} \biggr\|_{L^2,u} &\leq 8\delta_E(u).
\end{split}
\]
Comme les normes $(L^2_u)_{u\in [1,\infty]}$ sont uniform\'ement \'equivanlentes, il existe une constante $c''$ telle que
 la derni\`ere in\'egalit\'e devient:
\[
 \begin{split}
 \biggl\|\frac{\pt }{\pt u}(I+\Delta_{E,u})^{-1}\biggr\|_{L^2,\infty}\leq 8c''\delta_E(u).
\end{split}
\]
 donc, si $1<p<q$, on obtient:
\allowdisplaybreaks{
\begin{align*}
\biggl\|(I+\Delta_{E,p})^{-1}-(I+\Delta_{E,q})^{-1}\biggl\|_{L^2,\infty}\leq 8c''\int_{p}^q\delta_E(v)dv.
\end{align*}
Or, $\delta_E(v)=O\bigl(\bigl\|\frac{h_{[v]+1} }{h_{[v]}}-1\bigr\|_{\sup} \bigr)\quad \forall \,v\geq 1.
$, donc pour $p,q\gg 1$, on peut trouver des constantes $M''$ et $M^{(3)}$ telles que
\begin{align*}
\biggl\|(I+\Delta_{E,p})^{-1}-(I+\Delta_{E,q})^{-1}\biggl\|_{L^2,\infty}&\leq M''\sum_{k=p}^{q+1} \Bigl\|\frac{h_k}{h_{k-1}}-1\Bigr\|_{\sup}\\
&\leq M^{(3)}\sum_{k=p}^{q+1} \Bigl\|\frac{h_k}{h_{k-1}}-1\Bigr\|_{\sup}^{\frac{1}{2}}.
\end{align*} }
Par hypoth\`ese, voir d\'ebut du paragraphe, le dernier terme tend vers z\'ero lorsque $p,q\mapsto \infty$.
Par suite, la suite d'op\'erateurs compacts $\bigl((\Delta_{E,p}+I)^{-1}\bigr)_{p\in \N}  $ converge vers un op\'erateur $P$, qui est
compact par  \eqref{operateurcompact} et qui v\'erifie $(\Delta_{E,\infty}+I)P=I$.  On le note par  $(\Delta_{E,\infty}+I)^{-1}$.
\end{proof}

\subsection{Une extension maximale positive autoadjointe de $\Delta_{{E,\infty}}$}\label{extDelE}

Dans ce paragraphe, on se propose de montrer que $\Delta_{E,\infty}$ admet une extension maximal positive et autoadjointe
\`a  un sous espace qu'on note par $\h_2(X,E)$.

On commence par revoir la notion d'extension positive et autoadjointe des op\'erateurs Laplaciens g\'en\'eralis\'es. On se limite aux
surface de Riemann compactes. On note que le cas du fibr\'e hermitien trivial est trait\'e dans \cite[Chapitre 14.2]{Buser}.

Soit $X$ une surface de Riemann compacte et $E$ un fibr\'e en droites holomorphe sur $X$. Soit $\omega_X$ une forme de
K\"ahler normalis\'ee sur  $X$, et $h_E$ est une m\'etrique hermitienne continue sur  $E$. On note par
par $A^{(0,0)}(X,E)$ l'espace des fonctions $\cl$ sur $X$, \`a  valeurs dans $E$.  Pour tous $\phi, \psi\in A^{(0,0)}(X,E)$, on d\'efinit
un produit hermitien $(\phi,\psi)$ comme avant, la norme correspondante sera not\'ee par $\|\cdot\|$, et on
l'appellera la norme-$L^2$. Soit
$\h_0(X,E)$ la completion de l'espace  pre-Hilbertien $\bigl(A^{(0,0)}(X,E),(,)\bigr)$. On montre que $\h_0(X,E)$ ne depend pas
des m\'etriques. En fait, si on se donne deux m\'etrique continues  $X$ (resp. sur $E$) alors en utilisant la compacit\'e de $X$, on
obtient deux m\'etriques
\'equivalentes sur $A^{(0,0)}(X,E)$.\\

Lorsque
les metrics de $X$ et de $E$ sont de classe $\cl$, alors on sait que le Laplacien g\'en\'eralis\'e admet une famille
totale  $\phi_0,\phi_1,\phi_2,\ldots$ in $\h_0(X,E)$.  On a:
\[
 \h_0(X,E)=\bigl\{\phi=\sum_{k=0}^\infty a_k\phi_k\, \bigl|\,  \|\phi\|^2=\sum_{k=0}^\infty
|a_k|^2<\infty\bigr\}.
\]

Si l'on pose,
\[
 \h_2(X,E)=\bigl\{\phi=\sum_{k=0}^\infty a_k\phi_k\, \bigl|\,  \sum_{k=0}^\infty
\la_k^2|a_k|^2<\infty  \bigr\}.
\]
Alors, on a
\[
A^{(0,0)}(X,E)\subseteq \h_2(X,E)\subseteq \h_0(X,E).
\]
L'inclusion \`a  droite est \'evidente, quant \`a  l'autre inclusion, elle peut \^etre d\'eduite de \cite[ 14.2.2 p.367]{Buser}. Comme
$\h_0(X,E)$  est complet pour $\vc$, on note que $\h_2(X,E)$ est le complet\'e de $A^{(0,0)}(X,E)$ pour la norme hermitienne
 $\vc_2$, d\'efinie comme suit: $\|\phi\|_2^2=\|\phi\|^2+\|\Delta \phi\|^2$, pour tout $\phi\in
A^{(0,0)}(X,E)$.\\

On peut voir $\h_2(X,E)$ autrement, en effet, on montre que $\phi\in \h_2(X,E)$ tel qu'il existe $(\phi_j)_{j\in \N}$, une suite
dans    $A^{(0,0)}(X,E)$, qui converge vers $\phi$ pour la norme-$L^2$ et tel que la suite
 $(\Delta \phi_j)_{j\in \N}$ admet une  limite dans $\h_0(X,E)$. On peut alors \'ecrire:
\[
\h_2(X,E)=(I+\Delta)^{-1}\h_0(X,E).
\]

Soit $\phi \in \h(X,E)$, il existe $(\phi_j)_{j\in \N}$, une suite dans  $A^{(0,0)}(X,E)$ qui converge vers $\phi$ et telle que
 $(\Delta\phi_j)_{j\in\N}$ admet une  limite. On v\'erifie que la  limite est unique. On introduit alors l'op\'erateur lin\'eaire, $Q$ sur $\h_2(X,E)$ donn\'ee par $Q(\phi)=\lim_{j\in \N}\Delta\phi_j$ pour tout $\phi\in \h_2(X,E)$
et $(\phi_j)_{j\in \N}$ comme avant. Alors $Q$ est une extension maximale positive et autoadjointe de $\Delta$, avec domaine $\mathrm{Dom}(Q)=\h_2(X,E)$. V\'erifions que:
\begin{equation}\label{Qfi}
Q(\phi)=\psi-\phi,
\end{equation}
pour tout $\phi\in \h_2(X,E)$, o\`u $\psi$ est l'unique \'el\'ement dans $\h_0(X,E)$ tel que $\phi=(I+\Delta)^{-1}\psi$.
Soit $\phi\in \h_2(X,E)$, puisque $I+\Delta$ est inversible alors il existe un  unique $\psi\in \h_0(X,E)$ tel que
 $\phi=(I+\Delta)^{-1}\psi$. Soit
$(\psi_j)_{j\in \N}$
une suite dans $A^{(0,0)}(X,E)$ qui converge vers $\psi$ pour la norme $L^2$, alors  on conclut  que
 \[\bigl(\phi_j:=(I+\Delta)^{-1}\psi_j\bigr)
_{j\in \N}\]
converge vers vers $\phi$ pour la norme $L^2$. Comme, $Q(\phi_j)=\Delta\phi_j=\psi_j-\phi_j$ pour tout $j\in \N$, alors
 $(Q(\phi_j))_{j\in\N}$
 converge vers $\psi-\phi$. Alors,
 \[
 Q(\phi)=\psi-\phi,
 \]

Si $T$ est une  extension de $\Delta$, c'est \`a  dire
un op\'erateur lin\'eaire autoadjoint et positif  $T:\mathrm{Dom}(T)\lra
\h_0(X,E)$ tel que: $\h_2(X,E)\subseteq\mathrm{Dom}(T)$ et
que la restriction de $T$ \`a  $\h_2(X,E)$ co\"incide avec $Q$. Soit $\phi=\sum_{j=0}^\infty a_j \phi_j$ dans
$\h_0(X,E)$, alors il existe  $b_j\in \CC$ pour tout $j\in \N$ tels que:
\[
T\phi=\sum_{j=0}^\infty b_j \phi_j,
\]
On a,
\[
b_j=(T(\phi),\phi_j)=(\phi,T(\phi_j))=(\phi,Q(\phi_j))=\la_j(\phi,\phi_j)=\la_j a_j.
\]

Rappelons que $\|T\phi\|^2=\sum_{j=0}^\infty |b_j|^2<\infty $. De l\`a , on obtient que $\sum_{j=0}^\infty \la_j^2|a_j|^2<\infty$, par
suite $\phi\in \h_2(X,E)$. On conclut que  $T=Q$. On dit que
$Q$ est une extension maximale positive et autoadjointe de  $\Delta$.\\

On se propose dans la suite de construire une extension maximale positive et autoadjointe pour  $\Delta_{E,\infty}$, qu'on notera
par $Q_{E,\infty}$. Soit
 $(h_{u})_{u>1}$ comme avant, et soit  $(\Delta_{E,u})_{u>1}$ la suite des Laplaciens g\'en\'eralis\'es  associ\'es.

De \eqref{cvinvlap},   on a $\bigl((I+\Delta_{E,u})^{-1} \bigr)_{u>1}$ converge  vers $(I+\Delta_{E,\infty})^{-1}$,
pour la norme $L^2_{u}$ et par suite pour
norme $L^2_{v}$ avec $v$ fix\'e.  On a aussi que  $(I+\Delta_{E,\infty})^{-1}$ est un op\'erateur compact sur $\h_0(X,E)$. On note que
 $(I+\Delta_{E,u})^{-1}\h_0(X,E)$
 ne d\'epend pas de $u$. En effet,  cela r\'esulte de ce qui pr\'ec\`ede
, et du fait que les m\'etriques sont uniform\'ement \'equivalentes. Alors,
  \[\h_2(X,E)=(I+\Delta_{E,\infty})^{-1}\h_0(X,E).\]

  \begin{lemma}\label{suiteconverge}
 Soit $H$ un espace de Hilbert. Soit $(\vc_u)_{u\geq 1}$ une suite de norme hilbertienne uniform\'ement \'equivalente sur $H$, qui
 qui converge
vers $\vc_\infty$, une norme hilbetienne sur  $H$.

  Soient $(\eta_u)_{u\geq 1}$ et $(\eta'_u)_{u\geq 1}$ deux suites dans  $H$, qui convergent respectivement vers  $\eta_\infty$ et
  $\eta'_\infty$ pour une norme $\vc_v$ avec $v\geq 1$. Alors, la suite complexe $\bigl((\eta_u,\eta'_u)_u \bigr)_{u\geq 1}$ tends
  vers  $(\eta_\infty,\eta'_\infty)_\infty$.
    \end{lemma}

  \begin{proof}
  On a,
  \[ (\eta_u,\eta'_u)_u-(\eta_\infty,\eta'_\infty)_\infty=(\eta_u-\eta_\infty,\eta'_u)_u+(\eta_\infty,\eta'_u-\eta_\infty')_u+(\eta_\infty,\eta'_\infty)_u-(\eta_\infty,\eta'_\infty)_\infty.
  \]
  Par hypoth\`ese, il existe une constante $M$ telle que $|(\eta_u-\eta_\infty,\eta'_u)_u+(\eta_\infty,\eta'_u-\eta_\infty')_u |\leq M \bigl(
  \|\eta_u-\eta_\infty\|_\infty+\|\eta_u'-\eta_\infty'\|_\infty\bigr)$, et puisque  $(\vc_u)_{u\geq 1}$ tend vers $\vc_\infty$, on
conclut que  $\bigl( (\eta_\infty,\eta'_\infty)_u\bigr)_{u\geq 1}$ converge vers  $(\eta_\infty,\eta'_\infty)_\infty$.
  \end{proof}

Soit $\phi\in \h_2(X,E)$. Alors, il existe un unique \'el\'ement $\psi \in \h_0(X,E)$ tel que
 $\phi=(I+\Delta_{E,\infty})^{-1}\psi$.  On d\'emontrera l'unicit\'e dans la suite (voir preuve du
 \eqref{invertibleoperator}). On d\'efinit $Q_{E,\infty}$, l'extension de $\Delta_{E,\infty}$  comme suit: Soit $\phi\in \h_2$,
 donc, par hypoth\`ese,  il existe un unique \'el\'ement $\psi\in \h_0(X,E) $ tel que $\phi=(I+\Delta_{E,\infty})^{-1}\psi$, on pose:
 \[
 Q_{E,\infty}(\phi):=\psi-\phi,
 \]
V\'erifions que $Q_{E,\infty}$ est une extension positive autoadjointe de $\Delta_{E,\infty}$.  Afin d'\'etablir la positivit\'e de
$Q_{E,\infty}$ on a besoin du lemme suivant: Il existe une suite $(\phi_u)_{u>1}$ dans $\h_2(X,E)$ telle que
 $(\phi_u)_{u>1}$ converge vers $\phi$ pour une norme $L^2$, et telle que $(Q_{E,u}(\phi_u))_{u>1}$ tend vers
$Q_{E,\infty}(\phi)$. En effet, soit $\phi_u:=(I+\Delta_{E,u})^{-1}\psi$, $\forall u>1$. On a,
 \[
  \bigl((I+\Delta_{E,u})^{-1}\psi\bigr)_{u\geq 1}\xrightarrow[u\mapsto \infty]{} (I+\Delta_{E,\infty})^{-1}\psi\quad\text{see}\;
  \eqref{cvinvlap}
 \]

Notons que $Q_{E,u}(\phi_u)=\psi -\phi_u$ (voir \eqref{Qfi}), qui    converge vers $\psi-\phi=Q_{E,\infty}(\phi)$. Comme  $Q_{E,u}$ est
 positif pour $(,)_u$. En d'autres termes, $(Q_{E,u}(\phi_u),\phi_u)_u\geq 0$. Puisque  $\bigl((,)_u\bigr)_{u\geq 1}$ converge
uniform\'ement vers $(,)_\infty$ et d'apr\`es  lemme pr\'ec\'edent \eqref{suiteconverge}. On conclut que
\[
(Q_{E,\infty}(\phi),\phi)_\infty\geq 0.
\]
En utilisant le m\^eme argument, on montre que $Q_{E,\infty}$ est autoadjoint.\\

 Soit $\phi\in A^{(0,0)}(X,E)$. Par \eqref{lapintconv}, l'\'el\'ement $\psi:=(I+\Delta_{E,\infty})\phi$
 appartient \`a  $ \h_0(X,E)$
, et
 \[
 (I+\Delta_{E,u})^{-1}\psi\xrightarrow[u\mapsto \infty]{} (I+\Delta_{E,\infty})^{-1}\psi,
 \]
 donc,
 \[
 Q_{E,\infty}(\phi)=\psi-\phi=\Delta_{E,\infty} \phi.
 \]

Soit $T$ une  extension de $Q_{E,\infty}$, c'est \`a  dire un op\'erateur lin\'eaire  positif et autoadjoint $T:\mathrm{Dom}(T)\lra
\h_0(X,E)$ tel que $\h_2(X,E)\subseteq \mathrm{Dom}(T)$ and $T_{|_{\h_2(X,E)}}=Q_{E,\infty}$. Soit $\phi\in \mathrm{Dom}(T)$, on pose
 $\psi:=(I+T)\phi$. On a
 $\psi\in \h_0(X,E)$, donc $\theta:=(I+\Delta_{E,\infty})^{-1}\psi\in \h_2(X,E)$. Par suite,
\[
(I+T)(\theta)=\theta+Q_{E,\infty}(\theta)=\theta+(\psi-\theta)=\psi.
\]
Mais rappelons que  $\psi=(I+T)\phi$, alors
\[
(T+I)(\theta-\phi)=0.
\]
Comme $T$ un op\'erateur positif, et donc  $T+I$ aussi, alors on d\'eduit que
\[
\phi=\theta=(I+\Delta_{E,\infty})^{-1}\psi.
\]
Donc,
\[
\mathrm{Dom}(T)=\h_2(X,E)\quad \text{and}\quad T=Q_{E,\infty}.
\]
Par suite $Q_{E,\infty}$ est une extension  maximale positive et autoadjoint extension pour $\Delta_{E,\infty}$.

\begin{theorem}\label{invertibleoperator} L'op\'erateur $\Delta_{E,\infty}$ admet une extension  maximale positive et autoadjoint \`a
$\h_2(X,E)$, on note aussi cette  extension
par $\Delta_{E,\infty}$. On a:
\[
(I+\Delta_{E,\infty})(I+\Delta_{E,\infty})^{-1}=I,
\]
sur $\h_0(X,E)$, o\`u $I$ est l'operateur identit\'e de $\h_0(X,E)$.
\[
(I+\Delta_{E,\infty})^{-1}(I+\Delta_{E,\infty})=I,
\]
sur $\h_2(X,E)$, o\`u   $I$ est  l'op\'erateur identit\'e de $\h_2(X,E)$.
\end{theorem}

\begin{proof}
La premi\`ere assertion est d\'ej\`a  trait\'e dans la discussion pr\'ec\'edente.

On avait suppos\'e qu'il existe un unique $\psi\in \h_0(X,E)$ tel que $\phi=(I+\Delta_{E,\infty})^{-1}\psi$. Montrons cela. Il  suffit
de montrer cette \'egalit\'e:
\[
(I+\Delta_{E,\infty})(I+\Delta_{E,\infty})^{-1}=I,
\]
sur $\h_0(X,E)$.\\

Fixons $\xi\in A^{(0,0)}(X,E)$. D'apr\`es \eqref{lapintconv} on a,
\[
\underset{u\mapsto \infty}{\lim}\Delta_{E,u}\xi =\Delta_{E,\infty}\xi,
\]
et
\[
\|  \Delta_{E,u\infty}\xi \|_{L^2,\infty}<\infty.
\]

Soient  $\psi\in \h_0(X,E)$ et $\xi\in A^{(0,0)}(X,E)$. En utilisant \eqref{suiteconverge}, on obtient:
\begin{align*}
\bigl((\Delta_{E,\infty}+I)(\Delta_{E,\infty}+I)^{-1}\psi,\xi
\bigr)_{L^2,\infty}&=\bigl((\Delta_{E,\infty}+I)^{-1}\psi, (\Delta_{E,\infty}+I)\xi
\bigr)_{L^2,\infty}\\
&=\lim_{u\mapsto \infty}\bigl((\Delta_{E,u}+I)^{-1}\psi, (\Delta_{E,u}+I)\xi
\bigr)_{L^2,u}\\
&=\lim_{u\mapsto \infty}\bigl(\psi, \xi  \bigr)_{L^2,u}\\
&=\bigl(\psi, \xi  \bigr)_{L^2,\infty}.
\end{align*}
Donc, on a montr\'e que pour tout $\psi\in \h_0(X,E)$,
\[
\bigl((\Delta_{E,\infty}+I)(\Delta_{E,\infty}+I)^{-1}\psi-\psi,\xi
\bigr)_{L^2,\infty}=0\quad \forall \, \xi\in A^{(0,0)}(X,E).
\]
Afin de conclure, rappellons que si $D$ est un sous espace linear d'un Hilbert  $(H,(,)_H)$, et si on   suppose de plus qu'il existe
 $v\in H $ tel que  $(v,z)_H=0$ pour tout $z\in D$, alors $v=0$. En effet, on consid\`ere  $(z_j)_{j\in \N}$ une suite dans  $D$
qui converge vers  $v$. On a
$(v,v)_H=\lim_{j\mapsto \infty}(v,z_j)_H=0$.\\

On applique ce lemme \`a   $H=\h_0(X,E)$, $D=A^{(0,0)}(X,E)$ et $v=(\Delta_{E,\infty}+I)(\Delta_{E,\infty}+I)^{-1}\psi-\psi$.
On conclut que,
\[
(\Delta_{E,\infty}+I)(\Delta_{E,\infty}+I)^{-1}=I.
\]
sur $\h_0(X,E)$.\\

Montrons l'assertion qui reste. Soient $\xi \in \h_2(X,E)$ et $\psi \in \h_0(X,E)$, on a
\begin{align*}
\bigl((\Delta_{E,\infty}+I)^{-1}(\Delta_{E,\infty}+I)\xi,\psi
\bigr)_{L^2,\infty}&=\bigl((\Delta_{E,\infty}+I)\xi, (\Delta_{E,\infty}+I)^{-1}\psi
\bigr)_{L^2,\infty}\\
&=\bigl(\xi, (\Delta_{E,\infty}+I)(\Delta_{E,\infty}+I)^{-1}\psi
\bigr)_{L^2,\infty} \\
&=\bigl(\xi, \psi  \bigr)_{L^2,\infty}.
\end{align*}
Donc,
\[
(\Delta_{E,\infty}+I)^{-1}(\Delta_{E,\infty}+I)=I,
\]
sur $\h_0(X,E)$.\\

\end{proof}

\begin{Corollaire}
  $\Delta_{E,\infty}$ poss\`ede un spectre discret, positif et infini.
  \end{Corollaire}
\begin{proof}
C'est une cons\'equence de la th\'eorie spectrale des op\'erateurs compacts, positifs et autoadjoints, appliqu\'ee \`a
$(I+\Delta_{E,\infty})^{-1}$.
\end{proof}

\begin{theorem}
$\Delta_{E,\infty}$ admet un noyau de Chaleur, qu'on note par  $e^{-t\Delta_{E,\infty}}$, $t>0$.
\end{theorem}

\begin{proof}
On a montr\'e que $\Delta_{E,\infty}$ est un op\'erateur positif et autoadjoint. Par \eqref{semi}, on d\'eduit que  $\Delta_{E,\infty}$
engendre un  semi-groupe $e^{-t\Delta_{E,\infty}}$ pour tout $t>0$.\\
\end{proof}

\subsection{Trace et fonction Z\^eta}

Dans la suite, on notera par $0\leq \la_{\infty,1}\leq \la_{\infty,2}\leq \ldots$ l'ensemble des valeurs propres de $\Delta_{E,\infty}$ compt\'ees avec multiplicit\'e.\\

Dans le cas o\`u les m\'etriques $h_X$ sur $X$ et $h_E$ sur $E$ sont de classe $\cl$, alors on a le r\'esultat suivant:
\begin{proposition}
Le noyau de chaleur $e^{-t\Delta}$ d'un op\'erateur Laplacien sur une vari\'et\'e diff\'erentielle compacte, est un op\'erateur nucl\'eaire, (voir \eqref{definitionTrace})  et on a
\[
\mathrm{Tr}(P e^{-t\Delta})=\sum_{k\in \N}e^{-\la_kt} \quad \forall t>0,
\]
o\`u les $\la_1\leq \la_2\ldots$ sont les valeurs propres non nulles de $\Delta$ compt\'ees avec multiplicit\'es et $P$ est le projecteur de noyau $\ker \Delta$ pour la m\'etrique $L^2$.
\end{proposition}
\begin{proof}
Voir \cite[proposition 2.32 p86]{heat} et \cite{Soule}.
\end{proof}

\begin{theorem}\label{laurentexpansion}
 On a
\[
 \theta(t)=(4\pi t)^{-1}\mathrm{rg}(E)\mathrm{vol}(X)+O(1)
\]
quand $t\rightarrow 0$, avec $O(1)$ est une fonction en $t$ born\'ee au voisinage de $0$.
\end{theorem}
\begin{proof}
 Voir \cite[th\'eor\`eme 2.41]{heat}.
\end{proof}

Soit $\bigl(\Delta_{E,u}\bigr)_{u}$ d'une famille $ \cl$ de Laplaciens g\'en\'eralis\'es  sur surface de Riemann. On note par $\theta_u$ la fonction Th\^eta associ\'ee \`a  $\Delta_{E,u}$, rappelons que si l'on note par $\lambda_{1,u}\leq \lambda_{2,u}\leq \ldots$ les valeurs propres non nuls compt\'ees avec multiplicit\'es de $\Delta_{E,u}$, alors en fixant $u$, on a pour tout $k$ entier, il existe des r\'eels $a_{u,-1},a_{u,0},\ldots,a_{u,k}$ tels que

\[
\theta_u(t)=\sum_{i=-1}^k a_{u,i}t^i+O(t^{k+1})
\]
pour tout $t$ proche de z\'ero.\\

\begin{proposition}
 On consid\`ere $(h_u)_{u>1}$ une famille de classe $\cl$ de m\'etriques hermitiennes sur $E$. Alors $a_{0,u}$ et $a_{-1,u}$ sont des fonctions constantes en $u$.
\end{proposition}

\begin{proof}
D'apr\`es \eqref{laurentexpansion},
\[
 a_{-1,u}=4\pi \,\mathrm{rg}(E)\,\mathrm{vol}(X),
\]
qui ne d\'epend pas de $u$.

Pour montrer la deuxi\`eme assertion, soit $t\neq 1$ un r\'eel strictement positif et on consid\`ere la donn\'ee suivante $\bigl((TX,th_X); (E,h_u)\bigr)$. La variation de la m\'etrique de Quillen associ\'ee \`a  $t$ est donn\'ee par la formule des anomalies suivante:
\[
\begin{split}
 -\log h_{Q,((TX,th_X); (E,h_u))}+\log h_{Q,((TX,h_X); (E,h_u))}&=\int_X ch(E,h_u)\widetilde{Td}(TX,th_X,h_X),\\
\end{split}
\]
voit \cite{BGS1}.\\
On v\'erifie, en utilisant l'expression locale du Laplacien, que:
\[
\Delta_{t,u}:= \Delta_{((TX,th_X); (E,h_u)) }=t^{-1}\Delta_{((TX,h_X); (E,h_u))}=t^{-1}\Delta_{1,u},
\]
alors $\zeta_{\Delta_{t,u}}'(0)=\zeta_{\Delta_{1,u}}(0)\log t+ \zeta'_{\Delta_{1,u}}(0)$, o\`u on a not\'e par $\zeta_{\Delta_{t,u}}$ la fonction Z\^eta associ\'ee \`a  la donn\'ee $\bigl((TX,th_X); (E,h_u)\bigr)$. On montre que
\[\widetilde{Td}(TX,th_X,h_X)=\frac{1}{2}\log t+\frac{1}{6}\log t\, c_1(TX,h_X),\]
dans $\oplus_{p\geq 0}\widetilde{A}^{(p,p)}(X)$, voir \cite{Character} pour la d\'efinition de la classe de Bott-Chern.\\

Comme $\mathrm{Vol}_{th_X}=t^{\dim X}\mathrm{Vol}_{h_X}$, alors
\[h_{L^2,((TX,th_X); (E,h_u))}=t^{2\dim H^0(X,E)} h_{L^2,((TX,h_X); (E,h_u))}. \]
En regroupant cela dans la formule des anomalies, on obtient que
\[
 \begin{split}
  -2\dim H^0(X,E) \log t+ \zeta_{\Delta_{1,u}}(0)\log t &=\frac{1}{2}\log t\int_Xc_1(E,h_u)+ \frac{1}{6}\log t \int_Xc_1(TX,h_X).\\
 \end{split}
\]
On rappelle que $\zeta(0)=a_0$, voir par exemple \cite[Th\'eor\`eme. 1]{Soule}, alors
\begin{equation}\label{a_0}
 a_{0,u}=\frac{1}{2}\int_X c_1(E)+\frac{1}{6}\int_X c_1(TX)+2\dim H^0(X,E),\quad \forall u.
\end{equation}

\end{proof}

\begin{definition}
On pose pour tout $t>0$,
\[
 \theta_\infty(t):=\Bigl\|P^\infty e^{-t\Delta_{E,\infty}}\Bigr\|_{1,\infty}\footnote{Voir \eqref{definitionTrace} pour la d\'efinition de la norme $\vc_{1}$.}.
\]
avec $P^\infty$ est la projection orthogonale de noyau $\ker (\Delta_{E,\infty})$ pour la m\'etrique $L^2_\infty$. On l'appelle la fonction Th\^eta associ\'ee \`a  l'op\'erateur $\Delta_{E,\infty}$.
\end{definition}

\begin{lemma}
Pour tout $u\leq \infty$, soit $P^u$ la projection orthogonale \`a  $H^0(X,E)$ pour la m\'etrique $L^2_u$. On a $P^u$ est un op\'erateur born\'e et
\[
\frac{\pt P^u}{\pt u}=O\Bigl(\Bigl\|\frac{\pt}{\pt u}\log h_u\Bigr\|_{\sup}\Bigr),\quad \forall u<\infty.
\]

\end{lemma}
\begin{proof}
Soit $\{e_1,\ldots,e_r\}$ une base de $H^0(X,E)$. Soit $1< u<\infty$. Pour tout $\xi\in A^{0,0}(X,E)$, il existe $a_1^{(u)}(\xi),\ldots,a_r^{(u)}(\xi)\in \CC$ tels que
\[
P^u\xi=\xi+\sum_{i=1}^r a_i^{(u)}(\xi)\bigl(1\otimes e_i\bigr),
\]
qui sont donn\'es par
\begin{align*}
 \sum_{i=1}^r a_i^{(u)}(\xi)\bigl(1\otimes e_i,1\otimes e_j \bigr)_{L^2,u}=-\bigl( \xi,1\otimes e_j\bigr)_{L^2,u}
\end{align*}

On pose $A^{(u)}$ la matrice
\[
A^{(u)}=\bigl((1\otimes e_i,1\otimes e_j)_{L^2,u} \bigr)_{1\leq i,j\leq r},
\]
et $D_u$ l'op\'erateur $A^{(0,0)}(X,E)\rightarrow \CC^r$ d\'efini par $D_u(\xi)=-{}^t\bigl((\xi,e_1\otimes 1)_{L^2,u},\ldots,(\xi,e_r\otimes 1)_{L^2,u} \bigr)${\footnote{ ${}^t(\ast,\ldots,\ast)$ d\'esigne la transpos\'ee du vecteur ligne $(\ast,\ldots,\ast)$ }}.\\

Comme $A^{(u)}$
est inversible puisque $\det(A^{(u)})=\mathrm{Vol}^2_{L^2,u}$, alors
 \[
\begin{pmatrix}
         &a_1^{(u)}(\xi) \\
         &\vdots \\
         &a_r^{(u)}(\xi) \\
\end{pmatrix}=\bigl(A^{(u)}\bigr)^{-1} D_u(\xi)\quad \forall \xi\in A^{(0,0)}(X,E).
\]
On a
\[
 \frac{\pt P^{(u)}}{\pt u}\xi=\sum_{i=1}^r \frac{\pt a_i^{(u)}}{\pt u} 1\otimes e_i.
\]
Il suffit de montrer que
\[
 \frac{\pt}{\pt u}\Bigl(\bigl(A^{(u)}\bigr)^{-1} D_u \Bigr)=O\Bigl(\Bigl\|\frac{\pt}{\pt u}\log h_u\Bigr\|_{\sup}\Bigr).
\]
On a
\begin{align*}
 \frac{\pt}{\pt u}\Bigl(\bigl(A^{(u)}\bigr)^{-1} D_u \Bigr)=-\bigl(A^{(u)}\bigr)^{-1}\frac{\pt A^{(u)}}{\pt u}\bigl(A^{(u)}\bigr)^{-1}D_u+\bigl(A^{(u)}\bigr)^{-1}\frac{\pt D_u}{\pt u}.
\end{align*}
Puisque $(h_u)_u$ converge vers $h_\infty$ uniform\'ement sur $X$ alors on v\'erifie que $\Bigl(\bigl(A^{(u)}\bigr)^{-1}\Bigr)_u$ converge vers $\bigl(A^{(\infty)}\bigr)^{-1}$ pour un choix de norme matricielle quelconque, donc $(\bigl(A^{(u)}\bigr)^{-1}$ est born\'ee pour cette norme matricielle.\\

 Soit  $\xi,\eta\in A^{(0,0)}(X,E)$, on a
\begin{align*}
 \Bigl|\frac{\pt}{\pt u}\Bigl(\xi,\eta \Bigr)_{L^2,u}\Bigr|&=\Bigl|\int_X\frac{\pt }{\pt u}h_u\Bigl(\xi,\eta\Bigr)\omega_X\Bigr|\\
&=\Bigl|\int_X \Bigl(\frac{\pt }{\pt u}\log h_u\bigr) h_u(\xi,\eta)\omega_X\Bigr|\\
&=\Bigl|\Bigl( \xi,\bigl(\frac{\pt}{\pt u}\log h_u\bigr)\eta\Bigr)_{L^2,u}\Bigr|\\
&\leq \Bigl\|\frac{\pt}{\pt u}\log h_u\Bigr\|_{\sup} \bigl\|\xi\bigr\|_{L^2,u}\bigl\|\eta\bigr\|_{L^2,u}.
\end{align*}
On d\'eduit que
\[
 \frac{\pt P^u}{\pt u}=O\Bigl(\Bigl\|\frac{\pt}{\pt u}\log h_u\Bigr\|_{\sup} \Bigr),\quad u\gg 1.
\]

\end{proof}

\begin{theorem}\label{EEnuclear}
Pour tout $t>0$,  $e^{-t\Delta_{E,\infty}}$ est un op\'erateur nucl\'eaire. On a
\[
 \underset{u\mapsto \infty}{\lim}\theta_u(t)=\theta_\infty(t).
\]
 La fonction Z\^eta $\zeta_\infty$ d\'efinit par:
\[
 \zeta_\infty(s)=\frac{1}{\Gamma(s)}\int_0^\infty \theta_{\infty}(t) t^{s-1}dt,
\]
est holomorphe sur $\mathrm{Re}(s)>1$ et
\[
\zeta_\infty(s)=\sum_{k=1}^\infty \frac{1}{\la_{\infty,k}^s}\quad \forall\; \mathrm{Re}(s)>1.
\]
 admet un prolongement analytique au voisinage de 0 et on a
\[
 \zeta'_{\infty}(0)=\lim_{u\mapsto \infty}\zeta'_u(0).
\]
avec $\zeta_u$ est la fonction Z\^eta  associ\'ee \`a  $\Delta_{E,u}$.
\end{theorem}
\begin{proof}
Soit $R$ un op\'erateur de rang inf\'erieur \`a  $n$, (voir \eqref{paragrapheOpcompacts} pour la d\'efinition de $\si_n(\cdot)$).
Fixons t>0, on a
\begin{align*}
 \si_n(P^\infty e^{-t\Delta_{E,\infty}})_\infty&\leq \bigl\|P^\infty e^{-t\Delta_{E,\infty}}-R\bigr\|_{L^2,\infty}\\
 &\leq \bigl\|P^\infty e^{-t\Delta_{E,\infty}}-P^\infty e^{-t\Delta_{E,u}}\bigr\|_{L^2,\infty}+\bigl\|P^\infty e^{-t\Delta_{E,u}}-P^u e^{-t\Delta_{E,u}}\bigr\|_{L^2,\infty}+ \bigl\|P^u e^{-t\Delta_{E,u}}-R\bigr\|_{L^2,\infty}\\
 &\leq \bigl\|P^\infty\|_{L^2,\infty} \bigl\|e^{-t\Delta_{E,\infty}}-e^{-t\Delta_{E,u}}\bigr\|_{L^2,\infty}+\bigl\|P^\infty -P^u\bigr\|_{L^2,\infty}\bigl\|e^{-t\Delta_{E,u}}\bigr\|_{L^2,\infty}+  \bigl\|P^u e^{-t\Delta_{E,u}}-R\bigr\|_{L^2,\infty}.\\
\end{align*}
 On a alors
\[
 \si_n(P^\infty e^{-t\Delta_{E,\infty}})_\infty\leq  \bigl\|P^\infty\|_{L^2,\infty} \bigl\|e^{-t\Delta_{E,\infty}}-e^{-t\Delta_{E,u}}\bigr\|_{L^2,\infty}+\bigl\|P^\infty -P^u\bigr\|_{L^2,\infty}\bigl\|e^{-t\Delta_{E,u}}\bigr\|_{L^2,\infty}+  \si_n(P^u e^{-t\Delta_{E,u}})_\infty,
\]
Comme  $e^{-t\Delta_{E,u}}$ (resp. $P^u$) converge pour la norme $L^2$ vers $ e^{-t\Delta_{E,\infty}}$ (resp. vers $P^\infty$) et que la suite $\bigl(\bigl\|e^{-t\Delta_{E,u}}\bigr\|_{L^2,\infty}\bigr)_u$ est born\'ee, alors
\begin{equation}\label{SSSSS}
\si_n(P^\infty e^{-t\Delta_{E,\infty}})_\infty\leq
 \liminf_{u\mapsto \infty}\si_n(P^u e^{-t\Delta_{E,u}})_\infty.
\end{equation}

Rappelons que
\[
 \theta_u(t)=\bigl\|P^u e^{-t\Delta_u}\bigr\|_{1,u}
\]
Par \eqref{definitionTrace}, $\|P^u e^{-t\Delta_u}\bigr\|_{1,u}$ est la somme des valeurs propres non nulles de $\bigl((P^u e^{-t\Delta_u})^\ast (P^u e^{-t\Delta_u})\bigr)^\frac{1}{2}$ compt\'ees avec leur multiplicit\'ees o\`u on a not\'e par $\bigl(P^u e^{-t\Delta_u}\bigr)^\ast$ l'adjoint de $P^u e^{-t\Delta_u}$ pour la produit $(,)_{L^2_u}$. V\'erifions  cela: on va montrer que
\[
 e^{-t\Delta_u\ast}\bigl(P^u\bigr)^\ast P^u e^{-t\Delta_u}=0
\]
sur $H^0(X,E)$, et
\[
 e^{-t\Delta_u}\bigl(P^u\bigr)^\ast P^u e^{-t\Delta_u}=e^{-2t\Delta_u}
\]
sur l'orthogonal \`a  $H^0(X,E)$ pour la m\'etrique $L^2_u$.\\

 Puisque $e^{-t\Delta_u}$ laisse stable $H^0(X,E)$ et son orthogonal, alors il nous suffit de v\'erifier que
\[
 \bigl(P^u\bigr)^\ast P^u=0,
\]
sur $H^0(X,E)$, ce qui est \'evident, et
\[
 \bigl(P^u\bigr)^\ast P^u=id,
\]
sur son orthogonal. On peut v\'erifier que $P^u$ est autoadjoint pour $L^2_u$.\\

On a $\bigl(P^{u\ast}\xi,\xi' \bigr)_{L_u^2}=(\xi,P^u\xi' )_{L_u^2}=(\xi,\xi' )_{L_u^2}$ si $\xi'$ est orthogonal \`a  $H^0(X,E)$, donc $P^{u\ast}\xi-\xi\in H^0(X,E)$. Si $\xi$ est orthogonal \`a  $H^0(X,E)$ alors
\[
\begin{split}
 \bigl(P^{u\ast}\xi-\xi,P^{u\ast}\xi-\xi \bigr)_{L^2,u}&= \bigl(\xi,P^u(P^{u\ast}\xi-\xi) \bigr)_{L^2,u} -\bigl(\xi,P^{u\ast}\xi-\xi \bigr)_{L^2,u}\\
&=-\bigl(\xi,0\bigr)_{L^2,u}+0\\
&=0.
\end{split}
\]
donc,
\[
 P^{u\ast}\xi=\xi,\quad\forall \xi\in H^0(X,E)^{\perp_u}.
\]
Maintenant, si $\xi\in \h$, on sait qu'il existe $\eta\in H^{0}(X,E)$ tel que
\[
 P^u\bigl(e^{-t\Delta_u}\xi \bigr)=e^{-t\Delta_u}\xi+\eta,
\]
par suite,
\[
P^{u\ast} P^u\bigl(e^{-t\Delta_u}\xi \bigr)=P^{u\ast}e^{-t\Delta_u}\xi,
\]
On conclut en rappelant que $e^{-t\Delta_u}$ laisse stable  $H^0(X,E)$ et son orthogonal pour $L^2_{u}$.\\

On pose, pour tout $u\geq 1$:
\[
 \theta_{u,\infty}(t):=\bigl\|P^u e^{-t\Delta_{E,u}}\bigr\|_{1,\infty}.
\]
On a pour tout $u,u'>1$,
{\allowdisplaybreaks
\begin{align*}
\Bigl| \theta_{u,\infty}(t)- &\theta_{u',\infty}(t) \Bigr|=\Bigl|\Bigl\| P^u e^{-t\Delta_{E,u}}\Bigr\|_{1,\infty}-\Bigl\|  P^{u'} e^{-t\Delta_{E,u'}}\Bigr\|_{1,\infty}  \Bigr|\\
&\leq \Bigl\| P^u e^{-t\Delta_{E,u}}-P^{u'}e^{-t\Delta_{E,u'}}  \Bigr\|_{1,\infty}\\
&=\biggl\|\int_u^{u'} \frac{\pt}{\pt v}\bigl(P^v e^{-t\Delta_{E,v}}\bigr)dv  \biggr\|_{1,\infty}\\
&=\biggl\|\int_u^{u'}\Bigl( \frac{\pt P^v}{\pt v} e^{-t\Delta_{E,v}}+P^v \frac{\pt }{\pt v}\bigl(e^{-t\Delta_{E,v}} \bigr)\Bigr)dv  \biggr\|_{1,\infty}\\
&=\biggl\|\int_u^{u'} \Bigl(\frac{\pt P^v}{\pt v} P^ve^{-t\Delta_{E,v}}+P^v \frac{\pt }{\pt v}\bigl(e^{-t\Delta_{E,v}} \bigr)\Bigr)dv  \biggr\|_{1,\infty}\\
&\leq  \biggl\|\int_u^{u'} \frac{\pt P^v}{\pt v} P^ve^{-t\Delta_{E,v}}dv\biggl\|_{1,\infty}+ \biggl\|\int_u^{u'}P^v\biggl( \frac{\pt }{\pt v}\bigl(e^{-\frac{t}{2}\Delta_{E,v}}\bigr)e^{-\frac{t}{2}\Delta_{E,v}}+e^{-\frac{t}{2}\Delta_{E,v}}\frac{\pt }{\pt v}\bigl(e^{-\frac{t}{2}\Delta_{E,v}}\bigr) \biggr)dv  \biggr\|_{1,\infty}\\
&= \biggl\|\int_u^{u'} \frac{\pt P^v}{\pt v} P^ve^{-t\Delta_{E,v}}dv\biggl\|_{1,\infty}+\biggl\|\int_u^{u'} \Bigl(P^v\frac{\pt }{\pt v}\bigl(e^{-\frac{t}{2}\Delta_{E,v}}\bigr)\Bigr)\Bigl(P^v e^{-\frac{t}{2}\Delta_{E,v}}\Bigr)+\Bigl(P^v e^{-\frac{t}{2}\Delta_{E,v}}\Bigr)\frac{\pt }{\pt v}\bigl(e^{-\frac{t}{2}\Delta_{E,v}}\bigr) \Bigr)dv  \biggr\|_{1,\infty}\\
&\leq  \int_u^{u'} \biggl\|\frac{\pt P^v}{\pt v} \biggl\|_{L^2,\infty}\theta_{v,\infty}(t)dv+ 2\int_u^{u'} \biggl\|   \frac{\pt }{\pt v}e^{-\frac{t}{2}\Delta_{E,v}}\biggr\|_{L^2,\infty} \theta_{v,\infty}(\frac{t}{2})dv\quad \text{par}\;\eqref{normetrace} \\
&\leq c_7\int_u^{u'}\frac{1}{2^v}\theta_{v,\infty}(t)dv+ c_6\int_u^{u'}\frac{1}{2^v}t^{\frac{1}{4}}\theta_{v,\infty}(\frac{t}{2})dv\\
&\leq c_8\frac{1}{2^u}\int_u^{u'}\Bigl(\theta_{v,\infty}(t)+t^{\frac{1}{4}}\theta_{v,\infty}(\frac{t}{2})\Bigr)dv.\\
 \end{align*}}

On a utilis\'e  les faits suivants  $\frac{\pt }{\pt v}\bigl(e^{-t\Delta_v}\bigr)\xi=\frac{\pt }{\pt v}\bigl(e^{-t\Delta_v}\bigr)P^v\xi$ et que $\frac{\pt P^v}{\pt v}P^v=\frac{\pt P^v}{\pt v}Id$. Montrons la premi\`ere assertion: Soit $\xi \in \h$, il existe $a_0,\ldots,a_r\in \CC$ tels que $\xi-P\xi=\sum_i a_i1\otimes e_i$, donc
\[
\frac{e^{-t\Delta_u}-e^{-t\Delta_{u'}}}{u-u'}\Bigl(\xi-P^v(\xi)\Bigr)=\frac{1}{u-u'}\Bigl(\sum_i a_i(1\otimes e_i)-a_i(1\otimes e_i) \Bigr)=0\quad \forall\, u\neq u'.
\]

On introduit  la fonction $\zeta_{u,\infty}$ d\'efinie par:
\[
 \zeta_{u,\infty}(s):=\frac{1}{\Gamma(s)}\int_0^\infty \theta_{u,\infty}(t)t^{s-1}dt,\quad \forall s\in \CC.
\]

Si l'on pose $B=P^u$ dans  \eqref{thetainfty}, on obtient
\begin{equation}\label{ZZZZZ}
\frac{1-\eps}{1+\eps}\theta_{u,\infty}(t)\leq \theta_u(t)\leq \frac{1+\eps}{1-\eps}\theta_{u,\infty}(t),\quad \forall t>0,
\end{equation}
ce qui donne que
\begin{equation}\label{zetaeps11}
\frac{1-\eps}{1+\eps}\zeta_{u,\infty}(s)\leq \zeta_u(s)\leq \frac{1+\eps}{1-\eps}\zeta_{u,\infty}(s),\quad \forall s\in \R.
\end{equation}
Comme $\zeta_u(s)$ est fini pour tout $\mathrm{Re}(s)>1$ et tout $u\geq 1$, alors  $\zeta_{\infty,u}$ existe et fini pour tout $\mathrm{Re}(s)>1$.\\

D'apr\`es  ce qui pr\'ec\`ede, on a pour tout $Re(s)>1$:
\[
\begin{split}
\biggl|\int_0^\infty t^{s-1}\theta_{u,\infty}(t)dt-\int_0^\infty t^{s-1}\theta_{u',\infty}(t)dt\biggr|&\leq c_8\int_u^{u'}\frac{1}{2^v}\int_0^\infty\Bigl(t^{\mathrm{Re}(s)-1} \theta_{v,\infty}(t)+ t^{\mathrm{Re}(s)+\frac{1}{4}-1}\theta_{v,\infty}(\frac{t}{2})\Bigr)dtdv.\\
\end{split}
\]
Par suite, pour tout r\'eel $Re(s)>1$, on a
\[
\begin{split}
\biggl|\zeta_{u,\infty}(s)-\zeta_{u',\infty}(s)\biggr|&\leq \frac{c_8}{2^u} \int_u^{u'}\biggl(\frac{\Gamma\bigl(\mathrm{Re}(s)\bigr) }{\bigl| \Gamma(s)\bigr|}\zeta_{v,\infty}\bigl(\mathrm{Re}(s)\bigr)+2^{\mathrm{Re}(s)+\frac{1}{4}}\frac{\Gamma\bigl(\mathrm{Re}(s)+\frac{1}{4}\bigr)}{\bigl|\Gamma(s)\bigr|}\zeta_{v,\infty}\bigl(\mathrm{Re}(s)+\frac{1}{4}\bigr)\biggr)dv\\
\end{split}
\]

Notons que si l'on multiplie $\omega_X$ par $t>0$, il est possible de supposer que les premi\`eres valeurs propres de $\Delta_{E,v}$ sont $> 1$,  donc on aura pour tout $s>1$:
\begin{align*}
\zeta_v(s)-\zeta_v(s+\frac{1}{4})&=\sum_{k=1}^\infty \frac{1}{\la_{v,k}^s}-\sum_{k=1}^\infty \frac{1}{\la_{v,k}^{s+\frac{1}{4}}} \\
&\geq \sum_{k=1}^\infty \frac{1}{\la_{v,k}^s}-\frac{1}{\la_{v,1}^s}\sum_{k=1}^\infty \frac{1}{\la_{v,k}^{s}}\\
&=  \Bigl(1-\frac{1}{\la_{1,v}^{\frac{1}{4}}}\Bigr)\zeta_v(s).
\end{align*}
On conclut de \eqref{zetaeps11}, qu'on peut avoir
\[
\zeta_{v,\infty}(s)\geq \zeta_{v,\infty}(s+\frac{1}{4}),\quad \forall s>1,\forall v\gg1.
\]
Donc, on obtient
\begin{equation}\label{gronwall}
\begin{split}
\Bigl|\zeta_{u,\infty}(s)-\zeta_{u',\infty}(s)\Bigr|&\leq \frac{c_8}{2^u}\biggl(\frac{\Gamma\bigl(\mathrm{Re}(s)\bigr)}{|\Gamma(s)|}+2^{\mathrm{Re}(s)+\frac{1}{4}}\frac{\Gamma\bigl(\mathrm{Re}(s)+\frac{1}{4}\bigr)}{|\Gamma(s)|} \biggr) \int_u^{u'}\zeta_{v,\infty}\bigl(\mathrm{Re}(s)\bigr)dv,\; \forall\, \mathrm{Re}(s)>1.\\
\end{split}
\end{equation}
Maintenant on suppose que $s>1$. Montrons que pour tout $ \forall u,u'\gg 1$ et pour tout $ s>1$:
{\footnotesize{\begin{equation}\label{zetarapport}
\exp\biggl(-\frac{c_8}{\log 2}\Bigl|\frac{1}{2^u}-\frac{1}{2^{u'}} \Bigr|\biggl(\frac{\Gamma\bigl(\mathrm{Re}(s)\bigr)}{|\Gamma(s)|}+2^{s+\frac{1}{4}}\frac{\Gamma(s+\frac{1}{4})}{\Gamma(s)} \biggr) \biggr)\leq \frac{\zeta_{u,\infty}(s)}{\zeta_{u',\infty}(s)}\leq \exp\biggl(\frac{c_8}{\log 2}\Bigl|\frac{1}{2^u}-\frac{1}{2^{u'}} \Bigr|\biggl(\frac{\Gamma\bigl(\mathrm{Re}(s)\bigr)}{|\Gamma(s)|}+2^{s+\frac{1}{4}}\frac{\Gamma(s+\frac{1}{4})}{\Gamma(s)} \biggr) \biggr).
\end{equation}}}

Si $u'\mapsto \zeta_{u',\infty}(s)$ est diff\'erentiable pour $s$ fix\'e, on obtient de \eqref{gronwall}
\[
\begin{split}
\biggl|\frac{\pt }{\pt u}\zeta_{u,\infty}(s)\biggr|&\leq \frac{c_8}{2^u}\biggl(\frac{\Gamma\bigl(\mathrm{Re}(s)\bigr)}{|\Gamma(s)|}+2^{\mathrm{Re}(s)+\frac{1}{4}}\frac{\Gamma\bigl(\mathrm{Re}(s)+\frac{1}{4}\bigr)}{|\Gamma(s)|} \biggr)\zeta_{u,\infty}(s),\\
\end{split}
\]
par suite
\[
\biggl|\frac{\pt }{\pt u}\log \zeta_{u,\infty}(s)\biggr|\leq
\frac{c_8}{2^u}\biggl(\frac{\Gamma\bigl(\mathrm{Re}(s)\bigr)}{|\Gamma(s)|}+2^{s+\frac{1}{4}}\frac{\Gamma(s+\frac{1}{4}
)}{\Gamma(s)} \biggr),
\]
et donc,
\[
\Bigl|\log \zeta_{u,\infty}(s)-\log \zeta_{u',\infty}(s)\Bigr|\leq \frac{c_8}{\log
2}\Bigl|\frac{1}{2^u}-\frac{1}{2^{u'}}
\Bigr|\biggl(\frac{\Gamma\bigl(\mathrm{Re}(s)\bigr)}{|\Gamma(s)|}+2^{s+\frac{1}{4}}\frac{\Gamma(s+\frac{1}{4})}{\Gamma
(s)} \biggr) \quad \forall u,u'.
\]

Si $u'\mapsto \zeta_{u',\infty}(s)$ n'est pas diff\'erentiable, on applique le lemme de Gronwall \`a  la fonction
$u'\mapsto \zeta_{u',\infty}(s)$, notons que cette fonction est continue, car localement lipschitzienne, cela r\'esulte
de fait que $v\mapsto \zeta_{v}(s)$ est continue, de \eqref{zetaeps11} et de \eqref{gronwall}.\\

%(le cas $s\in \CC$ s'en d\'eduit du fait que $|\zeta(s)|=\sqrt{2}|\zeta(2\mathrm{Re}(s))|$).\\
De \eqref{zetaeps11} et \eqref{zetarapport}, il existe des constantes r\'eelles positives $c_{12}$ et $c_{13}$ telles que
{\footnotesize{
\begin{equation}\label{zetaboundunif}
c_{12}\exp\biggl(-\frac{c_8}{\log 2}\Bigl|\frac{1}{2^u}-\frac{1}{2^{u'}} \Bigr|\biggl(\frac{\Gamma\bigl(\mathrm{Re}(s)\bigr)}{|\Gamma(s)|}+2^{s+\frac{1}{4}}\frac{\Gamma(s+\frac{1}{4})}{\Gamma(s)} \biggr) \biggr)\leq \frac{\zeta_{u}(s)}{\zeta_{u'}(s)}\leq c_{13}\exp\biggl(\frac{c_8}{\log 2}\Bigl|\frac{1}{2^u}-\frac{1}{2^{u'}} \Bigr|\biggl(\frac{\Gamma\bigl(\mathrm{Re}(s)\bigr)}{|\Gamma(s)|}+2^{s+\frac{1}{4}}\frac{\Gamma(s+\frac{1}{4})}{\Gamma(s)} \biggr) \biggr)\;
\end{equation}}}
pour tout  $u,u'\gg1$ et $ \forall s>1$. \\

On sait que $\zeta_u$ est holomorphe sur $\mathrm{Re}(s)>1$. Montrons que $\bigl(\zeta_u\bigr)$ converge uniform\'ement sur tout domaine de la forme $\al\leq \mathrm{Re}(s)\leq \beta$, avec $1<\al\leq \beta$, vers une fonction holomorphe sur $\mathrm{Re}(s)>1$. On \'etablira apr\`es que cette limite est $\zeta_\infty$. \\

De \eqref{ZZZZZ}, on a
\begin{align*}
 -\frac{2\eps}{1+\eps}\theta_{\infty,u}(t)\leq \theta_u(t)-\theta_{\infty,u}(t)\leq \frac{2\eps}{1-\eps}\theta_{u,\infty}\quad \forall t>0\;\forall u,u'\gg 1
\end{align*}
d'o\`u on tire,
\begin{align*}
 \Bigl|\theta_u(t)-\theta_{\infty,u}(t) \Bigr|\leq \frac{2\eps}{1-\eps}\theta_{u,\infty}(t),
\end{align*}
Par suite,
\begin{align*}
 \Bigl| \zeta_u(s)-\zeta_{u,\infty}(s)\Bigr|\leq \frac{2\eps}{1-\eps}\frac{\Gamma(\mathrm{Re}(s))}{|\Gamma(s)|}\zeta_{u,\infty}(\mathrm{Re}(s))\quad \forall\, \mathrm{Re}(s)>1.
\end{align*}
En utilisant cette in\'egalit\'e, on obtient:
\begin{align*}
 \Bigl| \zeta_u(s)-\zeta_{u'}(s) \Bigr|&\leq   \Bigl| \zeta_{u,\infty}(s)-\zeta_{u',\infty}(s)\Bigr|+\frac{2\eps}{1-\eps}\frac{\Gamma(\mathrm{Re}(s))}{|\Gamma(s)|}\Bigl(\zeta_{u,\infty}(\mathrm{Re}(s))+\zeta_{u',\infty}\bigl(\mathrm{Re}(s)\bigr)\Bigr)\quad \forall \mathrm{Re}(s)>1\; \forall u,u'\gg 1.
\end{align*}
Si l'on choisit $1<\al<\beta$.  Alors on a d\'ej\`a  montr\'e que la suite $\bigl(\zeta_{u,\infty}(\mathrm{Re}(s))\bigr)_u$ est born\'ee sur $\al\leq \mathrm{Re}(s)\leq \beta$ uniform\'ement en $u$. Donc, on peut trouver une constante $K$, qui d\'epend uniquement de $\al$ et $\beta$ telle que
\begin{align*}
 \Bigl| \zeta_u(s)-\zeta_{u'}(s) \Bigr|&\leq   \Bigl| \zeta_{u,\infty}(s)-\zeta_{u',\infty}(s)\Bigr|+\frac{2\eps}{1-\eps}K\quad \forall\, \mathrm{Re}(s)>1\; \forall u,u'\gg 1.
\end{align*}
donc, $\bigl(\zeta_u\bigr)_{u\geq 1}$ converge uniform\'ement vers une limite sur $\al\leq\mathrm{Re}(s)\leq \beta$, cette limite est n\'ecessairement holomorphe sur $\al<\mathrm{Re}(s)<\beta$. \\

\begin{lemma}
Soit $\theta$ une fonction r\'eelle, d\'ecroissante et positive. On pose  $\zeta$, la fonction d\'efinie par
\[
 \zeta(s)=\frac{1}{\Gamma(s)}\int_0^\infty t^{s-1}\theta(t)dt,
\]
pour $s\in \R$.
 %et on suppose qu'il existe $ d>0$, tel que $\zeta(s)$ existe pour tout $\mathrm{Re}(s)>d$. Alors on a
On a
\begin{equation}\label{thetazeta1111}
 \theta(a)\leq \frac{\Gamma(s+1)}{a^{s}}\zeta(s),\quad \, \forall s>a>0.
\end{equation}

\end{lemma}
\begin{proof}
Soit $a>0$  et $s>a$, on a
\[
 \begin{split}
  \zeta(s)&=\frac{1}{\Gamma(s)}\int_0^\infty \theta(t)t^{s-1}dt\\
&=\frac{1}{\Gamma(s)}\int_0^a \theta(t)t^{s-1}dt+\frac{1}{\Gamma(s)}\int_a^\infty \theta(t)t^{s-1}dt\\
&\geq \frac{\theta(a)}{\Gamma(s)}\int_0^a t^{s-1}dt\\
&=\frac{\theta(a)}{\Gamma(s+1)}a^s.
 \end{split}
\]
\end{proof}

Appliquons ce lemme pour montrer que $\theta_u(t)$ est uniform\'ement born\'ee en $u$ pour tout $t>0$. De \eqref{thetazeta1111}, on a
\[
\theta_u(t)\leq \frac{\Gamma(s+1)}{t^{s}}\zeta_u(s),\quad \, \forall s>t>0.
\]
donc, si l'on choisit $s>1$, alors $\zeta_u(s)$ est fini. En utilisant \eqref{zetaboundunif}, on d\'eduit que  pour tout $t>0$ fix\'e, $\theta_u(t)$ est uniform\'ement born\'ee en $u$. Cela va nous permettre de montrer que $e^{-t\Delta_{E,\infty}}$ est un op\'erateur nucl\'eaire, mais avant, on rappelle un lemme technique qui nous sera utile:
 \begin{lemma}\label{liminfsomme}
  Soit $\{c_{n,i}: n\in \N, i\in N \}$ une famille de r\'eels positifs, alors on a
\[
 \underset{n\mapsto \infty}{\liminf}\sum_{i}c_{n,i}\geq \sum_i \underset{n\mapsto \infty}{\liminf}c_{n,i}.
\]

 \end{lemma}
\begin{proof}
 Soit $N$ un entier non nul. On a
\[
 \sum_{i=1}^\infty c_{k,i}\geq \sum_{i=1}^N\inf_{l\geq n} c_{l,i},\quad \forall n\;\forall k\geq n
\]
donc
\[
 \inf_{k\geq n}\sum_{i=1}^\infty c_{k,i}\geq \sum_{i=1}^N\inf_{l\geq n} c_{l,i},\quad \forall n
\]
ce qui donne
\[
\underset{n}{ \liminf}\sum_i c_{n,i}\geq \sum_{i=1}^N\underset{n}{\liminf} c_{n,i}.
\]
donc,
\[
\underset{n}{ \liminf}\sum_i c_{n,i}\geq \sum_{i}\underset{n}{\liminf} c_{n,i}.
\]
puisque tous les termes sont positifs.
\end{proof}

%Un lemme classique, dit que si $b_{i,k}\geq 0$, alors
On a
{\allowdisplaybreaks
\begin{align*}
\frac{1+\eps}{1-\eps}\liminf_{u\mapsto \infty}\theta_u(t)&\geq \liminf_{u\mapsto \infty}\theta_{u,\infty}(t)\quad \text{par}\; \eqref{ZZZZZ}\\
&=\liminf_{u\mapsto \infty}\sum_{n\geq 1}\si_n\bigl(P^u e^{-t\Delta_{E,u}}\bigr)_\infty\\
&\geq \sum_{n\geq 1}\liminf_{u\mapsto \infty} \si_n\bigl(P^u e^{-t\Delta_{E,u}}\bigr)_\infty\quad\text{par}\;\eqref{liminfsomme} \\
&\geq  \sum_{n\geq 1} \si_n\bigl(P^\infty e^{-t\Delta_{E,\infty}}\bigr)_\infty\quad \text{par}\; \eqref{SSSSS}\\
&=\theta_\infty(t).
\end{align*}
}
Comme on a montr\'e que $\theta_{u}(t) $ est  born\'ee pour $t>0$ fix\'e, alors
\[
 \theta_\infty(t)<\infty,\quad \forall t>0.
\]
C'est \`a  dire, on a montr\'e que
\[
 e^{-t\Delta_{E,\infty}},
\]
est un op\'erateur nucl\'eaire.\\

De \eqref{thetainfty}, \eqref{thetainfty11} et \eqref{thetainfty12} alors
\[
 \bigl(\theta_{u}(t)\bigr)_u\xrightarrow[u\mapsto \infty]{} \theta_\infty(t)\quad \forall t>0.
\]
d'o\`u on d\'eduit que
\[
\bigl(\rho_u(t)\bigr)_{u} \xrightarrow[u\mapsto \infty]{} \rho_\infty(t)\quad \forall t>0.
\]

Fixons $\eps>0$, comme $\bigl(\zeta_u(1+\eps) \bigr)_{u\geq 1}$ est convergente, on peut trouver une constante r\'eelle $c$ telle que
\[
 \theta_u(t)\leq \frac{1}{t^{\eps +1}}c,\quad\forall\, 0<t<1+\eps,\; \forall u\gg1.
\]
donc,
\[
 \theta_\infty(t)\leq \frac{1}{t^{\eps+1}}c.
\]
Soit $\mathrm{Re}(s)>1+\eps$. On a
\[
 \zeta_u(s)=\frac{1}{\Gamma(s)}\int_0^\delta t^{s-\eps-2} \bigl(t^{\eps+1}\theta_u(t)  \bigr)dt+\frac{1}{\Gamma(s)}\int_\delta^\infty t^{s-1}\theta_u(t)dt,\; \forall u\geq 1.
\]
Comme $\bigl(\theta_u\bigr)_u$ converge simplement vers $\theta_\infty$, alors par le th\'eor\`eme de convergence domin\'ee de Lebesgue:
\[
 \zeta_u(s)\xrightarrow[u\mapsto \infty]{} \zeta_\infty(s)\quad \forall\, \mathrm{Re}(s)>1+\eps.
\]

Montrons que
\[
 \zeta_\infty(s)=\sum_{k=1}^\infty \frac{1}{\la_{\infty,k}^s},\; \forall\, \mathrm{Re}(s)>1.
\]
Montrons d'abord que
\[
 \zeta_\infty(s)=\sum_{k=1}^\infty \frac{1}{\la_{\infty,k}^s},\; \forall\, s>1.
\]

Soit $\delta>0$, on a pour tout $\mathrm{s}> 1+\eps$:
\begin{align*}
 \zeta_\infty(s)&=\frac{1}{\Gamma(s)}\int_0^\delta t^{s-1}\theta_\infty(t)dt+\int_\delta^\infty t^{s-1}\theta_\infty(t)dt\\
&=\frac{1}{\Gamma(s)}\int_0^\delta \Bigl(\theta_\infty(t) t^{\eps+1} \Bigr)t^{s-2-\eps}dt+\int_\delta^\infty t^{s-1}\theta_\infty(t)dt.
\end{align*}
Comme $\Bigl|\frac{1}{\Gamma(s)}\int_0^\delta \Bigl(\theta_\infty(t) t^{\eps+1} \Bigr)t^{s-2-\eps}dt\Bigr|\leq \frac{c}{s-1-\eps}\delta^{s-1-\eps}$ et $\theta_\infty (t)\leq \theta_\infty(\delta)e^{-\la_{\infty,1}(t-\delta)}$ pour tout $t\geq \delta$, alors
\begin{align*}
 \zeta_\infty(s)&=O(\delta^{s-1-\eps})+\sum_{k=1}^\infty \frac{1}{\Gamma(s)}\int_\delta^\infty t^{s-1}e^{-\la_{\infty,k} t}dt\\
&\leq  O(\delta^{s-1-\eps})+\sum_{k=1}^\infty \frac{1}{\la_{\infty,k}^s}\frac{1}{\Gamma(s)}\int_{\la_{\infty,k}\delta}^\infty t^{s-1}e^{- t}dt\\
&\leq O(\delta^{s-1-\eps})+\sum_{k=1}^\infty \frac{1}{\la_{\infty,k}^s}.\\
\end{align*}
En remarquant que
\[
 \sum_{k=1}^\infty \frac{1}{\la_{\infty,k}^s}=\lim_{N\mapsto \infty}\sum_{k=1}^N \frac{1}{\la_{\infty,k}^s}\leq \frac{1}{\Gamma(s)}\int_0^\infty t^{s-1}\theta_\infty(s)dt.
\]
On conclut que
\[
 \zeta_\infty(s)=\sum_{k=1}^\infty \frac{1}{\la_{\infty,k}^s},\; \forall s>1.
\]
Soit $\mathrm{Re}(s)>1$, on pose
\[
 \zeta_{N,\infty}(s):=\frac{1}{\Gamma(s)}\int_0^\infty t^{s-1}\Bigl(\sum_{k=N}^\infty e^{-\la_{\infty,k}t}\Bigr)dt,
\]
donc
\[
\zeta_{N,\infty}(s)=\zeta_\infty(s)-\sum_{k=1}^{N-1}\frac{1}{\la_{\infty,k}^s},\quad \forall\, \mathrm{Re}(s)>1.
\]
On a
\[
\bigl|\zeta_{N,\infty}(s)\bigr|\leq \Bigl|\frac{1}{\Gamma(s)}\Bigr|\int_0^\infty t^{\mathrm{Re}(s)-1}\Bigl(\sum_{k=N}^\infty e^{-\la_{\infty,k}t}\Bigr)dt=\frac{\Gamma(\mathrm{Re}(s))}{\bigl| \Gamma(s)\bigr|}\Bigl(\zeta_\infty(\mathrm{Re}(s))-\sum_{k=1}^{N-1}\frac{1}{\la_{\infty,k}^{\mathrm{Re}(s)}}  \Bigr).
\]
Le terme \`a  droite tends vers z\'ero lorsque $N$ se rapproche de l'infini. On conclut que
\[
 \zeta_\infty(s)=\sum_{k=1}^\infty \frac{1}{\la_{\infty,k}^{s}},\; \forall \, \mathrm{Re}(s)>1.
\]

Soit $\mathrm{Re}(s)>1$, on a
\begin{align*}
 \bigl|\zeta_u(s)-\zeta_{u'}(s)\bigr|&\leq \bigl| \zeta_u(s)\bigr|+ \bigl| \zeta_{u'}(s)\bigr|\\
&\leq \zeta_u(\mathrm{Re}(s))+\zeta_{u'}(\mathrm{Re}(s)).
\end{align*}
car $\zeta_u(s)=\sum_{k=1}^\infty \frac{1}{\la_{u,k}^s}$ lorsque $\mathrm{Re}(s)>1$. \\

Soit $\mathrm{Re}(x)>0$. On pose pour tout $u\geq 1$
\begin{align*}
\widetilde{\theta}_u(x)=\frac{1}{2\pi  i}\int_{c-i\infty}^{c+i\infty}x^{-s}\Gamma(s)\zeta_u(s)ds,
\end{align*}
o\`u $c$ est un entier $>1$ fix\'e.\\

 On v\'erifie que $\widetilde{\theta}_u(x)=\sum_{k\geq 1} e^{-\la_{u,k}x}$, $\bigl|\widetilde{\theta}_u(x) \bigr|\leq \theta(\mathrm{Re}(x))$, $\widetilde{\theta}_u$ co\"incide avec $\theta_u$ sur $\R^{+\ast}$ et que $\widetilde{\theta}_u(x)=\frac{a_{-1}}{x}+a_0+\widetilde{\rho}(x)$ pour $x$ assez petit.\\

On a
\begin{align*}
 \biggl|\widetilde{\theta}_u(x)-\widetilde{\theta}_{u'}(x) \biggr|&\leq \frac{1}{2\pi i}\int_{c-i\infty}^{c+i\infty}\Biggl|\Bigl(x^{-s}\Gamma(s)\zeta_u(s)-x^{-s}\Gamma(s)\zeta_{u'}(s)\Bigr) \Biggr|ds\\
&\leq \frac{1}{2\pi  i}\int_{c-i\infty}^{c+i\infty }|x|^{\mathrm{Re}(s)}\bigl|\Gamma(s)\bigr|\Bigl(\zeta_u\bigl(\mathrm{Re}(s)\bigr)+\zeta_{u'}\bigl(\mathrm{Re}(s)\bigr) \Bigr)ds \\
&=\frac{1}{2\pi  }|x|^{-c}\Bigl(\zeta_u(c)+\zeta_{u'}(c)\Bigr)\biggl(\int_{-1}^{1 }\bigl|\Gamma(c+it)\bigr| dt+\int_{]-\infty,-1]\cup[1,\infty[}\bigl|\Gamma(c+it)\bigr| dt \biggr)\\
%&=\frac{1}{2\pi  }|x|^{-c}\Bigl(\zeta_u(c)+\zeta_{u'}(c)\Bigr)\int_{-\infty}^{\infty }\prod_{k=0}^{c-1}\sqrt{k^2+t^2}\bigl|\Gamma(it)\bigr| dt \\
&=\frac{1}{2\pi  }|x|^{-c}\Bigl(\zeta_u(c)+\zeta_{u'}(c)\Bigr)\biggl(\int_{-1}^1 \bigl|\Gamma(c+it) \bigr|dt+\int_{]-\infty,-1]\cup[1,\infty[}\prod_{k=0}^{c-1}\sqrt{k^2+t^2}\bigl|\Gamma(it)\bigr|dt\biggr)\\
&=\frac{1}{2\pi  }|x|^{-c}\Bigl(\zeta_u(c)+\zeta_{u'}(c)\Bigr)\biggl(\int_{-1}^1 \bigl|\Gamma(c+it) \bigr|dt+\int_{]-\infty,-1]\cup[1,\infty[}\prod_{k=0}^{c-1}\sqrt{k^2+t^2}\frac{\sqrt{\pi}}{\sqrt{|t|}\sqrt{|\sinh(\pi t)|}}\bigr| dt\biggr)\\
&\quad \text{voir formule}\;\cite[6.1.29]{Table2} \\
\end{align*}
on v\'erifie que la derni\`ere int\'egrale est convergente.  Comme $\bigl(\zeta_u\bigr)$ converge uniform\'ement sur tout domaine de la forme $\delta+1\geq \mathrm{Re}(s)$. Alors il existe une constante r\'eelle $K$ qui d\'epend uniquement de $c$ telle que
\[
 \biggl|\widetilde{\theta}_u(x)-\widetilde{\theta}_{u'}(x) \biggr|\leq K |x|^{-c}, \quad\forall \, \mathrm{Re}(x)>0.
\]
On en d\'eduit que
\[
\biggl|\widetilde{\rho}_u(x)-\widetilde{\rho}_{u'}(x) \biggr|\leq K |x|^{-c}, \quad\forall \, \mathrm{Re}(x)>0.
\]
Soit $r>0$ fix\'e. On note par $D$ la courbe param\'etr\'ee par $r e^{i\al} $, o\`u $-\frac{\pi}{2}\leq\al \leq \frac{\pi}{2}$.  Si l'on remplace $x$ par $x^2$ dans l'in\'egalit\'e ci-dessus, et qu'on consid\`ere $k$ un entier sup\'erieur \`a  $1$, alors on a
\begin{align*}
 \biggl|\int_D\frac{\widetilde{\rho}_u(x^2)-\widetilde{\rho}_{u'}(x^2)}{x^{2k}}dx \biggr|\leq K \int_D |x|^{-2c-2k}dx,
\end{align*}
qui donne
\begin{align*}
 \biggl|a_{u,k}-a_{u',k} \biggr|\leq K  r^{-2c-2k},\;\forall u,u'\; \forall k\in \N_{\geq 1},
\end{align*}
(on a utilis\'e le fait suivant: $\int_D x^{2(j-k)}dx=\delta_{k,j}\pi$).

Par suite, si l'on consid\`ere $0<t<r$ alors
{\allowdisplaybreaks
\begin{align*}
 \bigl|\rho_u(t^2) \bigr|&\leq \sum_{k\geq 1} \bigl|a_{u,k}\bigr|t^{2k}\\
&\leq \sum_{k\geq 1}\bigl| a_{u,k}-a_{u',k}\bigr|t^{2k}+\sum_{k=1}^\infty |a_{u',k}|t^{2k}\\
&\leq \sum_{k\geq 1}Kr^{-2c} \Bigl(\frac{t^2}{r^2}\Bigr)^{k}+ \sum_{k=1}^\infty |a_{u',k}|t^{2k}\\
&\leq K r^{-2c}\frac{t^2}{r^2-t^2}+\sum_{k=1}^\infty |a_{u',k}|t^{2k}.
\end{align*}}
En fixant $u'$, on peut trouver une constante $K'$ et un r\'eel $t_0>0$ telle que
\[
 \sum_{k=1}^\infty |a_{u',k}|t^{2k}\leq K'\frac{t^2}{1-t^2}\quad \forall \, t\in[0,t_0 ].
\]
 On conclut qu'il existe une constante $K''$ telle que
\begin{align*}
 \bigl|\rho_u(t) \bigr|\leq K'' t,\quad \forall\, u\gg 1,\; \forall\, 0\leq t\leq \min(\sqrt{r},t_0).
\end{align*}
Or, on a montr\'e que $\rho_u$ converge simplement vers $\rho_\infty$, donc
\[
 \bigl|\rho_\infty(t)\bigr|\leq K'' t,\quad \forall\,  0\leq t\leq \min(\sqrt{r},t_0).
\]
en d'autres termes $\rho_\infty(t)=O(t)$. Par cons\'equent $\zeta_\infty $ admet un prolongement analytique au voisinage de $s=0$ avec
\[
 \zeta'_\infty(0)=\int_1^\infty \frac{\theta_\infty(t)}{t}dt+a_{-1}+a_0+\int_0^1\frac{\rho_\infty(t)}{t}dt.
\]
Comme
\[
 \theta_u(t)\leq \theta_u(1)e^{-\la_{u,1}(t-1)}\quad \forall t\geq 1.
\] D'apr\`es \eqref{uniformelambda}, on peut trouver une constante $c>0$ telle que $\la_{u,1}\geq c$ pour tout $u\geq 1$. Rappelons que $\bigl(\theta_u(1)\bigr)_{u}$ est born\'ee. Par cons\'equent, il existe $M>0$ telle que
 \[
 \theta_u(t)\leq M e^{-ct}\quad \forall t\geq 1.
\]
Aussi, on a montr\'e que
\[
 \bigl|\rho_u(t)\bigr|\leq K'' t,\quad \forall u\gg 1\;\forall\,  0\leq t\leq \min(\sqrt{r},t_0).
\]

On d\'eduit \`a  l'aide du th\'eor\`eme de convergence domin\'ee de Lebesgue que
\[
 \biggl(\int_0^1\frac{\theta_u(t)}{t}dt\biggr)_u\xrightarrow[u\mapsto \infty]{}\int_1^\infty \frac{\theta_\infty(t)}{t}dt.
\]
et
\[
\biggl(\int_0^1 \frac{\rho_u(t)}{t}dt\biggr)_{u}\xrightarrow[u\mapsto \infty]{} \int_0^1\frac{\rho_\infty(t)}{t}dt.
\]
Par cons\'equent,
\[
 \bigl( \zeta'_u(0)\bigr)_{u\geq 1}\xrightarrow[u\mapsto \infty]{}\zeta'_\infty(0).
\]
En particulier
\begin{equation}\label{lllllll}
 \bigl( \zeta'_p(0)\bigr)_{p\in \N }\xrightarrow[p\mapsto \infty]{}\zeta'_\infty(0).
\end{equation}

Rappelons qu'on a montr\'e \`a  l'aide de la formule des anomalies que la suite des m\'etriques de Quillen suivante:
\[
\Bigl( h_{Q,((X,\omega_X);(E,h_p) )}\Bigr)_{p\in \N}
\]
converge vers une limite qu'on a not\'e par $h_{Q,((X,\omega_X);(E,h_\infty) )}$.\\

Montrons que
\[
\Bigl( h_{L^2,((X,\omega_X);(E,h_p) )}\Bigr)_{p\in \N}
\]
converge vers $h_{L^2,((X,\omega_X);(E,h_\infty) )}$. Pour cela, on a besoin de la proposition suivante:

\begin{proposition}\label{kerH1}
Soit $\bigl(X,\omega)$ une surface de Riemann compacte. Soit $\overline{E}=\bigl(E,h_E\bigr)$ un fibr\'e en
droites muni d'une m\'etrique hermitienne $\cl$. On note $K_X=\Omega_X^{(1,0)}$ qu'on munit de la m\'etrique induite par $\omega$ et $E^\ast$ le fibr\'e dual muni de la m\'etrique duale.

On a
\[
 \Delta^0_{{}_{K_X\otimes E^\ast}}\ast_{1,E}=-\ast_{1,E}\Delta_{E}^1.
\]
o\`u on a not\'e par $\Delta_{\ast}^\ast$ l'op\'erateur Laplacien g\'en\'eralis\'e agissant sur $A^{(0,\ast)}\bigl(X,\ast\bigr)$. En particulier,
\[
 \ker\bigl(\Delta_{E}^1\bigr)=\ast^{-1}_{1,E}\Bigl(H^0\bigl(X,K_X\otimes E^\ast \bigr) \Bigr).
\]

\end{proposition}
\begin{proof}
Soit $\xi\in A^{(0,1)}\bigl(X,E)$. On a
\begin{align*}
\Delta_{{}_{K_X\otimes E^\ast}}\bigl(\ast_{1,E}\xi \bigr)&=\overline{\pt}_{{}_{K_X\otimes E^\ast}}^\ast \overline{\pt}_{{}_{K_X\otimes E^\ast}}\bigl( \ast_{1,E}\xi\bigr)\\
&=\ast_{1,E}\overline{\pt}_E\ast_{0,E}^{-1}\overline{\pt}_{K_X\otimes E^\ast} \ast_{1,E}\xi\\
&=\ast_{1,E}\overline{\pt}_E\overline{\pt}_E^\ast \xi\\
&=\ast_{1,E}\Delta_E^1\xi.\\
\end{align*}
Comme $\ker \Delta^0_{K_X\otimes E^\ast}=H^0\bigl(X,K_X\otimes E^\ast \bigr)$, alors
\[
\ker \Delta_E^1=\ast_{1,E}^{-1}\Bigl(H^0\bigl(X,K_X\otimes E^\ast \bigr) \Bigr).
\]
\end{proof}
\begin{proposition}
Soit $\bigl(h_p\bigr)_{p\in \N}$ une suite de m\'etriques hermitiennes continues qui converge uniform\'ement vers $h_\infty$ sur $E$, un fibr\'e en droites sur une surface de Riemann compacte.

 Alors pour tout $\xi,\xi'\in A^{(0,1)}\bigl(X,K_X\otimes E^\ast\bigr)$,

\begin{enumerate}
\item
\[
\bigl(\ast_{1,\overline{E}_p}^{-1}\xi\bigr)_{p\in \N}\xrightarrow[p\mapsto \infty]{}\ast_{1,\overline{E}_\infty}^{-1}\xi,
\]
pour la m\'etrique $L^2_{q}$, o\`u $q\in \N\cup\{\infty\}$. \\
\item
\[
\Bigl(\bigl(\ast_{1,\overline{E}_p}^{-1}\xi,\ast_{1,\overline{E}_p}^{-1}\xi' \bigr)_{L^2,p}\Bigr)_{p\in \N}\xrightarrow[p\mapsto \infty]{}\bigl(\ast_{1,\overline{E}_\infty}^{-1}\xi,\ast_{1,\overline{E}_\infty}^{-1}\xi' \bigr)_{L^2,\infty}
\]
\end{enumerate}

\end{proposition}
\begin{proof}
Soit $\xi\in A^{(0,1)}\bigl(X,K_X\otimes E^\ast\bigr)$. Comme $X$ est compacte, on peut supposer que $\xi=g \,dz\otimes \tau^\ast$ avec $g\in A^{(0,0)}\bigl(X\bigr)$ et $\tau$ une section holomorphe locale de $E$. On a pour tout $p,q\in \N$

\begin{align*}
\biggl\|\ast_{1,\overline{E}_p}^{-1}\bigl(g dz\otimes \tau^\ast\bigr)-\ast_{1,\overline{E}_q}^{-1}\bigl(g dz\otimes \tau^\ast\bigr)\biggr\|_{L^2,\infty}&=\biggl\|\overline{g}\frac{\overline{\tau^\ast(\tau)}}{h_q(\tau,\tau)} d\z\otimes \tau-\overline{g}\frac{\overline{\tau^\ast(\tau)}}{h_p(\tau,\tau)} d\z\otimes \tau\biggr\|_{L^2,\infty}\\
&=\frac{i}{2\pi}\int_X|g|^2|\tau^\ast(\tau)|^2\biggl|\frac{1}{h_p(\tau,\tau)}-\frac{1}{h_q(\tau,\tau)} \biggr|^2h_\infty(\tau,\tau)dz\wedge d\z\\
&=\frac{i}{2\pi}\int_X|g|^2\biggl|\frac{h_p^\ast(\tau^\ast,\tau^\ast)-h_q^\ast(\tau^\ast,\tau^\ast)}{h_\infty(\tau^\ast,\tau^\ast)} \biggr|^2h_\infty^\ast(\tau^\ast,\tau^\ast)dz\wedge d\z.\\
\end{align*}
Donc,
\[
\bigl(\ast_{1,\overline{E}_p}^{-1}\xi\bigr)_{p\in \N}\xrightarrow[p\mapsto \infty]{}\ast_{1,\overline{E}_\infty}^{-1}\xi,
\]
pour la norme $L^2_\infty$.\\

Montrons que,
\[
\Bigl(\bigl(\ast_{1,\overline{E}_p}^{-1}\xi,\ast_{1,\overline{E}_p}^{-1}\xi' \bigr)_{L^2,p}\Bigr)_{p\in \N}\xrightarrow[p\mapsto \infty]{}\bigl(\ast_{1,\overline{E}_\infty}^{-1}\xi,\ast_{1,\overline{E}_\infty}^{-1}\xi' \bigr)_{L^2,\infty}
\]
On suppose que $\xi'=fdz\otimes \si^\ast$, on a
\begin{align*}
\bigl(\ast_{1,\overline{E}_p}^{-1}\xi,\ast_{1,\overline{E}_p}^{-1}\xi' \bigr)_{L^2,p}&=\frac{i}{2\pi}\int_X\overline{g}f\, \frac{\overline{\tau^\ast(\tau)}\si^\ast(\si)}{h_p(\tau,\tau)h_p(\si,\si)}h_p(\tau,\si)dz\wedge d\z\\
&=\frac{i}{2\pi}\int_X \overline{g}f\,h^\ast_p(\tau^\ast,\si^\ast)dz\wedge d\z.
\end{align*}
Comme $\bigl(h_p\bigr)_{p\in \N}$ converge uniform\'ement vers $h_\infty$ sur $E$, alors $\bigl(h_p^\ast\bigr)_{p\in \N}$ converge uniform\'ement vers $h_\infty^\ast$ sur $E^\ast$, donc
\[
\Bigl(\bigl(\ast_{1,\overline{E}_p}^{-1}\xi,\ast_{1,\overline{E}_p}^{-1}\xi' \bigr)_{L^2,p}\Bigr)_{p\in \N}\xrightarrow[p\mapsto \infty]{}\bigl(\ast_{1,\overline{E}_\infty}^{-1}\xi,\ast_{1,\overline{E}_\infty}^{-1}\xi' \bigr)_{L^2,\infty}.
\]

\end{proof}
 Comme application, on obtient:
 \[
 \Bigl(\mathrm{Vol}_{L^2,p}\bigl( H^1(X,E)\bigr) \Bigr)_{p\in \N}\xrightarrow[p\mapsto \infty]{}\mathrm{Vol}_{L^2,\infty}\bigl( H^1(X,E)\bigr).
 \]

  On r\'e\'ecrit donc \eqref{lllllll} sous la forme:
\[
 T\Bigl((X,\omega_X);(E,h_\infty) \Bigr)=\zeta'_\infty(0).
\]

\end{proof}

\section{Variation de la m\'etrique sur $TX$}\label{paragrapheLapX}
Dans cette partie on \'etudie les diff\'erents objets spectraux associ\'es aux m\'etriques int\'egrables sur le fibr\'e tangent d'une surface de Riemann compacte. Cette \'etude est plus simple contrairement au cas de $E$ o\`u on avait besoin de supposer que la m\'etrique de $E$ soit $1$-int\'egrable.\\

% \begin{remarque}
%{\rm Les m\'etriques $L^2$ sur $A^{0,0}(X,E)$ sont toutes \'equivalentes deux \`a  deux, mais si l'on consid\`ere une suite $\bigl(h_k\bigr)_{k\in \N}$ qui converge vers une limite $h_\infty$ (plus g\'en\'eralement une suite born\'ee) alors les m\'etriques $L^2$ sont uniform\'ement \'equivalentes.}
%\end{remarque}

Soit $X$ une surface de Riemann compacte et $E$ un fibr\'e en droites holomorphe sur $X$. On munit $X$ d'une m\'etrique int\'egrable $h_{X,\infty}$.
Par hypoth\`ese il existe une d\'ecomposition de $h_{X,\infty}=h_{1,\infty}\otimes h_{2,\infty}^{-1}$
en m\'etriques admissibles et des suites $(h_{1,n})_{n\in \N}$ et $(h_{2,n})_{n\in \N}$  de
m\'etriques positives $\cl$ qui convergent uniform\'ement vers  $h_{1,\infty}$ respectivement vers
$h_{2,\infty}$.\\

Fixons $h_E$, une m\'etrique hermitienne $\cl$ sur $E$.\\

On pose $h_{X,n}:=h_{n,1}\otimes h_{2,n}^{-1}$ $\forall \,n\in \N$. On consid\`ere la famille $\bigl(h_{X,u}\bigr)_{u>1}$ associ\'ee \`a
cette suite comme dans \eqref{suitefamille}, rappelons que $h_{X,u}$ est une m\'etrique hermitienne sur le fibr\'e $TX$.

On note par  $\omega_{X,u}$ la forme volume normalis\'ee associ\'ee et par $\Delta_{{}_{X,u}}$ le
Laplacien g\'en\'eralis\'e associ\'e \`a  $h_{X,u}$ et \`a  $h_E$ pour tout $u\in ]1,\infty[$. \\

Pour tout $u\in]1,\infty]$, on notera par $L^2_{X,u}$ (resp. $(\cdot,\cdot)_{L^2,u}$)  la m\'etrique hermitienne (resp.   le produit hermitien) induits par $h_{X,u}$ et
$h_E$ sur $A^{(0,0)}(X,E)$. \\

\begin{definition}
On pose, pour $\xi\in A^{0,0}(X,E)$:
\[\Delta_{{}_{X,\infty}}\xi:=-\sum_{i=1}^rh_{X,\infty}\Bigl(\dif,\dif\Bigr)^{-1}h_E(\si_i,\si_i)^{-1}\dif\Bigl(h_E(\si_i,\si_i)\frac{\pt f_i}{\pt \z} \Bigr)\otimes \si_i.
\]
o\`u $\xi=\sum_{i=1}^r f_i\otimes \si_i$ localement et $\bigl\{\dif \bigr\}$ une base locale de $TX$.
On l'appellera l'op\'erateur Laplacien associ\'e \`a  la m\'etrique $h_{X,\infty}$.\\

\end{definition}

Dans le th\'eor\`eme suivant, on montre que $\Delta_{X,\infty}$ est un op\'erateur lin\'eaire d\'efini sur $A^{(0,0)}(X,E)$ \`a  valeurs dans
$\overline{A^{(0,0)}(X,E)}_{\infty}$, ( le complet\'e de $A^{(0,0)}(X,E)$ pour la m\'etrique $L^2_\infty$).

\begin{theorem}\label{laplaceTX}
Avec les notations pr\'ec\'edentes, on a:
\begin{enumerate}
\item
\[
\underset{u\mapsto \infty}{\lim}\bigl\| \Delta_{{}_{X,u}}\xi \bigr\|^2_{L^2,u}=\bigl\|
\Delta_{{}_{X,\infty}}\xi
\bigr\|^2_{L^2,\infty}<\infty,
\]
\item
\[
\bigl(\Delta_{{}_{X,\infty}}\xi,\xi' \bigr)_{L^2,\infty}=
\bigl(\xi,\Delta_{{}_{X,\infty}}\xi' \bigr)_{L^2,\infty},
\]

\item
\[
\bigl(\Delta_{{}_{X,\infty}}\xi,\xi \bigr)_{L^2,\infty}\geq 0,
\]
\end{enumerate}

pour tout $\xi,\xi'\in A^{0,0}(X,E)$.

\end{theorem}

\begin{proof} Soit $\xi\in A^{(0,0)}(X,E)$, on a pour tout $u>1$:
\[
\bigl\| \Delta_{{}_{X,u}}\xi \bigr\|^2_{L^2,u}=\int_{x\in X}(\Delta_{X,u}\xi,\Delta_{X,u}\xi)_x
\omega_{X,u}=\frac{i}{2\pi}\int_{x\in X} \Bigl(\Delta_{{}_{X,u}}\xi, \Delta_{{}_{X,u}}\xi \Bigr)_x\, h_{X,u
}\Bigl(\dif(x),\dif(x)\Bigr)_x dz_x\wedge d\z_x\\,
\]
o\`u $\{\dif(x)\}$ est une base locale de $TX$ au voisinage de $x$.

Si l'on note par $U$ un ouvert dans lequel $\xi=\sum_{i=1}^r f_i\otimes \si_i$, alors le Laplacien $\Delta_{X,u}$ s'\'ecrit pour tout
$x\in U$, :
\[
\Delta_{{}_{X,u}}\xi=-\sum_{i=1}^rh_{X,u}\Bigl(\dif(x),\dif(x)\Bigr)^{-1}h_E(\si_i,\si_i)^{-1}\dif_x\Bigl(h_E(\si_i,\si_i)\frac{\pt f_i}{\pt \z_x} \Bigr)\otimes \si_i.
\]
Pour simplifier les notations, on \'ecrit $\|\si\|_E^2=h_E(\si,\si)$. On a
{\small
\begin{align*}
 h_{X,u }&\Bigl(\dif(x),\dif(x)\Bigr)_x\Bigl(\Delta_{{}_{X,u}}\xi, \Delta_{{}_{X,u}}\xi \Bigr)_x\\
 &= h_{X,u }\Bigl(\dif,\dif\Bigr)\biggl( \sum_{j=1}^rh_{X,u}\Bigl(\dif,\dif\Bigr)^{-1}\|\si_j\|_E^{-2}\dif\Bigl(\|\si_j\|_E^2\frac{\pt f_j}{\pt \z_x} \Bigr)\otimes \si_j,  \sum_{j=1}^rh_{X,u}\Bigl(\dif,\dif\Bigr)^{-1}\|\si_j\|_E^{-2}\dif\Bigl(\|\si_j\|_E^2\frac{\pt f_j}{\pt \z_x} \Bigr)\otimes \si_j\biggr)  \\
 &=h_{X,u }\Bigl(\dif,\dif\Bigr)^{-1}_x\sum_{k,j}\|\si_k\|_{E}^{-2}\|\si_j\|_E^{-2}\dif\Bigl(\|\si_k\|_E^2\frac{\pt f_k}{\pt \z} \Bigr)\frac{\pt}{\pt \z}\Bigl(\|\si_j\|_E^2\frac{\pt \overline{f_j}}{\pt z}
 \Bigr)h_E(\si_k,\si_j)\\
&\leq \frac{h_{X,\infty }\Bigl(\dif,\dif\Bigr)_x}{h_{X,u }\Bigl(\dif,\dif\Bigr)_x} h_{X,\infty }\Bigl(\dif,\dif\Bigr)_x\Bigl(\Delta_{{}_{X,\infty}}\xi, \Delta_{{}_{X,\infty}}\xi \Bigr)_x,
\end{align*}}
On a $x\mapsto h_{X,\infty }\Bigl(\dif,\dif\Bigr)_x h_{X,u }\Bigl(\dif,\dif\Bigr)_x^{-1} $ est la restriction sur $U$ d'une fonction globale born\'ee sur $X$, rappelons que $\bigl(h_{X,u}\bigr)_u\xrightarrow[u\mapsto\infty]{} h_{X,\infty}$.

Par une partition d'unit\'e, on \'etablit \`a  l'aide du th\'eor\`eme de convergence domin\'ee que:
{\allowdisplaybreaks
\begin{align*}
\bigl\| \Delta_{{}_{X,\infty}}\xi &\bigr\|^2_{L^2,\infty}=\frac{i}{2\pi}\int_X h_{X,\infty
}\Bigl(\dif,\dif\Bigr)_x\Bigl(\Delta_{{}_{X,\infty}}\xi, \Delta_{{}_{X,\infty}}\xi \Bigr)_x\,dz\wedge d\z\\
&=\frac{i}{2\pi}\int_Xh_{X,\infty}\Bigl(\dif,\dif\Bigr)^{-1}\sum_{k,j}\|\si_k\|_E^{-2}\|\si_j\|_E^{-2}\dif\Bigl(\|\si_k\|_E^2\frac{\pt f_k}{\pt \z} \Bigr)\frac{\pt}{\pt \z}\Bigl(\|\si_j\|_E^2\frac{\pt \overline{f_j}}{\pt z} \Bigr)h_E(\si_k,\si_j) dz\wedge d\z\\
&=\frac{i}{2\pi }\int_X\underset{u\mapsto \infty}{\lim}h_{X,u}\Bigl(\dif,\dif\Bigr)^{-1}\sum_{k,j}\|\si_k\|_E^{-2}\|\si_j\|_E^{-2}\dif\Bigl(\|\si_k\|_E^2\frac{\pt f_k}{\pt \z} \Bigr)\frac{\pt}{\pt \z}\Bigl(\|\si_j\|_E^2\frac{\pt \overline{f_j}}{\pt z} \Bigr)h_E(\si_k,\si_j)dz\wedge d\z\\
&=\frac{i}{2\pi }\underset{u\mapsto \infty}{\lim}\int_Xh_{X,u }\Bigl(\dif,\dif\Bigr)^{-1}\sum_{k,j}\|\si_k\|_E^{-2}\|\si_j\|_E^{-2}\dif\Bigl(\|\si_k\|_E^2\frac{\pt f_k}{\pt \z} \Bigr)\frac{\pt}{\pt \z}\Bigl(\|\si_j\|_E^2\frac{\pt \overline{f_j}}{\pt z} \Bigr)h_E(\si_k,\si_j) dz\wedge d\z\\
&=\underset{u\mapsto \infty}{\lim}\bigl\| \Delta_{{}_{X,u}}\xi \bigr\|^2_{L^2,u}.
\end{align*}}
On a aussi
\allowdisplaybreaks{
\begin{align*}
\bigl(\Delta_{{}_{X,\infty}}(f\otimes \si),g\otimes \tau \bigr)_{L^2,\infty}&=\int_X h_E(\si,\si)^{-1}\dif\Bigl(h_E(\si,\si)\frac{\pt f}{\pt \z} \Bigr)\,\overline{g}\,h_E(\si,\tau)dz\wedge d\z\\
&=\int_X h_E(\si,\tau)\frac{\pt f}{\pt \z}\frac{\pt \overline{g}}{\pt z}dz\wedge d\z, \quad \text{par}\quad \eqref{formesimple}\\
&=\bigl(f\otimes \si,\Delta_{{}_{X,\infty}}(g\otimes \tau )\bigr)_{L^2,\infty}.
\end{align*}
}
par lin\'earit\'e, on d\'eduit que
\[
\bigl(\Delta_{{}_{X,\infty}}\xi,\xi'\bigr)_{L^2,\infty}=\bigl(\xi, \Delta_{{}_{X,\infty}}\xi' \bigr)_{L^2,\infty}
\]
pour tout $\xi,\xi'\in A^{(0,0)}(X,E)$. On a
{\allowdisplaybreaks
\begin{align*}
\bigl(\Delta_{{}_{X,\infty}}\xi,\xi \bigr)_{L^2,\infty}&=\frac{i}{2\pi}\int_X (\Delta_{X,\infty}\xi,\xi)_{L^2,\infty}\omega_{X,\infty} \\
&=\frac{i}{2\pi}\int_X \biggl(-\sum_{j=1}^rh_{X,\infty}\Bigl(\dif,\dif\Bigr)^{-1}h_E(\si_j,\si_j)^{-1}\dif\Bigl(h_E(\si_j,\si_j)\frac{\pt f_i}{\pt \z} \Bigr)\otimes \si_j,\sum_{j=1}^r f_j\otimes \si_j  \biggr)\omega_{X,\infty}\\
&=-\frac{i}{2\pi}\int_X \sum_{j,k=1}^r h_E(\si_j,\si_j)^{-1}\frac{\pt}{\pt z}\Bigl(h_E(\si_j,\si_j)\frac{\pt f_j}{\pt \z}\Bigr)\overline{f_k} h_E(\si_j,\si_k) dz\wedge d\z\\
&=\frac{i}{2\pi}\int_X \sum_{j,k=1}^r h_E(\si_j,\si_k)\frac{\pt f_j}{\pt \z}\frac{\pt \overline{f_k}}{\pt z}dz\wedge d\z,\quad \text{par} \;\eqref{formesimple}\\
&=\frac{i}{2\pi}\int_X h_E\Bigl(\sum_{j=1}^r \frac{\pt f_j}{\pt \z} \otimes \si_j, \sum_{j=1}^r \frac{\pt f_j}{\pt \z} \otimes \si_j\Bigr)dz\wedge d\z\\
&\geq 0.
\end{align*}}
\end{proof}

\begin{Corollaire}
$\Delta_{{}_{X,\infty}}$ admet un noyau de chaleur, qu'on note par $e^{-t\Delta_{{}_{X,\infty}}}$, $t>0$.
\end{Corollaire}

\begin{proof}
D'apr\`es \eqref{laplaceTX}, $\Delta_{{}_{X,\infty}}$ est un op\'erateur autoadjoint et positif. Donc de \eqref{semi}, on d\'eduit que $\Delta_{{}_{X,\infty}}$ engendre un semi-groupe $e^{-t\Delta_{{}_{X,\infty}}}$ pour $t>0$.\\
\end{proof}

On introduit la fonction suivante:
 \[\delta_X(u):=\sup_{x\in X}\Bigl|\frac{\pt }{\pt u}\Bigl(\log h_u(\dif,\dif)^{-1} \Bigr)(x) \Bigr|\quad \forall u>1,\]
o\`u $\dif $ est une base locale de $TX$. Notons que $\delta_X$ ne d\'epend pas du choix de la base locale. Comme $h_u=(1-\rho(u))h_{p-1}+\rho(u)h_p$, pour tout $p\in \N^\ast, \forall u\in [p-1,p]$, alors $\frac{\pt }{\pt u}\log h_u(\dif,\dif)^{-1} =\rho(u)\frac{h_{p-1}-h_p}{h_u}=\rho(u)\frac{h_p}{h_u}\frac{h_{p-1}-h_p}{h_p}$ qui est une fonction d\'efinie sur $X$ entier et puisqu'on a suppos\'e que $(h_p)$ converge uniform\'ement vers $h_\infty$ alors il existe $c_1$ une constante r\'eelle telle que

\begin{equation}\label{deltaXU}
\delta_X(u)\leq c_1\biggl| \frac{h_{[u]}-h_{[u]+1}}{h_{[u]+1}}\biggr| \quad \forall\, u\geq 1,
\end{equation}

o\`u $[u]$ d\'esigne la partie enti\`ere de $u$.\\

\begin{proposition}\label{bornelapbelt}
 On a
\[
 \biggl\| \frac{\pt \Delta_{{}_{X,u}}}{\pt u}\xi \biggr\|_{L^2,u}\leq \delta_X(u)\bigl\|\Delta_{{}_{X,u}}\xi\bigr\|_{L^2,u},
\]
pour tout $\xi\in A^{0,0}(X,E)$ et $u>1$.
\end{proposition}

\begin{proof}
Soit $\xi\in A^{0,0}(X,E)$, On a localement,
\[
 \Delta_{{}_{X,u}}(f_i\otimes \si_i)=-h_u(\dif,\dif)^{-1}h(\si_i,\si_i)^{-1}\dif\bigl(h(\si_i,\si_i)\dif f_i \bigr)\otimes \si_i, \quad i=1,\ldots,r.
\]
o\`u $\xi=\sum_{i=1}^rf_i\otimes \si_i$, avec  $\dif$ est une base locale de $TX$.
Donc
\[
 \frac{\pt \Delta_{{}_{X,u}}}{\pt u}(f_i\otimes \si_i)=\frac{\pt }{\pt u}\bigl(\log h_u(\dif,\dif)^{-1} \bigr)\Delta_{{}_{X,u}}(f_i\otimes \si_i),
\]
et par suite,
\[
 \frac{\pt \Delta_{{}_{X,u}}}{\pt u}\xi=\frac{\pt }{\pt u}\bigl(\log h_u(\dif,\dif)^{-1} \bigr)\Delta_{{}_{X,u}}\xi \quad \forall \, \xi\in A^{(0,0)}(X,E).
\]
Rappelons que   $\frac{\pt }{\pt u}\bigl(\log h_u(\dif,\dif)^{-1} \bigr)$ est une fonction continue globale sur $X$ qui ne d\'epend pas du choix de la base.\\

On a
\begin{equation}\label{convcomp}
\begin{split}
\biggl( \frac{\pt \Delta_{{}_{X,u}}}{\pt u}\xi, \frac{\pt \Delta_{{}_{X,u}}}{\pt u}\xi\biggr)_{L^2,u}&=\int_X \biggl|\frac{\pt }{\pt u}\bigl(\log h_u(\dif,\dif)^{-1} \bigr) \biggr|^2h_u\bigl(\Delta_{{}_{X,u}}\xi,\Delta_{{}_{X,u}}\xi\bigr)\omega_{X,u}\\
&\leq |\delta_X(u)|²\int_X h_u(\Delta_{{}_{X,u}}\xi,\Delta_{{}_{X,u}}\xi)\omega_{X,u}\\
& =|\delta_X(u)|^2 \bigl( \Delta_{{}_{X,u}}\xi,  \Delta_{{}_{X,u}}\xi\bigr)_{L^2,u}.
\end{split}
\end{equation}

\end{proof}

\begin{proposition}
La suite $\bigl((I+\Delta_{X,u} )^{-1} \bigr)_{u\geq 1}$ converge vers un op\'erateur compact qu'on note par
\[(\Delta_{{}_{X,\infty}}+I)^{-1}:\mathcal{H}_0(X,E)\lra \mathcal{H}_0(X,E).\]
\end{proposition}
\begin{proof}
Commencons par rappeler que si $\Delta$ est un op\'erateur Laplacien associ\'e \`a  des m\'etriques de classe $\cl$ alors $\|(\Delta+I)^{-1}\|\leq 1$.

En effet, on sait que les vecteurs propres de $\Delta$ forment une base orthogonale pour le compl\'et\'e de $A^{0,0}(X,E)$ par rapport aux m\'etriques consid\'er\'ees, donc si l'on note par $(v_i)_{i}$ une base orthonormale form\'ee par des vecteurs propres de $\Delta$, alors  $\xi \in \overline{A^{0,0}(X,E)}$ s'\'ecrit sous la forme  $ \xi =\sum_i a_iv_i $ avec $a_i\in \CC$, et on v\'erifie que
\[
\Bigl\|(\Delta+I)^{-1}\xi\Bigr\|^2=\Bigl\| \sum_i \frac{a_i}{\la_i+1}v_i \Bigr\|^2\leq \sum_i |a_i|^2\|v_i\|^2=\|\xi\|^2.
\]
On a
\[
 \frac{\pt }{\pt u}\bigl(I+\Delta_{{}_{X,u}}\bigr)^{-1}=-\bigl(I+\Delta_{{}_{X,u}}\bigr)^{-1}\frac{\pt \Delta_{{}_{X,u}}}{\pt u} \bigl(I+\Delta_{{}_{X,u}}\bigr)^{-1}\quad \forall \, u>1.
\]

%Soit $0<q<p$, on a
%{{}\[
%(\Delta_p+I)^{-1}-(\Delta_q+I)^{-1}=(\Delta_p+I)^{-1}(\Delta_q-\Delta_p)
%(\Delta_q+I)^{-1}
%\]}
Soit $\eta \in A^{(0,0)}(X,E)$ et on pose $\xi=(\Delta_{{}_{X,u}}+I)^{-1}\eta$, on a
\[
\begin{split}
\biggl\| \frac{\pt \Delta_{X,u}}{\pt u} \, (\Delta_{{}_{X,u}}+I)^{-1}\eta\biggr\|^2_{L^2,u}&=\biggl\| \frac{\pt \Delta_{{}_{X,u}}}{\pt u} \, \xi\biggr\|^2_{L^2,u}\\
&\leq \bigl|\delta_X(u)\bigr|^2 \bigl(\Delta_{{}_{X,u}}\xi,\Delta_{{}_{X,u}}\xi\bigr)_{L^2,u}\quad \text{par} \;\eqref{convcomp}\\
&= |\delta_X(u)|^2 \bigl(\Delta_{{}_{X,u}}(\Delta_{{}_{X,u}}+I)^{-1}\eta,\Delta_{{}_{X,u}}(\Delta_{{}_{X,u}}+I)^{-1}\eta\bigr)_{L^2,u}\\
&=|\delta_X(u)|^2 \Bigl(\eta- (\Delta_{{}_{X,u}}+I)^{-1}\eta,(\Delta_{{}_{X,u}}+I)^{-1}\eta  \Bigr)_{L^2,u}\\
&\leq 2|\delta_X(u)|^2 \|\eta\|_{L^2,u}^2,
\end{split}
\]
donc
\[
\biggl\| \frac{\pt}{\pt u} \Delta_{{}_{X,u}}\cdot (\Delta_u+I)^{-1}\Bigr\|^2_{L^2,u}\leq 2\big|\delta_X(u)\bigr|^2\quad \forall \, u>1.
\]

Par suite
\[
\biggl\|\frac{\pt }{\pt u}\bigl(I+\Delta_{{}_{X,u}}\bigr)^{-1}\biggr\|_{L^2,u}\leq \bigl\|(\Delta_{{}_{X,u}}+I)^{-1} \bigr\|_{L^2,u}\biggl\| \frac{\pt}{\pt u} \Delta_{{}_{X,u}}\cdot (\Delta_{{}_{X,u}}+I)^{-1}\biggr\|^2_{L^2,u}\leq \sqrt{2}|\delta_X(u)|\quad \forall \,u>1.
\]

Comme  les normes $\bigl\{\|\cdot\|_{L^2,\infty},\, \|\cdot\|_{L^2,u},\, u>1\bigr\}$  sont uniform\'ement \'equivalentes, alors il existe une constante $c_2$ telle que
\[
\biggl\|\frac{\pt }{\pt u}(I+\Delta_{{}_{X,u}})^{-1}\biggr\|_{L^2,\infty}\leq c_2|\delta_X(u)| \quad \forall\, u>1.
\]

Il est possible d'avoir $\bigl|\delta_X(u)\bigr|=O\bigl(\frac{1}{u^2}\bigr)$, puisqu'on a d\'ej\`a  montr\'e, voir \eqref{deltaXU}, que \[\bigl|\delta_X(u)\bigr|\leq c_1\sup_{x\in X}\bigl|\frac{h_p}{h_{p-1}}-1 \bigr|\quad \forall\, u\in [p-1,p],\; p\in \N^\ast.\]

Par cons\'equent, on obtient pour $q>p$
\begin{align*}
\Bigl\|\bigl(\Delta_{{}_{X,p}}+I\bigr)^{-1}-\bigl(\Delta_{{}_{X,q}}+I\bigr)^{-1}\Bigr\|_{L^2,\infty}&=\biggl\|\int_p^q\frac{\pt }{\pt u}\bigl(I+\Delta_{{}_{X,u}}\bigr)^{-1}du\biggr\|_{L^2,\infty}\\
&\leq \int_p^qc_2 |\delta_X(u)|du\\
&=\int_p^q O(\frac{1}{u^2})du\\
 &=O\Bigl(\frac{1}{q}-\frac{1}{p}\Bigr)\quad \forall\, p,q\gg1.
\end{align*}
Notons que pour tout $p\in \N$, l'op\'erateur $(I+\Delta_{{}_{X,p}})^{-1}$ est encore compact dans
$\overline{A^{0,0}(X,E)}_{\infty}$. Par suite, la suite d'op\'erateurs compacts $\bigl((\Delta_{{}_{X,p}}+I)^{-1}\bigr)_{p\in \N}  $
converge vers un op\'erateur $P$ qui est compact par  \eqref{operateurcompact}, on le note par    $(\Delta_{{}_{X,\infty}}+I)^{-1}$.
\end{proof}

\subsection{Une extension maximale positive autoadjointe de $\Delta_{{X,\infty}}$}
Dans ce paragraphe on \'etablit que $\Delta_{X,\infty}$ admet une extension maximale autoadjointe et positive.

\begin{theorem}\label{invertibleoperator} L'op\'erateur $\Delta_{X,\infty}$ admet  une  extension maximale et positive \`a
$\h_2(X,E)$, on note cette extension aussi par $\Delta_{X,\infty}$. On a:
\[
(I+\Delta_{X,\infty})(I+\Delta_{X,\infty})^{-1}=I,
\]
sur $\h_0(X,E)$, o\`u $I$ est l'op\'erateur  identit\'e de $\h_0(X,E)$.
\[
(I+\Delta_{X,\infty})^{-1}(I+\Delta_{X,\infty})=I,
\]
sur $\h_2(X,E)$, avec    $I$ est l'op\'erateur identit\'e de $\h_2(X,E)$.
\end{theorem}

\begin{proof}
La preuve de la premi\`ere assertion est analogue \`a  celle faite dans \eqref{extDelE}.\\

Montrons que:
\[
(I+\Delta_{X,\infty})(I+\Delta_{X,\infty})^{-1}=I,
\]
sur $\h_0(X,E)$.\\

On fixe $\xi\in A^{(0,0)}(X,E)$. On a,
\[
\underset{u\mapsto \infty}{\lim}\bigl\| \Delta_{{}_{X,u}}\xi \bigr\|^2_{L^2,u}=\bigl\|
\Delta_{{}_{X,\infty}}\xi
\bigr\|^2_{L^2,\infty}<\infty\quad \text{par}\; \eqref{laplaceTX},
\]
et
\[
 \biggl\| \frac{\pt \Delta_{{}_{X,u}}}{\pt u}\xi \biggr\|_{L^2,u}\leq \delta_X(u)\bigl\|\Delta_{{}_{X,u}}\xi\bigr\|_{L^2,u}\quad
 \text{par}\; \eqref{bornelapbelt}.
\]
On d\'eduit qu'il existe une constante  $C$ telle que:
\[
 \biggl\| \frac{\pt \Delta_{{}_{X,u}}}{\pt u}\xi \biggr\|_{L^2,u}\leq C\delta_X(u),
\]
pour tout $u\gg 1$.  Rappelons que les  normes $L^2_{X,u}$ sont  uniform\'ement  \'equivalentes, on peut donc trouver une constante
$C'$ telle que:
\[
 \biggl\| \frac{\pt \Delta_{{}_{X,u}}}{\pt u}\xi \biggr\|_{L^2,\infty}\leq C'\delta_X(u),
\]
Donc,
\[
\bigl\| \Delta_{{}_{X,p}}\xi-\Delta_{{}_{X,q}}\xi \bigr\|_{L^2,\infty}\leq C'\int_p^q\delta_X(u)du,
\]
Par cons\'equent $(\Delta_{{}_{X,p}}\xi)_{p\in \N}$ converge vers $\Delta_{{}_{X,\infty}}\xi$ with respect to
$L^2_{X,\infty}$.\\

Si $\psi\in \h_0(X,E)$ et $\xi\in A^{(0,0)}(X,E)$. En utilisant \eqref{suiteconverge}, on obtient:
\begin{align*}
\bigl((\Delta_{X,\infty}+I)(\Delta_{X,\infty}+I)^{-1}\psi,\xi
\bigr)_{L^2,\infty}&=\bigl((\Delta_{X,\infty}+I)^{-1}\psi, (\Delta_{X,\infty}+I)\xi
\bigr)_{L^2,\infty}\\
&=\lim_{u\mapsto \infty}\bigl((\Delta_{X,u}+I)^{-1}\psi, (\Delta_{X,u}+I)\xi
\bigr)_{L^2,u}\\
&=\lim_{u\mapsto \infty}\bigl(\psi, \xi  \bigr)_{L^2,u}\\
&=\bigl(\psi, \xi  \bigr)_{L^2,\infty}.
\end{align*}
On a donc montr\'e que pour tout $\psi\in \h_0(X,E)$,
\[
\bigl((\Delta_{X,\infty}+I)(\Delta_{X,\infty}+I)^{-1}\psi-\psi,\xi
\bigr)_{L^2,\infty}=0\quad \forall \, \xi\in A^{(0,0)}(X).
\]

Comme  la preuve de \eqref{invertibleoperator}, on montre que
\[
(\Delta_{X,\infty}+I)(\Delta_{X,\infty}+I)^{-1}=I.
\]
sur $\h_0(X,E)$.\\

Montrons la derni\`ere  assertion du  th\'eor\`eme. Soient $\xi \in \h_2(X)$ et $\psi \in \h_0(X)$, on a
\begin{align*}
\bigl((\Delta_{X,\infty}+I)^{-1}(\Delta_{X,\infty}+I)\xi,\psi
\bigr)_{L^2,\infty}&=\bigl((\Delta_{X,\infty}+I)\xi, (\Delta_{X,\infty}+I)^{-1}\psi
\bigr)_{L^2,\infty}\\
&=\bigl(\xi, (\Delta_{X,\infty}+I)(\Delta_{X,\infty}+I)^{-1}\psi
\bigr)_{L^2,\infty} \\
&=\bigl(\xi, \psi  \bigr)_{L^2,\infty}.
\end{align*}
Donc,
\[
(\Delta_{X,\infty}+I)^{-1}(\Delta_{X,\infty}+I)=I,
\]
sur $\h_0(X,E)$.

\end{proof}

\begin{Corollaire}
 L'op\'erateur $\Delta_{{}_{X,\infty}}$ admet un spectre discret positif et non born\'e.
\end{Corollaire}
\begin{proof}
L'existence du spectre et sa nature  est une propri\'et\'e des op\'erateurs compacts, voir par exemple \cite[th\'eor\`eme 3.4 p.429]{functional}. Pour la positivit\'e du spectre, elle r\'esulte de la positivit\'e de l'op\'erateur $\Delta_{{}_{X,\infty}}$.
\end{proof}

\begin{proposition}
Il existe une famille $(h_u)_{u}$ de classe $\cl$  telle que

\[
\lim_{u\mapsto \infty}h_u=h_{\infty}
\]
et
\[
\biggl\|\frac{\partial}{\partial u}e^{-t\Delta^u}\biggr\|_{L^2,\infty}=O\bigl(\delta_X(u)\bigr)\quad \forall \, u\gg 1.
\]
\end{proposition}

\begin{proof}
On a
\[
\begin{split}
\frac{\partial }{\partial u}e^{-t\Delta^u}&=\int_{0}^t e^{-(t-s)\Delta^u}\bigl((\partial_u\log h_u) \Delta^u\bigr)e^{-s\Delta^u}ds\quad \text{par}\; \eqref{heatkernel}\\
&=-\int_{0}^t e^{-(t-s)\Delta^u}\bigl(\partial_u\log h_u \bigr)\pt_se^{-s\Delta^u}ds.\\
\end{split}
\]

Soit $\xi\in A^{(0,0)}(X,E)$. Fixons $u>0$. Soit $\bigl(v_{u,k}\bigl)_{k\in \N}$ une base orthonormale pour $L^2_u$ form\'ee par des vecteurs propres de $\Delta_{{}_{X,u}}$.   Il existe des r\'eels $a_k,k\in \N$ tels que $\xi=\sum_{k\in \N}a_{u,k} v_{u,k}$.

On a
\[
\begin{split}
\partial_t e^{-t\Delta_{{}_{X,u}}}v&=-\sum_{k\in\N} a_{u,k}\lambda_{u,k} e^{-\lambda_{u,k} t}v_{u,k}\\
&=-\frac{1}{t}\sum_i a_{i,u}t\lambda_{i,u} e^{-\lambda_{i,u}t},
\end{split}
\]
Comme $a^2 e^{-a}\leq 4e^{-2}, \forall a\geq 0 $, alors
\[
\begin{split}
\Bigl\|\partial_t e^{-\Delta_{{}_{X,u}}} v \Bigr\|^2_{L^2,u}&=\frac{1}{t^2}\sum_{k\in \N} a_{u,k}^2\bigl(\lambda_{u,k} t\bigr)^2e^{-2\lambda_{u,k}t }\\
&\leq \frac{e^{-2}}{t^2}\sum_{k\in \N}a_{u,k}^2\\
&=\frac{e^{-2}}{t^2}\bigl\|v\bigr\|_{L^2,u}^2,
\end{split}
\]
donc
\[
\Bigl\|\partial_t e^{-t\Delta_{{}_{X,u}} }\Bigl\|_{L^2,u}\leq \frac{e^{-1}}{t} \quad \forall\, t>0.
\]
On d\'eduit qu'il existe une constante $M$ ind\'ependante de $u$ telle que
\[
\Bigl\|\partial_te^{-t\Delta_{{}_{X,u}}}\Bigr\|_{L^2,\infty}\leq \frac{M}{t} \quad \forall\, u,\;\forall \, t>0.
\]
Fixons $t>0$. On a pour tout $u\gg 1$
\[
\begin{split}
\biggl\|\int_{\frac{t}{2}}^t e^{-(t-s)\Delta^u}\bigl((\partial_u\log h_u) \bigr)\partial_se^{-s\Delta^u}ds\biggr\|_{L^2,\infty}&\leq \int_{\frac{t}{2}}^t\Bigl\| e^{-(t-s)\Delta^u}\bigl((\partial_u\log h_u) \bigr)\partial_se^{-s\Delta^u} \Bigr\|_{L^2,\infty} \\
&\leq {M'}\delta_X(u)\int_{\frac{t}{2}}^t O_u(1)\frac{1}{s}ds\\
&\leq {M''}\delta_X(u)\int_{\frac{t}{2}}^t\frac{1}{s}ds\\
&={M''}\delta_X(u)\log 2.
\end{split}
\]
Par une int\'egration par parties, on a
{\allowdisplaybreaks
\begin{align*}
\int_{0}^{\frac{t}{2}} &e^{-(t-s)\Delta^u}\bigl((\partial_u\log h_u) \bigr)\partial_se^{-s\Delta^u}ds\\ &=\Bigl[e^{-(t-s)\Delta_{{}_{X,u}}}(\pt_u\log h_u)e^{-s\Delta_{{}_{X,u}}} \Bigr]_0^{\frac{t}{2}}
-\int_0^{\frac{t}{2}}\pt_s( e^{-(t-s)\Delta_{{}_{X,u}}})\bigl(\pt_u\log h_u\bigr) e^{-s\Delta_{{}_{X,u}}}ds     \\
&=e^{-\frac{t}{2}\Delta_{{}_{X,u}}}\bigl(\pt_u \log h_u\bigr) e^{-\frac{t}{2}\Delta_{{}_{X,u}}}-e^{-t\Delta_{{}_{X,u}}}\bigl(\pt_u \log h_u\bigr) I- \int_{0}^\frac{t}{2} \partial_s (e^{-(t-s)\Delta^u})\bigl(\partial_u\log h_u \bigr)e^{-s\Delta^u}ds\\
&=e^{-\frac{t}{2}\Delta_{{}_{X,u}}}(\pt_u \log h_u) e^{-\frac{t}{2}\Delta_{{}_{X,u}}}-e^{-t\Delta_{{}_{X,u}}}(\pt_u \log h_u) I+ \int_{\frac{t}{2}}^t \partial_s (e^{-s\Delta^u})\bigl(\partial_u\log h_u \bigr)e^{-(t-s)\Delta^u}ds,
\end{align*}}
o\`u $I$ est l'op\'erateur identit\'e. Il existe donc, une constante $M^{(3)}$ telle que

\[
\begin{split}
\biggl\|\int_{0}^{\frac{t}{2}} e^{-(t-s)\Delta^u}\bigl((\partial_u\log h_u) \bigr)\partial_se^{-s\Delta^u}ds\biggr\|_{L^2,\infty}&\leq  {M^{(3)}}\delta_X(u).
\end{split}
\]
On conclut que
\begin{equation}\label{key}
\biggl\|\frac{\partial}{\partial u} e^{-t\Delta_{{}_{X,u}}}\biggr\|_{L^2,\infty}=O\bigl(\delta_X(u)\bigr)\quad \forall\, u\gg 1.
\end{equation}

\end{proof}

\begin{proposition}[Cl\'e]
 On a, pour tout $t>0$ fix\'e:
\[
 \bigl(e^{-t\Delta_{{}_{X,u}}}\bigr)_u\xrightarrow[u\to+\infty]{}e^{-t\Delta_{X,\infty}},
\]
et $e^{-t\Delta_{{}_{X,\infty}}}$ est un op\'erateur compact.
\end{proposition}

\begin{proof}
Par \eqref{key} et \eqref{operateurcompact} la suite $\bigl(e^{-t\Delta_{{}_{X,u}}}\bigr)_u$ converge vers une limite qu'on notera $L_t$, cet op\'erateur est compact, par \eqref{operateurcompact}.\\

 On a  pour tout $u\gg 1$,
\[
 \begin{split}
  \bigl(\partial_t+\Delta_{{}_{X,\infty}}\bigr)e^{-t\Delta_{{}_{X,u}}}&=\bigl(\partial_t+\frac{h_\infty}{h_u}\Delta_{{}_{X,u}}\bigr)e^{-t\Delta_{{}_{X,u}})}\\
&=\bigl(1-\frac{h_\infty}{h_u}\bigr)\,\partial_te^{-t\Delta_{{}_{X,u}}}\\
&=O\bigl(\frac{1}{u}\bigr)\,\partial_t e^{-t\Delta_{{}_{X,u}}}.
 \end{split}
\]

Si l'on fixe $t_0>0$,  on obtient
\[
 \bigl(\partial_t+\Delta_{{}_{X,\infty}}\bigr)L_t=0 \quad \forall\, t\geq t_0.
\]
En plus on a $L_t\rightarrow I$ quand $t\mapsto 0$. Mais comme $e^{-t\Delta_{{}_{X,\infty}}}$ v\'erifie les m\^emes conditions et par unicit\'e du noyau de chaleur alors $L_t=e^{-t\Delta_{{}_{X,\infty}}}$.\\

\end{proof}

\subsection{Trace et fonction Z\^eta}

On consid\`ere la m\'etrique $L^2_{X,\infty}$ sur  $A^{(0,0)}(X,E)$, (rappelons que $L^2_{X,\infty}$ est  induite para $h_{X,\infty}$ et $h_E$). Si $T$ est un op\'erateur sur le complet\'e de $A^{(0,0)}(X,E)$ pour cette m\'etrique, on appelle trace de $T$ et on la note  par $\mathrm{Tr}(T)$  le r\'eel suivant:
\[
 \mathrm{Tr}(T)=\sum_{k\geq 0}\bigl(T\xi_{k,\infty},\xi_{k,\infty} \bigr)_{L^2,\infty},
\]
lorsque cette somme converge absolument, o\`u $\bigl(\xi_{k,\infty}\bigr)_{k\in \N}$ est une base orthonormale pour $L^2_{X,\infty}$.\\

\begin{definition}
On pose
\[
 \theta_{X,\infty}(t):=\mathrm{Tr}\bigl(P^\infty e^{-t\Delta_{X,\infty}}\bigr)\quad \forall\, t>0,
\]
o\`u $P^\infty$ est la projection orthogonale pour la m\'etrique $L^2_{X,\infty}$ et de noyau $H^0(X,E)$. On l'appelle la fonction Th\^eta associ\'ee \`a  $\Delta_{X,\infty}$.
\end{definition}

\begin{proposition}
 On a,
\[
 \theta_{X,\infty}(t)<\infty\quad \forall t>0,
\]
c'est \`a  dire que $e^{-t\Delta_{X,\infty}}$ est un op\'erateur nucl\'eaire pour tout $t>0$.
\end{proposition}
\begin{proof}

% On se propose de montrer que $ \theta_{X,\infty}(t)$ est fini pour tout $t>0$, en d'autres termes $e^{-t\Delta_{X,\infty}}$ est un op\'erateur nucl\'eaire pour tout $t>0$.\\

De \eqref{tracenorme1},
\[
\Bigl|\mathrm{Tr}\Bigl(\bigl(\Delta_{X,u}+I \bigr)^{-2}\Bigr)\Bigr|\leq \Bigl\|\bigl(\Delta_{X,u}+I \bigr)^{-2}\Bigr\|_{1,\infty}\quad \forall \, u\geq 1.
\]

et de \eqref{equivtracenorme11},
\[
\frac{1-\eps}{1+\eps}\Bigl\|\bigl(\Delta_{X,u}+I \bigr)^{-2}\Bigr\|_{1,u}\leq \Bigl\|\bigl(\Delta_{X,u}+I
\bigr)^{-2}\Bigr\|_{1,\infty}\leq \frac{1+\eps}{1-\eps}\Bigl\|\bigl(\Delta_{X,u}+I \bigr)^{-2}\Bigr\|_{1,u},
\]
pour $0<\eps\ll 1$ et $u\gg 1$. Or
\[
\Bigl\|\bigl(\Delta_{X,u}+I \bigr)^{-2}\Bigr\|_{1,u}=\sum_{k\in \N }\frac{1}{(\la_{u,k}+1)^2}\leq \zeta_{X,u}(2)<\infty\quad
\forall\,u\geq 1,
\]
o\`u $\zeta_{X,u}$ est la fonction Z\^eta associ\'ee \`a  l'op\'erateur $\Delta_{X,u}$.\\

On conclut  que $\bigl(\Delta_{X,u}+I \bigr)^{-2} $ sont des op\'erateurs nucl\'eaires pour $L^2_{X,\infty}$, donc
$\mathrm{Tr}\bigl((\Delta_{X,u}+I \bigr)^{-2}\bigr)$ est fini.\\

D'apr\`es le th\'eor\`eme de Lidskii,   (voir \eqref{lidskii}) et comme $\bigl(\Delta_{X,u}+I \bigr)^{-2} $ est un op\'erateur nucl\'eaire pour $L^2_\infty$, alors
$\mathrm{Tr}\bigl((\Delta_{X,u}+I \bigr)^{-2}\bigr)$ est la somme de ses valeurs propres, or elles sont toutes positives, donc
\[
 \mathrm{Tr}\bigl((\Delta_{X,u}+I \bigr)^{-2}\bigr)=\Bigl\|(\Delta_{X,u}+I \bigr)^{-2}\Bigr\|_{1,u}.
\]

On a
\[
\begin{split}
 \frac{\pt}{\pt u}\bigl(\Delta_{X,u}+I \bigr)^{-2}&=-(\Delta_{X,u}+I)^{-2}\frac{\pt}{\pt u}\bigl(\Delta_{X,u}+I
\bigr)^2 \bigl(\Delta_{X,u}+I \bigr)^{-2}\\
&=-(\Delta_{X,u}+I)^{-2}\frac{\pt \Delta_{X,u}}{\pt u} \bigl(\Delta_{X,u}+I
\bigr)^{-1}-(\Delta_{X,u}+I)^{-1}\frac{\pt \Delta_{X,u}}{\pt u} \bigl(\Delta_{X,u}+I \bigr)^{-2},
\end{split}
\]
donc,
\[
 \begin{split}
\biggl|\frac{\pt}{\pt u}\mathrm{Tr}_{{}_\infty}\Bigl(\bigl(\Delta_{{}_{X,u}}+I \bigr)^{-2} \Bigr) \biggr|&=\biggl|
\mathrm{Tr}_{{}_\infty}\Bigl((\Delta_{{}_{X,u}}+I)^{{}^{-2}}\frac{\pt \Delta_{{}_{X,u}}}{\pt u} \bigl(\Delta_{{}_{X,u}}+I
\bigr)^{{}^{-1}} \Bigr)
+\mathrm{Tr}_{{}_\infty}\Bigl((\Delta_{X,u}+I)^{{}^{-1}}\frac{\pt \Delta_{X,u}}{\pt u} \bigl(\Delta_{X,u}+I \bigr)^{{}^{-2}}
\Bigr)\biggr|\\
&=2\biggl| \mathrm{Tr}_\infty\bigl((\Delta_{{}_{X,u}}+I)^{{}^{-2}}\frac{\pt \Delta_{{}_{X,u}}}{\pt u}
\bigl(\Delta_{X,u}+I \bigr)^{-1} \biggr|\quad\text{par}\;\eqref{commutetrace11}\\
&\leq 2\delta_X(u)\Bigl\|\bigl(\Delta_{X,u}+I \bigr)^{-2}\Bigr\|_{1,\infty}  \quad
\text{par}\;\eqref{tracenorme1}\;\text{et}\; \eqref{bornelapbelt}\\
&\leq  c\,\delta_X(u) \Bigl\|\bigl(\Delta_{X,u}+I \bigr)^{-2}\Bigr\|_{1,u}\quad \text{l'existence de
}\,c\;\text{r\'esulte de}\; \eqref{equivtracenorme11}.\\
\end{split}
\]
Si l'on pose $\al(u)=\mathrm{Tr}_\infty\bigl((\Delta_{X,u}+I \bigr)^{-2}\bigr)$, alors on obtient:
\[
 \Bigl| \frac{\pt}{\pt u}\al(u)\Bigr|\leq c\, \delta_X(u) \al(u)\quad \forall\, u>1.
\]
%On a aussi $\al(u)=\Bigl\|\bigl(\Delta_{X,u}+I \bigr)^{-2}\Bigr\|_{1,u}$, donc
Donc,
\[
\biggl|\log \biggl(\frac{\al(u)}{\al(u')}\biggr) \biggr|\leq c\,\Bigl|\int_u^{u'}\delta_X(v)dv\Bigr|\quad \forall\,
u\gg 1.
\]
D'apr\`es \eqref{deltaXU},  on peut trouver une suite $(h_{X,u})_{u\geq 1}$ telle que $\delta_X(u)=O(\frac{1}{u^2})$, donc on aura
\[
\biggl|\log \biggl(\frac{\al(u)}{\al(u')}\biggr) \biggr|\leq \Bigl|\int_u^{u'}O(\frac{1}{v^2})dv\Bigr|=O(\frac{1}{u}-\frac{1}{u'})\quad \forall\,
u\gg 1,
\]
On en d\'eduit que   que $u\mapsto \al(u)$ est born\'ee sur un intervalle de la forme $[A,\infty[$, et on choisit $A>1$.\\
%\[
%\dim H^0(X,E)+\sum_{k\in \N}\frac{1}{(\la_{u',k}+1)^2}\leq \biggl(\dim H^0(X,E)+ \sum_{k\in \N }\frac{1}{(\la_{u,k}+1)^2}\biggr)\exp\Bigl(c\int_u^{u'}\delta_X(v)dv \Bigr)\quad \forall\, u,u'\geq 1,
%\]
%par suite, la suite $\bigl(\sum_{k\in \N }\frac{1}{(\la_{u,k}+1)^2}\bigr)_u $ est de Cauchy, mais puisque $\sum_{k\in \N}\frac{1}{\la_{u,k}^s}$ converge si $\mathrm{Re}(s)>1$, alors $\sum_{k\in \N }\frac{1}{(\la_{u,k}+1)^2} $ est finie $\forall\, u$. On conclut que $\bigl(\sum_{k\geq 1}\frac{1}{(\la_{u,k}+1)^2}\bigr)_u $ est born\'ee. \\

Pour $t>0$ fix\'e, il existe une constante $c_t$ telle que
\[
 e^{-t a}\leq \frac{c_t}{(1+a)^2}\quad \forall a>0.
\]
Si l'on note  par $\theta_{X,u}$ la fonction Th\^eta associ\'ee \`a  $\Delta_{X,u}$, alors
\[
 \theta_{X,u}(t)=\sum_{k= 1}^\infty e^{-t\la_{u,k}}\leq c_t \sum_{k=0}^\infty\frac{1}{(\la_{u,k}+1)^2}=c_t\al(u)\quad \forall\, u\geq 1.
\]
En particulier, on a cette in\'egalit\'e pour tout $u\geq A$. Comme $u\mapsto \al(u)$ est born\'ee sur $[A,\infty[$ alors on conclut, par ce qui pr\'ec\`ede que pour tout $t>0$ fix\'e, la suite:
\[\bigl(\theta_{X,u}(t)\bigr)_{u\geq A},\] est born\'ee.\\

Pour \'etablir que $e^{-t\Delta_{X,\infty}}$ est un op\'erateur nucl\'eaire  pour tout $t>0$, il suffit,
comme dans le cas de $E$, de montrer qu'il existe $0<\eps\ll 1$ tel que
\[
 \theta_{X,\infty}(t)\leq \frac{1+\eps}{1-\eps}\liminf_{u\mapsto \infty}\theta_{X,u}(t)\quad \forall\, t>1,
\]
cela implique que
\[
 \theta_{X,\infty}(t)<\infty\quad \forall\, t>0.
\]
Il reste donc \`a  montrer qu'il existe  $0<\eps\ll 1$ tel que
\[
 \theta_{X,\infty}(t)\leq \frac{1+\eps}{1-\eps}\liminf_{u\mapsto \infty}\theta_{X,u}(t)\quad \forall\, t>1.
\]

Soit $R$ un op\'erateur de rang inf\'erieur \`a  $n$, (voir \eqref{paragrapheOpcompacts} pour la d\'efinition de $\si_n(\cdot)$).
Fixons t>0, on a
\begin{align*}
 \si_n(P^\infty e^{-t\Delta_{X,\infty}})_\infty&\leq \bigl\|P^\infty e^{-t\Delta_{X,\infty}}-R\bigr\|_{L^2,\infty}\\
 &\leq \bigl\|P^\infty e^{-t\Delta_{X,\infty}}-P^\infty e^{-t\Delta_{X,u}}\bigr\|_{L^2,\infty}+\bigl\|P^\infty e^{-t\Delta_{X,u}}-P^u e^{-t\Delta_{X,u}}\bigr\|_{L^2,\infty}+ \bigl\|P^u e^{-t\Delta_{X,u}}-R\bigr\|_{L^2,\infty}\\
 &\leq \bigl\|P^\infty\|_{L^2} \bigl\|e^{-t\Delta_{X,\infty}}-e^{-t\Delta_{X,u}}\bigr\|_{L^2,\infty}+\bigl\|P^\infty -P^u\bigr\|_{L^2,\infty}\bigl\|e^{-t\Delta_{X,u}}\bigr\|_{L^2,\infty}+  \bigl\|P^u e^{-t\Delta_{X,u}}-R\bigr\|_{L^2,\infty}.
\end{align*}
 On a alors
\[
 \si_n(P^\infty e^{-t\Delta_{X,\infty}})_\infty\leq  \bigl\|P^\infty\|_{L^2} \bigl\|e^{-t\Delta_{X,\infty}}-e^{-t\Delta_{X,u}}\bigr\|_{L^2,\infty}+\bigl\|P^\infty -P^u\bigr\|_{L^2,\infty}\bigl\|e^{-t\Delta_{X,u}}\bigr\|_{L^2,\infty}+  \si_n(P^u e^{-t\Delta_{X,u}})_\infty,
\]
Comme  $e^{-t\Delta_{X,u}}$ (resp. $P^u$) converge pour la norme $L^2$ vers $ e^{-t\Delta_{X,\infty}}$ (resp. vers $P^\infty$) et que la suite $\bigl(\bigl\|e^{-t\Delta_{X,u}}\bigr\|_{L^2,\infty}\bigr)_u$ est born\'ee, alors
\begin{equation}\label{SSSSSX}
\si_n(P^\infty e^{-t\Delta_{X,\infty}})_\infty\leq
 \liminf_{u\mapsto \infty}\si_n(P^u e^{-t\Delta_{X,u}})_\infty.
\end{equation}
Donc,
\begin{align*}
\theta_{X,\infty}(t)&=\mathrm{Tr}(P^\infty e^{-t\Delta_{X,\infty}})\\
&=\sum_{n\in \N}\si_n(P^\infty e^{-t\Delta_{X,\infty}})_\infty\\
&\leq \sum_{n\in \N} \liminf_{u\mapsto \infty}\si_n(P^u e^{-t\Delta_{X,u}})_\infty\quad \text{par}\,\eqref{SSSSSX}\\
&\leq \liminf_{u\mapsto \infty}\sum_{n\in \N} \si_n(P^u e^{-t\Delta_{X,u}})_\infty\quad \text{par}\,\eqref{liminfsomme}\\
&=\liminf_{u\mapsto \infty}\|P^u e^{-t\Delta_{X,u}}\|_{1,\infty}\quad \text{voir la d\'efinition} \;\eqref{definitionTrace}.\\
\end{align*}
Soit $\eps>0$, par  le corollaire \eqref{equivtracenorme11}, on a
\[
\frac{1-\eps}{1+\eps}\|P^u e^{-t\Delta_{X,u}}\|_{1,\infty}\leq \|P^u e^{-t\Delta_{X,u}}\|_{1,u} \leq \frac{1+\eps}{1-\eps}\|P^u e^{-t\Delta_{X,u}}\|_{1,\infty},
\]
pour $u\gg 1$.\\

Rappelons que
\[
 \theta_{X,u}(t)=\bigl\|P^u e^{-t\Delta_{X,u}}\bigr\|_{1,u}.
\]
En combinant tout cela, on obtient:
\[
\theta_{X,\infty}(t)\leq \frac{1+\eps}{1-\eps}\liminf_{u\mapsto \infty}\theta_{X,u}(t).
\]
Ce termine la preuve de la proposition.
\end{proof}

\begin{theorem}\label{key2}
On a pour tout $t>0$ fix\'e
\[
\bigl(\theta_{X,u}(t)\bigr)_{u\geq 1}\xrightarrow[u\mapsto \infty]{} \theta_{X,\infty}(t),
\]
$e^{-t\Delta_{X,\infty}}$ est un op\'erateur nucl\'eaire et on a
\[
\zeta_{X,\infty}(s):=\sum_{k=1}^\infty \frac{1}{\la_{\infty,k}^s}=\frac{1}{\Gamma(s)}\int_0^\infty t^{s-1}\theta_{X,\infty}(t)dt,
\]
pour tout $s\in \CC$, avec $\mathrm{Re}(s)>1$, est prolongeable en  fonction m\'eromorphe sur $\CC$ avec un p\^ole en $s=1$  et holomorphe au voisinage de $s=0$.

On a
\[
\zeta'_{X,\infty}(0)=\int_1^\infty \frac{\theta_{X,\infty}(t)}{t}dt+\gamma b_{\infty,-1}-b_{\infty,0}+\int_0^1 \frac{\rho_{X,\infty}(t)}{t}dt=\lim_{u\mapsto \infty} \bigl(\zeta'_u(0)\bigr)_{u\geq 1}.\\
\]
\end{theorem}
\begin{proof}
Pour tout $k\in \N$, on a $\theta_u(t)=\sum_{i=-1}^kb_{u,i}t^i+\rho_u(t)$ avec $\rho_u(t)=O(t^{k+1})$ pour $t$ assez petit. \\

On sait que
\[
 b_{u,-1}=4\pi\, \mathrm{rg}(E)\,\mathrm{Vol}_{u}(X),
\]
et que
 $b_{u,0}$ est un invariant qui ne d\'epend pas de $u$, voir la formule \eqref{a_0}. Alors,
\[
\bigl(b_{u,i}\bigr)_{u\geq 1}\xrightarrow[u\to+\infty]{}\text{limite}=:b_{\infty,i} \quad\text{pour}\;
i=-1,\,0.
\]
On pose donc,
\[
 \rho_{X,\infty}(t):=\theta_\infty(t)-\frac{b_{\infty,-1}}{t}-b_{\infty,0}\quad \forall \, t>0.
\]
Comme dans le cas de $E$, on montre de la m\^eme mani\`ere que:
\begin{align*}
\bigl(\theta_{X,u}\bigr)_{u\geq 1}&\xrightarrow[u\mapsto \infty]{} \theta_{X,\infty}(t)\quad \forall \,
t>0,\\
\bigl(\rho_{X,u}(t) \bigr)_{u\geq 1}&\xrightarrow[u\mapsto \infty]{}\rho_{X,\infty}(t)\quad \forall\, t>0,\\
\rho_{X,\infty}(t)&=O(t)\quad \forall \, 0\leq t\ll 1.\\
\end{align*}
On conclut que $\zeta_{X,\infty}$ est holomorphe sur $\mathrm{Re}(s)>1$ avec un p\^ole en $s=1$ et admet un prolongement m\'eromorphe \`a  $\CC$ qui est holomorphe au voisinage de $s=0$.\\

D'apr\`es \eqref{lowerbound},
il existe une constante non nulle $C$ qui ne d\'epend pas de $p\in \N$ telle que \[\lambda_{p,1}\geq
C\quad\forall\, p\in \N_{\geq 2},\]
et comme   $\theta_{X,u}(t)\leq e^{-\lambda_{u,1}(t-1)}\theta_{X,u}(1)$ pour tout $t\geq 1$ et que $\bigl(\theta_{X,u}(t)\bigr)_{u\geq 1}$ est born\'ee, alors il existe une constante r\'eelle  $M$ telle que
\[
\theta_{X,p}(t)\leq M\,e^{-Ct}\quad \forall\, p\in \N_{\geq 2}\;\forall\, t\geq 1.
\]
Donc, par le th\'eor\`eme de convergence domin\'ee de Lebesgue, on obtient
\[
 \biggl(\int_{1}^\infty\theta_{X,p}(t)\frac{dt}{t}\biggr)_{p\in \N}\xrightarrow[p\to+\infty]{}
 \int_{1}^\infty\theta_{X,\infty}(t)\frac{dt}{t}.
\]
Comme,
$\forall u\geq 1$
\[
 \zeta'_{X,u}(0)=\int_{1}^\infty\theta_{X,u}(t)\frac{dt}{t}+\gamma b_{u,0}-b_{u,-1}+\int_{0}^1\rho_{X,u}(t)\frac{dt}{t},
\]
(voir, \cite[p. 99]{Soule}) on  d\'eduit que
\[
 \bigl(\zeta'_{X,p}(0)\bigr)_{p\in \N_{\geq 2}}\xrightarrow[p\to+\infty]{} \zeta'_{X,\infty}(0).
\]

D'apr\`es la proposition \eqref{kerH1}, (qui est biens\^ur    valable dans cette situation), on a $\ker
\Delta_{X,u}^1=\ast_{1,E}^{-1}\Bigl(H^0(X,K_X\otimes E^\ast) \Bigr)$, et rappelons que $\ast_{1,E}$ est donn\'e explicitement dans
\eqref{starhodge}, on d\'eduit que $\ast_{1,E}^{-1}$ ne d\'epend pas de la m\'etrique sur $X$, donc de $h_{X,u}$ pour tout $u\geq 1$. Par cons\'equent la suite de m\'etriques $L^2$ sur $\la(E)$: $\bigl( h_{L^2,p}\bigr)_{p\in \N}$ converge vers $h_{L^2,\infty}$.\\

Ce qui implique que la suite $\bigl(\vc_{Q,p}\bigr)_{p\in \N}$  des m\'etriques de Quillen associ\'ees aux m\'etriques $h_p$ admet une limite quand $p$ tend vers l'infini. Mais rappelons qu'on a  montr\'e que suite converge vers limite en utilisant la formule des anomalies et qu'on a not\'e cette limite par $T_g\bigl((X,\omega_{X,\infty});\overline{E}\bigr)$ qu'on a appel\'e la torsion analytique g\'en\'eralis\'ee associ\'ee \`a  la donn\'ee $\bigl((X,\omega_{X,\infty});\overline{E}\bigr)$. \\

Notre r\'esultat s'\'ecrit donc comme suit:

\[
\zeta'_{X,\infty}(0)=T_g\bigl((X,\omega_{X,\infty});\overline{E}\bigr).
\]
\end{proof}

\section{Les op\'erateurs compacts, un rappel}\label{paragrapheOpcompacts}
Soit $\mathcal{H}$ un espace de Hilbert et {{} $<\cdot,\cdot>$} le produit hermitien (resp. $\vc$ la norme ) associ\'e. Si $\mathcal{H}'$ est un autre espace de Hilbert, on note par $L(\h,\h')$ l'espace des op\'erateurs lin\'eaires continues de $\h$ vers $\h'$. Lorsque $\h'=\h$, on note $L(\h)=L(\h,\h)$.

Soit $T:\h\rightarrow \h'$ un op\'erateur lin\'eaire continue. $T$  est dit compact, si  l'image de tout born\'e de $\h$ est relativement compacte dans $\h'$. On dit que $T$ est de rang fini si son image est de dimension finie, en particulier c'est un op\'erateur compact. La dimension de l'image d'un op\'erateur de rang fini s'appelle le rang de l'op\'erateur.\\

\begin{proposition}\label{operateurcompact}
Soit $\mathcal{H}$ un espace de Hilbert, $\mathcal{B}(\mathcal{H})$ l'espace des op\'erateurs born\'es et $\mathcal{K}(\mathcal{H})$ l'espace des op\'erateurs compacts. On a, $\mathcal{K}(\mathcal{H})$ est sous espace ferm\'e de $\mathcal{B}(\mathcal{H})$.
\end{proposition}
\begin{proof}
Voir, par exemple \cite[proposition 1.4]{functional2}.
\end{proof}

Soit $T\in L(\h)$. Pour tout $n\in \N$, on pose
{{}
\[
\si_n(T)=\inf\Bigl\{ \bigl\|T-R\bigr\|:\, R\in L(\h), \text{rg}(R)\leq n\Bigr\},
\]
}
On sait que $T$ est compact si et seulement si la suite $\bigl( \si_n(T)\bigr)_{n\in \N}$ tend vers $0$, voir \cite[p. 232]{functional2}. On suppose dans la suite que $T$ est compact,
et on appelle $\si_n(T)$ la n-\'eme valeur singuli\`ere de l'op\'erateur compact $T$.\\

On a $P:=\bigl(T^\ast T\bigr)^{\frac{1}{2}}$ (o\`u $T^\ast$ est l'op\'erateur adjoint de $T$) est un op\'erateur autoadjoint et compact; on note par $\bigl( \mu_n(T)\bigr)_{n\in \N}$ l'ensemble de ses valeurs propres non nulles  ordonn\'e en une suite d\'ecroissante et compt\'ees avec leur multiplicit\'es (par d\'efinition, la multiplicit\'e de $\la$ une valeur propre non nulle, not\'ee $d_\la$ est la dimension de $\ker(\la I-P)$).\\

On montre que
\[
 \mu_n(T)=\si_n(T) \quad \forall\, n\in \N,
\]
voir par exemple \cite[p. 246]{functional2}.

\begin{definition}\label{definitionTrace}
Soit $T$ un op\'erateur compact. Soit $\bigl(\mu_n(T)\bigr)_{n\in \N}$ l'ensemble des valeurs singuli\`eres de $T$ ordonn\'e par ordre
d\'ecroissant. On pose
{{}
\[
 \|T\|_1:=\sum_{n\in \N} \mu_n(T).
\]
}
si $\|T\|_1<\infty$ alors $T$ est appel\'e op\'erateur nucl\'eaire et  $\|T\|_1$ sa norme nucl\'eaire.

\end{definition}

On note par $\mathcal{C}_1(\h)$ l'ensemble des op\'erateurs nucl\'eaires. On a
\begin{proposition}
$\mathcal{C}_1(\h)$ est un espace vectoriel et $\|\cdot\|_1$ est une norme sur $\mathcal{C}_1(\h)$.
\end{proposition}
\begin{proof}
Voir \cite[15.11 probl\`eme 7, c]{Dieudonne2}.
\end{proof}

\begin{proposition}
Soit $T$ un op\'erateur nucl\'eaire. Si $\bigl(\xi_n\bigr)_{n\in \N}$ est une base orthonormale de $\h$, alors la s\'erie {{} $\sum_{n\in \N}<T\xi_n,\xi_n>$} converge  de somme {{} $\|T\|_1$}.
\end{proposition}
\begin{proof}
 Voir \cite[15.11 probl\`eme 7, b)]{Dieudonne2}.
\end{proof}

\begin{proposition}\label{tracesup11}
Soit $T$ un op\'erateur compact, on a

\[\|T\|_1=\sup\Bigl\{ \sum_{k}\bigl|<T\xi_k,\eta_k>\bigr|\; \Bigl|\quad \{ \xi_k\},\{\eta_k\} \;\text{bases hilbertiennes de}\; \h \Bigr\},\]
et il existe deux sous-ensembles orthonorm\'es $\{ \xi_k\}$ et $\{\eta_k\}$ de $\h$ tels que $<T\xi_k,\eta_k>\geq 0$ pour tout $k$ et $\|T\|_1= \sum_{k}<T\xi_k,\eta_k>$.
\end{proposition}
\begin{proof}
 Voir \cite[15.11 probl\`eme 7, c)]{Dieudonne2}.
\end{proof}
\begin{proposition}
Soit $\bigl(\la_n\bigr)_{n\in \N}$ la suite de valeurs propres d'un op\'erateur nucl\'eaire $T$, compt\'ees avec leur
multiplicit\'e. On a
\[
 \sum_{n\in \N}|\la_n|\leq \|T\|_1.
\]

\end{proposition}
\begin{proof}
Voir \cite[Th\'eor\`eme 1.15]{Simon}.
\end{proof}
\begin{proposition}\label{lidskii}
Soit $T$ un  op\'erateur nucl\'eaire, et soit $(\la_n)_{n\in \N}$ la suite de ses valeurs propres compt\'ees avec leur multiplicit\'e.
Alors, $\sum_{n\in \N}\la_n$ converge absolument et on a
\[
\sum_{n\in \N}\la_n=\mathrm{Tr}(T).
\]
\end{proposition}
\begin{proof}
Voir \cite[(3.2)]{Simon}.
\end{proof}

\begin{proposition}\label{normetrace}
 Soient $A$ et $B$ deux op\'erateurs born\'es et $T\in \mathcal{C}_1(\h)$, alors
{{}
\[
 \|ATB\|_1\leq \|A\|\|T\|_1\|B\|.
\]

}
\end{proposition}

Soit $\mathcal{H}$ un espace de Hilbert s\'eparable. Soit $A$ un op\'erateur compact sur $\mathcal{H}$  et $\bigl(e_i\bigr)_i$ une base orthonormale de $\h$. Lorsque  $\sum_{i\geq 0}\bigl(Ae_i,e_i\bigr)$ est absolument convergente pour une base orthonormale $\bigl(e_i\bigr)_{i}$, et donc pour toute base orthonormale de $\mathcal{H}$, on appelle cette somme la trace de $A$  et elle est not\'ee $\mathrm{Tr}(A)$.\\

Si $T$ est nucl\'eaire, alors
\[
 \|T\|_1=\mathrm{Tr}\bigl((T^\ast T)^\frac{1}{2}\bigr).
\]

Si $A$ est un op\'erateur nucl\'eaire, on a les propri\'et\'es suivantes:
\begin{itemize}
\item[$\bullet$] \begin{equation}\label{commutetrace11}
 \mathrm{Tr}(AB)=\mathrm{Tr}(BA),
\end{equation}

si  $B$ est born\'e, cf. \cite[TR. 2 p.463]{functional}.
\item[$\bullet$]

\begin{equation}\label{tracenorme1}
 \bigl|\mathrm{Tr}(A) \bigr|\leq \|A\|_1.
\end{equation}

cf. \cite[TR. 7 p.463]{functional}.
\item[$\bullet$] Si $\bigl(T_n\bigr)_{n\in \N}$ est une suite d'op\'erateurs sur $\h$ qui converge faiblement vers  un op\'erateur $T$,
(c'est \`a  dire que $\forall v\in \mathcal{H}$, la suite $(T_nv)_{n\in \N}$ converge vers $Tv$ pour la norme de $\mathcal{H}$), alors
\[
 \mathrm{Tr}(TA)=\lim_{n\mapsto \infty}\mathrm{Tr}(T_nA).
\]
cf. \cite[TR. 8 p.463]{functional}.
\end{itemize}

\begin{proposition}
Soit $\h$ un espace de Hilbert. Soit $\bigl(<\cdot, \cdot>_u\bigr)_{u\in I}$  une famille de m\'etriques hermitiennes
sur $\h$ uniform\'ement \'equivalentes deux \`a  deux.

Soit  $u_0\in I$ et $T\in \mathcal{C}_{1,u_0}(\h)$ alors
 \begin{enumerate}
 \item  $T$ est un op\'erateur nucl\'eaire sur $\h $ muni de $<\cdot,\cdot>_u$, pour tout $u\in I$,
\item  Il existe $c_8$ et $c_9$ deux constantes positives non nulles telles que
\[
 c_8\|T\|_{1,u'}\leq \|T\|_{1,u}\leq c_9\|T\|_{1,u'} \quad \forall\, u,u'\in I.
\]
\end{enumerate}

\end{proposition}
\begin{proof}
On commence par rappeler que $L(\h)$ est muni de la norme suivante:

\[
 \|A\|=\sup_{x\in \h\setminus \{0\}}\frac{\|Ax\|}{\|x\|} \quad A\in L(\h).
\]
Par hypoth\`ese, il existe $c'_8,c'_9$ deux constantes positives non nulles telles que
\[
 c'_8\|x\|_{u'}\leq \|x\|_u\leq  c'_9\|x\|_{u'}\quad \forall x\in \h\quad\forall\, u,u'\in I.
\]

Donc,
\[
\frac{c'_8}{c'_9}\frac{\|Tx\|_{u'}}{\|x\|_{u'}}\leq \frac{\|Tx\|_{u}}{\|x\|_{u}}\leq
\frac{c'_9}{c'_8}\frac{\|Tx\|_{u'}}{\|x\|_{u'}}\quad \forall\, x\neq 0.
\]

Par suite,
{{}
\[
 \frac{c'_8}{c'_9}\|T\|_{u'}\leq \|T\|_{u}\leq \frac{c'_9}{c'_8}\|T\|_{u'}.
\]
}
On en d\'eduit que $T$ est compact pour tout $u\in I$, en effet, si l'on consid\`ere $F$, un ferm\'e born\'e dans
$(\mathcal{H},<\cdot,\cdot>_{u})$ pour un certain $u\in I$, alors par la derni\`ere in\'egalit\'e, $F$ est born\'e et ferm\'e pour
$<,>_{u_0}$ et comme $T\in \mathcal{C}_{1,u_0}(\mathcal{H})$, donc par d\'efinition
$T$ est compact dans $(\mathcal{H},<\cdot,\cdot>_{u_0})$, alors $T(F)$ est relativement compact dans ce dernier espace,
en utilisant la m\^eme in\'egalit\'e, on d\'eduit que $T(F)$ est relativement compact dans $(\mathcal{H},<,>_u)$.

Soit $R$ un op\'erateur de rang fini inf\'erieur \`a  $n$, on a

{{}
\[
 \frac{c'_8}{c'_9}\|T-R\|_{u'}\leq \|T-R\|_{u}\leq \frac{c'_9}{c'_8}\|T-R\|_{u'}.
\]
}
et donc,
{{}
\[
 \frac{c'_8}{c'_9}\si_n(T)_{u'}\leq \si_n(T)_{u}\leq \frac{c'_9}{c'_8}\si_n(T)_{u'}.
\]
}

et par cons\'equent
{{}
\[
 \frac{c'_8}{c'_9}\|T\|_{1,u'}\leq\|T\|_{1,u}\leq \frac{c'_9}{c'_8}\|T\|_{1,u'}.
\]
}

\end{proof}

\begin{Corollaire}\label{equivtracenorme11}
 Si $(\vc_u)_{u\geq 1}$ est une suite de m\'etriques hermitiennes sur $\h$ qui converge uniform\'ement quand $u$ tend vers l'infini vers une m\'etrique hermitienne qu'on note par $\vc_\infty$: On suppose que $\forall \eps>0$, il existe $\eta>0$ tel que
{{}
\begin{equation}\label{limitmulti}
(1-\eps)\|\xi\|_{u'}\leq  \|\xi\|_{u}\leq (1+\eps)\|\xi\|_{u'}\quad \forall u,u'>\eta.
\end{equation}

}
Alors, on a pour tout $0<\eps<1$

{{}
\[
\frac{1-\eps}{1+\eps}\si_n(T)_{u'}\leq  \si_n(T)_{u}\leq \frac{1+\eps}{1-\eps}\si_n(T)_{u'}\quad \forall
T\in L(\h),\,\forall n\in \N,\,\forall u,u'>\eta.
\]
}
En particulier,

{{}
\[
\frac{1-\eps}{1+\eps}\|T\|_{1,u'}\leq  \|T\|_{1,u}\leq \frac{1+\eps}{1-\eps}\|T\|_{u'}\quad \forall T\in L(\h),\quad\forall u,u'>\eta.
\]
}

\end{Corollaire}
\begin{proof}
 C'est une cons\'equence de la d\'emonstration de la proposition pr\'ec\'edente.
\end{proof}

On consid\`ere l'espace pr\'ehilbertien $A^{0,0}(X,E)$ muni de la m\'etrique $\vc_{L^2,u}$ associ\'ee \`a  la m\'etrique de $E$; $h_u$. On sait que $(h_u)_u$ converge uniform\'ement vers une limite $h_\infty$, cela donne que $\bigl(\vc_{L^2,u}\bigr)$ converge uniform\'ement vers $\vc_{L^2,\infty}$, {{} (plus exactement elle v\'erifie l'assertion \eqref{limitmulti})}.\\

 On prend $T=B e^{-t\Delta_{E,u}}$, o\`u $B$ est un op\'erateur born\'e  et tel que $t>0$ soit fix\'e. On a pour $u\gg 1$
{{}
\[
 \frac{1-\eps}{1+\eps}\si_n\bigl(Be^{-t\Delta_{E,u}}\bigr)_\infty\leq \si_n\bigl(Be^{-t\Delta_{E,u}}\bigr)_u\leq \frac{1+\eps}{1-\eps} \si_n\bigl(Be^{-t\Delta_{E,u}}\bigr)_\infty.
\]
}
Par suite,
{{}
\begin{equation}\label{thetainfty}
 \frac{1-\eps}{1+\eps}\bigl\|Be^{-t\Delta_{E,u}}\bigr\|_{1,\infty}\leq \bigl\|Be^{-t\Delta_{E,u}}\bigr\|_{1,u}\leq \frac{1+\eps}{1-\eps} \bigl\|Be^{-t\Delta_{E,u}}\bigr\|_{1,\infty}.
\end{equation}
}
%Soit {{} $\{\xi_{\infty,i} \}_{i\in \N}$} une base orthonormale pour $\h$ muni de la m\'etrique {{} $\vc_{L^2,\infty}$} {{}(Cela est possible car $\Delta_{E,\infty}$ admet une spectre..) }. On a, voir ...

%\[
 %\mathrm{Tr}\bigl(e^{-t\Delta_{E,u}}\bigr)_\infty=\sum_{i\geq 0}\bigl(e^{-t\Delta_{E,u}}\xi_{\infty,i},\xi_{\infty,i} \bigr)_\infty.
%\]
%(\textbf{cela est inutile}).

On munit $\h$ de la m\'etrique $L^2_\infty$. On a, pour tout $t>0$ fix\'e

\begin{equation}\label{thetainfty11}
 \begin{split}
  \bigl|\bigl\|Be^{-t\Delta_{E,u}}e^{-t\Delta_{E,\infty}} \bigr\|_{1,\infty}- \bigl\| Be^{-2t\Delta_{E,\infty}}\bigr\|_{1,\infty}\bigr|&\leq  \bigl\|Be^{-t\Delta_{E,u}}e^{-t\Delta_{E,\infty}}-  Be^{-2t\Delta_{E,\infty}}\bigr\|_{1,\infty}\\
&=\bigl\|B(e^{-t\Delta_{E,u}}-  e^{-t\Delta_{E,\infty}})e^{-t\Delta_{E,\infty}}\bigr\|_{1,\infty}\\
&\leq \bigl\|e^{-t\Delta_{E,u}}-  e^{-t\Delta_{E,\infty}}\bigr\|_{L^2,\infty}\bigl\|Be^{-t\Delta_{E,\infty}}\bigr\|_{1,\infty}
 \end{split}
\end{equation}

et
\begin{equation}\label{thetainfty12}
\begin{split}
 \Bigl|\bigl\|Be^{-t\Delta_{E,u}}e^{-t\Delta_{E,\infty}} \bigr\|_{1,\infty}- \bigl\| Be^{-2t\Delta_{E,u}}\bigr\|_{1,\infty}\Bigr|&\leq \bigl\|Be^{-t\Delta_{E,u}}e^{-t\Delta_{E,\infty}}- B e^{-2t\Delta_{E,u}}\bigr\|_{1,\infty}\\
&\leq \bigl\| Be^{-t\Delta_{E,u}}(e^{-t\Delta_{E,u}}-e^{-t\Delta_{E,\infty}})\bigr\|_{1,\infty}\\
&\leq \bigl\|e^{-t\Delta_{E,u}}-e^{-t\Delta_{E,\infty}}  \bigr\|_{L^2,\infty}\bigl\|Be^{-t\Delta_{E,u}}\|_{1,\infty}.
\end{split}
\end{equation}

\section{Appendice}\label{Quelqueslemmes}
Dans ce chapitre on regroupe quelques r\'esultats et lemmes techniques qui seront utilis\'es dans le texte.\\

La proposition suivante a pour but de construire \`a  partir d'une suite discr\`ete $(h_n)_{n\in \N^\ast}$ de m\'etriques
hermitiennes sur un fibr\'e vectoriel sur une vari\'et\'e riemannienne, une famille de m\'etriques  $(h_u)_{u\geq 1}$ \`a
param\`etre continue qui varie de facon $\cl$ et qui pr\'eserve les propri\'et\'es de la suite $(h_n)_{n\in \N^\ast}$ par
exemple si $(h_n)_n$ est born\'ee pour la topologie de convergence uniforme, alors $(h_u)_{u\geq 1}$ l'est aussi. Cela
nous sera utile pour \'etudier les variations infinit\'esimales des diff\'erents objets attach\'es \`a  ces suites.

\begin{proposition}\label{suitefamille}
Soit $X$ une vari\'et\'e diff\'erentielle complexe de dimension quelconque. Soit $E$ un fibr\'e vectoriel holomorphe, par exemple $TX$.

On note $\mathcal{M}et(E)$ l'espace des m\'etriques hermitiennes continues (int\'egrables, de classe $\cl$...) sur $E$. Soit   $(h_n)_{n\in \N}$ une suite de m\'etriques hermitiennes
sur $E$ (non n\'ecessairement de classe $\cl$), alors  il existe une famille continue $(H(u))_{u\geq 1}$  v\'erifiant:
 % Il existe une famille $(h_u)_{u\in [2,\infty [}$ de m\'etriques $\mathcal{C}^\infty$ sur $\p^1$  v\'erifiant:
\begin{enumerate}
\item  $H(u)$ est une m\'etrique hermitienne sur $E$, $\forall u$.
\item Pour toute section locale $s$ de $E$, l'application
\[
 \begin{split}
 H: [1,\infty[&\lra \R^+\\
u&\longmapsto H(u)(s,s),
 \end{split}
\]
 est de classe $\mathcal{C}^\infty$.
\item $H(n)=h_n$, $\forall\, n\in \N$.
 \item Si l'on suppose de plus que $(h_n)_{n\in \N}$ converge uniform\'ement vers une m\'etrique $h_\infty$, alors la fonction $u\mapsto  H(u)$ converge uniform\'ement vers  la fonction constante $u\mapsto h_\infty$ sur $[1,\infty[$.
\item Si $E$ est un fibr\'e en droites,  alors
{{}
\[\frac{\pt}{\pt u}  \log H(u)=\mathrm{O}\biggl(\frac{h_{[u]+1}-h_{[u]}}{h_{[u]}}  \biggr),\]}
sur $X$, avec $[u]$ est la partie enti\`ere de $u$.
\end{enumerate}

\end{proposition}

\begin{proof} On consid\`ere une suite discr\`ete de m\'etriques $(h_p)_{p\geq 2}$.

 Pour tout $n$, soit $\rho_n$ une fonction de classe $\cl$ sur $\R^+$ croissante positive avec

\begin{equation}
\rho_n(x)=\left\{
    \begin{array}{ll}
        0,& \quad x\leq n   \\
        1, & \quad x\geq n+1
 \end{array}
\right.
\end{equation}

On  peut supposer que $\rho_n(x)=\rho_1(x-n), \forall x\in \R^+ $.\\

On pose $H_1:[1,\infty[\lra \mathcal{M}et$; $H_1(u)=h_1 $, $\forall u$. Si $H_2$ est la fonction  d\'efinie comme suit: $H_2(u)=(1-\rho_1(u))H_1(u)+\rho_1(u)h_2 $ alors c'est une fonction de classe $\cl$ qui v\'erifie $H_2(1)=H_1(1)=h_1$ et $H_2(2)=h_2$. Par r\'ecurrence, on pose $H_k(u)=(1-\rho_{k-1}(u))H_{k-1}(u)+\rho_{k-1}(u)h_k$, et on montre que  $H_k(i)=h_i, $ pour $i\leq k-1$ et $H_k(k)=h_k$.\\

On consid\`ere $H:\R^+\lra \mathcal{M}et(E)$ en posant $H(u)=H_n(u)$ si $u\leq n-1$ (notons que   $H_{n+1}(u)=H_n(u)$) donc $H$ est bien d\'efinie, de classe $\cl$ et on a  $H(n)=h_n$.\\

Supposons que $(h_n)_{n\in \N}$ converge uniform\'ement vers $h_\infty$. On montre par r\'ecurrence sur $k$ que
{{}
\begin{equation}\label{1}
 H(u)=\bigl(1-\rho_{k-1}(u)\bigr)h_{k-1}+\rho_{k-1}(u)h_k\quad \forall\, u\in [k-1,k] \; \forall k\,\in \N^\ast.
\end{equation}}
Par suite, si $s$ est une section locale de $E$ non nulle sur un ouvert $U$,
\[
 \bigl|H(u)(s,s)-h_\infty(s,s)\bigr|\leq \bigl|h_{k-1}(s,s)-h_\infty(s,s)\bigr|+\bigl|h_k(s,s)-h_\infty(s,s)\bigr| \quad \forall\, u\in [k-1,k].
\]
Et on a

\[
\begin{split}
\biggl|\frac{\pt }{\pt u}\log H(u)(s,s)\biggr|&=\bigl|h_u(s,s)^{-1}(\partial_u \rho_{k-1})(u)(h_k(s,s)-h_{k-1}(s,s))\bigr|\\
&=\bigl|\partial_u \rho_{k-1}(u) \bigr|\,\Bigl|\frac{h_k(s,s)-h_{k-1}(s,s)}{h_u(s,s)}\Bigr|\\
&\leq \bigl|\partial_u \rho_{k-1}(u) \bigr|\,\Bigl|\frac{h_k-h_{k-1}}{\min(h_{k-1},h_k)}\Bigr|\\
&=\bigl|\partial_u \rho_{k-1}(u) \bigr|\,\max\biggl(\frac{\bigl|h_k-h_{k-1}\bigr|}{h_{k-1}},\frac{\bigl|h_k-h_{k-1}\bigr|}{h_{k}} \biggr),
\end{split}
\]
donc, il existe une constante $M>0$ telle que
\[
 \biggl|\frac{\pt }{\pt u}\log H(u)\biggr|\leq M \frac{\bigl|h_{[u]+1}-h_{[u]}\bigr|}{h_{[u]}},
\]
pour tout $u\geq 1$.

\end{proof}

%\begin{lemma}\label{estimation}
 %Soit $k\in \N^\ast$, on a

%et

Dans la suite, on note par $h_u$ la m\'etrique $H(u)$.\\

On suivra la d\'efinition de \cite[Appendice, D]{Ma}, pour la d\'efinition du noyau de chaleur associ\'e \`a  un laplacien g\'en\'eralis\'e. Soit donc $\Delta$ un laplacien g\'en\'eralis\'e, on d\'efinit l'op\'erateur de chaleur  qu'on note par $e^{-t\Delta}$ en utilisant la th\'eorie des op\'erateurs. Pour tout $t>0$, $e^{-t\Delta}$ est un op\'erateur de $L^2(X,E)$ vers $L^2(X,E)$ qui est de classe $\mathcal{C}^1$ et qui v\'erifie les propri\'et\'es suivantes: pour tout $s\in L^2(X,E)$

\[
\begin{split}
\Bigl( \frac{\pt}{\pt t}+\Delta \Bigr)e^{-t\Delta}s&=0,\\
\underset{t\mapsto 0}{\lim}\, e^{-t\Delta}s&=s\quad \text{dans}\quad L^2(X,E).
\end{split}
\]
On montre que $e^{-t\Delta}$ est unique.\\

\begin{theorem}
Soit $\Delta^u$ une famille de classe $\cl$ de Laplaciens g\'en\'eralis\'es, alors pour tout $t>0$, la famille de noyaux de la chaleur $e^{-t\Delta_u}$ d\'efinit une famille de classe $\cl$ d'op\'erateurs sur $E$. En plus, on a la d\'eriv\'ee de $e^{-t\Delta_u}$ par rapport \`a  $u$ est donn\'ee par la formule du {}{Duhamel}
{{}
\begin{equation}\label{heatkernel}
\frac{\partial }{\partial u}e^{-t\Delta^u}=-\int_{0}^t e^{-(t-s)\Delta^u}(\partial_u \Delta^u)e^{-s\Delta^u}ds.
\end{equation}
}
\end{theorem}

\begin{proof}
Voir \cite[th\'eor\`eme D.1.6]{Ma} ou \cite[th\'eor\`eme 2.48]{heat}.
\end{proof}

On rappelle le th\'eor\`eme suivant qui sera utile dans la suite de l'article:
\begin{theorem}\label{semi}
Soit $V$ un espace de Banach. Si $A$ un op\'erateur auto-adjoint et positif, alors $-A$ engendre un semi-groupe $P(t)=e^{-tA}$ form\'e d'op\'erateurs positifs, auto-adjoints et de norme $\leq 1$.
\end{theorem}
\begin{proof}
Voir \cite[Proposition  9.4]{Taylor}.
\end{proof}
\subsection{Les m\'etriques int\'egrables}\label{rappelmetint}
Soit $X$ une vari\'et\'e complexe analytique et $\overline{L}=(L,\vc)$ un fibr\'e en droites hermitien muni d'une m\'etrique continue sur $L$.
\begin{definition}
On appelle premier courant de Chern de $\overline{L}$ et on note $c_1\bigl( \overline{L}\bigr)\in D^{(1,1)}(X)$ le courant d\'efini localement par l'\'egalit\'e:
\[
c_1\bigl(\overline{L}\bigr)=dd^c\bigl( -\log \|s\|^2\bigr),
\]
o\`u $s$ est une section holomorphe locale et ne s'annulant pas du fibr\'e  $L$.
\end{definition}

\begin{definition}
La m\'etrique $\vc$ est dite positive si $c_1\bigl(L,\vc\bigr)\geq 0$.
\end{definition}
\begin{definition}
 La m\'etrique $\vc$ est dite admissible s'il existe une famille $\bigl(\vc_n \bigr)_{n\in \N}$ de m\'etriques positives de classe $\cl$  convergeant uniform\'ement vers $\vc$ sur $L$. On appelle fibr\'e admissible sur $X$ un fibr\'e en droites holomorphe muni d'une m\'etrique admissible sur $X$.
\end{definition}
On dira que $\overline{L}$ est un fibr\'e en droites int\'egrable s'il existe $\overline{L}_1$ et $\overline{L}_2$ admissibles tels que
\[
 \overline{L}=\overline{L}_1\otimes \overline{L}_2^{-1}.
\]

\begin{example}
Soit $n\in \N^\ast$. On note par $\mathcal{O}(1)$ le fibr\'e de Serre sur $\p^n$ et on le munit de la m\'etrique d\'efinie pour toute section m\'eromorphe de $\mathcal{O}(1)$ par:
\[
 \|s(x)\|_\infty=\frac{|s(x)|}{\max(|x_0|,\ldots,|x_n|)}.
\]
Cette  m\'etrique est admissible.
\end{example}
En fait, c'est un cas particulier d'un r\'esultat plus g\'en\'eral combinant   la construction Batyrev
et Tschinkel  sur une vari\'et\'e torique projective et  la construction de Zhang. Dans la premi\`ere
construction permet d'associer canoniquement \`a  tout fibr\'e en droites sur une vari\'et\'e torique
projective complexe une m\'etrique continue not\'ee $\vc_{BT}$ et d\'etermin\'ee uniquement par la
combinatoire de la vari\'et\'e, voir \cite[proposition 3.3.1]{Maillot} et \cite[proposition
3.4.1]{Maillot}. L'approche de Zhang est moins directe, elle utilise un endomorphisme \'equivariant
(correspondant \`a  la multiplication par $p$, un entier sup\'erieur \`a  2) afin de construire par
r\'ecurrence une suite de m\'etriques qui converge uniform\'ement vers une limite not\'ee $\vc_{Zh,p}$ et
qui, en plus, ne d\'epend pas  du choix de la m\'etrique de d\'epart, voir \cite{Zhang} ainsi que
\cite[th\'eor\`eme 3.3.3]{Maillot}. Mais d'apr\`es \cite[th\'eor\`eme 3.3.5]{Maillot} on montre que
\[
 \vc_{BT}=\vc_{Zh,p}.
\]
Que l'on appelle la m\'etrique canonique associ\'ee \`a  $L$. Notons que lorsque $L$ n'est pas trivial,
alors cette m\'etrique  n'est pas $\cl$.\\

% qu'on en rappelle les principaux r\'esultats:
%Soit $\p(\Delta)$ une vari\'et\'e torique projective  complexe d\'efini par un \'eventail $\Delta$ et $L$ un fibr\'e en droites sur $\p(\Delta)$
%\begin{proposition}[Construction de Batyrev et Tschinkel] Soit $s$ une section holomorphe de $L$ au dessus d'un ouvert $\Omega\subset \p(\Delta)$. Pour tout point $x\in \Omega$, soit $\si\Delta$

%\end{proposition}

\subsection{Constante isop\'erimetrique de Cheeger}\label{paragraphecheeger}
%On note par $\mathcal{S}$ l'ensemble des sous-vari\'et\'es de codimension r\'eelle $1$ dans $X$.\\
On rappelle  un r\'esultat d\^u \`a Cheeger donnant une borne inf\'erieure  pour la premi\`ere valeur propre non nulle du
Laplacien $\Delta_{\overline{\mathcal{O}}_0}$, en termes de la g\'eom\'etrie de la vari\'et\'e consid\'er\'ee, o\`u
$\overline{\mathcal{O}}_0$ est le fibr\'e en droites trivial muni d'une m\'etrique constante . Ce r\'esultat, nous a permis
de r\'epondre compl\`etement \`a  la question \eqref{isoperproblem111} dans la situation suivante: Soit $(h_p)_{p\in \N} $
une suite born\'ee de m\'etriques hermitiennes sur $\p^1$ et si l'on note par $\la_{1,p}$,  pour tout $p\in \N$, la
premi\`ere valeur propre non nulle du Laplacien associ\'e \`a  $\bigl((\p^1,h_p),\overline{\mathcal{O}}_0 \bigr)$ alors il
existe une constante
$C>0$ telle que
\[
 \la_{1,p}\geq C,\quad \forall\, p\in\N.
\]
On va \'etendre la constante de Cheeger et on \'etablira une in\'egalit\'e isop\'erim\'etrique du m\^eme type.\\

On termine par r\'esultat plus faible sur la variation de la premi\`ere valeur propre en famille, c'est l'objet de la proposition \eqref{uniformelambda}.\\

\begin{definition}[Constante isop\'erimetrique de Cheeger]\label{Cheeger0} Soit $(M,g)$ une vari\'et\'e riemannienne compacte de dimension $n$ sans bord. On pose

\[
  h(M):=\inf\frac{A(S)}{\min\bigl(V(M_1),V(M_2 )\bigr)}
 \]
o\`u $A(\cdot)$ d\'enote le $(n-1)$-volume dimensionnel, $V(\cdot)$ d\'esigne le volume et l'inf est pris sur l'ensemble des sous-vari\'et\'es \`a  coins compactes $S$ de dimension $n-1$, $M_1$ et $M_2$ sont deux sous-vari\'et\'es avec bords  telles que $M=M_1\cup M_2$ et $\partial M_i=S$, pour $i=1,2$.
\end{definition}

\begin{theorem}\label{Cheeger}
 Si l'on note par $\lambda_1$ la plus petite valeur propre non nulle du Laplacien associ\'e \`a  $(M,g)$ alors
\[
 \lambda_1\geq \frac{1}{4}h^2.
\]

\end{theorem}
\begin{proof}
 Voir \cite{Cheeger}.
\end{proof}
\begin{remarque}
Il est important de noter que $h$ est non nulle, voir \cite[p.198]{Cheeger} pour le cas $n=2$.
\end{remarque}

\begin{proposition}\label{Cheeger2}
 Soit $(M,g)$ une vari\'et\'e riemannienne compacte. Soit $g'$ une autre m\'etrique riemannienne telle que
\begin{equation}\label{GGGGGG}
 C_1g\leq g'\leq C_2 g,
\end{equation}

 o\`u $C_i$ deux constantes r\'eelles non nulles, alors
\[
 \frac{C_1}{C_2}h_g(M)\leq h_{g'}(M)\leq \frac{C_2}{C_1}h_g(M).
\]

\end{proposition}
\begin{proof}
 Il suffit de remarquer que l'in\'egalit\'e \eqref{GGGGGG} est stable par restriction.
\end{proof}

\begin{lemma}\label{Cheeger3}
On consid\`ere la suite $\bigl(h_p \bigr)_{p\in \N_{\geq 2}}$ de m\'etriques sur $\p^1$ d\'efinie par
\[
 h_p\bigl(\cdot,\cdot\bigr)=\frac{|\cdot|^2}{\bigl(1+|z|^p \bigr)^{\frac{4}{p}}},
\]
alors
\[
 2^{-\frac{2\pi^2}{3}}\leq \frac{h_q}{h_p}\leq  1,
\]
Pour tout $z\in \p^1$ et $2\leq q\leq p$.
\end{lemma}
\begin{proof}
 L'in\'egalit\'e de droite est \'evidente. On v\'erifie que la fonction $F_{p,q}:x (\in\R^+)\mapsto \frac{(1+x^p)^\frac{1}{q}}{(1+x^q)^\frac{1}{q}}$ est minor\'ee par $ 2^{\frac{1}{p}-\frac{1}{q}}$,
donc
\[
 2^{-\frac{1}{n^2}}\leq F_{n+1,n}(x)\leq 1\quad\forall n\in \N_{\geq 2}\;\forall x\in \R^+,
\]
donc
\[
 \begin{split}
  2^{-\frac{\pi^2}{6}}&\leq 2^{\sum_{n=q}^p\frac{1}{n^2}}\\
&\leq \prod_{q}^pF_ {n+1,n}(x)\\
&=F_{p,q}(x)\\
 \end{split}.
\]

\end{proof}

\begin{Corollaire}\label{lowerbound}
On consid\`ere la suite $\bigl(h_p\bigr)_{p\in \N_{\geq 2}}$ pr\'ec\'edente, alors
\[
 2^{-\frac{\pi^2}{6}}h_p\bigl(\p^1\bigr)\leq h_q\bigl(\p^1\bigr)\leq 2^{\frac{\pi^2}{6}}h_p\bigl(\p^1\bigr)\quad \forall \, p,q\in \N_{\geq 2}.
\]
En particulier, il existe une constante $C$ non nulle telle que
\[
 \lambda_{p,1}\geq C, \quad \forall p\in \N_{n\geq 2}.
\]
o\`u $\lambda_{p,1}$ est la plus petite valeur propre non nulle du Laplacien associ\'e \`a  la m\'etrique $h_p$.
\end{Corollaire}
\begin{proof}
 C'est une application du \eqref{Cheeger}, \eqref{Cheeger2} et \eqref{Cheeger3}.
\end{proof}

\subsection{Sur la premi\`ere valeur propre du Laplacien g\'en\'eralis\'e}

\begin{lemma}\label{spectrale}
On a

\[
\inf_{\xi \in  \ker(\Delta_{\overline{E}})^\perp}\frac{(\Delta_{\overline{E}}\xi,\xi)_{L^2}}{(\xi,\xi)_{L^2}}=\la_{1,h}
\]
\end{lemma}

\begin{proof}
Par la th\'eorie spectrale des op\'erateurs compacts positifs et autoadjoints, on a les vecteurs propres de $\Delta_{\overline{E}}$ forment un syst\`eme orthogonal complet pour la m\'etrique $L^2$. Si $\xi \in A^{0,0}(X,E)\cap \ker(\Delta_{\overline{E}} )^\perp$, alors il existe des r\'eels $a_k$ tels que

\[
\xi=\sum_{k=1}^\infty a_k v_k,
\]
o\`u $v_k$ est le vecteur propre associ\'e \`a  la valeur propre $\la_k$ o\`u on a pos\'e $\la_0=0$. On a, alors
\[
\begin{split}
\bigl(\Delta_{\overline{E}}\xi,\xi\bigr)_{L^2}&=\sum_{k= 1}^\infty |a_k|^2 \la_k \bigl(v_k,v_k\bigr)_{L^2}\\
&\geq \la_1\sum_{k= 1}^\infty |a_k|^2 \bigl(v_k,v_k\bigr)_{L^2}\\
&=\la_1 \bigl(\xi,\xi\bigr)_{L^2} \quad\text{si}\quad \xi\in \ker(\Delta_{\overline{E}})^\perp.
\end{split}
\]
\end{proof}
%\begin{remarque}
%On a $\ker(\Delta_{\overline{E}})^\perp=\ker(\Delta_{\overline{E}'})^\perp$.
%\end{remarque}

\begin{theorem}\label{thmprmvp}
Soit $(X,\omega)$ une surface de Riemann compacte et $L$ un fibr\'e en droites holomorphe. Soit $h_\infty$ une m\'etrique hermitienne continue sur $L$. Soit $(h_n)_{n}$ une suite de m\'etriques hermitiennes $\cl$ sur $L$ qui converge uniform\'ement vers $h_\infty$. Alors il existe $\al$ et $\beta$ deux constantes r\'eelles positives non nulles telle que
\begin{equation}\label{normeequivalente1}
\al \bigl(v,v\bigr)_{L^2,q} \leq \bigl(v,v\bigr)_{L^2,p}\leq \beta\bigl(v,v\bigr)_{L^2,q},
\end{equation}
et
\begin{equation}\label{normeequivalente2}
\al \bigl(\Delta_q v,v\bigr)_{L^2,q}\leq \bigl(\Delta_p v,v\bigr)_{L^2,p}\leq \beta\bigl(\Delta_q v,v\bigr)_{L^2,q}
\end{equation}
pour tout $1\ll p\leq q$ et $v\in A^{0,0}(X,L)$.
\end{theorem}
\begin{proof}
%Commencons par fixer quelques notations. On note par $\{e_1,\ldots,e_r\}$ une base de $H^0(X,L)$. Si  $h$ est une m\'etrique hermitienne continue sur $L$, on pose
Soit $h$ une m\'etrique hermitienne de classe $\cl$ sur $L$. On note par $\Delta_h$ le Laplacien g\'en\'eralis\'e associ\'e.

%\[
%H_h(x):=\bigl(h(e_i,e_j)(x)\bigr)_{1\leq i,j\leq r}, \quad \forall x\in X.
%\]
%(c'est une fonction-matrice $r\times r$ d\'efinie globalement sur $X$, dont chaque coefficient est le produit hermitien ponctuel en $x$ de deux sections globales.)

%Si l'on choisit  un plongement de $X$ dans un espace projectif, on peut voir toute section globale $s$ de $L$ comme un polyn\^ome homog\`ene, donc on peut trouver une fonction $l$ homog\`ene sur $X$ telle que

%\[
%h(s,s)(x)=\frac{s(x)\overline{s(x)}}{l(x)}, \quad \forall x\in X,\quad \forall s\in H^0(X,L).
%\]

%On peut donc \'ecrire $H_h$ sous la forme
%\[
%H_h(x)=\frac{1}{l(x)}\bigl(e_i(x) \overline{e_j(x)}\bigr)_{ij} =:\frac{1}{l(x)}A(x),\quad \forall x\in X.
%\]

Soit $v\in A^{0,0}(X,L)$. Localement, il existe $f_1,\ldots,f_r\in A^{0,0}(X)$ et des sections locales holomorphes $e_1,\ldots,e_r$ de $L$   tels que $v=\sum_{k=1}^r f_k \otimes e_k$. On a donc,
\[
\bigl(v,v\bigr)_{L^2,h}=\int_{x\in X} \sum_{kj} h\bigl(e_k,e_j\bigr)(x) f_k(x) \overline{f_j(x)} \omega(x),
\]
%(cela n'est autre que la sesquilin\'earit\'e du produit hermitien).
et
\[
\bigl(\Delta_h v,v\bigr)_{L^2,h}=\frac{i}{2\pi }\int_X \sum_{kj} h(e_k,e_j)(x) \frac{\pt f_k}{\pt
\z}(x)\frac{\pt \overline{f_j}}{\pt z}(x)dz\wedge d\z \quad \text{voir} \; \eqref{formesimple}.
\]
Si l'on pose, localement,
\[
\textbf{f}:=\bigl(f_1,\ldots,f_r\bigr) \quad \text{et}\quad {\frac{\pt  \textbf{f}}{\pt \z}}:=\bigl(\frac{\pt f_1}{\pt \z},\ldots,\frac{\pt f_r}{\pt \z} \bigr),
\]
alors en utilisant les notations introduites, les deux derniers produits hermitiens deviennent:
\begin{equation}\label{Expression1}
\bigl(v,v\bigr)_{L^2,h}=\int_X \Bigl({}^t \overline{\textbf{f}(x)} H_h(x) \textbf{f}(x)\Bigr)\omega,
\end{equation}
et
\begin{equation}\label{Expression2}
(\Delta_h v,v)_{L^2,h}=\int_X \biggl({}^t \Bigl(\overline{\frac{\pt \textbf{f}}{\pt z}(x)}\Bigr) H_h(x) \frac{\pt \textbf{f}}{\pt z}(x)\biggr)dz\wedge d\z.
\end{equation}
%(\textbf{Les quantit\'es entre les parenth\`eses sont sous forme matricielle}).\\

On consid\`ere une m\'etrique hermitienne continue $h_\infty$ sur $L$ et $\bigl(h_n\bigr)_{n\in \N}$ une suite  de m\'etriques hermitiennes $\cl$ sur $L$ qui converge uniform\'ement vers $h_\infty$ sur $X$. On pose:
\[
H_n(x)=\bigl(h_n(e_i,e_j)(x)\bigr)_{1\leq i,j\leq r}\quad \forall n\in \N\cup \{\infty\}.
\]
Comme
\[
H_n(x)=\frac{1}{l_n(x)}A(x)\quad \forall\, n\in \N\cup\{\infty\},
\]
comme  la suite $\bigl(h_n\bigr)_{n\in \N}$ converge uniform\'ement vers $h_\infty$, on peut trouver $\al<1<\beta$ deux constantes  r\'eelles telles que
\begin{equation}\label{encadrementuniforme}
\al \leq \frac{l_q(x)}{l_p(x)}\biggl(=\frac{h_p(s,s)(x)}{h_q(s,s)(x)}\biggr)\leq \beta \quad \forall\, p,\, q\gg 1,\; \forall\, x\in X,
\end{equation}
o\`u $s$ est une section locale non nulle en $x$ de $L$.\\

%(l'in\'egalit\'e \`a  droite r\'esulte de la monotonie de la suite $(h_n)_n$, celle \`a  gauche est cons\'equence de la convergence uniforme sur $X$ tout entier).\\

Fixons $x\in X$, et soient $p\leq q$, on a
\[
\al\cdot {}^t \overline{\textbf{f}}(x)H_q(x) {\textbf{f}}(x) \leq {}^t \overline{\textbf{f}}(x)H_p(x) {\textbf{f}}(x)\leq {}^t \overline{\textbf{f}}(x)H_q(x) {\textbf{f}}(x),
\]
et
\[
\al \cdot {}^t \overline{\frac{\pt \textbf{f}}{\pt z}(x)} H_q(x) \frac{\pt \textbf{f}}{\pt z}(x) \leq {}^t \overline{\frac{\pt \textbf{f}}{\pt z}(x)} H_p(x) \frac{\pt \textbf{f}}{\pt z}(x)\leq {}^t \overline{\frac{\pt \textbf{f}}{\pt z}(x)} H_q(x) \frac{\pt \textbf{f}}{\pt z}(x).
\]
V\'erifions cela, pour simplifier on notera par $\textbf{u}$ le vecteur $\textbf{f}$ ou $\frac{\pt \textbf{f}}{\pt \z}$. On a
\[
h_k\bigl(\textbf{u},\textbf{u}\bigr)(x)={}^t \overline{\textbf{u}} H_k(x) \textbf{u}=\frac{1}{l_k(x)}{}^t \overline{\textbf{u}} A(x) \textbf{u} \quad \forall\, k \in \N_{\geq 1}.
\]

 De \eqref{encadrementuniforme} et notons que la matrice $A$ est positive, on trouve que
\[
\al  {}^t \overline{\textbf{u}} H_q(x) \textbf{u}\leq {}^t \overline{\textbf{u}} H_p(x) \textbf{u} \leq\beta {}^t \overline{\textbf{u}} H_q(x) \textbf{u}.
 \]
On conclut, en  utilisant les expressions \eqref{Expression1} et \eqref{Expression2}, que
\[
\al \bigl( v,v\bigr)_{L^2,q}\leq
\bigl(v,v\bigr)_{L^2,p}\leq \beta\bigl( v,v\bigr)_{L^2,q},
\]
et
\[
\al \bigl(\Delta_q v,v\bigr)_{L^2,q}\leq \bigl(\Delta_p v,v\bigr)_{L^2,p}\leq \beta\bigl(\Delta_q v,v\bigr)_{L^2,q},
\]
$\forall \,v\in A^{0,0}(X,L)$ et $p, q\gg 1$.
\end{proof}
\begin{example}
Sur $\p^1$. Soit $m$ un entier positif et $p\geq 2$. On consid\`ere $\mathcal{O}(m)$ muni de la m\'etrique $p$-\'eme d\'efinie comme suit:
\[
h_p(s,s)(x)=\frac{s(x)\overline{s(x)}}{\bigl(|x_0|^p+|x_1|^p\bigr)^{\frac{2m}{p}}}\quad \forall \,x=[x_0:x_1]\in \p^1,\; \forall s\in H^0(\p^1,\mathcal{O}(m)).
\]
\end{example}

\begin{proposition}\label{uniformelambda} En gardant les m\^emes hypoth\`eses, on montre qu'il existe une constante $\al\neq 0$ telle que
\[
\al \leq \frac{\la_{1,q}}{\la_{1,p}}\leq \frac{1}{\al},\quad \forall 1\ll p\leq q.
\]

\end{proposition}

\begin{proof}
%Soit $\xi \in H^0(X,E)^{\perp}$...
%Par \eqref{ordre} et \eqref{spectrale}, on a
%(Soit $p\leq q$)

Soient $q\geq p\geq 1$. Commencons  tout d'abord par montrer le r\'esultat technique suivant: Soit $v$ un vecteur
propre non nul associ\'e \`a  $\la_{1,q}$, alors il existe $a_1,\ldots,a_r\in \CC$ tel  \[v':=v+\sum_{i=1}^r a_i (1\otimes
e_i)\in \ker(\Delta_p)^{\perp}\setminus \{0\}.\]

Pour cela, on va montrer que le syst\`eme d'\'equations lin\'eaires ci-dessous avec $(a_1,\ldots,a_r)$ comme inconnu admet une solution:
\begin{equation}\label{syslin}
\bigl(v,1\otimes e_j\bigr)_{L^2,p}=-\sum_{i=1}^r a_i \bigl(1\otimes e_i,1\otimes e_j\bigr)_{L^2,p},
\end{equation} pour tout $j=1,\ldots,r$.

Comme la matrice $A:=\bigl((1\otimes e_i,1\otimes e_j)_{L^2,p}\bigr)_{1\leq i,j\leq r}$ est inversible car $e_1,\ldots, e_r$ est une famille libre sur $\CC$: v\'erifions le, si $b=(b_1,\ldots,b_r)\in \CC^r$ tel que $A\cdot b=0$ alors ${}^t \overline{b}A b=0$, en d\'eveloppant ce dernier on obtient:
 \[\sum_{ij} \bigl(1\otimes e_i,1\otimes e_j\bigr)_{L^2,p} b_i\overline{b_j}=0.\]
Par sesquilin\'earit\'e du produit $\bigl(\cdot,\cdot\bigr)_{L^2,p}$, ce dernier terme n'est autre que
\[
\bigl(\sum_i b_i e_i,\sum_i b_i e_i\bigr)_{L^2,p}=0,
\]
donc  $\sum_i b_i e_i=0$, par suite  $A$ est inversible. On conclut que  le syst\`eme lin\'eaire  \eqref{syslin} admet une solution, on
pose  $v'=v+\sum_{i=1}^r a_i e_i$, de \eqref{syslin} on v\'erifie que $v'\in \ker(\Delta_p)^{\perp}$. Il reste \`a  montrer que $v'\neq
0$. Par l'absurde, on aura $v=-\sum_{i=1}^r a_i (1\otimes e_i)$ qui appartient \`a  $\ker(\Delta_{p})$, voir
\eqref{isosection}, mais si l'on applique $\Delta_q$ on trouve que $\la_{1,q}v=0$ ce qui contredit l'hypoth\`ese
($\Delta_{q}v=\la_{1,q}v\neq 0$ ).\\

Soit maintenant $v$ un vecteur propre non nul associ\'e \`a  $\la_{1,q}$ et on consid\`ere $v'$ comme avant, cela nous permet de dire que
\begin{align*}
\bigl(\Delta_{q}v',v'\bigr)_{L^2,q}&=\bigl(\Delta_q v,v+\sum_{i=1}^r a_i (1\otimes e_i)\bigr)_{L^2,q}\\
&=\bigl(\Delta_{q}v,v\bigr)_{L^2,q}+\bigl(\Delta_{q}v,\sum_{i=1}^r a_i (1\otimes e_i)\bigr)_{L^2,q}\\
&=\bigl(\Delta_{q}v,v\bigr)_{L^2,q}\\
&=\la_{1,q}\bigl(v,v\bigr)_{L^2,q},
\end{align*}
et
\begin{align*}
\bigl(v',v'\bigr)_{L^2,q}&=\bigl(v,v\bigr)_{L^2,q}+\bigl(v,\sum_{i=1}^r a_i \bigl(1\otimes e_i)\bigr)_{L^2,q}+\bigl(\sum_{i=1}^r a_i (1\otimes e_i), v\bigr)_{L^2,q}+\bigl(\sum_{i=1}^r a_i (1\otimes e_i\bigl),\sum_{i=1}^r a_i (1\otimes e_i)\bigr)_{L^2,q}\\
&=\bigl(v,v\bigr)_{L^2,q}+\bigl(\sum_{i=1}^r a_i (1\otimes e_i),\sum_{i=1}^r a_i (1\otimes e_i)\bigr)_{L^2,q}\\
&\geq \bigl(v,v\bigr)_{L^2,q}.
\end{align*}
 (On a utilis\'e que $\bigl(v,1\otimes e_i\bigr)_{L^2,q}=0$, cela r\'esulte facilement du fait: $\bigl(v,1\otimes e_i\bigr)_{L^2,q}=\frac{1}{\la_{1,q}}\bigl(\Delta_q v,1\otimes e_i\bigr)_{L^2,q}= \frac{1}{\la_{1,q}}\bigl(v,\Delta_{q}(1\otimes e_i)\bigr)_{L^2,q}=0 $).\\

 On a donc, pour $p\leq q$:
{\allowdisplaybreaks
\begin{align*}
\la_{1,p}\leq &\frac{\bigl(\Delta_p v',v'\bigr)_{L^2,p}}{\bigl(v',v'\bigr)_{L^2,p}}\quad \text{par}\;\eqref{spectrale}\\
\leq & \frac{\bigl(\Delta_q v',v'\bigr)_{L^2,q}}{\bigl(v',v'\bigl)_{L^2,p}}\quad \text{par}\; \eqref{ordre}\\
\leq & \frac{1}{\al}\frac{\bigl(\Delta_q v',v'\bigr)_{L^2,q}}{\bigl(v',v'\bigr)_{L^2,q}}\quad\text{par}\; \eqref{normeequivalente1}\\
\leq & \frac{1}{\al}\frac{\bigl(\Delta_q v,v\bigr)_{L^2,q}}{\bigl(v,v\bigr)_{L^2,q}+\bigl(\sum_{i=1}^r a_i 1\otimes e_i,\sum_{i=1}^r a_i 1\otimes e_i\bigr)_{L^2,q}}\\
\leq & \frac{1}{\al}\frac{\bigl(\Delta_q v,v\bigr)_{L^2,q}}{\bigl(v,v\bigr)_{L^2,q}}\\
= & \frac{1}{\al} \la_{1,q},
\end{align*}}
d'o\`u
\begin{equation}\label{part1}
\la_{1,p}\leq \frac{1}{\al} \la_{1,q}\quad \forall\, p\leq q.
\end{equation}
%§\textbf{Cette \'egalit\'e est suffisante pour nos applications, \`a  savoir montrer qu'il existe une constante $C>0$ telle que $C\leq \la_{q,1}$ $\forall q$}. En effet, on avait montr\'e que $c_{i,k}\mapsto c_{i,\infty}$ quand $k\mapsto \infty$ pour $i=1,2$, voir \eqref{limiteratio}.\\

On peut aussi montrer que
\[ \frac{\la_{1,q}}{\la_{1,p}}\leq \al\quad \forall\, p\leq q.\]

%Pour ce faire, rappelons que si $\xi=\sum_i f_i\otimes e_i$ alors $(\Delta \xi,\xi)_{L^2,h}=\int_X \sum_{ij} \frac{\pt f_i}{\pt \z} \frac{\pt \overline{f_j}}{\pt z} h(e_i,e_j)dz\wedge d\z$, voir \eqref{formesimple}.  Dans la d\'emonstration de \eqref{lacl\'e}, on
% pose $\textbf{g}:=(\frac{\pt f_1}{\pt \z},\ldots,\frac{\pt f_r}{\pt \z})$ \`a  la place de $\textbf{f}$, et on montre de la m\^eme facon que
%\begin{equation}\label{Encadrement2}
%\int_X {}^t \overline{\textbf{g}(x)}H_q(x) \textbf{g}(x)\omega =(\Delta_q\xi,\xi)_q\leq \frac{c_{2,q}}{c_{1,p}} (\Delta_p\xi,\xi)_p=\frac{c_{2,q}}{c_{1,p}}  \int_X {}^t \overline{\textbf{g}(x)}H_p(x) \textbf{g}(x)\omega ,\quad \forall p\leq q
%\end{equation}

De la m\^eme mani\`ere que pr\'ec\'edemment: Si $v$ est un vecteur propre non nul associ\'e \`a  $\la_{p,1}$ alors il existe $v':=v+\sum_i a_i (1\otimes e_i)\in \ker(\Delta_q)^\perp\setminus\{0\}$. On va avoir que
\[
\bigl(\Delta_p v',v'\bigr)_{L^2,p}=\bigl(\Delta_pv,v\bigr)_{L^2,p}=\la_{1,p}\bigl(v,v\bigr)_{L^2,p},
\]
et
\[
\bigl(v',v'\bigr)_{L^2,p}\geq \bigl(v,v\bigr)_{L^2,p}.
\]
Donc
\[
\begin{split}
\la_{1,q}&\leq \frac{\bigl(\Delta_q v',v'\bigr)_{L^2,q}}{\bigl(v',v'\bigr)_{L^2,q}}\quad  \eqref{spectrale}\\
&\leq  \frac{1}{\al} \frac{\bigl(\Delta_pv',v'\bigr)_{L^2,p}}{\bigl(v',v'\bigr)_{L^2,p}},\quad \text{de}\;\eqref{normeequivalente2}\\
&\leq  \frac{1}{\al} \frac{\bigl(\Delta_pv,v\bigr)_{L^2,p}}{\bigl(v,v\bigr)_{L^2,p}}\\
&\leq  \frac{1}{\al} \la_{1,p}.
\end{split}
\]
Par suite
\begin{equation}\label{part2}
\la_{1,q}\leq \frac{1}{\al} \la_{1,p} \quad \forall\, p\leq q.
\end{equation}

De \eqref{part1} et \eqref{part2}, on a
\[
\al\leq \frac{\la_{1,q}}{\la_{1,p}}\leq \frac{1}{\al}\quad \forall \,p\leq q.
\]
%On termine la d\'emonstration de la proposition en utilisant \eqref{limiteratio} pour montrer que $c_{i,k}\mapsto c_{i,\infty}$ quand $k\mapsto \infty$ et en remarquant que $c_{i,\infty}\neq 0$, pour $i=1,2$.
\end{proof}

%\section{DES LEMMES}

\bibliographystyle{plain}
\bibliography{biblio}
\vspace{1cm}

\begin{center}
{\sffamily \noindent National Center for Theoretical Sciences, (Taipei Office)\\
 National Taiwan University, Taipei 106, Taiwan}\\

 {e-mail}: {hajli@math.jussieu.fr}

\end{center}

\end{document}